\documentclass[11pt]{article}
\usepackage{fullpage}

\usepackage[utf8]{inputenc} 
\usepackage[T1]{fontenc}    
\usepackage{hyperref}       
\usepackage{url}            
\usepackage{booktabs}       
\usepackage{amsfonts}       
\usepackage{nicefrac}       
\usepackage{microtype}      
\usepackage{tcolorbox}
\usepackage{diagbox}
\usepackage{enumerate}
\usepackage[shortlabels]{enumitem}
\usepackage{algorithmic}
\usepackage{algorithm}
\usepackage{subfigure}
\usepackage{graphicx} 
\usepackage{caption}
\usepackage{amsmath}
\usepackage{amsthm}
\usepackage{amssymb}
\usepackage{tablefootnote}
\usepackage{multirow}
\usepackage{enumerate}
\usepackage{color}
\usepackage{xcolor}
\usepackage{natbib}

\allowdisplaybreaks[4]

\usepackage{mathrsfs}

\usepackage{hyperref}
\usepackage{bm,todonotes}

\allowdisplaybreaks

\newtheorem{thm}{Theorem}[section]

\newtheorem{lem}{Lemma}[section]
\newtheorem{cor}{Corollary}[section]
\newtheorem{prop}{Proposition}[section]
\newtheorem{assump}{Assumption}[section]
\makeatletter
\newcounter{assumptH}
\newenvironment{assumpt}[1]{%
  \begingroup
  \stepcounter{assumptH}%
  \renewcommand\theassump{#1}%
  \assump
}{%
  \endassump
  \addtocounter{assump}{-1}%
  \endgroup
}
\makeatother

\newtheorem{rema}{Remark}[section]
\newtheorem{exam}{Example}[section]

\newcommand{\inner}[2]{\langle #1, #2 \rangle}

\def\LGx{{L}_{G,x}}
\def\LGy{{L}_{G,y}}
\def\Hsmooth{{\delta_H}}
\def\Hsmoothsq{{\delta_H^2}}
\def\Fsmooth{{\delta_F}}
\def\Gsmooth{{\delta_G}}
\def\cde{c^{de}}
\def\ccde{C^{de}}
\def\nablaH{\nabla H}
\def\Hholder{{\widetilde{S}_H}}
\def\SAF{{S}_{A,F}}
\def\SAG{{S}_{A,G}}
\def\SBF{{S}_{B,F}}
\def\SBG{{S}_{B,G}}

\def\invdiffslow{\tilde{\beta}}

\def\hyz#1 {\textcolor{red}{Hyz: #1 }}
\def\lx#1 {\textcolor{magenta}{Lx: #1 }}

\hypersetup{
	colorlinks=true,
	filecolor=blue,
	citecolor = blue,
	urlcolor=cyan,
}

\usepackage{amsmath,amsfonts,bm}

\def\1{\bm{1}}

\DeclareMathAlphabet{\mathsfit}{\encodingdefault}{\sfdefault}{m}{sl}
\SetMathAlphabet{\mathsfit}{bold}{\encodingdefault}{\sfdefault}{bx}{n}

\def\tU{{\tens{U}}}

\def\0{{\bf 0}}
\def\1{{\bf 1}}

\def\CM{{\mathcal C}}
\def\DM{{\mathcal D}}
\def\EM{{\mathcal E}}

\def\FM{{\mathcal F}}

\def\NM{{\mathcal N}}
\def\OM{{\mathcal O}}

\def\TM{{\mathcal T}}

\def\ZM{{\mathcal Z}}

\def\RB{{\mathbb R}}

\def\EB{{\mathbb E}}

\def\PB{{\mathbb P}}

\def\tU{\widetilde{U}}

\def\tr{\mathrm{tr}}

\newcommand{\xhat}{{\hat{x}}}
\newcommand{\yhat}{{\hat{y}}}

\DeclareMathOperator{\sign}{sign}

\title{
	Finite-Time Decoupled Convergence in Nonlinear Two-Time-Scale Stochastic Approximation
}

\author{
Yuze Han 
\thanks{Center for Applied Statistics and School of Statistics, Renmin University of China;
email: \texttt{hanyuze97@ruc.edu.cn}.}
\and
Xiang Li 
\thanks{School of Mathematical Sciences, Peking University; email: \texttt{lx10077@pku.edu.cn}. } 
\and
Zhihua Zhang 
\thanks{School of Mathematical Sciences, Peking University; email: \texttt{zhzhang@math.pku.edu.cn}. }
}

\begin{document}

\maketitle

\begin{abstract}%
In two-time-scale stochastic approximation (SA), two iterates are updated at varying speeds using different step sizes, with each update influencing the other. 
Previous studies on linear two-time-scale SA have shown  that the convergence rates of the mean-square errors for these updates depend solely on their respective step sizes,
a phenomenon termed decoupled convergence.
However, achieving decoupled convergence in nonlinear SA remains less understood.
Our research investigates the potential for finite-time decoupled convergence
in nonlinear two-time-scale SA.
We demonstrate that, under a nested local linearity assumption, finite-time decoupled convergence rates can be achieved with suitable step size selection. 
To derive this result, we conduct a convergence analysis of the matrix cross term between the iterates and leverage fourth-order moment convergence rates to control the higher-order error terms induced by local linearity.
To further investigate the necessity of local linearity for decoupled convergence, we also construct an example showing that, even when the fast-time-scale update is linear, the nonlinearity of the slow-time-scale update alone can destroy decoupled convergence.
\end{abstract}

\section{Introduction}\label{sec:intro}

Stochastic approximation (SA), initially introduced by \citet{robbins1951stochastic}, is an iterative method for finding the root of an unknown operator based on noisy observations. This method has gained substantial attention over the past few decades, finding applications in stochastic optimization and reinforcement learning 
~\citep{kushner2003stochastic, borkar2009stochastic,
	mou2020linear,
	mou2022banach, mou2022optimal, li2021polyak, li2023online}.
 However, certain scenarios require managing two iterates updated at different time scales using varying step sizes. 
 Examples include 
stochastic bilevel optimization~\citep{ghadimi2018approximation, chen2021closing, hong2023two}, 
temporal difference learning~\citep{sutton2009fast, xu2019two, xu2021sample, wang2021non}
and actor-critic methods~\citep{borkar1997actor, konda2003onactor, wu2020finite, xu2020non}.
These cases highlight the limitations of traditional SA and underscore the need for a two-time-scale SA approach to better capture the complexities involved.

In this paper, we study the two-time-scale SA~\citep{borkar1997stochastic}, a variant of the classical SA algorithm, designed to identify the roots of systems comprising two coupled, potentially nonlinear equations.
We focus on two unknown Lipschitz operators, $F:\RB^{d_x}\times\RB^{d_y}\rightarrow\RB^{d_x}$ and $G:\RB^{d_x}\times\RB^{d_y}\rightarrow\RB^{d_y}$, with the aim of finding the root pair $(x^{\star}, y^{\star})$ satisfying:
\begin{align}
	\left\{
	\begin{aligned}
		&F(x^{\star},y^{\star}) = 0,\\
		&G(x^{\star},y^{\star}) = 0.
	\end{aligned}\right.\label{prob:FG}
\end{align}
Given that $F$ and $G$ are unknown, we assume access to a stochastic oracle that provides noisy evaluations of $F(x,y)$ and $G(x,y)$ at any input pair $(x,y)$.
Specifically, for any $x$ and $y$, the oracle returns $F(x,y)\, +\, \xi$ and $G(x,y)\, +\, \psi$ where $\xi$ and $\psi$ represent noise components.
Utilizing this stochastic oracle, we apply the nonlinear two-time-scale SA to solve the problem defined in \eqref{prob:FG} by iteratively updating estimates $x_t$ and $y_t$ for $x^{\star}$ and $y^{\star}$, respectively, as follows:
\begin{subequations} \label{alg:xy}
\begin{align}
		x_{t+1} &= x_{t} - \alpha_{t}\left(F(x_{t},y_{t}) + \xi_{t}\right), \label{alg:x} \\   
		y_{t+1} &= y_{t} - \beta_{t}\left(G(x_{t},y_{t}) + \psi_{t}\right), \label{alg:y}
\end{align}
\end{subequations}
where $x_{0}$ and $y_{0}$ are initialized arbitrarily in $\RB^{d_x}$ and $\RB^{d_y}$.
The noise components $\xi_t$ and $\psi_t$ are modeled as martingale difference sequences.
The step sizes $\alpha_{t}$ and $\beta_{t}$ are chosen such that $\beta_{t}\ll \alpha_{t}$,
making $y_t$ ``quasi-static'' relative to $x_t$.
This difference in step sizes simulates a setting where $y_t$ remains nearly fixed while $x_t$ undergoes rapid updates, as discussed in \citet{konda2004convergence}.

Assuming that for a given $y$, the equation $F(x,y)=0$ has a unique solution $H(y)$, where $H$ is a Lipschitz operator.
Under this assumption, we can regard  two-time-scale SA as a single-loop approximation of the following double-loop process:
\begin{align} \label{eq:two-loop-approx}
\begin{aligned}    
    & \text{Inner loop: compute } x = H(y)
    , \text{ or equivalently, solve } F(x, y) = 0
    \text{ for a fixed } y, \\
    & \text{Outer loop: iteratively find the root of } G(H(y), y) = 0.
    \end{aligned}
\end{align}
This formulation ensures that the solution to \eqref{prob:FG} satisfies $x^\star = H(y^\star)$; in other words, for $y=y^{\star}$, $x^{\star}$ is precisely the root of the inner loop. We designate the update of $x_t$ as the ``fast-time-scale'' update and that of $y_t$ as the ``slow-time-scale'' update, referring $x_t$ and $y_t$ as the \textit{fast iterate} and \textit{slow iterate}, respectively. 

Our primary interest lies in the behavior of the \textit{slow iterate} $y_t$, a focus we illustrate through a classic example. Consider the case of Stochastic Gradient Descent (SGD) with Polyak-Ruppert averaging, which is an instance of two-time-scale SA. 
To minimize a strongly convex objective $f$, the SGD with PR-averaging updates are defined as follows:
  \begin{subequations} \label{eq:SGD-both}
  \begin{align}
      x_{t+1} &= x_{t} - \alpha_t(\nabla f(x_{t}) + \xi_{t})~\text{with}~\alpha_t = \frac{\alpha_0}{ (t+1)^a }~(0 < a \le 1)  \label{eq:SGD}, \\
      y_{t+1} &= \frac{1}{t+1}\sum_{\tau=0}^t x_{\tau} = y_t - \beta_t(y_t-x_t)~\text{with}~\beta_t = \frac{1}{t+1},\label{eq:SGD-averaging}
  \end{align}
  \end{subequations}
 where \eqref{eq:SGD} is the classic SGD update rule~\citep{robbins1951stochastic} and \eqref{eq:SGD-averaging} is the Polyak-Ruppert averaging step~\citep{polyak1992acceleration, ruppert1988efficient}.
This setup aligns with the formulation \eqref{eq:two-loop-approx}, where we have $F(x,y) = \nabla f(x)$ and $G(x,y) = y - x$, resulting in $x^\star = y^\star = \arg\min_{x} f(x)$ and $H(y) \equiv x^\star$.
 As shown in \citet{moulines2011non}, the convergence rate for $\EB \| x_{t} - x^\star \|^2$ is $\OM(\alpha_t)$, achieving the optimal rate $\OM(1/t)$
only when the step size parameter $a=1$ and the initial step size $\alpha_0$ is sufficiently large.
In contrast, with two-time-scale SA, the slow-time-scale update satisfies $\EB \| y_t - y^\star \|^2 = \OM(1/t)$ as long as $a \ge 0.5$.
This result implies that using the slow-time-scale update~\eqref{eq:SGD-averaging} enables optimal convergence rates with greater flexibility in choosing step sizes for the fast-time-scale update in \eqref{eq:SGD}.
Since the convergence rate of $y_t$ is the desired result,
the behavior of $x_t$, which serves as an auxiliary iterate, becomes secondary.

In the above example, each iterate’s convergence rate depends only on its own step size. Similar results were achieved by \citet{kaledin2020finite} for linear operators $F$ and $G$, yielding
\begin{equation}\label{eq:target-decouple}
	\EB \| y_t - y^\star \|^2 = \OM(\beta_t)\qquad\text{and}\qquad\EB \| x_t - H(y_t) \|^2 = \OM(\alpha_t).
\end{equation}
Here $y_t - y^\star$ and $x_t - H(y_t)$ represent the errors of the outer and inner loops in \eqref{eq:two-loop-approx}, respectively.
We refer to the phenomenon where each iterate’s convergence rate depends solely on its own step size as \textit{decoupled convergence}.

However, when $F$ and $G$ are nonlinear operators, the interactions between the two iterates become more complex, making the path to \eqref{eq:target-decouple} less straightforward. The asymptotic analysis in \citet{mokkadem2006convergence}, under an additional local linearity assumption, provides evidence supporting decoupled convergence for sufficiently large $t$. Nevertheless, there is no corresponding non-asymptotic guarantee for any finite $t$.
Consequently, our research goal is to
\begin{center}
	Establish the \textit{finite-time decoupled convergence} in \eqref{eq:target-decouple} for nonlinear two-time-scale SA.
\end{center}

The significance of decoupled convergence rates is twofold. First, as analyzed for the SGD example in \eqref{eq:SGD-both}, decoupled convergence allows greater flexibility in selecting step sizes for the fast iterate without affecting the convergence behavior of the main focus, the slow iterate. Additionally, decoupled convergence rates offer a more refined analysis than previous work \citep{doan2022nonlinear, shen2022single}, especially in strict two-time-scale scenarios. Such refined rates are valuable for further asymptotic trajectory analysis \citep{liang2023asymptotic} and online statistical inference \citep{li2022statistical, li2023online}.
Moreover, we emphasize that our main objective is not merely to show that the rate $\EB \| y_t - y^\star \|^2 = \Theta(t^{-1})$ can be attained under some particular choice of step sizes. Rather, our goal is to show that, over a broad regime of step-size choices, the convergence rate of the slow iterate can be made essentially \emph{independent} of the step size on the fast time scale.

\paragraph{Contributions}
In this paper, we focus on nonlinear two-time-scale SA under the assumptions of strong monotonicity and martingale difference noise, as specified in Assumptions~\ref{assump:sm:F} and \ref{assump:noise-s}. 
Our contributions are summarized as follows.

\begin{itemize}
    \item \textbf{Theoretical contribution}: 
    Under the nested local linearity assumption (Assumption~\ref{assump:near-linear}), we derive detailed convergence rates for $\EB \| x_t - H(y_t) \|^2$, $\EB \| y_t - y^\star \|^2$ and $\| \EB (x_t - H(y_t) )( y_t - y^\star )^\top \|$ in Theorem~\ref{thm:decouple-short}, establishing finite-time decoupled convergence in Corollary~\ref{cor:decouple-rates} with appropriate  step size selection.
    We further investigate the necessity of local linearity for decoupled convergence. 
    In particular, we construct an example and prove in Proposition~\ref{prop:lower} that, even when $F$ and $H$ are linear, the nonlinearity of $G$ alone can already slow down the convergence rate of the slow iterate. 
    Taken together, these upper and lower bounds provide a relatively complete characterization of when decoupled convergence can be achieved in the nonlinear setting.
    Moreover, this lower bound also complements the approximation perspective in \eqref{eq:two-loop-approx}: although two-time-scale SA can be viewed as solving $F(x,y)=0$ (for fixed $y$) and $G(H(y),y)=0$, the original form of $G(x,y)$ may still influence the convergence rates.
    
    \item \textbf{Technical contribution}: 
    We develop a systematic proof framework for establishing decoupled convergence in the nonlinear setting. A key ingredient is the treatment of the matrix cross term $\| \mathbb{E}(x_t - H(y_t))(y_t - y^*)^\top \|$, which is crucial for obtaining a sharp convergence characterization of the interacting sequences $\{x_t\}_{t=0}^\infty$ and $\{y_t\}_{t=0}^\infty$. While the use of such a cross term is related to the linear-case analysis in \citet{kaledin2020finite}, the nonlinear setting requires several additional ingredients: an initial coarse convergence-rate analysis, in the spirit of \citet{doan2022nonlinear}; local linear approximations of $F$, $G$, and $H$; control of the resulting higher-order error terms; a convergence analysis of fourth-order moments; and a final integration of these ingredients to derive the decoupled rates. This proof framework could provide a useful foundation for future finite-time analyses of nonlinear interacting stochastic approximation schemes with multiple time scales and variable step sizes.
\end{itemize}

\subsection{Related Work}
\label{sec:intro:related}

Our research investigates the finite-time convergence rate of nonlinear two-time-scale SA and endeavors to establish the decoupled convergence.
To contextualize our results, we provide more
background of two-time-scale SA.

\paragraph{Decoupled convergence for the linear case}
When both $F$ and $G$ are linear, $H$ is also linear.
Leveraging this linear structure, \citet{konda2004convergence} focused on a linearly transformed error $z_t = x_t - H(y_t) + L_t(y_t - y^\star)$, demonstrating that its update does not depend on $y_t$. Here $\{ L_t \}$ is a matrix sequence converging to zero.
Based on this insight, they prove that $\beta_t^{-1/2} (y_t - y^\star)$ converges in distribution to a normal distribution under martingale difference noise.
It can be extrapolated from their analysis that the $\alpha_t^{-1/2} (x_t - x^\star)$ and $\alpha_t^{-1/2} (x_t - H(y_t))$ also converge in distribution to a normal distribution.
Drawing inspiration from this technique, \citet{kaledin2020finite} derived finite-time convergence rates, establishing that $\EB \| y_t - y^\star \|^2 = \OM (\beta_t)$ and $\EB \| x_t - H(y_t) \|^2 = \OM (\alpha_t)$ hold for both martingale difference and Markovian noise.
Concurrent to our work, \citet{haque2023tight} achieved the same rates with asymptotically optimal leading terms;
\citet{kwon2024two} examined constant stepsize schemes and provided a refined characterization of the bias and variance terms.
For a variant of two-time-scale SA with sparse projection,
\citet{dalal2020tale} derived that $\| y_t - y^\star \| = \widetilde{\OM} (\beta_t)$ and $\| x_t - x^\star \| = \widetilde{\OM} (\alpha_t)$ hold with high probability.

\paragraph{Decoupled convergence for the nonlinear case}
When either $F$ or $G$ is nonlinear, the coupling between $x_t$ and $y_t$ is in a more complex nonlinear form, complicating the analysis of nonlinear two-time-scale SA.
Decoupled convergence analysis remains largely in the domain of asymptotic results.
Under local linearity condition of $F$ and $G$ around $(x^\star, y^\star)$ and assuming stability, \citet{mokkadem2006convergence} demonstrated that 
$\begin{pmatrix}
	\alpha_t^{-1/2} (x_t - x^\star) \\
	\beta_t^{-1/2} (y_t - y^\star )
\end{pmatrix}$
converges weakly to a normal distribution.
Recently, this central limit theorem has been extended to settings with Markovian noise by \citet{hu2024central} and to continuous-time dynamics by \citet{sharrock2022two}.
To the best of our knowledge, our work provides the first finite-time (non-asymptotic) decoupled convergence rate for the nonlinear case.

\paragraph{Other convergence rates for the nonlinear case}
Under the strongly monotone and Lipschitz conditions, \citet{doan2022nonlinear} and \citet{doan2021finite2} achieved the $\OM(t^{-2/3})$ and $\widetilde{\OM}(t^{-2/3})$ convergence rates for the slow iterate with martingale difference noise and Markovian noise, respectively.
Further assuming the Lipschitz continuity of $\nabla H$, \citet{shen2022single} improved this rate to $\OM(1/t)$.
\citet{huang2024single} and \citet{chandak2025finite} obtained the same rate without this additional condition by introducing auxiliary sequences.
An alternative approach to achieve the  $\OM(1/t)$ rate without relying on the Lipschitz condition of $\nabla H$ is to leverage an averaging step to improve the
estimates of the operators \citep{zeng2024fast, doan2024fast}.
Moreover,
\citet{chen2025convergence} studied the convergence rates for state- and  time-dependent noises; concentration bounds have been developed in~\citet{borkar2018concentration} and
a functional central limit theorem is provided in~\citet{faizal2023functional}.
In the absence of the strong monotonicity in the outer loop of~\eqref{eq:two-loop-approx}, two-time-scale SA has been studied within the framework of bilevel optimization \citep{hong2023two, zeng2021two}.

\paragraph{Notation}
For a vector $x$, $\| x \|$ denotes the Euclidean norm;
for a matrix $A$, $\| A \|$ to denote the spectral norm.
We use $o(\cdot)$, $\OM(\cdot)$, $\Omega(\cdot)$, and $\Theta(\cdot)$ to hide universal constants and $\widetilde{\OM}(\cdot)$ to hide both universal constants and log factors.
We denote $\max\{a,b\}$ as $a \vee b$ and $\min\{a, b \}$ as $a \wedge b$.
The ceiling function $\lceil \cdot \rceil$ denotes the smallest integer greater than or equal to the input number.
For two non-negative numbers $a$ and $b$, $a \lesssim b$ ($a \propto b$) indicates the existence of a positive number $C$ such that $a \le Cb$ ($a = Cb$) with $C$ depending on parameters of no interest.
For two positive sequence $\{ a_n \}$ and $\{ b_n \}$, $a_n \sim b_n$ signifies $\lim_{n \to \infty} a_n / b_n = 1$.
By $\overset{a.s.}{\rightarrow}$ we denote almost sure convergence;
by $\overset{d}{\rightarrow}$ we denote the convergence in distribution.
We abbreviate $\{ 1,2,\dots, n \}$ as $[n]$.

\paragraph{Organization}
The remainder of this paper is organized as follows.
Section~\ref{sec:prelim} introduces several motivating examples and the basic assumptions.
The main theoretical results are presented in Section~\ref{sec:decouple}, with the proof framework outlined in Section~\ref{sec:decouple:sketch}.
Section~\ref{sec:expe} illustrates some numerical results, and Section~\ref{sec:conclude} concludes the paper.
Detailed proofs are provided in the appendices.

\section{Preliminaries}\label{sec:prelim}
In this section, we present some motivating examples and the main assumptions.

\subsection{Motivating Examples}\label{sec:prelim:app}

In this subsection, we present several examples of two-time-scale SA.
In most of these examples, our primary focus is on evaluating the performance of the slow iterate $y_t$, with the fast iterate $x_t$ playing a secondary role as an auxiliary sequence.

In the next three examples, we assume $f(\cdot) \colon \RB^{d_x} \to \RB$ is a strongly convex function and the unique minimizer is $x_o^{\star} = \arg\min_{x \in \RB^{d_x}} f(x)$.
\begin{exam}[SGD with Polyak-Ruppert averaging]\label{exmp:PRave}
	SGD with Polyak-Ruppert averaging has been introduced in Section~\ref{sec:intro}.
	Suppose that we want to minimize a function $f$ 
	with access only to the noisy observations of the true gradients.
	To find the true minimizer,
	the stochastic gradient method (SGD)~\citep{robbins1951stochastic} iteratively updates the iterate $x_t$ by~\eqref{eq:SGD}.
	In order to improve the convergence of SGD, an additional averaging step~\eqref{eq:SGD-averaging} is often used~\citep{polyak1992acceleration,ruppert1988efficient}
	Obviously, these two updates are a special case 
	of the nonlinear two-time-scale SA in \eqref{alg:xy} with $F(x, y) = \nabla f(x), G(x, y) = y -x$ and $H(y) \equiv x^{\star}$.
	It follows that $G(H(y), y) = y - x^\star$ and $y^{\star} = x^{\star} = x_o^{\star}$.

	Now we contextualize this example within the framework of \eqref{eq:two-loop-approx}.
	In the inner loop, our goal is to reduce the norm of $\nabla f(x)$, while in the outer loop, we strive to approach $x_o^\star$.
	For a strongly convex objective, the two objectives coincide.
	\citet{moulines2011non} have demonstrated that $\EB \| x_t - x^\star \|^2 = \OM(\alpha_t)$ and $\EB\| y_t - y^\star \|^2 = \OM (1/t)$ as long as $a \ge 1/2$.
	Given the improved convergence rate of the averaged sequence, our primary interest lies in the slow iterate $y_t$.
\end{exam}

\begin{exam}[SGD with momentum]\label{exmp:momentum}
	Stochastic heavy ball (SHB) is a variant of SGD based on momentum and adaptive step sizes and has been shown to be effective~\citep{gadat2018stochastic}.
	A ``normalized''  version of SHB~\citep{gupal1972stochastic,gitman2019understanding} iteratively runs
	\begin{equation}\label{eq:SHB}
		\begin{aligned}
			x_{t+1} &= x_{t} - \alpha_t (x_t - \nabla f(y_t) + \xi_t),\\
			y_{t+1} &= y_{t} - \beta_t  x_{t}.
		\end{aligned}
	\end{equation}
	Here one should interpret $x_{t}$ as a (stochastic) search direction that is defined to be a combination of the current stochastic gradient $\nabla f(y_t) + \xi_t$
	and past search direction $x_t$.\footnote{Here we employ a slightly different update rule for $y_t$ to be consistent with~\citet{doan2022nonlinear}. }
	 These two updates are a special case of the nonlinear two-time-scale SA in \eqref{alg:xy} with $F(x, y) = x-\nabla f(y), G(x, y) = x$ and $H(y) = \nabla f(y)$.
	It follows that $G(H(y), y) = \nabla f(y)$, $y^{\star} = x_o^{\star}$ and $x^{\star} = 0$.

	Now, we integrate this example into the framework of \eqref{eq:two-loop-approx}.
	In the inner loop, our objective is to approximate the gradient of a fixed $y$, while 
	in the outer loop, we aim to locate the stationary point (also the minimizer) $x_o^\star$, 
	Thus, our primary focus lies in the slow iterate $y_t$.
\end{exam}

\begin{exam}[Constrained optimization with Lagrangian multipliers]\label{exmp:lagrange}
	Consider the linearly constrained optimization $\min_{x \in \RB^{d_x}} f(x), \ \mathrm{s.t.}\ Ax = b$.
	This problem can be solved by introducing Lagrange function $L(x,y) = f(x) + y^\top (Ax - b)$ and applying the following primal-dual method~\citep{platt1987neural}
	\begin{equation}\label{eq:lagrange}
	\begin{aligned}
		x_{t+1} & = x_t - \alpha_t (\nabla f(x_t) + A^\top y_t + \xi_t ), \\
		y_{t+1} & = y_t + \beta_t (Ax_t - b),
	\end{aligned}
	\end{equation}
	where $y_t$ denotes the Lagrange multiplier. 
	This update scheme can be viewed as an approximation to dual ascent~\citep{boyd2011distributed} and is a special case of the nonlinear two-time-scale SA in \eqref{alg:xy} with $F(x, y) = \nabla f(x) + A^\top y, G(x, y) = -Ax+b$ and $H(y) = [\nabla f]^{-1} (- A^\top y)$.
	The equations $F(x,y) = 0$ and $G(x,y) = 0$ correspond precisely to the KKT conditions. Therefore, $x^\star$ and $y^\star$, satisfying $x^\star = H(y^\star)$, are the solutions to the primal and dual problems, respectively.
	
	We now integrate this example into the framework of \eqref{eq:two-loop-approx}. In the inner loop, for a fixed $y$, the goal is to solve  $\min_{x \in \RB^{d_y}} L(x,y)$, where $H(y)$ specifies the corresponding solution. In the outer loop, the aim is to solve the dual problem  $\max_{y \in \RB^{d_y}} \min_{x \in \RB^{d_x}} L(x,y)$.
\end{exam}

The final example is bilevel optimization, a topic with a long history in the optimization literature~\citep{bracken1973mathematical,colson2007overview}. 
With some abuse of notation, we  examine the following (unconstrained) bilevel optimization problem:
\begin{align}\label{eq:prob:bilevel}
	\min_{y \in \RB^{d_y} } \ell(y) := g(\tilde{x}^\star(y), y)
	\quad \text{s.t.} \quad 
	\tilde{x}^\star(y) := \arg\min_{x \in \RB^{d_x}} f(x,y),
\end{align}
where $f(x,y)$ is the \textit{inner objective function} and $\ell(y)$ is the \textit{outer objective function}.
We refer to $\min_{x \in \RB^{d_x}} f(x,y)$ as the \textit{inner problem} and $\min_{y \in \RB^{d_y}} \ell (y)$ as the \textit{outer problem}.
Under some regularization conditions~\citep{ghadimi2018approximation, chen2021closing, shen2022single},
we have 
$$ \nabla \ell(y) = \nabla_y g(\tilde{x}^\star(y), y)
- \nabla_{yx}^2 f(\tilde{x}^\star(y), y)  [ \nabla_{xx}^2 f(\tilde{x}^\star(y), y) ]^{-1} \nabla_x g(\tilde{x}^\star(y), y)$$ and the solution of the inner problem is unique for any $y$. Thus $\tilde{x}^\star(y)$ is well-defined.

\begin{exam}[Stochastic bilevel optimization]\label{exmp:bilevel}
	Suppose that we want to solve the problem~\eqref{eq:prob:bilevel} with access only to the noisy observations of the true gradients and Hessians.
	To employ two-time-scale SA~\eqref{alg:xy}, we first give the definitions of $F(x,y)$ and $G(x,y)$:
	\begin{equation}\label{eq:bilevel-FG}
		\begin{aligned}
			F(x,y) = \nabla_x f(x, y), \ 
			G(x,y)  = \nabla_y g(x, y) - \nabla_{yx}^2 f(x,y) [ \nabla_{xx}^2 f(x,y) ]^{-1} \nabla_x g(x,y),
		\end{aligned}
	\end{equation}
	where
	$G(x,y)$ is a surrogate of $\nabla \ell (y)$ by replacing $\tilde{x}^\star(y)$ with $x$.
	It follows that $H(y) = \tilde{x}^\star (y)$ and $G(H(y), y) = \nabla \ell (y)$.
	One can apply the following approximation method~\citep{shen2022single}:
	\begin{equation}\label{eq:bilevel-alg}
		\begin{aligned}
			x_{t+1} &= x_{t} - \alpha_{t} h_F^t = x_{t} - \alpha_{t}\left(F(x_{t},y_{t}) + \xi_{t}\right),\\   
			y_{t+1} &= y_{t} - \beta_{t} h_G^t = y_{t} - \beta_{t}\left(G(x_{t},y_{t}) + \psi_{t}\right),
		\end{aligned}
	\end{equation}
	where $h_F^t$ and $h_G^t$ are estimations of $F(x_t, y_t)$ and $G(x_t, y_t)$, respectively.
	For their explicit forms, please refer to \citet{hong2023two}.
	Given the two-level structure inherent in bilevel optimization, it is convenient to contextualize this example within the framework of \eqref{eq:two-loop-approx}, with the outer problem being the ultimate goal. 
\end{exam}

\subsection{Assumptions}
\label{sec:prelim:assump}
In this subsection, we present the main assumptions required for our analysis.

\begin{assump}[Lipschitz conditions of $F$ and $H$]
	\label{assump:smooth:FH}
For any fixed $y\in\RB^{d_y}$, there exists an operator $H:\RB^{d_y}\rightarrow\RB^{d_x}$ such that $x = H(y)$ is the unique solution of
$F(x,y) = 0$.    
$H$ and $F$ satisfy that for
$\forall\, x \in \RB^{d_x}, y_{1}, y_{2} \in\RB^{d_y}$,
\begin{align}
\|H(y_{1}) - H(y_{2})\| & \leq L_{H}\|y_{1}-y_{2}\|,\label{assump:smooth:FH:ineqH}\\
\|F(x,y_1) - F(H(y_1),y_1)\| & \leq L_{F} \|x - H(y_1) \|.     \label{assump:smooth:FH:ineqF}
\end{align}
\end{assump}

Condition~\eqref{assump:smooth:FH:ineqF} introduces a star-type Lipschitz condition, which is less restrictive and allows for greater flexibility in modeling and analysis.

\begin{assump}[Nested Lipschitz condition of $G$]
	\label{assump:G}
The operator $G$ satisfies that
$\forall\, x \in \RB^{d_x}, y \in\RB^{d_y}$, 
\begin{align}
    \hspace{-0.3cm}
    \|G(x,y) - G(H(y),y)\| & \leq \LGx \|x - H(y)\|, \label{assump:G:smooth} \\   
    \| G(H(y),y) - G(H(y^\star), y^\star) \| & \leq \LGy \| y - y^\star \|. \label{assump:G:smooth2} 
\end{align}
Here $y^\star$ is the unique solution to $G(H(y),y) = 0$.
\end{assump}

The nested structure within Assumption~\ref{assump:G} 
aligns with the conceptual framework that views a two-time-scale SA as an approximation of the formulation presented in \eqref{eq:two-loop-approx}.
This type of nested Lipschitz condition is also adopted in \citet{shen2022single}.

\begin{assump}[Star-type strong monotonicity]
\label{assump:sm:F}
For any fixed $y \in \RB^{d_y}$,
$F(\cdot, y)$ is strongly monotone at $H(y)$, i.e., there exists a constant $\mu_{F} > 0$ 
\begin{align}
\left\langle x - H(y), F(x,y) - F(H(y),y) \right\rangle \geq \mu_{F} \|x-H(y)\|^2, \quad \forall\, x \in \RB^{d_x}. \label{assump:sm:F:ineq}    
\end{align}
$G(H(\cdot), \cdot)$ is strongly monotone at $y^{\star}$, i.e., there exists a constant $\mu_{G} > 0$ such that
\begin{align}
\left\langle y - y^{\star}, G(H(y),y) - G(H(y^\star), y^\star) \right\rangle \geq \mu_{G} \|y - y^{\star}\|^2, \quad \forall\, y\in\RB^{d_y}. \label{assump:G:sm}    
\end{align}
\end{assump}

The star-type strong monotonicity assumption of $F$ and $G$ also appears in previous work \citep{doan2022nonlinear, shen2022single}.
A direct consequence of the above assumptions is
$L_F \ge \mu_F$ and
$\LGy \ge \mu_G$.

\begin{assump}[Uniform local linearity of $H$]
	\label{assump:smoothH}
	Assume that $H$ is differentiable and there exist	constants $S_H \ge 0$ and $\Hsmooth \in [0.5, 1]$ such that $\forall\, y_1, y_2 \in \RB^{d_y}$,
	\begin{equation*}
		\| H(y_1) -  H(y_2) - \nablaH(y_2) (y_1-y_2)  \| \le  S_H  \|y_1-y_2\|^{1+\Hsmooth}.
	\end{equation*}
\end{assump}

An assumption closely related to Assumption~\ref{assump:smoothH} is the H\"older continuity of $\nabla H$.

\begin{assumpt}{\ref*{assump:smoothH}$\dagger$}[$\Hsmooth$-H\"older continuity of $\nabla H$]
	\label{assump:H:holder-assump}
	We assume that $H$ is differentiable and there exists constants $\Hholder \ge 0$ and $\Hsmooth \in [0.5,1]$ such that $\forall \, y_1, y_2 \in \RB^{d_y}$, 
	\begin{equation}
		\label{assump:H:holder}
		\| \nabla H(y_1) - \nabla H(y_2) \|
		\le \Hholder \| y_1 - y_2 \|^{\Hsmooth}.
	\end{equation}
\end{assumpt}

Assumption~\ref{assump:H:holder-assump} relaxes the requirement of Lipschitz continuity for $\nabla H$ used in \citet{shen2022single}.
Assumptions~\ref{assump:smoothH} and \ref{assump:H:holder-assump}
are equivalent, as shown in Proposition~\ref{prop:holder-equiv}. 
Therefore, in this paper, we do not distinguish between these two assumptions.
Proposition~\ref{prop:holder-equiv} extends \citet[Theorem~4.1, Euclidean case]{berger2020quality} from scalar-valued functions to vector-valued functions.
The proof is provided in Appendix~\ref{proof:holder-equiv}.

\begin{prop}\label{prop:holder-equiv}
	Under Assumption~\ref{assump:H:holder-assump},
	Assumption~\ref{assump:smoothH} holds with $S_H = \frac{\Hholder}{1+\Hsmooth} $;
	under Assumption~\ref{assump:smoothH}, Assumption~\ref{assump:H:holder-assump} holds with $\Hholder = 2^{1-\Hsmooth} \sqrt{1+\Hsmooth} \left( \frac{1+\Hsmooth}{\Hsmooth} \right)^\frac{\Hsmooth}{2} S_H$.\footnote{We make the contention that $\left( \frac{1+\Hsmooth}{\Hsmooth} \right)^\frac{\Hsmooth}{2} = 1$ if $\Hsmooth = 0$. }
	For this equivalence, we do not require $\Hsmooth \ge 0.5$.
\end{prop}

A direct conclusion of Assumptions~\ref{assump:smoothH} and~\ref{assump:smooth:FH} is that $\forall\, y_1, y_2 \in \RB^{d_y}$, with $R_H = \frac{2L_H}{S_H}$, $\| \nabla H(y_1) \| \le L_H$ and
\begin{equation} \label{assump:H:smooth-main}
	\| H(y_1) -  H(y_2) - \nablaH(y_2) (y_1-y_2)  \| \le S_H \|y_1-y_2\| \cdot \min\left\{\|y_1-y_2\|^{\Hsmooth}, R_H \right\}.
\end{equation}

Although a finite-time decoupled convergence rate for nonlinear two-time-scale SA is not established in the literature, asymptotic decoupled convergence has been explored under a local linearity assumption~\citep{mokkadem2006convergence}. Inspired by this work, we consider the following nested local linearity assumption.

\begin{assump}	[Nested local linearity up to order ($1+\Fsmooth, 1+\Gsmooth$)]
	\label{assump:near-linear}
	There exist matrices $B_1, B_2, B_3$ with compatible dimensions, constants $\SBF, \SBG \ge 0$ and $\Fsmooth, \Gsmooth \in (0,1]$ such that
	\begin{align}
		\|F(x, y) - B_1(x - H(y))\| &\le \SBF \left( \|x - H(y)\|^{1+\Fsmooth} +  \|y-y^{\star}\|^{1+\Fsmooth} \right),
		 \label{assump:near-linear-F}
		\\
		\|G(x, y) - B_2(x - H(y))- B_3(y-y^{\star}) \| &\le \SBG \left( \|x - H(y)\|^{1+\Gsmooth} + \|y-y^{\star}\|^{1+\Gsmooth} \right). 
		 \label{assump:near-linear-G}
	\end{align}
\end{assump}

This assumption follows the spirit that two-time-scale SA can be viewed as an approximation of the two-loop procedure~\eqref{eq:two-loop-approx}.
The parameters $\Fsmooth$ and $\Gsmooth$ quantify the order of errors in this linear approximation condition.
To demonstrate how Assumption~\ref{assump:near-linear} can be guaranteed, we introduce the following standard local linearity assumption.
\begin{assumpt}{\ref*{assump:near-linear}$\dagger$}[Standard local linearity up to order ($1+\Fsmooth, 1+\Gsmooth$)]
	\label{assump:local-linear}
	There exists matrices $A_{11}, A_{12}, A_{21}, A_{22}$ with compatible dimensions, constants $S_{A,F}, S_{A,G} \ge 0$ and $\Fsmooth, \Gsmooth \in (0,1]$ such that
	\begin{align}
		\|F(x, y) - A_{11}(x -x^{\star}) - A_{12} (y-y^{\star}) \| &\le S_{A,F} \left( \|x - x^{\star}\|^{1+\Fsmooth} + \|y-y^{\star}\|^{1+\Fsmooth} \right), \label{assump:local-linear-F} \\
		\|G(x, y) - A_{21}(x -x^{\star}) - A_{22}(y-y^{\star} )\| &\le S_{A,G} \left( \|x - x^{\star}\|^{1+\Gsmooth} + \|y-y^{\star}\|^{1+\Gsmooth} \right). \label{assump:local-linear-G}
	\end{align}
\end{assumpt}
Assumption~\ref{assump:local-linear} implies
both $F$ and $G$ are differentiable at $(x^{\star}, y^{\star})$, leading to 
$A_{11} = \nabla_x F(x^{\star}, y^{\star})$, 
$A_{12} = \nabla_y F(x^{\star}, y^{\star})$, 
$A_{21} = \nabla_x G(x^{\star}, y^{\star})$,  
$A_{22} = \nabla_y G(x^{\star}, y^{\star})$.
Notably, Assumption~\ref{assump:local-linear} with $\Fsmooth = \Gsmooth = 1$ is the local linearity assumption in \citet{mokkadem2006convergence}. 
The following
proposition provides some properties of the existing assumptions.
In particular,
it shows that
if the operators $F$ and $G$ can be locally approximated around the root $(x^{\star}, y^{\star})$ by linear functions of $x-x^{\star}$ and $y-y^{\star}$, up to a higher-order error, then the nested local linearity condition follows naturally—assuming certain prior assumptions are satisfied.
The proof is provided in Appendix~\ref{proof:de:ensure-linearity}.

\begin{prop}
	\label{prop:ensure-linearity}
	Suppose that Assumptions~\ref{assump:smooth:FH} --~\ref{assump:smoothH} hold. 
	\begin{enumerate}[(i)]
		\item \label{prop:ensure-linearity-1} If Assumption~\ref{assump:local-linear} holds, then $\| A_{11} \| \le L_F$, $ \| A_{21} \| \le \LGx $, $\frac{A_{11}+A_{11}^\top}{2} \succeq \mu_F I$, and $\nabla H(y^\star) =  - A_{11}^{-1} A_{12}$.
		If we further assume $\Hsmooth \ge \Fsmooth \vee \Gsmooth$, 	then Assumption~\ref{assump:near-linear} holds with parameters
		\begin{align*}
			B_1 & = A_{11},\ \ B_2 = A_{21},\ \ B_3 = A_{22}-A_{21}A_{11}^{-1}A_{12},\\
			S_{B,F} &= S_{A,F} + L_F (S_H \vee 2 L_H), \ 
			S_{B,G} = S_{A,G} + \LGx (S_H \vee 2 L_H).
		\end{align*}
		\item \label{prop:ensure-linearity-2}  If Assumption~\ref{assump:near-linear} holds, then $\|B_1\| \le L_F, \| B_2\| \le \LGx , \|B_3\|\le \LGy,  \frac{B_1 + B_1^\top}{2} \succeq \mu_F I$, and $\frac{B_3 + B_3^\top}{2} \succeq \mu_G I$.
	\end{enumerate}
\end{prop}

The inclusion of the condition $\Hsmooth \ge \Fsmooth \vee \Gsmooth$ in Proposition~\ref{prop:ensure-linearity}~\ref{prop:ensure-linearity-1} stems from the nested structure in Assumption~\ref{assump:near-linear}.
This structure 
is introduced to capture the nested nature of the two-loop procedure in \eqref{eq:two-loop-approx}, where $H$ is a crucial link between the inner and outer loops. 
Transforming Assumption~\ref{assump:local-linear} into a nested form necessitates a more refined smoothness condition on $H$ to preserve local linearity.
Additionally, 
$B_3$ equals the Schur complement of $A_{11}$ in the matrix $(A_{11}, A_{12}; A_{21}, A_{22})$.

Moreover, if we allow $H$ to be non-smooth or non-differentiable, Assumption~\ref{assump:near-linear} can cover a broader class of problems in which the nonlinearity is ``absorbed'' into the solution map $H$, whereas Assumption~\ref{assump:local-linear} cannot.
Typical examples include cases where $H$ is a projection-type map (e.g., under simplex or box constraints) or a proximal map (e.g., soft-thresholding for $\ell_1$ regularization).
Although these scenarios are beyond the scope of this paper, we believe that our nested local linearity condition is more consistent with the perspective in \eqref{eq:two-loop-approx} and may inspire future research.

For the examples in Section~\ref{sec:prelim:app}, Assumptions~\ref{assump:smooth:FH} --\ref{assump:near-linear} can be verified under standard regularity conditions, such as the strong convexity of the objective function and the Lipschitz continuity of the Hessian.
Details on the verification of the local linearity conditions are deferred to Appendix~\ref{sec:verify-assump}.

Next, we focus on noise models.
We denote by $\FM_{t}$ the filtration containing all the history generated by \eqref{alg:xy} before time $t$, i.e.,
	$\FM_{t} = \sigma\{x_{0},y_{0},\xi_{0},\psi_{0},\xi_{1},\psi_{1},\ldots,\xi_{t-1},\psi_{t-1}\}$.

\begin{assump}[Conditions on the noise]
	\label{assump:noise-s}

	The sequences of random variables $\{ \xi_{t} \}_{t=0}^\infty$ and $\{ \psi_{t} \}_{t=0}^\infty$ are martingale difference sequences satisfying
	\begin{align}
		\label{assump:variance-s}
		\begin{split}
			\EB [\xi_t \,|\, \FM_t ]  = 0, &\quad 
			\EB [\psi_t \,|\, \FM_t] = 0, \quad 		\EB [\|\xi_{t}\|^4|\FM_t] \le \Gamma_{11}^2,\quad
			\EB [\|\psi_{t}\|^4|\FM_t] \le \Gamma_{22}^2, \\
			\|\EB[\xi_{t}\xi_{t}^\top \,|\,\FM_{t}] \| &\le \Sigma_{11}, \quad
			\|\EB[  \psi_{t} \psi_{t}^\top \,|\,\FM_{t}]\| \le \Sigma_{22},\quad
			\| \EB [\xi_{t} \psi_{t}^\top \,|\, \FM_{t}]\| \le \Sigma_{12}.
		\end{split}
	\end{align}
\end{assump}
In Assumption \ref{assump:noise-s}, the boundedness of the fourth-order moments is imposed to control the higher-order error terms introduced by Assumption~\ref{assump:near-linear}.

The final assumption is the requirements for the step sizes.
\begin{assump}[Conditions on step sizes]
	\label{assump:stepsize-new}
	With the constants $\mu_F$, $\mu_G$, $\Fsmooth$ and $\Gsmooth$ defined in Assumptions~\ref{assump:sm:F} and \ref{assump:near-linear},
	the step sizes $\{ \alpha_{t} \}_{t=0}^\infty$ and $ \{ \beta_{t} \}_{t=0}^\infty $ satisfy the following conditions:
	\begin{enumerate}[(i)]
		\item Constant bounds:\label{assump:stepsize-new:constant} $\alpha_{t} \le \iota_1,\beta_t \le \iota_2,\frac{\beta_t}{\alpha_{t}} \le \kappa, \frac{\beta_t^2}{\alpha_{t}} \le \rho$, where $\iota_1$, $\iota_2$, $\kappa$ and $\rho$ are problem-dependent constants with specific forms defined in \eqref{assump:stepsize:constants} in Appendix~\ref{sec:assump-append}.
		
		\item Growth conditions:\label{assump:stepsize-new:growth} $1 \le \frac{\alpha_{t-1}}{\alpha_{t}} \le 1 + (\frac{ \Fsmooth \mu_F}{16}\alpha_t) \wedge (  \frac{ \Fsmooth \mu_G}{16}\beta_t) \wedge ( \frac{\Gsmooth \mu_G}{8} \beta_t )$ and $1 \le \frac{\beta_{t-1}}{\beta_{t}} \le 1 + \frac{\mu_G}{64}\beta_t$ for any $t \ge 1$.
		
		\item \label{assump:stepsize-new:other} $\frac{\beta_t}{\alpha_t}$ is non-increasing in $t$, 
		and 
		$\prod_{\tau=0}^t\left(1-\frac{\mu_{G}\beta_{\tau}}{4} \right) = \OM (\alpha_t^2) $.
	\end{enumerate}
\end{assump}
Conditions~\ref{assump:stepsize-new:constant} and \ref{assump:stepsize-new:growth} are similar to \citet[Assumption A2]{kaledin2020finite}, though our conditions are more intricate to handle non-linearity.
Condition~\ref{assump:stepsize-new:other} is a technical requirement introduced to simplify the proof.
Notably, these conditions are naturally satisfied as long as both $\alpha_{t}$ and $\beta_t$ decrease with $t$, with $\beta_t$
decreasing faster than $\alpha_t$.
Additionally, our setup includes single-time-scale SA as a special case if a constant ratio $\beta_t / \alpha_t \equiv \kappa' \le \kappa$ that satisfies Assumption~\ref{assump:stepsize-new} is adopted.

\begin{rema}[Discussion on Assumption~\ref{assump:stepsize-new}]\label{rema:step-size}
We make the following remarks on Assumption~\ref{assump:stepsize-new}.
\begin{itemize}
    \item The growth conditions imply that $\alpha_t$ and $\beta_t$ are non-increasing.
Moreover, $\alpha_t^{-1} \le \alpha_{t-1}^{-1} + \frac{\delta_F \mu_F}{16} \alpha_t / \alpha_{t-1} \le  \alpha_{t-1}^{-1} + \frac{\delta_F \mu_F}{16} =\OM(t) $ and hence $\alpha_t = \Omega(t^{-1})$.
Similarly, $\beta_t = \OM(t^{-1})$.
Consequently, $ \sum_{t=0}^\infty \alpha_t = \infty$ and $\sum_{t=0}^\infty \beta_t = \infty$, 
which is a standard condition in the study of SA to ensure convergence~\citep{borkar2009stochastic}.
    \item This assumption is formulated for a broad class of operators so as to cover all examples satisfying our operator assumptions.
    As a consequence, the associated step-size conditions are intentionally conservative, and the constants appearing therein (e.g., $1/64$) are chosen mainly for technical convenience in the proofs.
    If one restricts attention to a narrower problem class or to a specific example, these constants can often be significantly improved.
    \item As indicated by \citet{kaledin2020finite}, this assumption encompasses diminishing, piecewise constant, and constant step size schedules.
    Here we focus on diminishing step sizes, which are a standard and widely used choice.
    A typical example is $\alpha_t = \Theta(t^{-a})$ and $\beta_t = \Theta(t^{-b})$ with $0< a \le b < 1$.
    In practice, it is sufficient to impose Assumption~\ref{assump:stepsize-new} only for $t \ge t_0$, where $t_0$ is a prescribed integer, since the early stage of optimization is usually dominated by transient behavior. For the purpose of establishing finite-time convergence rates, it is therefore not necessary to require the assumption to hold from the start.
    In particular, when $a < b$, corresponding to the strict two-time-scale regime, all the constant bounds in \ref{assump:stepsize-new:constant} are of order $o(1)$ as $t_0 \to \infty$.
    Treating these terms as $o(1)$ can greatly simplify the analysis of the constants; see Remark~\ref{rema:const}.
\end{itemize}
\end{rema}

\section{Theoretical Analysis}
\label{sec:decouple}

In Section~\ref{sec:decoupled:result}, we establish upper bounds for $\EB\| y_t - y^\star \|^2$ and $\EB \|x_t - H(y_t) \|^2$, and prove finite-time decoupled convergence under the local linearity condition.
To demonstrate the necessity of local linearity for decoupled convergence, we also construct an example in which local linearity fails and derive a corresponding lower bound in Section~\ref{sec:decouple:lower}.

\subsection{Upper bounds and Decoupled Convergence}
\label{sec:decoupled:result}

To analyze the convergence rate of $(x_{t},y_{t})$ to $(x^{\star},y^{\star})$, it is common to examine the mean square errors $\EB[\|x_{t}-x^{\star}\|^2]$ and $\EB[\|y_{t}-y^{\star}\|^2]$. 
Recall that two-time-scale SA can be viewed as an approximation of the two-loop procedure in \eqref{eq:two-loop-approx}.
It is thus more fundamental to consider the following residual variables~\citep{doan2022nonlinear}:
\begin{align}
	\begin{aligned}
		\xhat_{t} = x_{t} - H(y_{t}), \quad 
		\yhat_{t} = y_{t} - y^{\star}.
	\end{aligned}    \label{alg:xyhat}
\end{align}
Here, $\xhat_t$ and $\yhat_t$ represent the errors of the inner and outer loops in \eqref{eq:two-loop-approx},
respectively.
Furthermore, given the Lipschitz continuity of
$H$ in Assumption~\ref{assump:smooth:FH}, we can bound $\| x_t - x^\star \|$ as $\| x_t - x^\star \| \le \| \xhat_t \| + L_H \| \yhat_t \|$. Thus, it suffices to focus on $(\xhat_t, \yhat_t)$.

Our main results, based on the assumptions in Section~\ref{sec:prelim:assump}, are stated as follows.

\begin{thm}
	\label{thm:decouple-short}
	Suppose that Assumptions~\ref{assump:smooth:FH} -- 
	\ref{assump:stepsize-new} hold. 
	Then we have for all $t \ge 0$,
	\begin{align}
		\EB \| \xhat_{t+1} \|^2 & \le \label{thm:decouple:x}
		\CM_{x}\,\alpha_t,
		\\
		\| \EB \xhat_{t+1} \yhat_{t+1}^\top \| 
		& \le \label{thm:decouple:xy}
		\CM_{xy,1} \, \beta_t
		+ 
		\CM_{xy,2}
		\, \alpha_t \beta_t 
		\left( \frac{\alpha_t}{\beta_t} \right)^\frac{2}{\Fsmooth}, \\
		\EB \| \yhat_{t+1} \|^2 & \le \label{thm:decouple:y}
		\CM_{y,1}
		\, \beta_t
		+ 
		\CM_{y,2} \, \alpha_t \beta_t 
		\left( \frac{\alpha_t}{\beta_t} \right)^{\frac{2}{\Fsmooth} } 
		+ 
		\CM_{y,3}
		\, \alpha_t \beta_t 
		\left( \frac{\alpha_t}{\beta_t} \right)^{\frac{1}{\Gsmooth} }.
	\end{align}
	The exact constants are given in \eqref{eq:thm:decouple-constants} in Appendix~\ref{proof:de}.
	\end{thm}

In Theorem~\ref{thm:decouple-short}, observe that the step size $\alpha_t$ of the fast iterate also influences the convergence rate of the slow iterate through terms involving the parameters
$\Fsmooth$ and $\Gsmooth$.
In the following discussion, we focus on polynomially diminishing step sizes and examine the conditions required for achieving decoupled convergence.

\begin{cor}[Decoupled convergence rates]
	\label{cor:decouple-rates}
	Under the same setting of Theorem~\ref{thm:decouple-short},
	if we use polynomially diminishing step sizes $\alpha_{t} = \frac{\alpha_{0}}{(t+T_0)^a}$ and $\beta_{t} = \frac{\beta_{0}}{(t+T_0)^b}$ with $a,b \in (0,1]$,
	$1 \le \frac{b}{a} \le 1 + \frac{\Fsmooth}{2} \wedge \Gsmooth$ and properly chosen $\alpha_{0}$, $\beta_{0}$ and $T_0$,
	then we have 
	\begin{align*}
		\EB\|\xhat_{t} \|^2 = \OM (\alpha_t),
		\quad 
		\| \EB \xhat_t \yhat_t^\top \| = \OM (\beta_t),
		\quad 
		\text{and}
		\quad 
		\EB\|\yhat_{t} \|^2 = \OM (\beta_t).
	\end{align*}
	For an example choice of the constants, please refer to \eqref{eq:poly0-decou} in Appendix~\ref{proof:de}.
\end{cor}

\begin{rema}[Comparison with previous work]\label{rema:comparison}
	Note that $x_t - x^\star = \xhat_t + H(y_t) - H(y^\star)$.
	Theorem~\ref{cor:decouple-rates} and the Lipschitz continuity of $H$, we have $\EB \| x_t - x^\star \|^2 
	= \OM(\alpha_t) $ and $ \| \EB (x_{t} - x^\star) (y_t - y^\star)^\top \| = \OM(\beta_{t}) $.
	This result is consistent with the central limit theorem established in \citet{mokkadem2006convergence}:
	\begin{align}
		\label{eq:nonlinear-clt}
		\begin{pmatrix}
			\alpha_t^{-1/2} (x_t - x^\star) \\
			\beta_t^{-1/2} (y_t - y^\star )
		\end{pmatrix}
		\overset{d}{\rightarrow}
		\NM \left( 0, \begin{pmatrix}
			\Sigma_x & 0 \\
			0 & \Sigma_y
		\end{pmatrix} \right).
	\end{align}
	Moreover, our analysis provides a more refined characterization of the matrix cross term, as \eqref{eq:nonlinear-clt} only implies $(x_{t} - x^\star) (y_t - y^\star)^\top = o(\sqrt{\alpha_{t} \beta_{t}}) $ in probability.
	
	To establish the convergence rate in \eqref{thm:decouple:y}, the convergence rate of the matrix cross term in \eqref{thm:decouple:xy} is an essential intermediate result. 
	Neither decoupled convergence for the strict two-time-scale case nor an analysis of the matrix cross term is present in prior work on general nonlinear cases~\citep{shen2022single, doan2022nonlinear}.
\end{rema}

\begin{rema}[Step size selection for the optimal convergence rate]\label{rema:decouple:optimal_rate}
To achieve the optimal convergence rate of the slow iterate $\EB \| \yhat_t \|^2 = \OM(1/t)$, we could choose $\beta_t \sim \beta_0 t^{-1}$ and $\alpha_t \sim \alpha_0 t^{-a}$ with $(1 + \frac{\Fsmooth}{2} \wedge \Gsmooth)^{-1} \le a \le 1$.
In particular, when $\Fsmooth = \Gsmooth = 1$, the feasible range for $a$ in $\alpha_t \sim \alpha_0 t^{-a}$ is $2/3 \le a \le 1$.
Our results ensure that achieving $\OM(1/t)$ convergence for the slow iterate allows greater flexibility in the step size selection for the fast iterate, extending beyond the single-time-scale case considered in \citet{shen2022single}.
\end{rema}

\begin{rema}[Leading terms in the constants]\label{rema:const}
The complete expressions of the constants in \eqref{eq:thm:decouple-constants} are fairly complicated and difficult to analyze directly.
To simplify the analysis and isolate the most essential parameter dependence, we focus on diminishing step sizes of the form $\alpha_t = \Theta(t^{-a})$ and $\beta_t = \Theta(t^{-b})$, where $a,b \in (0,1]$ and $1 < \frac{b}{a} < 1 + \frac{\Fsmooth}{2} \wedge \Gsmooth$.
Compared with the requirement in Corollary~\ref{cor:decouple-rates}, we additionally require the inequality to be strict.
The condition $\frac{b}{a} < 1 + \frac{\Fsmooth}{2} \wedge \Gsmooth$ ensures that the leading terms in \eqref{thm:decouple:x}--\eqref{thm:decouple:y} are all given by the first terms
\begin{equation*}
    \EB \| \xhat_{t+1} \|^2 \le \CM_{x}\,\alpha_t,\quad
    \| \EB \xhat_{t+1} \yhat_{t+1}^\top \| \le \CM_{xy,1} \, \beta_t
    + o(\beta_t),\quad
    \EB \| \yhat_{t+1} \|^2 \le 
    \CM_{y,1} \, \beta_t
    + o(\beta_t).
\end{equation*}
Consequently we only need to focus on the constants $\CM_{x}$, $\CM_{xy,1}$, and $\CM_{y,1}$.
We also require $b>a$, corresponding to the ``strict'' two-time-scale case, in order to highlight the role of decoupled convergence, since decoupled convergence is trivial when $a=b$.

As discussed in Remark~\ref{rema:step-size}, it suffices to focus on $t \ge t_0$ for a prescribed $t_0$.
Under strict two-time-scale diminishing step sizes, the constants $\iota_1, \iota_2, \kappa, \rho$ in Assumption~\ref{assump:stepsize-new} are all of order $o(1)$ as $t_0 \to \infty$.
Therefore, in the expressions of $\CM_{x}$, $\CM_{xy,1}$, and $\CM_{y,1}$ in \eqref{eq:thm:decouple-constants}, all terms involving these constants can be treated as higher-order infinitesimals.
In particular, the leading terms in $\CM_{x}$, $\CM_{xy,1}$, and $\CM_{y,1}$ can be summarized as
\begin{equation}
\label{eq:const-leading-sum}
\begin{aligned}
    \CM_{x} & \lesssim \frac{\Gamma_{11}}{\mu_F} + o(1), \quad
    \CM_{xy,1} \lesssim \frac{\Sigma_{12}}{\mu_F} + \frac{\LGx \Gamma_{11}}{\mu_F^2} + o(1), \\ 
    \CM_{y,1} & \lesssim \frac{\Gamma_{22}}{\mu_G} + \frac{d_y \LGx \Sigma_{12}}{\mu_F \mu_G} + \frac{d_y \LGx^2 \Gamma_{11}}{\mu_F^2 \mu_G} + o(1).
\end{aligned}
\end{equation}
For the detailed derivation, see Appendix~\ref{proof:de:const}.

Meanwhile, from the CLT in \eqref{eq:nonlinear-clt}, we can obtain $\lim_{t\to\infty} \alpha_t^{-1} \EB \| x_t \|^2 = \tr(\Sigma_x)$ and $\lim_{t\to\infty} \beta_t^{-1} \EB \| y_t \|^2 = \tr(\Sigma_y)$
under additional regular conditions.
By analyzing the detailed expression of $\Sigma_x$ and $\Sigma_y$, we obtain
\begin{equation}
\label{eq:tr-cov-x-y}
    \tr(\Sigma_x) \le \frac{\Gamma_{11}}{2\mu_F},\quad
    \tr(\Sigma_y)
    \le
    \frac{1}{\mu_G}
    \left[
    \Gamma_{22}
    +2d_y\,\frac{\LGx}{\mu_F}\,\Sigma_{12}
    +\Big(\frac{\LGx}{\mu_F}\Big)^{2}\Gamma_{11}
    \right].
\end{equation}
Moreover, let $\Sigma_{x,y} = \lim_{t\to\infty} \beta_t^{-1} \EB [x_t y_t^\top]$. Then $\lim_{t\to\infty} \beta_t^{-1} \|\EB x_t y_t^\top\| = \| \Sigma_{x,y} \|$.
\citet[Theorem~2.6]{konda2004convergence} provide a characterization of $\Sigma_{x,y}$ in the linear case, from which we obtain
\begin{equation}
\label{eq:norm-cov-xy}
    \| \Sigma_{x,y} \| \le \frac{1}{\mu_F} \left( \Sigma_{12} + \frac{\LGx \Gamma_{11}}{2 \mu_F} \right).
\end{equation}
The detailed derivations of \eqref{eq:tr-cov-x-y} and \eqref{eq:norm-cov-xy} are also deferred to Appendix~\ref{proof:de:const}.

We now compare the above results.
The upper bounds in \eqref{eq:const-leading-sum} come from the non-asymptotic analysis.
Accordingly, we omit parameter-independent constants, since the primary goal of the non-asymptotic analysis is to establish the convergence rates rather than to optimize the constants.
By contrast, the upper bounds in \eqref{eq:tr-cov-x-y} and \eqref{eq:norm-cov-xy} are derived from the asymptotic analysis and may be viewed as nearly tight upper bounds.
We find that all parameter dependencies match except for the third term in the expression of $\CM_{y,1}$.
\vspace{0.1cm}

\textbf{Dimension dependence}.
In \eqref{eq:const-leading-sum}, the third term involving $\Gamma_{11}$ in the upper bound of $\CM_{y,1}$ 
depends on the dimension, whereas in \eqref{eq:tr-cov-x-y}, the third term involving $\Gamma_{11}$ in the upper bound of $\tr(\Sigma_y)$ does not exhibit such dimensional dependence.
By contrast, the dimensional dependence in the term involving $\Sigma_{12}$ is unavoidable.
Recall that Assumption~\ref{assump:noise-s} implies $\EB [\|\xi_{t}\|^2|\FM_t] \le \Gamma_{11}$, 
$\EB [\|\psi_{t}\|^2|\FM_t] \le \Gamma_{22}$, and
$\| \EB [\xi_{t} \psi_{t}^\top \,|\, \FM_{t}]\| \le \Sigma_{12}$.
This discrepancy comes from the fact that $\Sigma_{12}$ is an upper bound on the operator norm of the cross-covariance matrix, whereas $\EB [\|\xi_{t}\|^4 | \FM_t] \le \Gamma_{11}^2$ implies $\tr( \EB [\xi_t \xi_t^\top|\FM_t] ) =  \EB [\|\xi_{t}\|^2 | \FM_t] \le \Gamma_{11}$, so that $\Gamma_{11}$ serves as an upper bound on the trace of the covariance matrix.
Clearly, an upper bound on the trace is also an upper bound on the operator norm; however, there may be a dimensional gap between the trace and the operator norm.

We believe that the third term of $\CM_{y,1}$ in \eqref{eq:const-leading-sum} can be improved to $\frac{d_y \LGx^2 \Sigma_{11}}{\mu_F^2 \mu_G}$ by carrying out a more refined analysis of $\EB \xhat_t \xhat_t^\top$, rather than $\EB \| \xhat_t \|^2$,  under a tighter upper bound $\| \EB [\xi_t \xi_t^\top | \FM_t ] \| \le \Sigma_{11}$.
However, the dimensional dependence may still be unavoidable under the current proof framework.
To illustrate this point, consider the special case in which all eigenvalues of $\EB [\xi_t \xi_t^\top | \FM_t ]$ are equal and $d_x = d_y$. In this case, we have $\Gamma_{11} = d_x \Sigma_{11} = d_y \Sigma_{11}$.
Such dimensional dependence also appears in the finite-time analysis of the linear case~\citep{kaledin2020finite}.
\vspace{0.1cm}

\textbf{Inverse dependence on the strong monotonicity parameters}.
Recall that $\mu_F$ and $\mu_G$ are the strong monotonicity parameters of the inner operator $F(\cdot, y)$ and the outer operator $G(H(\cdot), \cdot)$, respectively.
The leading term in the upper bound for $\EB \| \xhat_t \|^2$ scales as $\mu_F^{-1}$, and the leading term in the upper bound for $\EB \| \yhat_t \|^2$ scales as $\mu_G^{-1}$.
This is consistent with the classical behavior of SGD for strongly convex optimization, where the leading constant is proportional to the inverse of the strong convexity parameter; see, for example, \citet[Theorem~1]{moulines2011non}.
\vspace{0.1cm}

\textbf{Amplification factor $\LGx / \mu_F$}.
In addition to the fact that the convergence rates for $\EB \| \xhat_t \|^2$ and $\EB \| \yhat_t \|^2$ depend only on their respective step sizes, the leading constant term for $\EB \| \xhat_t \|^2$ is also almost unaffected by the noise in the slow-time-scale update. By contrast, the leading constant term for $\EB \| \yhat_t \|^2$ is affected by both the fast- and slow-time-scale updates.
The effect of the cross-covariance between the fast- and slow-time-scale noises, denoted by $\Sigma_{12}$, is amplified by a factor of $\LGx / \mu_F$, and the effect of the covariance of the fast-time-scale noise, denoted by $\Gamma_{11}$, is amplified by a factor of $(\LGx / \mu_F)^2$.

This amplification can be understood from the asymptotic viewpoint in \citet[Remark~4.1]{han2024decoupled}: the slow iterate behaves asymptotically like a standard SA iterate for the operator $G(H(\cdot), \cdot)$ with a modified noise term
$ \breve{\psi}_t = \psi_t - B_2 B_1^{-1} \xi_t$.
By Proposition~\ref{prop:ensure-linearity}, the ratio $\LGx / \mu_F$ upper bounds $\| B_2 B_1^{-1} \|$.
The amplification effect therefore comes from controlling the covariance of $\breve{\psi}_t$.
\vspace{0.1cm}

\textbf{Behavior of the matrix cross term}.
The leading constant term in the bound for $\|\EB \xhat_t \yhat_t^\top\|$ is proportional to $\mu_F^{-1}$.
This suggests that the behavior of $\xhat_t \yhat_t^\top$ is governed primarily by the strong monotonicity of the inner operator $F(\cdot, y)$, rather than that of the outer operator $G(H(\cdot), \cdot)$.
The key reason is the separation of step sizes, namely $\alpha_t \gg \beta_t$.
Moreover, the leading term depends on both $\Sigma_{12}$ and $\Gamma_{11}$, with the effect of $\Gamma_{11}$ amplified by a factor of $\LGx / \mu_F$.
Under the definition of the modified noise term $\breve{\psi}_t = \psi_t - B_2 B_1^{-1} \xi_t$, this amplification arises from controlling the cross-covariance between $\xi_t$ and $\breve{\psi}_t$.

\end{rema}

\subsection{A Lower Bound without Local Linearity}
\label{sec:decouple:lower}

As shown in the previous subsection, decoupled convergence can be achieved under the local linearity condition with appropriately chosen step sizes.
This naturally raises the following question:
\textit{Is local linearity essential for decoupled convergence?}
In this subsection, we answer this question in the affirmative and show that the local linearity condition in Assumption~\ref{assump:near-linear} is necessary for decoupled convergence.

To this end, we construct the following example in which the local linearity condition fails.
\begin{exam}
	\label{exp:slow-time}
	Consider the following nonlinear SA problem with the operators $F, G\colon \RB \times \RB \to \RB$ given by
	\begin{align}\label{eq:exp:slow-time}
		F(x, y) = x- y,\ G(x, y) = - |x-y| \cdot \sign(y) + y,
	\end{align}
	where $\sign(x) = 1_{x > 0} -1_{x < 0}$ is the sign function.
\end{exam}

In this example, $G(x,y)$ involves both the sign and absolute value functions, and therefore does not satisfy the local linearity condition in Assumption~\ref{assump:near-linear}.
However, since $F(x,y)=x-y$ is linear, the induced solution map $H(y)=y$ is also linear.
Consequently, the reduced operator $G(H(y),y)=y$ is linear as well.
That is, the nonlinearity appears only in $G(x,y)$ before substituting $x=H(y)$.
Meanwhile, this example satisfies Assumptions~\ref{assump:smooth:FH}--\ref{assump:sm:F} with all corresponding parameters equal to $1$, as well as Assumption~\ref{assump:smoothH} with $S_H=0$.
The unique root is $(x^\star,y^\star)=(0,0)$.
Applying the two-time-scale SA algorithm~\eqref{alg:xy} to this example, we obtain the following lower bound, whose proof is deferred to Appendix~\ref{append:proof:lower}.
\begin{prop}[Lower bound for Example~\ref{exp:slow-time}]\label{prop:lower}
 Suppose that: (a) the noise terms satisfy $\psi_t = 0$, and the $\xi_t$ are i.i.d. with $\EB \xi_t = 0$ and $\EB \xi_t^2 = \Sigma_1 > 0$; (b) the step sizes satisfy $\beta_t / \alpha_t \to 0$ and Assumption~\ref{assump:stepsize-new} with the parameters $\delta_F$ and $\delta_G$ in \ref{assump:stepsize-new:growth} replaced by $1$;\footnote{For $d_x = d_y = 1$, after this replacement, Assumption~\ref{assump:stepsize-new} reduces to the weaker version, Assumption~\ref{assump:stepsize-new-weak}, introduced in Section~\ref{sec:decouple:proof-3}.} (c) the initialization satisfies $y_0 \neq y^\star$.
 Then we have $\EB |\xhat_t|^2 = \Omega(\alpha_t)$ and $\EB |\yhat_t|^2 =\Omega(\alpha_t) $.
\end{prop}
The condition on the noise is imposed to simplify the analysis.
The requirement $\beta_t / \alpha_t \to 0$ on the step sizes restricts attention to the strict two-time-scale regime, since in the single-time-scale regime, where $\beta_t = \Theta(\alpha_t)$, decoupled convergence is trivial.

Proposition~\ref{prop:lower} shows that, without local linearity of $G(x,y)$, the convergence rate on the slow time scale is indeed degraded by the larger step size associated with the fast-time-scale update, even though $F(x,y)$, $H(y)$, and $G(H(y),y)$ are all linear.
This complements the approximation perspective in \eqref{eq:two-loop-approx}: 
\begin{center}
    Although the two iterates in two-time-scale SA can be interpreted as solving $F(x,y)=0$ (for fixed $y$) and $G(H(y),y)=0$, the detailed form of $G(x,y)$ before substituting $x=H(y)$ still affects the convergence rates.
\end{center}
Moreover, in Proposition~\ref{prop:lower}, the slow-time-scale update is deterministic.
This indicates that the main obstacle to decoupled convergence on the slow time scale is the interdependence between the two time scales, rather than the noise in the slow-time-scale update.

This observation also has a possible algorithmic implication.
Consider $\widetilde{G}_\alpha(x,y)=\alpha(x-y)+y$ for $\alpha \in \RB$.
Then $\widetilde{G}_\alpha(H(y),y)=G(H(y),y)$.
In other words, the linear operator $\widetilde{G}_\alpha(x,y)$ yields the same outer operator $\widetilde{G}_\alpha(H(y),y)$ as in Example~\ref{exp:slow-time}.
This suggests that, when there are multiple possible choices of the operators $F$ and $G$ leading to the same reduced operator $G(H(\cdot),\cdot)$, it is preferable to choose linear or nearly linear operators in order to ensure a faster convergence rate.

\section{Proof Framework for Our Main Theorem}
\label{sec:decouple:sketch}
In this section, we outline the framework of our proof for Theorem~\ref{thm:decouple-short}.
We first present the high-level idea and comparison with prior works, with the detailed procedure developed in Sections~\ref{sec:decouple:proof-1} -- \ref{sec:decouple:proof-4}.

\paragraph{High-level idea}
Our main technical contribution is a systematic framework for handling the cross term 
$\| \mathbb{E} \hat{x}_t \hat{y}_t^\top \|$
 in the nonlinear setting. It's crucial for the sharp convergence characterization of the interacting sequences $\{x_t\}_{t=0}^\infty$ and $\{y_t\}_{t=0}^\infty$. While similar analysis exists for the linear case (e.g., \citealt{kaledin2020finite}), the nonlinear case is harder due to the nonlinearity of $F$, $G$, and $H$. We approximate these mappings by their linear parts, which introduces residual errors. The key challenge is that the cross term's dynamics intertwine with these (higher-order) residuals. Our framework systematically tracks this interaction and controls the error accumulation via a fourth-moment analysis, showing the residuals are indeed higher-order. 

The whole proof is organized into the following four steps:
\begin{enumerate}[leftmargin=2cm]
    \item[Step~1:] Derive Convergence Rates without Local Linearity of $F$ and $G$
    \item[Step~2:] Introduce the Matrix Cross Term and Derive Refined One-Step Descent Lemmas
    \item[Step~3:] Analyze the Convergence Rates of Fourth-Order Moments
    \item[Step~4:] Integrate the Above Ingredients and Derive Decoupled Convergence Rates
\end{enumerate}
\paragraph{Comparison with prior works} 
Overall, the nonlinear setting requires several additional components beyond those in \citet{doan2022nonlinear} and \citet{kaledin2020finite}, 
even though a few individual steps are similar in spirit to these earlier analyses.
The first step is close in spirit to \citet{doan2022nonlinear}, but it only provides a preliminary convergence rate and serves as a starting point for our later analysis. The second step is partly inspired by \citet{kaledin2020finite}, where we introduce the matrix cross term $\| \EB \xhat_t \yhat_t^\top \|$ to refine the one-step descent analysis for the slow iterate; however, in the nonlinear setting, the local linear approximations of $F$, $G$, and $H$ also produce higher-order error terms, which require additional control arguments. The third step, namely the convergence analysis of fourth-order moments, does not appear in either \citet{doan2022nonlinear} or \citet{kaledin2020finite}, and is needed to handle the higher-order terms caused by nonlinearity. Finally, the last step integrates all these ingredients to establish decoupled convergence, while simultaneously handling the nonlinear remainder terms and the matrix cross term.

\subsection{Step~1: Derive Convergence Rates without Local Linearity of $F$ and $G$.}
\label{sec:decouple:proof-1}
We first establish a coarse convergence rate without local linearity as a starting point.
To this end, we first present the one-step descent lemmas for the squared errors
The analysis in this step does not require the local linearity of $F$ and $G$ in Assumption~\ref{assump:near-linear}, nor does it require the bounded fourth-order moments in Assumption~\ref{assump:noise-s}.
The requirement on the step sizes in Assumption~\ref{assump:stepsize-new} can also be weakened.
Thus, the results in this step is based on the weaker version of Assumptions~\ref{assump:noise-s} and \ref{assump:stepsize-new}, denoted as Assumptions~\ref{assump:noise} and \ref{assump:stepsize}, with the details deferred to Appendix~\ref{proof:without}.

\begin{lem}[One-step descent lemma of the fast iterate]
	\label{lem:xhat}
	Suppose that Assumptions \ref{assump:smooth:FH}  -- \ref{assump:smoothH} and 
    \ref{assump:noise} -- \ref{assump:stepsize}  hold.
	For any $t \ge 0$, we have
	\begin{equation}
		\begin{aligned}
			\EB \left[ \|\xhat_{t+1} \|^2\,|\,\FM_{t}\right]
			& \le \left(1- \mu_F \alpha_t \right) \|\xhat_{t}\|^2 + c_{x,1} \beta_t^2 \|\yhat_{t}\|^2 + 2 \Gamma_{11} \alpha_{t}^2   \\
			& \qquad + c_{x,2} \beta_t\sqrt{ 1-\alpha_t \mu_F}  \|\xhat_{t}\|\|\yhat_{t}\|+ c_{x,3} \beta_{t}^2 + c_{x,4} \frac{\beta_t^{2 + 2\Hsmooth}}{\alpha_t},
		\end{aligned}
		\label{lem:xhat:Ineq}
	\end{equation}
	where $\{c_{x,i}\}_{i \in [4]}$ are problem-dependent constants defined in 
	\eqref{eq:constantx}.
\end{lem}

\begin{lem}[One-step descent lemma of the slow iterate]\label{lem:yhat}
	Suppose that Assumptions \ref{assump:smooth:FH}  -- \ref{assump:smoothH} and \ref{assump:noise} -- \ref{assump:stepsize}  hold.
	For any $t \ge 0$, we have 
	\begin{align}
		\begin{split}	\EB\left[\|\yhat_{t+1}\|^2\,|\,\FM_{t}\right]
			& \leq \left(1-\mu_{G}\beta_{t} \right)\|\yhat_{t}\|^2 +  \LGx^2 \beta_{t}^2\|\xhat_{t}\|^2 \\
			& \qquad +2 \LGx \beta_t\sqrt{1-\mu_G\beta_t} \|\xhat_{t}\|\|\yhat_{t}\|  + \Gamma_{22}\beta_{t}^2.
		\end{split}
		\label{lem:yhat:Ineq}
	\end{align}
\end{lem} 
By combining Lemmas~\ref{lem:xhat} and \ref{lem:yhat} and using carefully designed Lyapunov functions, we can achieve the following convergence results.

\begin{thm}[Convergence rates without local linearity of $F$ and $G$]
	\label{thm:first}
	Suppose that Assumptions \ref{assump:smooth:FH}  -- \ref{assump:smoothH} and \ref{assump:noise} -- \ref{assump:stepsize}  hold.
	Then we have
		\begin{align*}
			\EB\|\xhat_{t+1} \|^2 &\le  \prod_{\tau=0}^t\left(1-\frac{\mu_{G}\beta_{\tau}}{4} \right) \left( 3 \EB\|\xhat_{0}\|^2  +  \frac{7 L_H \LGy \EB\|\yhat_{0}\|^2 }{\LGx} \right) + \frac{8\Gamma_{11}}{\mu_F} \alpha_t \\
			& \qquad + c_{x,5} \beta_t + c_{x,6} \frac{\beta_t^2}{\alpha_t} + c_{x,7} \frac{\beta_t^{2+2\Hsmooth}}{\alpha_t^2}, \\
			\EB\|\yhat_{t+1} \|^2 &\le  \prod_{\tau=0}^t\left(1-\frac{\mu_{G}\beta_{\tau}}{4} \right) \left( \EB\|\yhat_{0}\|^2 +  \frac{ 2 \LGx \EB\|\xhat_{0}\|^2 }{7 L_H \LGy}  \right)
			+ \frac{128 \LGx^2 \Gamma_{11}}{\mu_F\mu_G^2} \alpha_t \\
			& \qquad +  c_{y,1} \beta_t +  c_{y,2} \frac{\beta_t^{2+2\Hsmooth}}{\alpha_t^2},
		\end{align*}
		where $\{ c_{x,i} \}_{i \in [7] \setminus [4]}$ and $\{ c_{y,i} \}_{i \in [2]}$ are problem-dependent constants defined in 
		\eqref{eq:constanty} and \eqref{eq:constantx2}.
		In particular, when $\Hsmooth \ge 0.5$,
		\begin{align}
			\EB\|\xhat_{t+1} \|^2 + \EB\|\yhat_{t+1} \|^2 = \OM(\alpha_{t}).
			\label{eq:rate-w/o}
		\end{align} 
\end{thm}
The details proofs of Lemmas~\ref{lem:xhat}, \ref{lem:yhat} and Theorem~\ref{thm:first} are given in Appendix~\ref{proof:without}.
Although the results in this subsection follow \cite{doan2022nonlinear} at a high level, our analysis imposes refined conditions (e.g., Assumption 2.4), leading to explicit dependence on $\delta_H$, whereas Doan’s result corresponds to the special case $\delta_H = 0$. As a result, the proof must be reworked despite the similar overall idea.

\subsection{Step~2: Introduce the Matrix Cross Term and Derive Refined One-Step Descent Lemmas}
\label{sec:decouple:proof-2}

With Assumption~\ref{assump:near-linear}, we could replace 
$\EB\|\xhat_{t}\|\|\yhat_{t}\|$ in Lemma~\ref{lem:xhat} and~\ref{lem:yhat} with the matrix cross term $\| \EB \xhat_{t} \yhat_{t}^\top \|$ at the cost of introducing the higher-order residual terms. 
We also need to derive the one-step descent lemma of the matrix cross term $\| \EB \xhat_t \yhat_{t}^\top \| $.
To simplify the notation, we define
\[
\ZM_{t,\delta} := \EB \| \xhat_t \|^\delta + \EB \| \yhat_t \|^\delta.
\]

\begin{lem}[Refined one-step descent lemma of the fast iterate]\label{lem:xhat-new}
	Suppose that Assumptions~\ref{assump:smooth:FH} -- \ref{assump:stepsize-new} 
	hold. 
	We have for any $t \ge 0$,
	\begin{equation}
		\begin{aligned}
			\EB  \|\xhat_{t+1}\|^2 
			& \le \left(1- \mu_F \alpha_t
			\right) \EB \|\xhat_{t}\|^2 
			+ \cde_{x,1} \beta_t^2 \EB \|\yhat_{t}\|^2 +  \cde_{x,2} \beta_t \| \EB \xhat_t \yhat_t^\top \| 
			+ 2\Gamma_{11} \alpha_{t}^2 \\
			& \qquad + \cde_{x,3} \beta_{t}^2 + \cde_{x,4} \frac{\beta_t^{2+2\Hsmooth}}{\alpha_t}  + \Delta_{x, t}.
		\end{aligned}
		\label{lem:xhat-new:Ineq}
	\end{equation}
	where $\Delta_{x,t}$ is a higher-order residual given in the following
	\begin{equation}
		\label{eq:delta_x}
		\begin{aligned}
			\Delta_{x,t}
			& = \cde_{x,5} \beta_t \ZM_{t,2+\Hsmooth}
			+ \cde_{x,6} \alpha_t \beta_t \ZM_{t,2+\Fsmooth} 
			+ \cde_{x,7} \beta_t \ZM_{t,2+2\Gsmooth}
			+ \cde_{x,8} \beta_t^{1+\Hsmooth} \ZM_{t,2+2\Hsmooth},
		\end{aligned}
	\end{equation}
	and $\{\cde_{x,i}\}_{i \in [8]}$ are problem-dependent constants defined in 
	\eqref{eq:constantx:new}.
\end{lem}

\begin{lem}[Refined one-step descent lemma of the slow iterate]\label{lem:yhat-new}
	Suppose that Assumptions~\ref{assump:G} -- \ref{assump:sm:F} and~\ref{assump:near-linear} --~\ref{assump:stepsize-new} hold. 
	We have for any $t \ge 0$,
	\begin{align}
		\begin{split}    
			\EB  \|\yhat_{t+1}\|^2 
			& \le \left(1- \frac{2 \mu_G \beta_t}{3} \right) \EB \|\yhat_{t}\|^2 
			+ 2 \LGx^2 \beta_t^2 \EB \|\xhat_{t}\|^2 
			+ 2 d_y \LGx \beta_t \| \EB \xhat_t \yhat_t^\top \| \\
			& \qquad + \Gamma_{22} \beta_{t}^2 
			+ \Delta_{y, t},    
		\end{split}
		\label{lem:yhat-new:Ineq}
	\end{align}
	where $\Delta_{y,t}$ is a higher-order residual given in the following
	\begin{equation}
		\label{eq:delta_y}
		\begin{aligned}
			\Delta_{y,t}
			& = \SBG^2 \beta_t \left( 15 d_y^2 / \mu_G + { d_y^2 } \beta_t + 8 d_y \beta_t \right)  \ZM_{t,2+2\Gsmooth}.
		\end{aligned}
	\end{equation}
\end{lem}

\begin{lem}[One-step descent lemma of the matrix cross term]\label{lem:xyhat}
	Suppose that Assumptions~\ref{assump:smooth:FH} -- 
	\ref{assump:stepsize-new} hold. 
	We have that for any $t \ge 0$,
	\begin{equation}
		\begin{aligned}
			\|\EB\xhat_{t+1}\yhat_{t+1}^\top\| 
			&\le \left(1-\frac{\mu_F\alpha_t}{2}\right)\|\EB \xhat_t\yhat_{t}^\top \|  + \beta_t \left( \LGx \EB \|\xhat_{t}\|^2 + \cde_{xy,1} \EB\|\yhat_{t}\|^2\right)   \\
			& \qquad + \Sigma_{12} \alpha_t \beta_t + \cde_{xy,2} \beta_t^2 + \cde_{xy,3} \beta_t^{1+2\Hsmooth} + \Delta_{xy,t},
		\end{aligned}
		\label{lem:xyhat:Ineq}
	\end{equation}
	where $\Delta_{xy,t}$ is a higher-order residual given in the following
	\begin{equation}
		\label{eq:delta_xy}
		\begin{split}
			\Delta_{xy,t} 
			& = 
			2 \alpha_t \SBF \ZM_{t,2+\Fsmooth}
			+ \beta_t \SBG (1+2L_H) \ZM_{t,2+\Gsmooth}
			\\
			& \qquad
			+ \beta_t \cde_{xy,4} \ZM_{t,2+\Hsmooth}
			+ 2 \alpha_t \beta_t \SBF \SBG \ZM_{t,2+\Fsmooth+\Gsmooth},
		\end{split}
	\end{equation}
	and $\{ \cde_{xy,i} \}_{i \in [4]}$ are problem-dependent constants defined in 
	\eqref{eq:constant-xy}.
\end{lem} 
The proof of these lemmas can be found in Appendices~\ref{proof:de:xhat-new}, \ref{proof:de:yhat-new} and \ref{proof:de:xyhat}.

We conclude this step by briefly outlining the proof idea. 
Upon incorporating the update rules of $\xhat_{t+1}$ and $\yhat_{t+1}$ into our desired error metrics (e.g., $\EB\|\xhat_{t+1}\|^2$), we decompose the errors into primary components and higher-order terms. For each component, we determine individual upper bounds and then aggregate them accordingly. Specifically, we summarize the higher-order terms into single quantities, $\Delta_{x,t}$, $\Delta_{y,t}$, and $\Delta_{xy,t}$.

\subsection{Step~3: Analyze the Convergence Rates of Fourth-Order Moments}
\label{sec:decouple:proof-3}

To analyze the residual terms $\Delta_{x,t}$, $\Delta_{y,t}$ and $\Delta_{xy,t}$, 
we focus on a single quantity $\ZM_{t,4}$ due to the observation from Jensen's inequality: $\ZM_{t,\delta} \le (\ZM_{t,4})^{\delta/4}$ if $\delta \le 4$.\footnote{Here we choose the fourth-order for simplicity. It is feasible and natural to use an order smaller than $4$, which, however, would increase the complexity of the proof.}
It motivates us to analyze fourth-order moments of errors, i.e., $\EB\|\xhat_{t+1} \|^4$ and $\EB\|\yhat_{t+1}\|^4$. 
To that end, we derive one-step recursions for the conditional fourth-order moments.
The derivation process closely parallels that of Lemmas~\ref{lem:xhat-new} and \ref{lem:yhat-new}.
Moreover, we emphasize that the analysis of fourth-order moments does not require the local linearity of $F$ and $G$ in Assumption~\ref{assump:near-linear}.
Because Assumption~\ref{assump:stepsize-new} involves the parameter $\Fsmooth$ and $\Gsmooth$ in Assumption~\ref{assump:near-linear}.
The analysis relies on the following weaker version of Assumption~\ref{assump:stepsize-new} instead.

\begin{assumpt}{\ref*{assump:stepsize-new}$\dagger$}[Conditions on step sizes]
	\label{assump:stepsize-new-weak}
    The conditions in Assumption~\ref{assump:stepsize-new} hold, with $\mu_F$ and $\mu_G$ defined in Assumptions~\ref{assump:sm:F}, $\Fsmooth = \Gsmooth = 1$, and $d_x$ and $d_y$ in \eqref{assump:stepsize:constants} replaced by $1$.
\end{assumpt}

\begin{lem}[One-step descent of the fourth-order moment of the fast iterate]\label{lem:xhat-quartic}
	
	Suppose that Assumptions~\ref{assump:smooth:FH} -- \ref{assump:smoothH}, \ref{assump:noise-s}, and \ref{assump:stepsize-new-weak} hold. 
	We have for any $t \ge 0$,
	\begin{align}
		\label{lem:xhat-quartic:Ineq}
		\begin{split}
			\EB \left[ \| \xhat_{t+1} \|^4 \,|\, \FM_t \right]
			& \le \left( 1 - 
			{\mu_F \alpha_t}
			\right) \| \xhat_t \|^4 
			+ \cde_{xx, 1}  \beta_t \| \xhat_t \|^3 \| \yhat_t \|
			+ \cde_{xx,2} \beta_t^2 \| \xhat_t \|^2 \| \yhat_t \|^2 
			\\
			& \quad \ 
			+ \cde_{xx,3}  \beta_t^4 \| \yhat_t \|^4 + 32 \alpha_t^4 \Gamma_{11}^2 + 224 L_H^4 \beta_t^4 \Gamma_{22}^2 \\
			& \quad \     + \left( 20 \Gamma_{11} \alpha_t^2 + 
			20 L_H^2 \Gamma_{22} \beta_t^2 + \cde_{xx,4} \frac{\beta_t^{2+2\Hsmooth}}{\alpha_t} \right) \| \xhat_t \|^2,
		\end{split}
	\end{align}
	where $\{ c_{xx,i} \}_{i \in [5]}$ are problem-dependent constants defined in 
	\eqref{eq:constantx-quartic}.
\end{lem}

\begin{lem}[One-step descent of the fourth-order moment of the slow iterate]\label{lem:yhat-quartic}
	Suppose that Assumptions~\ref{assump:G}, \ref{assump:sm:F}, \ref{assump:noise-s}, and \ref{assump:stepsize-new-weak} hold. 
	We have for any $t \ge 0$,
	\begin{align}
		\label{lem:yhat-quartic:Ineq}
		\begin{split}
			\EB \left[ \| \yhat_{t+1} \|^4 \,|\, \FM_t \right]
			& \le \left( 1 - \frac{ 3 \mu_G \beta_t}{2} \right) \| \yhat_t \|^4
			+ 4 \LGx \beta_t \| \xhat_t \| \| \yhat_t \|^3 
			+ 18 \LGx^2 \beta_t^2 \| \xhat_t \|^2 \| \yhat_t \|^2 \\
			& \quad \ + 20 \LGx^4 \beta_t^4 \| \xhat_t \|^4
			+ 18 \Gamma_{22} \beta_t^2 \| \yhat_t \|^2 
			+ 28 \beta_t^4 \Gamma_{22}^2.
		\end{split}
	\end{align}
\end{lem}
The proof of these lemmas can be found in Appendices~\ref{proof:de:xhat-quartic} and \ref{proof:de:yhat-quartic}.

By integrating Lemmas~\ref{lem:xhat-quartic} and \ref{lem:yhat-quartic} and using a carefully designed Lyapunov function  $V_{t} = \varrho_3  \frac{\beta_t}{\alpha_t} \| \xhat_t \|^4 + \| \yhat_t \|^4$ for a properly specified $\varrho_3$, we can apply the results of Theorem~\ref{thm:first} to determine the convergence rates for the fourth-order moments. 

\begin{lem}[Convergence rates of the fourth-order moments]\label{lem:x+y_quartic}
	Suppose that Assumptions~\ref{assump:smooth:FH} -- \ref{assump:smoothH}, \ref{assump:noise-s}, and \ref{assump:stepsize-new-weak} hold. 
	Then we have for all $t \ge 0$,
	\begin{align}
		\EB \| \xhat_{t+1} \|^4 
		& \le \prod_{\tau=0}^t \left( 1 - \frac{\mu_G \beta_\tau}{4} \right)
		\left( 2 \EB \| \xhat_0 \|^4 + \frac{\mu_G^4\, \EB \| \yhat_0 \|^4 }{27 \LGx^4} 
		\right) + \frac{4 \cde_{xx,7}}{\mu_F} \alpha_t^2 + \frac{3 \cde_{xx,8}}{\mu_F} \frac{\beta_t^{4+4\Hsmooth}}{\alpha_t^4}, \notag \\
		\EB \| \yhat_{t+1} \|^4
		& \le \prod_{\tau=0}^t \left( 1 - \frac{\mu_G \beta_\tau}{4} \right) \left( \frac{\LGx^3 \EB \| \xhat_0 \|^4  }{3 \mu_G^2 L_H \LGy} 
		+ \EB \| \yhat_0 \|^4 \right) + \frac{8 \cde_{yy,1}}{\mu_G} \alpha_t^2 + \frac{10 \cde_{yy,2}}{\mu_G} \frac{\beta_t^{4+4\Hsmooth}}{\alpha_t^4}, \notag
	\end{align}
	where 
	$\{ \cde_{xx,i} \}_{i \in [7,8]}$ and $\{ \cde_{yy,i} \}_{i \in [2]}$ 
	are problem-dependent constants defined
	in 
	\eqref{eq:constanty-quartic-12} and \eqref{eq:constantx-quartic-78}.
	Moreover, when $\Hsmooth \ge 0.5$, then 
	\begin{align}\label{eq:rate-fourth}
		\EB \| \xhat_{t+1} \|^4 + \EB \| \yhat_{t+1} \|^4 = \OM(\alpha_{t}^2).
	\end{align}
\end{lem}
The proof can be found in Appendix~\ref{proof:de:x+y-quartic}.

\subsection{Step~4: Integrate the Above Ingredients and Derive Decoupled Convergence Rates}
\label{sec:decouple:proof-4}

With the aforementioned lemmas, we could integrate them to derive the convergence rates in Theorem~\ref{thm:decouple-short}.
Figure~\ref{fig:intergrate} provides a visual representation of the process. The integration follows the following procedure.

\begin{figure}[t!]\centering
	\begin{tikzpicture}[>=stealth,every node/.style={shape=rectangle,draw,rounded corners, minimum width=1cm, node distance=0.6cm},]
		\node (fourth) {\small \begin{tabular}{c} 
				Fourth-order \\ moments \eqref{eq:rate-fourth}	\end{tabular}};
		
		\node (delta) [right=of fourth]
		{ \small \begin{tabular}{c} 
				Control $\Delta_{x,t}$, \\ $\Delta_{y,t}$ and $\Delta_{xy,t}$	\end{tabular}};
		
		\node (coarse)
		[right=of delta]
		{ \small\begin{tabular}{c} 
				Coarse \\ rate \eqref{eq:rate-w/o}	\end{tabular}};
		
		\node (xyrecur) [right=of coarse]
		{\small\begin{tabular}{c} 
				Recursions \\ \eqref{lem:xhat-new:Ineq} and \eqref{lem:xyhat:Ineq} \end{tabular}};
		
		\node (yrecur) [right=of xyrecur]
		{\small\begin{tabular}{c} 
				Recursion \\ \eqref{lem:yhat-new:Ineq} \end{tabular}};
		
		\node (xrate) [below=of delta] 		
		{\small\begin{tabular}{c} 
				Rate of \\ $\EB \| \xhat_{t} \|^2$ \eqref{thm:decouple:x}
		\end{tabular}};
		
		\node (xyrate) [right=of xrate] 		
		{\small\begin{tabular}{c} 
				Rate of \\ $\| \EB \xhat_{t} \yhat_{t}^\top \| $ \eqref{thm:decouple:xy}
		\end{tabular}};
		
		\node (yrate) [right=of xyrate] 		
		{\small\begin{tabular}{c} 
				Rate of \\ $ \EB \| \yhat_{t} \|^2 $ \eqref{thm:decouple:y}
		\end{tabular}};

		\draw[->, line width=.3mm, color=brown] (fourth) to
		(delta);
		\draw[->, line width=.3mm] (coarse) to (xrate);
		\draw[->, line width=.3mm, color=blue] (delta) to (xyrate);
		\draw[->, line width=.3mm, color=blue] (coarse) to (xyrate);
		\draw[->, line width=.3mm, color=blue] (xyrecur) to (xyrate);
		\draw[->, line width=.3mm, color=red] (delta) to (yrate);
		\draw[->, line width=.3mm, color=red] (coarse) to (yrate);
		\draw[->, line width=.3mm, color=red] (yrecur) to (yrate);
		\draw[->, line width=.3mm, color=red] (xyrate) to (yrate);
	\end{tikzpicture}
	\caption{Illustration for Step~4 of the proof sketch.}
	\label{fig:intergrate}
\end{figure}

\begin{itemize}
	\item Black arrow: The convergence rate of $\EB \| \xhat_{t} \|^2$ in \eqref{thm:decouple:x-formal} directly follows from the coarse rate \eqref{eq:rate-w/o}. 
	\item Brown arrow: Using the fourth-order convergence rates in \eqref{eq:rate-fourth} 
	we could manage the higher-order residual terms $\Delta_{x,t}$, $\Delta_{y,t}$, and $\Delta_{xy,t}$. 
	For the detailed upper bounds, refer to \eqref{eq:delta_x_upper}, \eqref{eq:delta_y_upper} and \eqref{eq:delta_xy_upper}.
	
	\item Blue arrows: Combining the recursions \eqref{lem:xhat-new:Ineq} and \eqref{lem:xyhat:Ineq} with a properly chosen Lyapunov function and applying the coarse rate in \eqref{eq:rate-w/o}, we derive the convergence rate of $\| \EB \xhat_{t} \yhat_{t}^\top \|$ in \eqref{thm:decouple:xy}.
		
	\item Red arrows: Substituting the convergence rates of $\EB \| \xhat_{t} \|^2$ and $\| \EB \xhat_{t} \yhat_{t}^\top \|$ into the recursion \eqref{lem:yhat-new:Ineq} yields the convergence rate of $\EB \| \yhat_{t} \|^2$ in \eqref{thm:decouple:y}.
\end{itemize}

The details of proof can be found in Appendix~\ref{proof:de:decouple}.

\section{Numerical Experiments}\label{sec:expe}
This section presents the numerical experiments.
In Section~\ref{sec:ana:nume}, we report the results for Example~\ref{exp:slow-time} and its locally linear variant, illustrating the necessity of local linearity for decoupled convergence.
In Sections~\ref{sec:expe:decouple-toy} and \ref{sec:expe:decouple-logistic}, we consider one-dimensional toy examples and logistic regression, respectively, to illustrate the decoupled convergence rates established in Section~\ref{sec:decouple}.

\subsection{Example~\ref{exp:slow-time} and Its Locally Linear Variant}
 \label{sec:ana:nume}

In Section~\ref{sec:decouple:lower}, we have devised Example~\ref{exp:slow-time} to show that local linearity is necessary for decoupled convergence.
To illustrate this necessity, we start from $(x_0, y_0) = (2,1)$, consider noise terms $\xi_t \overset{{i.i.d.}}{\sim} \NM(0, 1)$ and 
$\psi_t =0$,
and step sizes $\alpha_t = \alpha_0 (t+1)^{-a}$ and $\beta_t = \beta_0 (t+1)^{-b}$.
To find the optimal values of $(\alpha_0, \beta_0)$, we perform a grid search on $\{ 10,3,1,0.3,0.1 \}^2$ for each pair, running $10^5$ iterations.
For each $ (a, b) \in \{ 0.7, 0.6 \} \times \{ 1.0, 0.9 \} $, we run $10^6$ steps of \eqref{alg:xy} across $10^3$ repetitions.

\begin{figure}
	\centering
	\includegraphics[width=\linewidth]{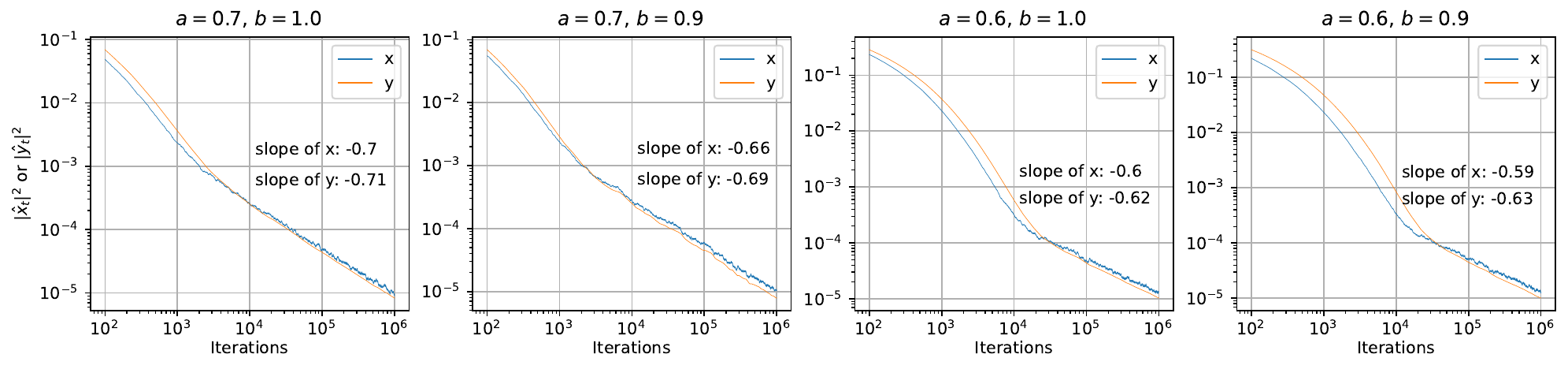}
	\caption{The convergence results within Example~\ref{exp:slow-time}. We calculate the line slopes using data from the $2\times10^5$ to $10^6$ iteration range. 
	}
	\label{fig:not-decouple}
\end{figure}
Figure~\ref{fig:not-decouple} presents results on a log-log scale, plotting the averaged values of $| \xhat_t |^2$ and $| \yhat_t |^2$ against the number of iterations for different $(a,b)$ pairs.
The slope of each line in the log-log plot reflects the convergence rate, as a relationship of the form $y = r \, x^{-s}$ corresponds to $\log y = -s \log x + \log r$.
Despite using distinctly different time scales, the slope of the orange line (representing the slow iterate $y_{t}$) nearly matches that of the blue line (representing the fast iterate $x_t$). This indicates that the nonlinear interaction in Example~\ref{exp:slow-time} hinders the convergence of the slow iterate $y_t$, 
preventing decoupled convergence, 
consistent with the theoretical results in Proposition~\ref{prop:lower}.

Next, we consider a local linear variant of Example~\ref{exp:slow-time}.
Define the auxiliary function 
\begin{equation}\label{eq:expe-aux}
    \tilde{h}_\delta (x) = \begin{cases}
    \sign(x) |x|^\delta / \delta, & |x| \le 1, \\
    \sign(x) ( |x| - 1 + 1 / \delta ), & |x| > 1.
\end{cases}
\end{equation}
One can check that when $\delta \ge 1$, $\tilde{h}_\delta$ is $1$-Lipschitz continuous.

\paragraph{A local linear variant of Example~\ref{exp:slow-time}}
Consider the following variant of Example~\ref{exp:slow-time}: 
\begin{align}
\label{eq:nume-counter-decouple}
	F(x, y) = x- y, \  G(x, y) = - \tilde{h}_{1.5}(|x-y|) \cdot \sign(y) + y.
\end{align}
Assumption~\ref{assump:near-linear} holds with $\SBF=0$ (allowing 
$\Fsmooth$ to be set to $1$) and $\Gsmooth=0.5$.
We start from $(x_0,y_0) = (2,2)$, consider noise terms $\xi_t \overset{{i.i.d.}}{\sim} \NM(0, 1)$ and $\psi_t \overset{{i.i.d.}}{\sim} \NM(0, 0.01)$.
Other settings are the same as Figure~\ref{fig:not-decouple}.
The results are shown in Figure~\ref{fig:counter-decouple}.
Based on Corollary~\ref{cor:decouple-rates}, decoupled convergence is expected within the range $1 \le b/a \le 1.5$.
Interestingly, as Figure~\ref{fig:counter-decouple} demonstrates, decoupled convergence is observed even when $b/a = 1/0.6 > 1.5$.

\begin{figure}[ht]
    \centering
    \includegraphics[width=\linewidth]{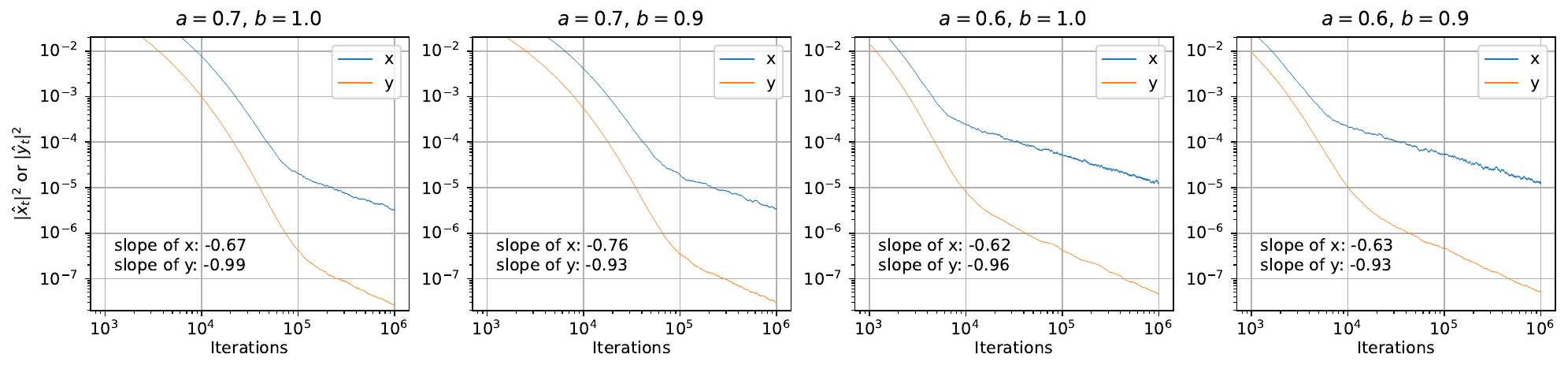}
    \caption{The convergence results for the example in \eqref{eq:nume-counter-decouple}.
    We calculate the line slopes using data from the $3 \times 10^5$ to $10^6$ iteration range.
    }
    \label{fig:counter-decouple}
\end{figure}

\subsection{Toy Examples}
\label{sec:expe:decouple-toy}

In this subsection, we illustrate the decoupled convergence in Section~\ref{sec:decouple} through numerical results on one-dimensional toy examples.
To reduce fluctuation, all experiments are repeated $1000$ times. The errors $|\xhat_t|^2$ and $|\yhat_t|^2$ are averaged over these $1000$ repetitions.


\paragraph{SGD with Polyak-Ruppert averaging}
We employ SGD with Polyak-Ruppert averaging~\eqref{eq:SGD-both} to minimize $f(x) = x^2 + \sin x$ with $(x_0,y_0) = (2,2)$, $\alpha_t = \alpha_0 (t+1)^{-a}$, $\beta_t = (t+1)^{-1}$, $\xi_t \overset{i.i.d.}{\sim} \NM (0, 1)$ and $a \in\{0.6, 0.7, 0.8, 0.9\}$. 
For each $a$, a grid search is performed on $\{10, 3, 1, 0.3, 0.1\}$ to find the optimal choice for $\alpha_0$ and each grid search is conducted with $10^4$ iterations.
The results are depicted in Figure~\ref{fig:PRave}.
The value of $a$ does not affect the convergence rate of $| \yhat_t |^2$, which is roughly $\Theta(1/t)$.

\begin{figure}[ht]
	\centering
	\includegraphics[width=\linewidth]{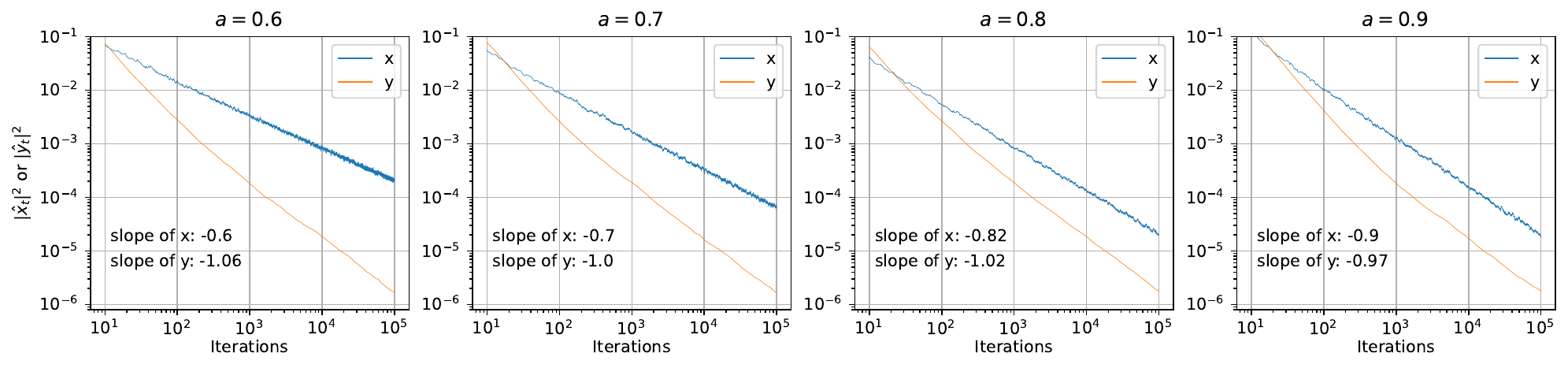}
	\caption{The convergence results for SGD with Polyak-Ruppert averaging~\eqref{eq:SGD-both}.
    We calculate the line slopes using data from the $10^4$ to $10^5$ iteration range.
	}
	\label{fig:PRave}
\end{figure}

\paragraph{SGD with momentum}
We employ SHB \eqref{eq:SHB}
to minimize $f(x) = x^2 + \sin x$ with $(x_0, y_0) = (2,2)$, $\alpha_t = \alpha_0 (t+1)^{-a}$, $\beta_t = \beta_0(t+1)^{-b}$,
$\xi_t \overset{i.i.d.}{\sim} \NM (0, 1)$ and $ (a, b) \in \{ 0.7, 0.6 \} \times \{ 1.0, 0.9 \} $.
For each $(a,b)$, a grid search is performed on $\{10, 3, 1, 0.3, 0.1\}^2$ to find the optimal choices for $(\alpha_0, \beta_0)$ and each grid search is conducted with $10^4$ iterations.
The results are depicted in Figure~\ref{fig:SHB}. 
Decoupled convergence is achieved.

\begin{figure}[ht]
	\centering
	\includegraphics[width=\linewidth]{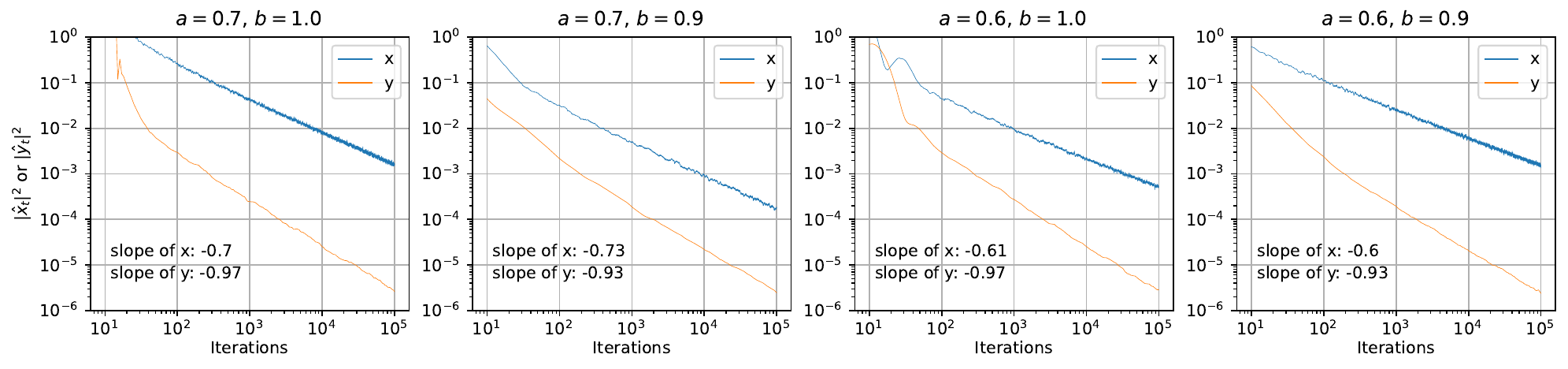}
	\caption{The convergence results for SHB~\eqref{eq:SHB}.
    We calculate the line slopes using data from the $10^4$ to $10^5$ iteration range.
	}
	\label{fig:SHB}
\end{figure}

\paragraph{Stochastic bilevel optimization}
Consider the following problem with $f(x,y) = (x + \tilde{h}_2(y))^2 + \sin ( x + \tilde{h}_2(y) )$ and $g(x,y) = (x + \tilde{h}_2 (y) )^2 + y^2 + \sin (y)$, with $\tilde{h}_2$ defined in \eqref{eq:expe-aux}:
\begin{equation}\label{eq:nume-bilevel}
\begin{aligned}
    & 
    \min_{y \in \RB } 
    (\tilde{x}^\star(y) + \tilde{h}_2 (y) )^2 + y^2 + \sin (y), \
    \\
    & 
    \text{s.t.} \
    \tilde{x}^\star(y) := \arg\min_{x \in \RB} 
    (x + \tilde{h}_2 (y) )^2 + \sin ( x + \tilde{h}_2(y) ).
\end{aligned}
\end{equation}
We apply two-time-scale SA, with $F$ and $G$ defined in \eqref{eq:bilevel-FG} 
to solve this problem,
with $(x_0, y_0) = (2,2)$, $\alpha_t = \alpha_0 (t+1)^{-a}$, $\beta_t = \beta_0(t+1)^{-b}$, $\xi_t, \psi_t \overset{{i.i.d.}}{\sim} \NM(0, 1)$ and $ (a, b) \in \{ 0.7, 0.6 \} \times \{ 1.0, 0.9 \} $.
For each $(a,b)$, a grid search is performed on $\{10, 3, 1, 0.3, 0.1\}^2$ to find the optimal choices for $(\alpha_0, \beta_0)$ and each grid search is conducted with $10^5$ iterations.
The results, depicted in Figure~\ref{fig:bilevel-unbias}, illustrate decoupled convergence for different $(a,b)$ pairs.

\begin{figure}[ht]
    \centering
    \includegraphics[width=\linewidth]{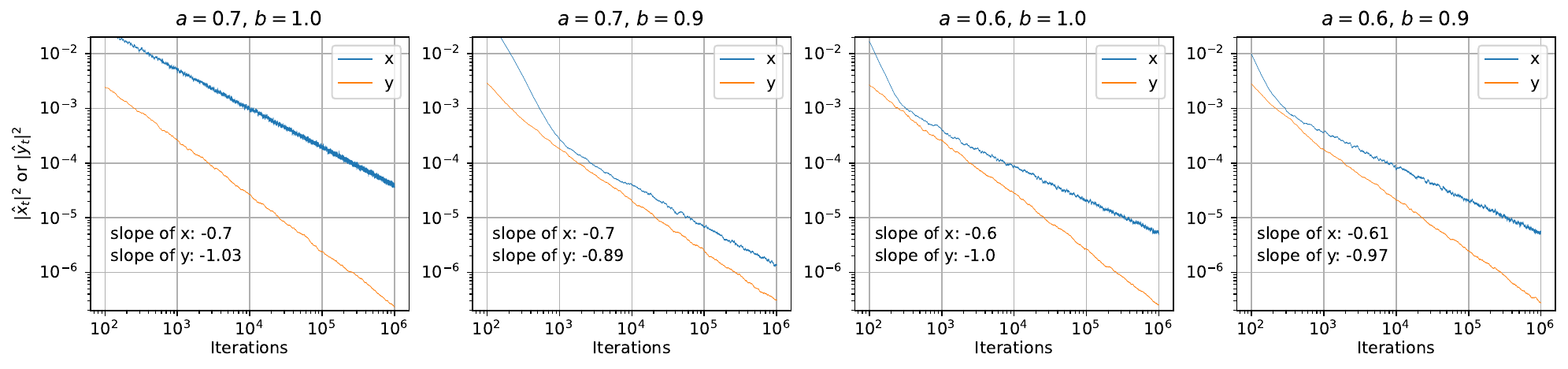}
    \caption{The convergence results for two-time-scale SA to solve \eqref{eq:nume-bilevel}.
    We calculate the line slopes using data from the $3 \times 10^5$ to $10^6$ iteration range.
    }
    \label{fig:bilevel-unbias}
\end{figure}

\subsection{Logistic Regression}
\label{sec:expe:decouple-logistic}

In this subsection, we consider the following \(\ell_2\)-regularized logistic regression problem:
\begin{equation}\label{eq:exp-logreg}
    \min_{x \in \mathbb{R}^d}
    f(x)
    :=
    \frac{1}{n}\sum_{i=1}^n
    \log\!\bigl(1+\exp(-b_i a_i^\top x)\bigr)
    + \frac{\lambda}{2}\|x\|^2,
\end{equation}
where \(a_i \in \mathbb{R}^d\) is the covariate and \(b_i \in \{-1,1\}\) is the label. 
We set the regularization parameter to \(\lambda=0.01\), so that the objective is strongly convex and therefore admits a unique minimizer.

\paragraph{Data generation}
We use a synthetic logistic regression model with dimension \(d=20\) and sample size \(n=1000\). 
To generate the dataset, we first sample a ground-truth parameter \(w_{\mathrm{true}} \in \mathbb{R}^{20}\) from a standard Gaussian distribution. 
Then we generate covariates \(a_i \sim \mathcal{N}(0,I_{20})\) independently, and produce binary labels according to the logistic model
\(\mathbb{P}(b_i = 1 \mid a_i)
    =
    \sigma(a_i^\top w_{\mathrm{true}}),\ 
    \sigma(z)=\frac{1}{1+e^{-z}}\).
Equivalently, \(b_i \in \{-1,1\}\) is sampled with
\(\mathbb{P}(b_i=1 \mid a_i)=\sigma(a_i^\top w_{\mathrm{true}}),\ 
\mathbb{P}(b_i=-1 \mid a_i)=1-\sigma(a_i^\top w_{\mathrm{true}})\).
After generation, the dataset is fixed throughout the whole experiment.

\paragraph{SGD with Polyak-Ruppert averaging}
We employ SGD with Polyak-Ruppert averaging~\eqref{eq:SGD-both} to minimize \eqref{eq:exp-logreg}, with $(x_0, y_0) = (0,0)$, $\alpha_t = \alpha_0 (t+1)^{-a}$, $\beta_t = (t+1)^{-1}$, and $a \in \{0.6,\,0.7,\,0.8,\,0.9\}$.
For each $a$, we perform a grid search over $\{10, 3, 1, 0.3, 0.1\}$ to determine the optimal choice of $\alpha_0$, and each grid search is conducted for $5\times10^4$ iterations.
The noise for the fast iterate comes from minibatch sampling with batch size $32$.
The results are shown in Figure~\ref{fig:logistic-ave}.
The y-axis represents the average of $\|\xhat_t\|^2$ or $\| \yhat_t \|^2$ over $100$ repetitions.
The figure shows that decoupled convergence can be achieved.
\begin{figure}[ht]
    \centering
    \includegraphics[width=\linewidth]{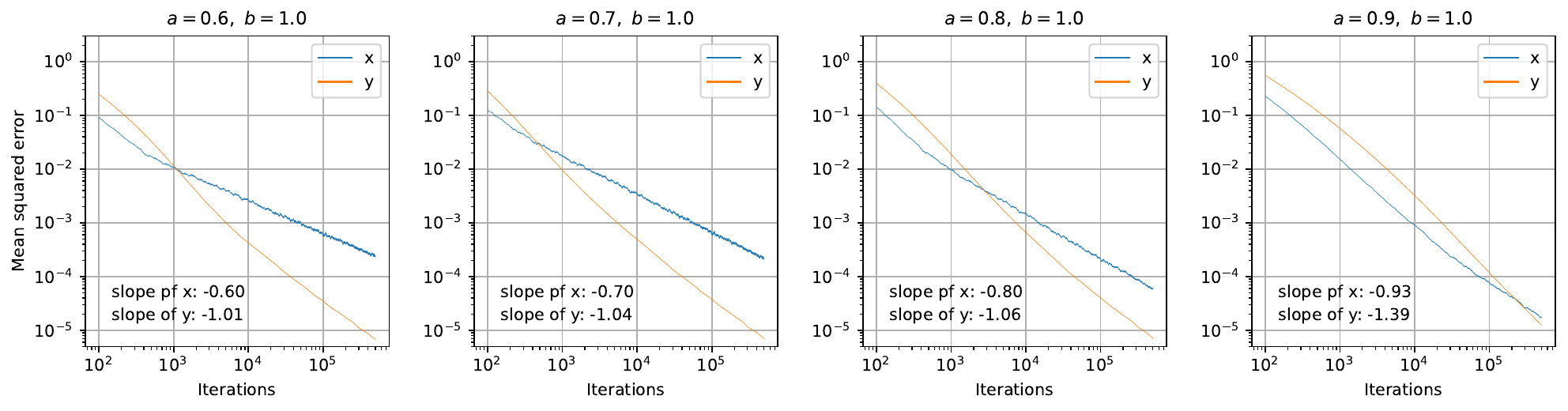}
    \caption{The convergence results for SGD with Polyak-Ruppert averaging~\eqref{eq:SGD-both} to solve \eqref{eq:exp-logreg}.
    We calculate the line slopes using data from the $10^5$ to $5\times10^5$ iteration range.
    }
    \label{fig:logistic-ave}
\end{figure}

\paragraph{SGD with momentum}
We employ SHB~\eqref{eq:SHB} to minimize \eqref{eq:exp-logreg}, with $(x_0, y_0) = (0,0)$, $\alpha_t = \alpha_0 (t+1)^{-a}$, $\beta_t = \beta_0(t+100)^{-b}$, and $ (a, b) \in \{ 0.7, 0.6 \} \times \{ 1.0, 0.9 \}$. 
For each $(a,b)$, we perform a grid search over $\{10, 3, 1, 0.3, 0.1\} \times \{1000,300,100,30,10\}$ to determine the optimal choices of $(\alpha_0, \beta_0)$, and each grid search is conducted for $5\times10^4$ iterations.
The noise for the fast iterate again comes from minibatch sampling with batch size $32$.
The results are shown in Figure~\ref{fig:logistic-shb}.
Unlike Figure~\ref{fig:logistic-ave}, the y-axis here represents the average of $\|x_t - x^\star\|^2$ or $\| \yhat_t \|^2$ over $100$ repetitions.
We plot $\|x_t - x^\star\|^2$ instead of $\|\xhat_t \|^2$ to reduce the computational cost, because computing $H(y) = \nabla f(y)$ requires passing through the entire dataset, whereas $x^\star = H(y^\star) = 0$.
As discussed in Remark~\ref{rema:comparison}, $\EB \| x_t - x^\star \|^2$ is also of order $\OM(\alpha_t)$.
Figure~\ref{fig:logistic-shb} again shows that decoupled convergence can be achieved.
\begin{figure}[ht]
    \centering
    \includegraphics[width=\linewidth]{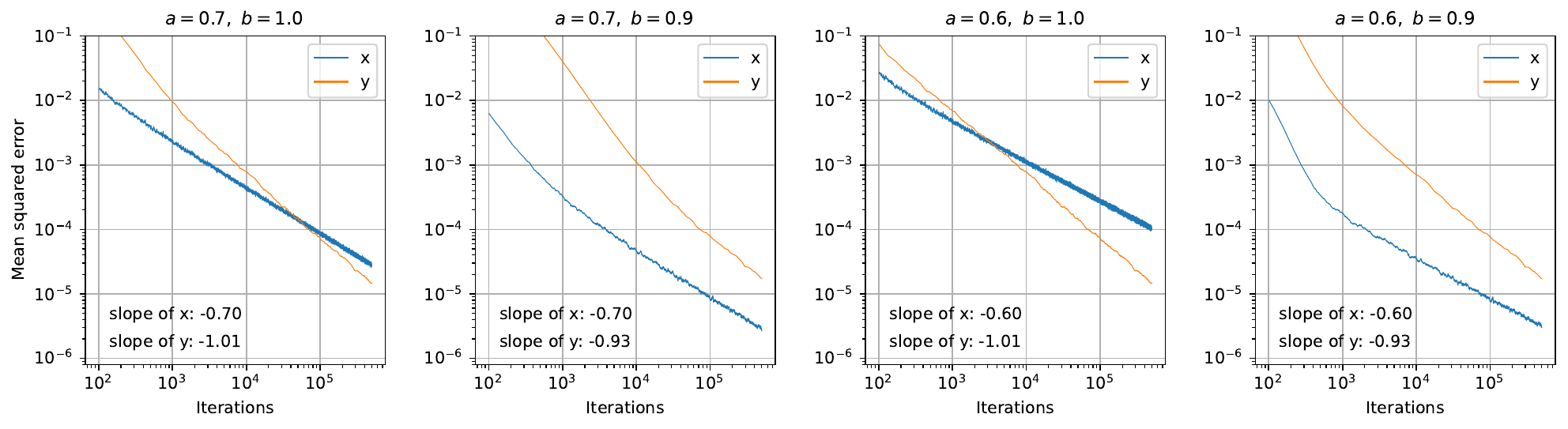}
    \caption{The convergence results for SHB~\eqref{eq:SHB} to solve \eqref{eq:exp-logreg}.
    We calculate the line slopes using data from the $10^5$ to $5\times10^5$ iteration range.
    }
    \label{fig:logistic-shb}
\end{figure}

\section{Concluding Remarks}\label{sec:conclude}
In this paper, we have investigated the potential for finite-time decoupled convergence in nonlinear two-time-scale SA under the strongly monotone condition, wherein the mean-square errors of different iterates depend solely on their respective step sizes.
Viewing the two-time-scale SA 
as an approximation of a two-loop procedure,
our primary focus is on the outer-loop iterate, i.e., the slow iterate $y_t$.

Under a nested local linearity assumption, we have established the first finite-time decoupled convergence for nonlinear two-time-scale SA with appropriate step size selection.
This decoupled convergence offers greater flexibility in choosing the step size for the fast iterate $x_t$, without impacting the convergence rate of the main focus, the slow iterate $y_t$.
Our analytical framework advances the approach for the linear operators in \citet{kaledin2020finite} to adapt complexities introduced by non-linearity.
In particular, we have derived a refined characterization of the matrix cross term, surpassing previous asymptotic results~\citep{mokkadem2006convergence}, and applied fourth-order moment convergence rates to manage higher-order error terms induced by local linearity.
In addition, we provide an example showing that decoupled convergence may fail even when the fast-time-scale update is linear, as long as the slow-time-scale update remains nonlinear. 
Together with our upper bound, this lower-bound result helps clarify when decoupled convergence should be expected in the nonlinear setting. 
It also sheds further light on the approximation perspective in \eqref{eq:two-loop-approx}: even if two-time-scale SA can be viewed as solving $F(x,y)=0$ (for fixed $y$) and $G(H(y),y)=0$, the original form of $G(x,y)$ may still affect the convergence behavior. We hope that this observation may also be useful in inspiring future algorithm design.

Despite the progress made in our paper, several avenues for future research remain. 
First, our results could be extended to include scenarios with Markovian noise or non-strongly monotone operators, broadening the applicability of our approach. 
Second, investigating the asymptotic trajectory behavior and developing online statistical inference methods for two-time-scale SA based on our non-asymptotic convergence results would be interesting future directions.
Finally, another natural direction is to generalize our framework to more complex algorithms, such as those involving multiple iterates or multiple time scales.

\bibliography{bib/twotime, bib/distributed, bib/federated, bib/optimization, bib/privacy, bib/sde, bib/stat, bib/SP}
\bibliographystyle{plainnat}

\newpage
\appendix
\setcounter{section}{1}

\section{Omitted Details in Section~\ref{sec:prelim:assump}}
\label{sec:assump-append}

In this section, we present the omitted details in Section~\ref{sec:prelim:assump}.

First,
we give the detailed definition of the constants in Assumption~\ref{assump:stepsize-new}.
\begin{align}\label{assump:stepsize:constants}
	\begin{split}
	\iota_1 &=   \frac{\mu_F}{4L_F^2} \wedge \frac{1}{12 \mu_F}, \ 
	\iota_2 = \frac{\mu_G}{\LGx^2} \wedge \frac{1}{14 \mu_G},
	\ \kappa =  \frac{\mu_{F}\mu_{G}}{( 28 d_x \vee 200 \LGy ) L_H \LGx
	} \wedge \frac{\mu_F}{5 \mu_G}, \\
	\rho & = \frac{\mu_{F} }{
		(16 d_x \vee 200) L_H^2 \LGx^2 }.
		\end{split}
\end{align}

As long as $\beta_t / \alpha_t = o(1)$ and $\beta_{t} = o(1)$, e.g., $\alpha_{t} \sim \alpha_{0} t^{-a}$ and $\beta_{t} \sim \beta_{0} \sim t^{-b}$, Assumption~\ref{assump:stepsize-new} will hold for sufficiently large $t$, regardless of the initial values $
\alpha_{0}$ and $\beta_{0}$.
However, if $\beta_{t} / \alpha_{t}$ remains constant, $\alpha_{0}$ and $\beta_{0}$ must be appropriately chosen to satisfy  Assumption~\ref{assump:stepsize-new}.
This comparison highlights the advantage of using different time scales over the single-time-scale case in terms of flexible parameter selection.

In the remaining part,
the proofs of Propositions~\ref{prop:holder-equiv} and \ref{prop:ensure-linearity} are given in Appendices~\ref{proof:holder-equiv} and \ref{proof:de:ensure-linearity}, respectively.
The verification of Assumptions~\ref{assump:smoothH} and \ref{assump:near-linear} is provided in Appendix~\ref{sec:verify-assump}.
For completeness, we repeat Propositions~\ref{prop:holder-equiv} and \ref{prop:ensure-linearity} below.

\begin{prop}[Proposition~\ref{prop:holder-equiv}]\label{prop:holder-equiv-rep}
	Under Assumption~\ref{assump:H:holder-assump},
	Assumption~\ref{assump:smoothH} holds with $S_H = \frac{\Hholder}{1+\Hsmooth} $;
	under Assumption~\ref{assump:smoothH}, Assumption~\ref{assump:H:holder-assump} holds with $\Hholder = 2^{1-\Hsmooth} \sqrt{1+\Hsmooth} \left( \frac{1+\Hsmooth}{\Hsmooth} \right)^\frac{\Hsmooth}{2} S_H$.\footnote{We make the contention that $\left( \frac{1+\Hsmooth}{\Hsmooth} \right)^\frac{\Hsmooth}{2} = 1$ if $\Hsmooth = 0$. }
	For this equivalence, we do not require $\Hsmooth \ge 0.5$.
\end{prop}

\begin{prop}[Proposition~\ref{prop:ensure-linearity}]
	\label{prop:ensure-linearity-rep}
	Suppose that Assumptions~\ref{assump:smooth:FH} --~\ref{assump:smoothH} hold. 
	\begin{enumerate}[(i)]
		\item 
        If Assumption~\ref{assump:local-linear} holds, then $\| A_{11} \| \le L_F$, $ \| A_{21} \| \le \LGx $, $\frac{A_{11}+A_{11}^\top}{2} \succeq \mu_F I$, and $\nabla H(y^\star) =  - A_{11}^{-1} A_{12}$.
		If we further assume $\Hsmooth \ge \Fsmooth \vee \Gsmooth$, 	then Assumption~\ref{assump:near-linear} holds with parameters
		\begin{align*}
			B_1 & = A_{11},\ \ B_2 = A_{21},\ \ B_3 = A_{22}-A_{21}A_{11}^{-1}A_{12},\\
			S_{B,F} &= S_{A,F} + L_F (S_H \vee 2 L_H), \ 
			S_{B,G} = S_{A,G} + \LGx (S_H \vee 2 L_H).
		\end{align*}
		\item 
        If Assumption~\ref{assump:near-linear} holds, then $\|B_1\| \le L_F, \| B_2\| \le \LGx , \|B_3\|\le \LGy,  \frac{B_1 + B_1^\top}{2} \succeq \mu_F I$, and $\frac{B_3 + B_3^\top}{2} \succeq \mu_G I$.
	\end{enumerate}
\end{prop}

\subsection{Proof of Proposition~\ref{prop:holder-equiv}}
\label{proof:holder-equiv}

\begin{proof}[Proof of Proposition~\ref{prop:holder-equiv}]
The proof is divided into two parts.
The first part is straightforward while the second part follows a similar procedure as the proof of \citet[Theorem~4.1]{berger2020quality}

\textbf{Assumption~\ref{assump:H:holder-assump} $\implies$ Assumption~\ref{assump:smoothH} }.
For any unit vector $e \in \RB^{d_x}$,
we define $f(y) = \inner{e}{H(y)}$. Then we have $ f'(y) = \nabla H(y)^\top e$ and
\begin{align*}
    & \quad f(y_1) - f(y_2) - \inner{ f'(y_2) }{y_1 - y_2} 
    = \int_0^1 \inner{ f' ( y_2 + t(y_1 - y_2) ) - f' (y_2) }{ y_1 - y_2 } \mathrm{d} t
   \\
   & 
    \le \| y_1 - y_2 \| \int_0^1 \| \nabla H ( y_2 + t(y_1 - y_2) ) - \nabla H (y_2) \|  \mathrm{d} t 
   \\
   &
	\overset{\eqref{assump:H:holder}}{\le} \Hholder \| y_1 - y_2 \|^{1+\Hsmooth} \int_0^1  t^{\Hsmooth} \mathrm{d} t 
    \le \frac{\Hholder}{1+\Hsmooth} \| y_1 - y_2 \|^{1+\Hsmooth}.
\end{align*}
From the definition of $f$, we have
$
    \inner{e}{ H(y_1) - H(y_2) - \nabla H(y_2) (y_1 - y_2) }
    \le \frac{\Hholder}{1+\Hsmooth} \| y_1 - y_2 \|^{1+\Hsmooth}.
    $
Since $e$ is an arbitrary unit vector,  Assumption~\ref{assump:smoothH} holds with $S_H = \frac{\Hholder}{1+\Hsmooth}$.

\textbf{Assumption~\ref{assump:smoothH} $\implies$ Assumption~\ref{assump:H:holder-assump} }.
Under Assumption~\ref{assump:smoothH}, we have for  any $y_1, y_2 \in \RB^{d_y}$ and any unit vector $e \in \RB^{d_x}$, it holds that
\begin{align}
\label{prop:holder-equiv:eq1}
    - S_H \| y_1 - y_2 \|^{1+\Hsmooth}
    \le \inner{e}{ H(y_1) - H(y_2) - \nabla H(y_2) (y_1 - y_2) }
    \le S_H \| y_1 - y_2 \|^{1+\Hsmooth}.
\end{align}
Now we fix two arbitrary points $\bar{y}_1$ and $\bar{y}_2 \in \RB^{d_y}$ with $\bar{y}_1 \neq \bar{y}_2$.
Without loss of generality, we assume $H(\bar{y}_1) = 0$ and $\nabla H( \bar{y}_1 ) = 0$.
Otherwise, we could replace $H(y)$ by $\widetilde{H} (y) =: H(y) - H(\bar{y}_1) - \nabla H(\bar{y}_1) (y-\bar{y}_1)$ in \eqref{prop:holder-equiv:eq1}.
For any $z_1, z_2 \in \RB^{d_y}$, setting $(y_1, y_2) = (z_1, \bar{y}_1)$ and $(z_1, \bar{y}_2)$ in \eqref{prop:holder-equiv:eq1} yields
\begin{align}
\label{prop:holder-equiv:eq2}
    - S_H \| z_1 - \bar{y}_1 \|^{1+\Hsmooth}
    \le \inner{e}{ H(z_1) }
    \le \inner{e}{ H(\bar{y}_2) + \nabla H( \bar{y}_2 ) (z_1 - \bar{y}_2) } + S_H \| z_1 - \bar{y}_2 \|^{1+\Hsmooth}.
\end{align}
Similarly, setting $(y_1, y_2) = (z_2, \bar{y}_1)$ and $(z_2, \bar{y}_2)$ in \eqref{prop:holder-equiv:eq1} yields
\begin{align}
\label{prop:holder-equiv:eq3}
    S_H \| z_2 - \bar{y}_1 \|^{1+\Hsmooth}
    \ge \inner{e}{H(z_2)}
    \ge \inner{e}{ H(\bar{y}_2) + \nabla H(\bar{y}_2) (z_2 - \bar{y}_2) } - S_H \| z_2 - \bar{y}_2 \|^{1+\Hsmooth}.
\end{align}
Subtracting the rightmost and leftmost sides of \eqref{prop:holder-equiv:eq2} from \eqref{prop:holder-equiv:eq3}, we obtain
\begin{align*}
    \inner{e}{ \nabla H(\bar{y}_2) (z_2 - z_1) }
    \le S_H \left( \| z_1 - \bar{y}_1 \|^{1+\Hsmooth} + \| z_2 - \bar{y}_1 \|^{1+\Hsmooth} + \| z_1 - \bar{y}_2 \|^{1+\Hsmooth} + \| z_2 - \bar{y}_2 \|^{1+\Hsmooth} \right).
\end{align*}
Let $\tilde{e} \in \RB^{d_y}$ be an arbitrary unit vector, $z_1 = \frac{ \bar{y}_1 + \bar{y}_2 - \alpha \tilde{e} }{2} $
and $z_2 = \frac{ \bar{y}_1 + \bar{y}_2 + \alpha \tilde{e} }{2} $ with $\alpha$ a positive constant determined later.
Then we have
\begin{align}
\label{prop:holder-equiv:eq4}
    \alpha \inner{e}{ \nabla H(\bar{y}_2)\, \tilde{e} }
    \le 2 S_H \left( \left\| \frac{ \bar{y}_2 - \bar{y}_1 + \alpha \tilde{e} }{2}  \right\|^{1+\Hsmooth} + \left\| \frac{ \bar{y}_2 - \bar{y}_1 - \alpha \tilde{e} }{2}  \right\|^{1+\Hsmooth} \right).
\end{align}
Since the function $h(x) = x^\frac{1+\Hsmooth}{2}$ is concave, Jensen's inequality together with the parallelogram identity implies
\begin{align}
    \left\| \frac{ \bar{y}_2 - \bar{y}_1 + \alpha \tilde{e} }{2}  \right\|^{1+\Hsmooth} + \left\| \frac{ \bar{y}_2 - \bar{y}_1 - \alpha \tilde{e} }{2}  \right\|^{1+\Hsmooth}
    & \le 2^\frac{1-\Hsmooth}{2} \left( \frac{ \| \bar{y}_2 - \bar{y}_1 \|^2 + \alpha^2 }{2} \right)^\frac{1+\Hsmooth}{2}, \label{prop:holder-equiv:eq5}
\end{align}
Plugging \eqref{prop:holder-equiv:eq5} into \eqref{prop:holder-equiv:eq4}, dividing both sides by $\alpha$ and setting $\alpha = k \left\| \bar{y}_2 - \bar{y}_1 \right\| $, we obtain
$
    \inner{e}{ \nabla H(\bar{y}_2) \, \tilde{e} }
    \le 2^{1-\Hsmooth} S_H \frac{ (1+k^2)^\frac{1+\Hsmooth}{2} }{k} \| \bar{y}_2 - \bar{y}_1 \|^\Hsmooth.
    $

If $\Hsmooth = 0$, letting $k \to \infty$ yields
$
    \inner{e}{ \nabla H(\bar{y}_2) \, \tilde{e} }
    \le 2^{1-\Hsmooth} S_H \| \bar{y}_2 - \bar{y}_1 \|^\Hsmooth.
    $
If $\Hsmooth > 0$, setting $k = 1/ \sqrt{\Hsmooth}$ yields
$
    \inner{e}{ \nabla H(\bar{y}_2) \, \tilde{e} }
    \le 2^{1-\Hsmooth} S_H \sqrt{1+\Hsmooth} \left( \frac{1+\Hsmooth}{\Hsmooth} \right)^\frac{\Hsmooth}{2} \| \bar{y}_2 - \bar{y}_1 \|^\Hsmooth.
    $
Since $e$ and $\tilde{e}$ are two arbitrary unit vectors and $\nabla H(\bar{y}_1) = 0$, summarizing the two cases shows that Assumption~\ref{assump:H:holder-assump} holds with $\Hholder = 2^{1-\Hsmooth} \sqrt{1+\Hsmooth} \left( \frac{1+\Hsmooth}{\Hsmooth} \right)^\frac{\Hsmooth}{2} S_H$.
\end{proof}

\subsection{Proof of Proposition~\ref{prop:ensure-linearity}}
\label{proof:de:ensure-linearity}

\begin{proof}[Proof of Proposition~\ref{prop:ensure-linearity}]
We prove the results step by step.
\vspace{0.2cm}
		
\textbf{Proof of Part~\ref{prop:ensure-linearity-1}}.	
We first prove $\| A_{11} \| \le L_F$ and $\| A_{21} \| \le \LGx$.
Setting $y = y^\star$ in \eqref{assump:local-linear-F} yields
$
	\| F(x,y^\star) - A_{11} (x - x^\star) \| \le \SAF \| x - x^\star \|^{1+\Fsmooth}.
    $
Then by Condition~\eqref{assump:smooth:FH:ineqF},
we have
$
	L_F \| x - H(y^\star) \| 
	\ge \| F(x,y^\star) \|
	\ge \| A_{11} (x - x^\star) \| - \| F(x,y^\star) - A_{11} (x - x^\star) \|.
    $
Recall that $x^\star = H(y^\star)$.
It follows that
\begin{align*}
	\| A_{11} (x - x^\star) \| \le L_F \| x - x^\star \| + \SAF \| x - x^\star \|^{1+\Fsmooth}.
\end{align*}
Let $x - x^\star = t e_1$ where $e_1 \in \RB^{d_x}$ is an arbitrary unit vector and $t > 0$.
Dividing both sides by $t$ and letting $t \to 0$ yields $\| A_{11} e_1 \| \le L_F$.
Since $e_1$ is arbitrary, we obtain $\| A_{11} \| \le L_F$.
Setting $y = y^\star$ in \eqref{assump:local-linear-G} and applying Condition~\eqref{assump:G:smooth},
we can obtain
$
	\| A_{21} (x - x^\star) \| \le \LGx \| x - x^\star \| + \SAG \| x - x^\star \|^{1+\Gsmooth}.
    $
Similarly, we can obtain $\| A_{21} \| \le \LGx$.

Next, we prove $ A_{11} \nabla H(y^\star) +  A_{12} = 0$.
Recall that the first inequality of Assumption~\ref{assump:local-linear} is
$
	\|F(x, y) - A_{11}(x -x^{\star}) - A_{12} (y-y^{\star}) \| \le S_{A,F} \left( \|x - x^{\star}\|^{1+\Fsmooth} + \|y-y^{\star}\|^{1+\Fsmooth} \right).
    $
Since $F(H(y) , y) = 0$,
setting $x = H(y)$ yields
\begin{align*}
	\| A_{11} (H(y) - H(y^\star)) + A_{12} (y-y^\star) \| \le S_{A,F} (1 + L_H^{1+\Fsmooth}) \|y-y^{\star}\|^{1+\Fsmooth}.
\end{align*}
By Assumption~\ref{assump:smoothH} and the triangle inequality,
we have
\begin{align*}
	\| (A_{11} H(y^\star) + A_{12} )  (y-y^\star) \| \le  S_{A,F} (1 + L_H^{1+\Fsmooth}) \|y-y^{\star}\|^{1+\Fsmooth} + L_F S_H \|y-y^{\star}\|^{1+\Hsmooth}. 
\end{align*}
Setting $y - y^* = t e_2$ for an arbitrary unit vector $e_2 \in \RB^{d_y}$ and letting $t \to 0$ yields $\|  (A_{11} H(y^\star) + A_{12} ) e_2 \| = 0 $.
Since $e_2$ is arbitrary, we have $ A_{11} \nabla H(y^\star) +  A_{12} = 0$.

In the following, we prove that the first inequality of Assumption~\ref{assump:near-linear} holds.
Since $\| A_{11} \| \le L_F$ and $\| A_{21} \| \le \LGx$.
we have
	\begin{align*}
		& \quad \| F(x, y) - A_{11}(x - H(y))\| \\
		&\le 	\|F(x, y) - A_{11}(x -x^{\star}) - A_{12} (y-y^{\star}) \|  \notag 
		+  \|A_{11} \left( H(y) - H(y^{\star}) - \nabla H(y^{\star})(y-y^{\star}) \right) \|  \notag \\
		& \overset{(a)}{\le} S_{A,F} \left( \|x - x^{\star}\|^{1+\Fsmooth} + \|y-y^{\star}\|^{1+\Fsmooth} \right) \notag 
		+ L_F S_H \|y-y^{\star} \| \cdot \min \left\{ \| y - y^\star \|^{\Hsmooth}, R_H \right\},
	\end{align*}
	where (a) also uses \eqref{assump:H:smooth-main} and $R_H = \frac{2 L_H}{S_H}$.
	If $\|y - y^\star \| \le 1$, then we have
	\begin{align*}
		\| F(x, y) - A_{11}(x - H(y))\|
		& \le  S_{A,F} \left( \|x - x^{\star}\|^{1+\Fsmooth} + \|y-y^{\star}\|^{1+\Fsmooth} \right) 
		+ L_F S_H \| y - y^\star \|^{1+\Hsmooth} \\
		& \le S_{A,F} \|x - x^{\star}\|^{1+\Fsmooth} + (S_{A,F} + L_F S_H) \| y - y^\star \|^{1 + \Fsmooth}.
	\end{align*}
	Otherwise, we have
	\begin{align*}
		\| F(x, y) - A_{11}(x - H(y))\|
		& \le  S_{A,F} \left( \|x - x^{\star}\|^{1+\Fsmooth} + \|y-y^{\star}\|^{1+\Fsmooth} \right) 
		+ 2 L_F L_H \| y - y^\star \| \\
		& \le S_{A,F}  \|x - x^{\star}\|^{1+\Fsmooth} + (S_{A,F} + 2 L_F L_H) \|y - y^\star \|^{1+\Fsmooth}. 
	\end{align*}
	Combining the two cases yields
	\begin{align}
		\label{eq:linear-F}
		\| F(x, y) - A_{11}(x - H(y))\|
		\le S_{A,F}  \|x - x^{\star}\|^{1+\Fsmooth} + (S_{A,F} + L_F \max\{ S_H, 2L_H \}) \|y - y^\star \|^{1+\Fsmooth}. 
	\end{align}
	
Before proving the second inequality in Assumption~\ref{assump:local-linear}, we first check the lower bound for $ \frac{A_{11} + A_{11}^\top}{2}$.
	By Assumption~\ref{assump:sm:F} and~\eqref{eq:linear-F}, it follows that
	\begin{align*}
		\mu_{F} \|x-x^{\star}\|^2 \le \left\langle x - x^{\star}, F(x,y^{\star}) \right\rangle 
		&\le \left\langle x - x^{\star}, A_{11}(x - x^{\star})  \right\rangle  + \SBF \|x - x^{\star}\|^{2+\Fsmooth} \\
		&= \left\langle x - x^{\star}, \frac{A_{11}+A_{11}^\top}{2}\left(x - x^{\star}\right)  \right\rangle  + \SBF \|x - x^{\star}\|^{2+\Fsmooth}.
	\end{align*}
	Dividing $ \|x-x^{\star}\|^2 $ on the both sides of the last inequality and letting $x$ converge to $x^{\star}$ along the direction $v$, we then have that $\mu_F \le  v^\top  \frac{A_{11}+A_{11}^\top}{2} v$ for any unit one vector $v$, which implies that $ \frac{A_{11} + A_{11}^\top}{2} \succeq \mu_F I$.	This condition also implies $A_{11}$ is non-singular and consequently $H(y^\star) = - A_{11}^{-1} A_{12}$.
	
Finally, we prove the second inequality in Assumption~\ref{assump:local-linear}. We have
	\begin{align*}
		&\| G(x, y) - A_{21}(x - H(y))  - (A_{22}-A_{21}A_{11}^{-1}A_{12})(y-y^{\star}) \|\\
		&\le \|G(x, y) - A_{21}(x -x^{\star}) - A_{22}(y-y^{\star}) )\| 
		+ \|A_{21}\left(H(y) - x^{\star} + A_{11}^{-1} A_{12}(y-y^{\star})\right)\|\\
		&\le S_{A,G} \left( \|x - x^{\star}\|^{1+\Gsmooth} + \|y-y^{\star}\|^{1+\Gsmooth} \right)  + \LGx S_H\|y-y^{\star}\|\cdot \min \{ \| y - y^\star \|^\Hsmooth, R_H \}.
	\end{align*}
	Similar to the proof of \eqref{eq:linear-F}, we can obtain
	\begin{align*}
		& \| G(x, y) - A_{21}(x - H(y))  - (A_{22}-A_{21}A_{11}^{-1}A_{12})(y-y^{\star}) \|\\
		& \quad \le  S_{A,G} \| x - x^\star \|^{1+\Gsmooth} + (S_{A,G} + \LGx \max\{ S_H, 2L_H \}) \|y - y^\star \|^{1+\Gsmooth}. 
	\end{align*}
	Hence, the near linearity conditions in Assumption~\ref{assump:near-linear} follow with
	$B_1 = A_{11}$, $B_2 = A_{21}$, $B_3 = A_{22}-A_{21}A_{11}^{-1}A_{12}$,
	$\SBF = S_{A,F} + L_F (S_H \vee 2 L_H)$ and $\SBG = S_{A,G} + \LGx (S_H \vee 2 L_H)$.
\vspace{0.1cm}
	
\textbf{Proof of Part~\ref{prop:ensure-linearity-2}}.	
Setting $y = y^\star$ in \eqref{assump:near-linear-F} and applying Condition~\eqref{assump:smooth:FH:ineqF}, we obtain
\begin{align*}
	\| B_1 (x - H(y^\star)) \| \le L_F \| x - H(y^\star) \| + \SBF \| x - H(y^\star) \|^{1+\Fsmooth}.
\end{align*}
Similar to the proof in Part~\ref{prop:ensure-linearity-1}, we can obtain $\| B_1 \| \le L_F$.
Setting $y = y^\star$ in \eqref{assump:near-linear-G} and applying Condition~\eqref{assump:G:smooth}, we can obtain $\| B_2 \| \le \LGx$.
Setting $x = H(y)$ in \eqref{assump:near-linear-G} and applying Condition~\eqref{assump:G:smooth2},
we can obtain $\|B_3\| \le \LGy$.
The proof of the lower bounds for $ \frac{B_1 + B_1^\top}{2}$ and  $ \frac{B_3 + B_3^\top}{2}$ is also similar ot that of the lower bound for $ \frac{A_{11} + A_{11}^\top}{2}$ in Part~\ref{prop:ensure-linearity-1}.
\end{proof}

\subsection{Verification of Local Linearity Conditions on Examples}
\label{sec:verify-assump}
In this subsection, we present the verification of local linearity conditions in Assumptions~\ref{assump:smoothH} and \ref{assump:near-linear} on the examples in Section~\ref{sec:prelim:app}.

First, we would like to emphasize an important simplification in verifying Assumption~\ref{assump:near-linear}.
Under the Lipschitz continuity of $F$, $G$, and $H$ in Assumptions~\ref{assump:smooth:FH} and \ref{assump:G}, together with the uniform local linearity of $H$ in Assumption~\ref{assump:smoothH}$,$ if either \eqref{assump:near-linear-F}--\eqref{assump:near-linear-G} or \eqref{assump:local-linear-F}--\eqref{assump:local-linear-G} holds in a neighborhood of the solution $(x^\star,y^\star)$, then these inequalities automatically extends to all $(x,y)$, thereby yielding Assumption~\ref{assump:near-linear}; see \citet[Propositions~A.2 and~A.3]{han2024decoupled}.
Therefore, for the examples, it is sufficient to verify \eqref{assump:near-linear-F}--\eqref{assump:near-linear-G} or \eqref{assump:local-linear-F}--\eqref{assump:local-linear-G} only locally around $(x^\star,y^\star)$.

For Examples~\ref{exmp:PRave}--\ref{exmp:lagrange}, we assume that $f \colon \RB^{d_x} \to \RB$ is strongly convex, with unique minimizer $x_o^{\star} = \arg\min_{x \in \RB^{d_x}} f(x)$.
\vspace{0.1cm}

\textbf{Example~\ref{exmp:PRave}: SGD with Polyak-Ruppert averaging}.
SGD with Polyak-Ruppert averaging in \eqref{eq:SGD-both} an example of two-time-scale SA \eqref{alg:xy} with $F(x,y)=\nabla f(x)$, $G(x,y)=y-x$, and $H(y)\equiv x^\star$.
It follows that $G(H(y),y)=y-x^\star$, and hence $y^\star=x^\star=x_o^\star$.
If $\nabla f(x)$ is Lipschitz continuous, then $F$, $G$, and $H$ are all Lipschitz continuous.
It therefore remains only to verify \eqref{assump:near-linear-F} and \eqref{assump:near-linear-G} locally.
Since $H$ is constant, Assumption~\ref{assump:smoothH} holds trivially with $S_H=0$ and $\Hsmooth=1$.
If $\nabla f(x)$ is differentiable in a neighborhood of $x_o^\star$ and $\nabla^2 f(x)$ is $\delta$-H\"older continuous there, then
$\|F(x,y)-\nabla^2 f(x^\star)(x-H(y))\|=\|\nabla f(x)-\nabla^2 f(x^\star)(x-x_o^\star)\|=\OM(\|x-x^\star\|^{1+\delta})$
in that neighborhood, while
$\|G(x,y)+(x-H(y))-(y-y^\star)\|=0$.
Thus, Assumption~\ref{assump:near-linear} holds with $B_1=\nabla^2 f(x_o^\star)$, $B_2=-I$, $B_3=I$, $\Fsmooth=\delta$, $\SBG=0$, and $\Gsmooth=1$.
\vspace{0.1cm}

\textbf{Example~\ref{exmp:momentum}: SGD with momentum}.
SGD with momentum in \eqref{eq:SHB} is an example of two-time-scale SA \eqref{alg:xy} with $F(x,y)=x-\nabla f(y)$, $G(x,y)=x$, and $H(y)=\nabla f(y)$.
It follows that $G(H(y),y)=\nabla f(y)$, $y^\star=x_o^\star$, and $x^\star=0$.
If $\nabla f(x)$ is Lipschitz continuous, then $F$, $G$, and $H$ are all Lipschitz continuous.
It therefore remains only to verify \eqref{assump:near-linear-F} and \eqref{assump:near-linear-G} locally.
If $\nabla^2 f(x)$ is $\delta$-H\"older continuous, then Assumption~\ref{assump:smoothH} holds with $\Hsmooth=\delta$.
Moreover, $\|F(x,y)-(x-H(y))\|=0$, and
$\|G(x,y)-(x-H(y))-\nabla^2 f(x_o^\star)(y-y^\star)\|=\OM(\|y-y^\star\|^{1+\delta})$.
Therefore, Assumption~\ref{assump:near-linear} holds with $B_1=I$, $B_2=I$, $B_3=\nabla^2 f(x_o^\star)$, $\SBF=0$, $\Fsmooth=1$, and $\Gsmooth=\delta$.
\vspace{0.1cm}

\textbf{Example~\ref{exmp:lagrange}: Constrained optimization with Lagrange multipliers}.
The algorithm in \eqref{eq:lagrange} is an example of two-time-scale SA \eqref{alg:xy} with $F(x,y)=\nabla f(x)+A^\top y$, $G(x,y)=-Ax+b$, and $H(y)=[\nabla f]^{-1}(-A^\top y)$.
If $\nabla f(x)$ and $[\nabla f]^{-1}(x)$ are Lipschitz continuous, then $F$, $G$, and $H$ are all Lipschitz continuous.
For this example, it is more convenient to first verify \eqref{assump:local-linear-F} and \eqref{assump:local-linear-G} locally.
If $\nabla^2 f(x)$ is $\delta$-H\"older continuous, then \eqref{assump:local-linear-F} and \eqref{assump:local-linear-G} hold locally with $A_{11}=\nabla^2 f(x^\star)$, $A_{12}=A^\top$, $A_{21}=-A$, $A_{22}=0$, $\SAG=0$, and $\Fsmooth=\Gsmooth=\delta$.
Since $H(y)=[\nabla f]^{-1}(-A^\top y)$, we have $\nabla H(y)=-[\nabla^2 f(H(y))]^{-1}A^\top$.
We now show that $\nabla H(y)$ is also $\delta$-H\"older continuous.

Because $f$ is strongly convex, there exists $m>0$ such that $\nabla^2 f(x)\succeq mI$, and hence $\|[\nabla^2 f(x)]^{-1}\|\le 1/m$.
Therefore,
$\|\nabla H(y_1)-\nabla H(y_2)\|
\le
\|A\|\,\|[\nabla^2 f(H(y_1))]^{-1}-[\nabla^2 f(H(y_2))]^{-1}\|$.
Using the matrix inverse identity $M^{-1}-N^{-1}=M^{-1}(N-M)N^{-1}$, we obtain
$\|M^{-1}-N^{-1}\|\le \|M^{-1}\|\,\|N^{-1}\|\,\|M-N\|$.
Taking $M=\nabla^2 f(H(y_1))$ and $N=\nabla^2 f(H(y_2))$ yields
\begin{equation}\label{eq:exmp-lagrange-matrix-inverse}
    \|
    [\nabla^2 f(H(y_1))]^{-1}
    -
    [\nabla^2 f(H(y_2))]^{-1}
    \|
    \le
    m^{-2}\|\nabla^2 f(H(y_1))-\nabla^2 f(H(y_2))\|.
\end{equation}
Combining the $\delta$-H\"older continuity of $\nabla^2 f$ with the Lipschitz continuity of $H$, we obtain
$\|\nabla H(y_1)-\nabla H(y_2)\| \lesssim \|y_1-y_2\|^\delta$.
By Proposition~\ref{prop:holder-equiv}, Assumption~\ref{assump:smoothH} holds with $\Hsmooth=\delta$.
Therefore, by Proposition~\ref{prop:ensure-linearity} and \citet[Proposition~A.3]{han2024decoupled}, Assumption~\ref{assump:near-linear} holds with $B_1=A_{11}=\nabla^2 f(x^\star)$, $B_2=A_{21}=-A$, $B_3=A_{22}-A_{21}A_{11}^{-1}A_{12}=A[\nabla^2 f(x^\star)]^{-1}A^\top$, and $\Fsmooth=\Gsmooth=\delta$.
For the verification of the strong monotonicity condition, please refer to \citet{chandak2025non}.
\vspace{0.1cm}

\textbf{Example~\ref{exmp:bilevel}: Stochastic bilevel optimization}.
With a slight abuse of notation, consider the unconstrained bilevel optimization problem in \eqref{eq:prob:bilevel}.
Suppose that $\ell(y)$ is strongly convex and that $f(x,y)$ is strongly convex in $x$ for each fixed $y$.
To apply two-time-scale SA, $F(x,y)$ and $G(x,y)$ are of the following form
\begin{equation*}
    \begin{aligned}
        F(x,y) &= \nabla_x f(x,y),\\
        G(x,y) &= \nabla_y g(x,y)-\nabla_{yx}^2 f(x,y)[\nabla_{xx}^2 f(x,y)]^{-1}\nabla_x g(x,y).
    \end{aligned}
\end{equation*}
Then $H(y)=\tilde{x}^\star(y)$ and the solution $(x^\star,y^\star)$ satisfies $y^\star=\arg\min_{y \in \RB^{d_y}} \ell(y)$ and $x^\star=\tilde{x}^\star(y^\star)$.
\citet[Lemma~1]{shen2022single} provide conditions under which $F$, $G$, $H$, and $\nabla H$ are Lipschitz continuous.
For simplicity, we impose the slightly stronger assumption that $\nabla^2 f(x,y)$, $\nabla f(x,y)$, $\nabla g(x,y)$, and $g(x,y)$ are all Lipschitz continuous.
Then, by Proposition~1, Assumption~\ref{assump:smoothH} holds with $\Hsmooth=1$.
Under these conditions, $F(x,y)$ satisfies \eqref{assump:local-linear-F} with $A_{11}=\nabla_{xx} f(x^\star,y^\star)$, $A_{12}=\nabla_{xy} f(x^\star,y^\star)$, and $\Fsmooth=1$.

For $G$, we further assume that $\nabla^2 g(x,y)$ is Lipschitz continuous in a neighborhood of $(x^\star,y^\star)$.
For convenience, define
\[
A(x,y):=\nabla_y g(x,y),\
B(x,y):=\nabla_{yx}^2 f(x,y),\
C(x,y):=\nabla_{xx}^2 f(x,y),\
D(x,y):=\nabla_x g(x,y),
\]
so that $G(x,y)=A(x,y)-B(x,y)C(x,y)^{-1}D(x,y)$.
Since $f(x,y)$ is strongly convex in $x$, there exists $\mu>0$ such that $C(x,y)\succeq \mu I$, and hence $\|C(x,y)^{-1}\|\le \mu^{-1}$.
Let $z:=(x,y)$.
Arguing as in the derivation of \eqref{eq:exmp-lagrange-matrix-inverse}, we obtain
$\|C(z_1)^{-1}-C(z_2)^{-1}\|\le \mu^{-2}\|C(z_2)-C(z_1)\|\lesssim \|z_1-z_2\|$,
which shows that $C^{-1}$ is Lipschitz continuous.
For the gradient of $G$, the product rule together with the derivative of the inverse gives
\begin{equation}\label{eq:exmp-bilevel-nablaG}
\nabla G = \nabla A - (\nabla B)C^{-1}D + BC^{-1}(\nabla C)C^{-1}D - BC^{-1}\nabla D.
\end{equation}
Here, $\nabla A$, $\nabla B$, $\nabla C$, and $\nabla D$ denote derivatives with respect to the full variable $z=(x,y)$.
Under our regularity conditions, $\nabla A$, $\nabla B$, $\nabla C$, and $\nabla D$ are all Lipschitz continuous in a neighborhood of $(x^\star,y^\star)$, while $B$, $C^{-1}$, and $D$ are bounded there.
Hence $\nabla G$ is Lipschitz continuous in a neighborhood of $(x^\star,y^\star)$.
This implies that \eqref{assump:local-linear-G} holds with $A_{21}=\nabla_x G(x^\star,y^\star)$, $A_{22}=\nabla_y G(x^\star,y^\star)$, and $\Gsmooth=1$.
Therefore, by Proposition~\ref{prop:ensure-linearity}, Assumption~\ref{assump:near-linear} holds with $B_1=A_{11}=\nabla_{xx} f(x^\star,y^\star)$, $B_2=A_{21}=\nabla_x G(x^\star,y^\star)$, $B_3=A_{22}-A_{21}A_{11}^{-1}A_{12}=\nabla_y G(x^\star,y^\star)-\nabla_x G(x^\star,y^\star)[\nabla_{xx} f(x^\star,y^\star)]^{-1}\nabla_{xy} f(x^\star,y^\star)$, and $\Fsmooth=\Gsmooth=1$,
where $\nabla G$ is given by \eqref{eq:exmp-bilevel-nablaG}.


\section{Convergence Rates without Local Linearity}
\label{proof:without}

In this section, we focus on the convergence rates of nonlinear two-time-scale SA without imposing local linearity on $F$ and $G$, i.e., without Assumption~\ref{assump:near-linear}.
For Assumption~\ref{assump:smoothH}, we consider the more general case $\Hsmooth \in [0,1]$.
Moreover, we only need the following weaker version of Assumptions~\ref{assump:noise-s} and \ref{assump:stepsize-new} throughout this section.
\begin{assump}
	[Martingale difference noise with bounded variance]
	\label{assump:noise}
	The sequences of random variables $\{ \xi_{t} \}_{t=0}^\infty$ and $\{ \psi_{t} \}_{t=0}^\infty$ are martingale difference sequences satisfying
	\begin{equation}
		\EB [\xi_t \,|\, \FM_t ] = 0, \quad
		\EB [\psi_t \,|\, \FM_t] = 0, \quad 
		\EB[ \|\xi_{t}\|^2 \,|\,\FM_{t}] \le \Gamma_{11},\quad 
		\EB[  \|\psi_{t}\|^2 \,|\,\FM_{t}] \le \Gamma_{22}.
		\label{assump:variance}
	\end{equation}
\end{assump}

\begin{assump}[Conditions on step sizes]
	\label{assump:stepsize}
	The step sizes $\{ \alpha_{t} \}_{t=0}^\infty$ and $ \{ \beta_{t} \}_{t=0}^\infty $ satisfy the following conditions that for $\forall t \ge 1$:
	\begin{itemize}
		\item Constant bounds: $\alpha_{t} \le \iota_1,\beta_t \le \iota_2,\frac{\beta_t}{\alpha_{t}} \le \kappa, \frac{\beta_t^2}{\alpha_{t}} \le \rho$ with
		\[
		\iota_1 =   \frac{\mu_F}{4L_F^2},\  
		\iota_2 = \frac{\mu_G}{\LGy^2},\ 
		\kappa =  \frac{\mu_F \mu_G}{28 L_H \LGx (1 \vee \LGy) }
		\wedge \frac{\mu_F}{\mu_G},\ 
		\rho =   \frac{\mu_{F}}{
			16L_H^2 \LGx^2 }.
		\]
		\item Growth conditions: $1 \le \frac{\alpha_{t-1}}{\alpha_{t}} \le 1 + (\frac{\mu_F}{4}\alpha_t) \wedge (\frac{\mu_G}{8}\beta_t)$ and $1 \le \frac{\beta_{t-1}}{\beta_{t}} \le 1 + \frac{\mu_G}{32}\beta_t$.
		\item $\frac{\beta_t}{\alpha_t}$ is non-increasing in $t$, and $\prod_{\tau=0}^t\left(1-\frac{\mu_{G}\beta_{\tau}}{4} \right) = \OM (\alpha_t) $
	\end{itemize}
	
\end{assump}

\subsection{Proof of Lemma \ref{lem:xhat}}
\label{proof:xhat}
\begin{proof}[Proof of Lemma \ref{lem:xhat}]
	We first present the specific forms of the constants.
	\begin{equation}
		\label{eq:constantx}
		c_{x,1} =  4 L_{H}^2 \LGy^2, \  c_{x,2} = 7 L_H \LGy, \   
		c_{x,3} = 2 L_H^2 \Gamma_{22}, \  c_{x,4} = \frac{16 S_H^2 \Gamma_{22}^{1+\Hsmooth} }{\mu_F}.
	\end{equation}
	
	Recall that $\xhat_t = x_t-H(y_t)$. We then consider 
	\begin{align}
		\xhat_{t+1} = x_{t+1} - H(y_{t+1})
		&= \xhat_{t} -  \alpha_{t}F(x_{t},y_{t}) - \alpha_{t}\xi_{t}  + H(y_{t}) - H(y_{t+1}), \label{eq:xhat_update}
	\end{align}
	which implies
	\begin{align}
		\|\xhat_{t+1}\|^2 
		&= \|\xhat_{t}  - \alpha_{t}F(x_{t},y_{t})\|^2 + \left\|H(y_{t}) - H(y_{t+1}) - \alpha_{t}\xi_{t}\right\|^2\notag\\
		&\quad + 2 \langle\xhat_{t}-\alpha_t F(x_{t},y_{t}), H(y_{t}) - H(y_{t+1}) \rangle - 2\alpha_{t} \langle \xhat_{t} - \alpha_{t}F(x_{t},y_{t}), \xi_{t}\rangle. \label{lem:xhat:Eq1}
	\end{align}
	We then analyze each term on the right-hand side of \eqref{lem:xhat:Eq1}. 
	
	For the first term, noting that $F(H(y_{t}),y_{t}) = 0$, we have 
	\begin{align}
		& \quad \  \|\xhat_{t}  - \alpha_{t}F(x_{t},y_{t})\|^2 \notag \\ 
		&= \|\xhat_{t}\|^2 - 2\alpha_{t} \langle \xhat_{t}, F(x_{t},y_{t}) - F(H(y_{t}),y_{t}) \rangle \notag
		+ \alpha_{t}^2\| F(x_{t},y_{t}) - F(H(y_{t}),y_{t})\|^2\notag\\
		&\leq \|\xhat_{t}\|^2 - 2\mu_{F}\alpha_{t}\|\xhat_{t} \|^2 + L_{F}^2\alpha_{t}^2\|\xhat_{t} \|^2
		\le \left(1-\frac{7\mu_{F}\alpha_{t}}{4} \right) \|\xhat_{t}\|^2,\label{lem:xhat:Eq1a}
	\end{align}
	where the inequality uses strong monotone and Lispchitz continuity of $F$ in~\eqref{assump:sm:F:ineq} and~\eqref{assump:smooth:FH:ineqF} respectively. 
	
	For the second term, recall that $\FM_{t} = \sigma\{x_{0},y_{0},\xi_{0},\psi_{0},\xi_{1},\psi_{1},\ldots,\xi_{t-1},\psi_{t-1}\}.$
	We then take the conditional expectation of the second term on the right-hand side of \eqref{lem:xhat:Eq1} w.r.t. $\FM_{t}$ and using Assumption \ref{assump:noise} to have
	\begin{align}
		&\EB\left[\left\|H(y_{t}) - H(y_{t+1}) - \alpha_{t}\xi_{t}\right\|^2\,|\,\FM_{t}\right]\notag\\
		&\le 2L_{H}^2\beta_{t}^2\left\|G(x_{t},y_{t})\right\|^2 + 2L_{H}^2\beta_{t}^2\EB\left[\|\psi_{t}\|^2\,|\,\FM_{t}\right]  +  2\alpha_{t}^2\EB\left[\left\|\xi_{t}\right\|^2\,|\,\FM_{t}\right]\notag\\
		&\leq 4L_{H}^2 \LGx^2\beta_{t}^2\|\xhat_{t}\|^2 + 4 L_{H}^2 \LGy^2\beta_{t}^2\|\yhat_{t}\|^2     +  2\beta_{t}^2L_H^2\Gamma_{22} +  2\alpha_{t}^2\Gamma_{11}, \label{lem:xhat:Eq1b}
	\end{align}
	where the last inequality uses the following inequality that depends on the fact $G(H(y^{\star}),y^{\star}) = 0$ and Assumption~\ref{assump:G}
	\begin{align}
		\|G(x_{t},y_{t})\|
		&\leq \|G(x_{t},y_{t}) - G(H(y_{t}),y_{t})\| +  \|G(H(y_{t}),y_{t}) - G(H(y^{\star}),y^{\star})\|\notag\\
		&\leq \LGx \|\xhat_{t}\| + \LGy \|\yhat_{t}\|,\label{lem:xhat:Eq1b1}
	\end{align}
	and thus $\|G(x_{t},y_{t})\|^2 \le 2 \LGx^2\|\xhat_{t}\|^2 + 2\LGy^2 \|\yhat_{t}\|^2.$
	
	For the last term, we have that
	\begin{align}
		\label{eq:xhat_third}
		\begin{split}
			\langle \xhat_{t} - \alpha_{t}F(x_{t},y_{t}), &H(y_{t}) - H(y_{t+1}) \rangle
			=	\underbrace{\langle \xhat_{t} - \alpha_{t}F(x_{t},y_{t}), \nablaH(y_{t})(y_{t+1}-y_{t})  \rangle}_{\spadesuit_1}
			\\
			&	+	\underbrace{\langle \xhat_{t} - \alpha_{t}F(x_{t},y_{t}), \nablaH(y_{t})(y_{t}-y_{t+1}) + H(y_{t}) - H(y_{t+1}) \rangle}_{\spadesuit_2}.
		\end{split}
	\end{align}
	For one thing, Assumptions \ref{assump:smooth:FH} and \ref{assump:smoothH} imply that $\| \nablaH (y_t) \| \le L_H$.
	Then we have
	\begin{align*}
		\EB\left[ \spadesuit_1 |\FM_t  \right] &= \beta_t \langle \xhat_{t} - \alpha_{t}F(x_{t},y_{t}), \nablaH(y_{t}) G(x_{t}, y_{t}) \rangle\\
		&\le  \beta_t L_H \|  \xhat_{t} - \alpha_{t}F(x_{t},y_{t})\| \|G(x_{t}, y_{t}) \|
		\overset{\eqref{lem:xhat:Eq1a}}{\le}  \beta_t \sqrt{ 1-\alpha_t \mu_F} L_H \|  \xhat_{t}\| \|G(x_{t}, y_{t}) \|.
	\end{align*}
	For another thing, by Assumption~\ref{assump:smoothH}, it follows that
	\begin{align}
		\label{eq:residual}
		&\|\nablaH(y_{t})(y_{t}-y_{t+1}) + H(y_{t}) - H(y_{t+1})\| 
		\le S_H \|y_{t} - y_{t+1}\| \cdot \min \left\{ \|y_{t} - y_{t+1}\|^{\Hsmooth} , R_H  \right\} \notag \\
		& \overset{(a)}{\le} S_H\beta_t   \left( \|G(x_t, y_t)\| + \|\psi_{t}\|  \right) \cdot \min \left\{
		\beta_t^\Hsmooth   \left( \|G(x_t, y_t)\|^\Hsmooth + \|\psi_{t}\|^\Hsmooth  \right), R_H \right\} \notag \\
		& \overset{(b)}{\le}  S_H\beta_t   \left( \|G(x_t, y_t)\| + \|\psi_{t}\|  \right) \cdot  \left( \min \left\{
		\beta_t^\Hsmooth   \|G(x_t, y_t)\|^\Hsmooth , R_H \right\}
		+ \min \left\{\beta_t^\Hsmooth  \|\psi_{t}\|^\Hsmooth, R_H \right\}\right)\notag \\
		& \overset{(c)}{\le} S_H  \Bigg( 2\beta_t \|G(x_t, y_t)\|  R_H + \frac{1}{1+\Hsmooth} \beta_t^{1{+}\Hsmooth} \|\psi_t\|^{1{+}\Hsmooth} + \frac{\Hsmooth}{1 + \Hsmooth}  \min \left\{
		\beta_t^\Hsmooth   \|G(x_t, y_t)\|^\Hsmooth , R_H \right\}^\frac{1+\Hsmooth}{\Hsmooth} \notag \\
		& \qquad \qquad  + \beta_t^{1+\Hsmooth} \|\psi_t\|^{1+\Hsmooth}
		\Bigg) \notag \\
		& 
		\le S_H  \left( \frac{2 + 3\Hsmooth}{1+\Hsmooth} \beta_t \|G(x_t, y_t)\|  R_H + \frac{2+\Hsmooth}{1+\Hsmooth} \beta_t^{1+\Hsmooth} \|\psi_t\|^{1+\Hsmooth}
		\right),
	\end{align}
	where (a) uses $(a+b)^\Hsmooth \le a^\Hsmooth + b^\Hsmooth$ for any $a, b > 0$ and $\Hsmooth \in [0,1]$, (b) uses the inequality that $\min \{ a+b, c\} \le  \min \left\{a, c\right\}+ \min \left\{b, c\right\}$ for any non-negative $a, b,c$, and
	(c) follows from Young's inequality $ab \le \frac{a^p}{p} + \frac{b^q}{q}$ for any $a,b, p, q \ge 0$ and $\frac{1}{p} + \frac{1}{q} = 1$.
	Note that Jensen's inequality implies $\EB [ \| \phi_t \|^{1+\Hsmooth} | \FM_t ] \le \left( \EB [ \| \phi_t \|^{2} | \FM_t ] \right)^\frac{1+\Hsmooth}{2} \le \Gamma_{22}^\frac{1+\Hsmooth}{2}$.
	Hence,
	\[
	\EB \left[
	\|\nablaH(y_{t})(y_{t}-y_{t+1}) + H(y_{t}) - H(y_{t+1})\|  |\FM_t
	\right]
	\le 6 L_H\beta_t  \|G(x_t, y_t)\|  + 4 S_H \beta_t^{1+\Hsmooth} \Gamma_{22}^\frac{1+\Hsmooth}{2},
	\]
	\begin{align*}
\text{and}~\EB\left[ \spadesuit_2 |\FM_t  \right] &
		\le \|\xhat_{t} - \alpha_{t}F(x_{t},y_{t})\| \cdot \EB \left[ \|\nablaH(y_{t})(y_{t}-y_{t+1}) + H(y_{t}) - H(y_{t+1})\|  |\FM_t\right] \\
		&\overset{\eqref{lem:xhat:Eq1a}}{\le}  \sqrt{1-\alpha_t \mu_F}\|  \xhat_{t}\| \left(
		6 L_H\beta_t  \|G(x_t, y_t)\|  + 4 S_H \beta_t^{1+\Hsmooth} \Gamma_{22}^\frac{1+\Hsmooth}{2}
		\right).
	\end{align*}
	Combing these two inequalities, we then obtain the preceding relation
	\begin{align}
		&\EB\left[ 	\langle \xhat_{t} - \alpha_{t}F(x_{t},y_{t}), H(y_{t}) - H(y_{t+1}) \rangle|\FM_t  \right]   \notag\\
		&  \le 7 L_H \beta_t \sqrt{ 1-\alpha_t \mu_F}  \|  \xhat_{t}\| \|G(x_{t}, y_{t}) \| +  4 S_H\sqrt{ 1-\alpha_t \mu_F} \beta_t^{1+\Hsmooth} \|\xhat_{t}\|  \Gamma_{22}^\frac{1+\Hsmooth}{2} \notag\\
		& \overset{\eqref{lem:xhat:Eq1b1}}{\le}
		7 L_H \beta_t \left( \LGx \sqrt{ 1-\alpha_t \mu_F} \|\xhat_{t}\|^2 +  \LGy   \sqrt{ 1-\alpha_t \mu_F} \|\xhat_{t}\|\|\yhat_{t}\| \right)+  4 S_H \sqrt{ 1-\alpha_t \mu_F} \beta_t^{1+\Hsmooth} \|\xhat_{t}\|  \Gamma_{22}^\frac{1+\Hsmooth}{2} \notag\\
		\notag\\
		& \le \frac{ \mu_F \alpha_{t}}{2}\|\xhat_{t}\|^2 +  7 L_H \LGy  \beta_t \sqrt{ 1-\alpha_t \mu_F}  \|\xhat_{t}\|\|\yhat_{t}\| +  \frac{16 S_H^2 \Gamma_{22}^{1+\Hsmooth} }{\mu_F} \frac{\beta_t^{2 + 2\Hsmooth} }{\alpha_t} 
		\label{lem:xhat:Eq1c}
	\end{align}
	where the last inequality uses the relation obtained from the choices of step sizes $7L_H \LGx \beta_t \le \frac{ \mu_F \alpha_{t}}{4}$ and Cauchy–Schwarz inequality.


	Combing the three bounds in~\eqref{lem:xhat:Eq1a},~\eqref{lem:xhat:Eq1b}, and~\eqref{lem:xhat:Eq1c}, we obtain that
	\begin{align*}
		\EB\left[\|\xhat_{t+1}\|^2\,|\,\FM_{t}\right]
		&\le \left(1-\frac{7\mu_{F}\alpha_{t}}{4} \right) \|\xhat_{t}\|^2 
		+ 4 L_{H}^2 \LGx^2 \beta_{t}^2 \|\xhat_{t}\|^2 + 4 L_{H}^2 \LGy^2 \beta_{t}^2 \|\yhat_{t}\|^2 +  2L_H^2\Gamma_{22} \beta_{t}^2 \\
		& \qquad  +  2 \Gamma_{11} \alpha_{t}^2  {+} \frac{\mu_F \alpha_{t}}{2}\|\xhat_{t}\|^2 {+}  7 L_H \LGy  \beta_t \sqrt{ 1{-}\alpha_t \mu_F} \|\xhat_{t}\|\|\yhat_{t}\| {+}  \frac{16 S_H^2 \Gamma_{22}^{1+\Hsmooth} }{\mu_F} \frac{\beta_t^{2 + 2\Hsmooth} }{\alpha_t}  \\
		&\le \left(1- \mu_F \alpha_t \right) \|\xhat_{t}\|^2  + c_{x,1} \beta_t^2 \|\yhat_{t}\|^2 + c_{x,2} \beta_t\sqrt{ 1-\alpha_t \mu_F}  \|\xhat_{t}\|\|\yhat_{t}\| + 2\alpha_{t}^2 \Gamma_{11} \notag \\
		& \qquad + c_{x,3} \beta_{t}^2 +  c_{x,4} \frac{\beta_t^{2 + 2\Hsmooth} }{\alpha_t},
	\end{align*}
	where the second inequality uses the fact that $4L_{H}^2 \LGx^2 \beta_{t}^2 \le  \frac{ \mu_F \alpha_{t}}{4}$ and the constants are defined as
	the last inequality uses the notations in~\eqref{eq:constantx}.
\end{proof}

\subsection{Proof of Lemma \ref{lem:yhat}}
\label{proof:yhat}
\begin{proof}[Proof of Lemma \ref{lem:yhat}]
	Recall that $\yhat = y - y^{\star}$. Using \eqref{alg:xy} we consider
	\begin{align}
		\yhat_{t+1} 
		&= \yhat_{t} - \beta_{t}G(H(y_{t}),y_{t}) + \beta_{t} \left(G(H(y_{t}),y_{t}) - G(x_{t},y_{t})\right)  - \beta_{t}\psi_{t}, \label{eq:yhat_update}
	\end{align}
	which implies that
	\begin{align}
		& \|\yhat_{t+1}\|^2 = \|\yhat_{t} - \beta_{t}G(H(y_{t}),y_{t})\|^2 + \left\|\beta_{t} \left(G(H(y_{t}),y_{t}) - G(x_{t},y_{t})\right)  - \beta_{t}\psi_{t}\right\|^2\notag\\
		&\quad + 2\beta_{t} \langle \yhat_{t} - \beta_{t} G(H(y_{t}),y_{t}), G(H(y_{t}),y_{t}) - G(x_{t},y_{t})   \rangle
		- 2\beta_{t} \langle \yhat_{t} - \beta_{t}G(H(y_{t}),y_{t}), \psi_{t}  \rangle. \label{lem:yhat:Eq1}
	\end{align}
	We next analyze each term on the right-hand side of \eqref{lem:yhat:Eq1}. 
	
	For the first term, we have that
	\begin{align}
		\left\|\yhat_{t} - \beta_{t}G(H(y_{t}),y_{t})\right\|^2 
		&\overset{\eqref{assump:G:sm}}{\le} \|\yhat_{t}\|^2 - 2\mu_{G}\beta_{t}\|\yhat_{t}\|^2  + \beta_{t}^2 \|G(H(y_{t}),y_{t}) - G(H(y^{\star}),y^{\star})\|^2\notag\\
		& \leq (1- 2\mu_{G}\beta_{t})\|\yhat_{t}\|^2  + \LGy^2 \beta_t^2 \| \yhat_t \|^2 
		\leq \left(1-\mu_{G}\beta_{t} \right)\|\yhat_{t}\|^2,
		\label{lem:yhat:Eq1a}
	\end{align}
	where the first inequality also uses ${G(H(y^{\star}),y^{\star}) = 0}$ and the last inequality uses the relation $\LGy^2 \beta_t^2 \le \mu_G \beta_t$.
	
	For the second term, taking the conditional expectation on its both sides w.r.t $\FM_{t}$ and using Assumptions \ref{assump:noise} and \ref{assump:G}, we have
	\begin{align}
		&\EB\left[\left\|\beta_{t} \left(G(H(y_{t}),y_{t}) - G(x_{t},y_{t})\right)  - \beta_{t}\psi_{t}\right\|^2\;|\;\FM_{t}\right]
        \leq \LGx^2 \beta_{t}^2\|\xhat_{t}\|^2 + \beta_{t}^2\Gamma_{22}. 
		\label{lem:yhat:Eq1b}
	\end{align}
	
	For the third term, it follows that
	\begin{align}
		&2\beta_{t} \langle \yhat_{t} - \beta_{t}G(H(y_{t}),y_{t}), G(H(y_{t}),y_{t}) - G(x_{t},y_{t})  \rangle \notag\\
		&\overset{\eqref{assump:G:smooth}}{\leq}2\beta_{t}\|\yhat_{t}- \beta_{t}G(H(y_{t}),y_{t})\| \cdot \LGx \|\xhat_{t}\|
		\overset{\eqref{lem:yhat:Eq1a}}{\le} 2 \LGx \beta_t\sqrt{1-\mu_G\beta_t} \|\xhat_{t}\|\|\yhat_{t}\|
		.\label{lem:yhat:Eq1c}
	\end{align}
	Finally, taking the conditional expectation of \eqref{lem:yhat:Eq1} w.r.t $\FM_{t}$ and using \eqref{lem:yhat:Eq1a}--\eqref{lem:yhat:Eq1c} yields
	\begin{align*}
		\EB\left[\|\yhat_{t+1}\|^2\,|\,\FM_{t}\right]
		&\leq \left(1-\mu_{G}\beta_{t} \right)\|\yhat_{t}\|^2 +  \LGx^2 \beta_{t}^2\|\xhat_{t}\|^2  + \beta_{t}^2\Gamma_{22} +2 \LGx \beta_t\sqrt{1-\mu_G\beta_t} \|\xhat_{t}\|\|\yhat_{t}\|,
	\end{align*}
	which concludes our proof. 
\end{proof}


\subsection{Proof of Theorem~\ref{thm:first}}
\label{proof:thm:first}

Before proving Theorem~\ref{thm:first}, we first present a refined one-step descent lemma
by applying Cauchy-Schwarz inequality to the cross term $\OM(\beta_t \| \xhat_t \| \| \yhat_t \| )$.

\begin{lem}[One-step descent lemma]
	\label{lem:iterate1}
	Under Assumptions \ref{assump:smooth:FH}  -- \ref{assump:smoothH}, \ref{assump:noise} and \ref{assump:stepsize}, it follows that
	\begin{align}
			\EB[|\xhat_{t+1}\|^2|\FM_t] &\le \left(1- \frac{\mu_F \alpha_t }{2}\right) \|\xhat_{t}\|^2  + \frac{c_{x,2}^2}{2\mu_{F}} \frac{\beta_{t}^2}{\alpha_{t}} \|\yhat_{t}\|^2 + 2\Gamma_{11} \alpha_{t}^2 + c_{x,3} \beta_{t}^2 + c_{x,4} \frac{\beta_t^{2+2\Hsmooth}}{\alpha_t}, \label{lem:iterate1-xhat} \\
			\EB [\|\yhat_{t+1}\|^2|\FM_t] &\le  \left(1-\frac{\mu_{G}\beta_{t}}{2} \right)\|\yhat_{t}\|^2 +  \frac{2\LGx^2 }{\mu_G} \beta_{t} \|\xhat_{t}\|^2    + \Gamma_{22}\beta_{t}^2, \label{lem:iterate1-yhat}
	\end{align}
	where $c_{x,2}$ and $c_{x,3}$ are defined in Lemma~\ref{lem:xhat}.
\end{lem}
\begin{proof}[Proof of Lemma~\ref{lem:iterate1}]
	By Lemma~\ref{lem:xhat}, using the Cauchy-Schwarz inequality, we have
	\begin{align*}
		\EB\left[\|\xhat_{t+1}\|^2\,|\,\FM_{t}\right]
		&\le\left(1- \frac{\mu_F \alpha_t }{2}\right) \|\xhat_{t}\|^2  + \left( c_{x,1} \beta_t^2 +\frac{c_{x,2}^2\beta_{t}^2(1-\mu_F\alpha_t)}{2\mu_F\alpha_{t}} \right)\|\yhat_{t}\|^2 + 2\Gamma_{11} \alpha_{t}^2 \\
		& \qquad + c_{x,3} \beta_{t}^2 + c_{x,4} \frac{\beta_t^{2+2\Hsmooth}}{\alpha_t} \\
		&\le \left(1- \frac{\mu_F \alpha_t }{2}\right) \|\xhat_{t}\|^2  + \frac{c_{x,2}^2}{2\mu_{F}} \frac{\beta_{t}^2}{\alpha_{t}} \|\yhat_{t}\|^2 + 2\Gamma_{11} \alpha_{t}^2 + c_{x,3} \beta_{t}^2 + c_{x,4} \frac{\beta_t^{2+2\Hsmooth}}{\alpha_t},
	\end{align*}
	where the last inequality uses the fact that $2c_{x,1} - c_{x,2}^2\le 0$.
	Similarly, 	by Lemma~\ref{lem:yhat}, using the Cauchy-Schwarz inequality, we have
	\begin{align*}
		\EB\left[\|\yhat_{t+1}\|^2\,|\,\FM_{t}\right]
		& \leq \left(1-\frac{\mu_{G}\beta_{t}}{2} \right)\|\yhat_{t}\|^2 +  \left( \LGx^2 \beta_{t}^2 + \frac{2 \LGx^2 (1-\mu_G\beta_t) \beta_{t}}{\mu_G} \right)\|\xhat_{t}\|^2    + \Gamma_{22}\beta_{t}^2\\
		&\leq \left(1-\frac{\mu_{G}\beta_{t}}{2} \right)\|\yhat_{t}\|^2 +  \frac{2 \LGx^2 }{\mu_G} \beta_{t} \|\xhat_{t}\|^2    + \Gamma_{22}\beta_{t}^2.
	\end{align*}
	We then complete the proof.
\end{proof}

Lemma~\ref{lem:iterate1} is akin to \citet[Lemmas~1 and 2]{doan2022nonlinear}. While there are minor differences, the key distinction lies in the incorporation of $\Hsmooth$ in the term $\OM ( \beta_t^{2+2\Hsmooth} / \alpha_t )$, which is $\OM(\beta_t^2 / \alpha_t)$ in \citet{doan2022nonlinear}.
Combining \eqref{lem:iterate1-xhat} and \eqref{lem:iterate1-yhat} through the careful construction of Lyapunov functions, we can establish both almost sure convergence and $L_2$-convergence.

\begin{proof}[Proof of Theorem~\ref{thm:first}]
	We first present the specific forms of the constants:
	\begin{align}
		\label{eq:constanty}
		& c_{y,1} = \frac{ 8 \Gamma_{22} (2 L_H \LGx + 3 \LGy)}{3 \mu_G \LGy}, \ \,
		c_{y,2} = \frac{1280 \LGx^2 S_H^2 \Gamma_{22}^{1+\Hsmooth}}{\mu_F^2 \mu_G^2}, \\
		& \begin{aligned}
			c_{x,5} &= \frac{ 1120 L_H \LGx \LGy \Gamma_{11}}{\mu_{F}^2\mu_{G}}, \quad \quad
			c_{x,6} = \frac{21 \mu_G L_H \LGy c_{y,1}}{4 \mu_F \LGx} + \frac{8 L_H^2 \Gamma_{22}}{\mu_F}, \\
			c_{x,7} & = \frac{7 \mu_G^2 c_{y,2}}{2 \LGx^2} + \frac{48 S_H^2 \Gamma_{22}^{1+\Hsmooth}}{\mu_F^2}. 
		\end{aligned}
		\label{eq:constantx2}
	\end{align}
	

	We 
	characterize the $L_2$-convergence rates by introducing the
	Lyapunov function 
    $
	\tU_t = \varrho_1 \frac{\beta_t}{\alpha_t} \|\xhat_{t}\|^2 + \|\yhat_{t}\|^2~\text{with}~\varrho_1=\frac{8 \LGx^2}{\mu_F\mu_G}.
    $
	Note that
	$
    \frac{\beta_t}{\alpha_t} \le \frac{\mu_G\mu_F}{4c_{x,2} \LGx }\implies
	\frac{\mu_G}{4} \beta_t \ge \varrho_1  \frac{ c_{x,2}^2}{2\mu_F} \frac{\beta_{t}^3}{\alpha_{t}^2}.
    $
	Since
	$\frac{\beta_{t+1}}{\alpha_{t+1}} \le \frac{\beta_t}{\alpha_t}$, it follows that
	\begin{align}
		\label{eq:U-iterate}
		\EB\tU_{t+1} 
		&\le \varrho_1 \frac{\beta_t}{\alpha_t} \cdot \left[\left(1- \frac{\mu_F \alpha_t }{2}\right) \EB\|\xhat_{t}\|^2  + \frac{c_{x,2}^2}{2\mu_{F}} \frac{\beta_{t}^2}{\alpha_{t}}\EB \|\yhat_{t}\|^2 + \left( 2\Gamma_{11} \alpha_{t}^2 + c_{x,3} \beta_{t}^2 + c_{x,4} \frac{\beta_t^{2+2\Hsmooth}}{\alpha_t} \right)
		\right] \notag \\
		&\qquad + \left(1-\frac{\mu_{G}\beta_{t}}{2} \right)\EB\|\yhat_{t}\|^2 +  \frac{2 \LGx^2 }{\mu_G} \beta_{t} \EB\|\xhat_{t}\|^2    + \Gamma_{22}\beta_{t}^2\notag \\
		&\le \left(1- \frac{\mu_G\beta_t}{4}\right) \EB \tU_t
		+ 2\varrho_1 \Gamma_{11} \alpha_t \beta_t + \left( \Gamma_{22} + \frac{c_{x,3} \LGx }{3 L_H \LGy} \right)  \beta_{t}^2
		+ \varrho_1 c_{x,4} \frac{\beta_t^{3 + 2\Hsmooth}}{\alpha_t^2},
	\end{align}
	where the last inequality uses $\frac{\beta_t}{\alpha_t} \le \kappa \le \frac{\mu_F}{\mu_G}$.
	For simplicity, we set 
    $
	\alpha_{j, t} = \prod_{\tau=j}^t\left(1-\frac{\mu_{F}\alpha_{\tau}}{2} \right)~\text{and}~\beta_{j, t} = \prod_{\tau=j}^t\left(1-\frac{\mu_{G}\beta_{\tau}}{4} \right).
    $
	Then iterating~\eqref{eq:U-iterate} yields
	\begin{align}
		\label{eq:U-final}
		\EB\tU_{t+1}  
		&\le  \beta_{0, t} \EB \tU_0 {+} 2\varrho_1 \Gamma_{11} {\cdot} \sum_{\tau=0}^t\beta_{\tau+1, t}   \alpha_{\tau} \beta_{\tau} {+} \left( \Gamma_{22} {+} \frac{c_{x,3} \LGx}{3 L_H \LGy} \right)  \cdot \sum_{\tau=0}^t \beta_{\tau+1, t} \beta_{\tau}^2 {+} \varrho_1 c_{x,4} {\cdot} \sum_{\tau=0}^t \beta_{\tau+1, t} \frac{\beta_\tau^{3 + 2\Hsmooth}}{\alpha_\tau^2} \notag \\
		&\le \beta_{0, t} \EB \tU_0 
		+ \frac{16\varrho_1 \Gamma_{11}}{\mu_G} \alpha_t
		+ \frac{ 8 \Gamma_{22} (2 L_H \LGx + 3 \LGy)}{3 \mu_G \LGy}  \beta_t + \frac{ 10 \varrho_1 c_{x,4} }{\mu_G} \frac{\beta_t^{2+2\Hsmooth}}{\alpha_t^2}.
	\end{align}
	where the last inequality follows from Lemma~\ref{lem:step-size-ineq}, whose proof is 
    deferred to Appendix~\ref{proof:step-size-ineq}.
	
	\begin{lem}[Step sizes inequalities]
		\label{lem:step-size-ineq}
		Under Assumption~\ref{assump:stepsize}, it follows that
		\begin{enumerate}[(i)]
			\item $\sum_{\tau=0}^t \beta_{\tau+1, t} \beta_{\tau}^2 \le \frac{8\beta_t}{\mu_G}$ and $\sum_{\tau=0}^t \alpha_{\tau+1, t} \alpha_{\tau}^2 \le \frac{4\alpha_t}{\mu_F}$. 
			\item $\sum_{\tau=0}^{t}  \alpha_{\tau+1, t}\beta_{\tau}^2 \le \frac{4}{\mu_F} \frac{\beta_t^2}{\alpha_{t}}$ and $\sum_{\tau=0}^{t}  \alpha_{\tau+1, t}\beta_{\tau}  \beta_{\tau-1} \le  \frac{6}{\mu_F} \frac{\beta_t^2}{\alpha_t}$.
			\item $\sum_{\tau=0}^t\beta_{\tau+1, t}   \alpha_{\tau} \beta_{\tau} \le \frac{8\alpha_t}{\mu_G}$ and $\sum_{\tau=0}^{t}  \alpha_{\tau+1, t}\beta_{\tau}  \alpha_{\tau-1} \le  \frac{10}{\mu_F} \beta_t$.
			\item $\sum_{\tau=0}^{t}  \alpha_{\tau+1, t}\beta_{\tau}  \beta_{0, \tau-1} \le   \frac{8}{\mu_G}\beta_{0, t}$.
			\item $\sum_{\tau=0}^t \beta_{\tau+1, t} \frac{\beta_\tau^{3 + 2\Hsmooth}}{\alpha_\tau^2} \le \frac{10 }{\mu_G} \frac{\beta_t^{2 + 2\Hsmooth}}{\alpha_t^2}$.
			\item $\sum_{\tau=0}^t \alpha_{\tau+1, t} \frac{\beta_\tau^{2 + 2\Hsmooth}}{\alpha_\tau} \le \frac{3 }{\mu_F} \frac{\beta_t^{2 + 2\Hsmooth}}{\alpha_t^2}$ and $\sum_{\tau=0}^t \alpha_{\tau+1, t} \beta_\tau \frac{\beta_{\tau-1}^{1 + 2\Hsmooth}}{\alpha_{\tau-1}} \le \frac{4 }{\mu_F} \frac{\beta_t^{2 + 2\Hsmooth}}{\alpha_t^2}$.
		\end{enumerate}
	\end{lem}
	
	As a result of~\eqref{eq:U-final}, it follows that
	\begin{align}
		\label{eq:y-iterate}
			\EB\|\yhat_{t+1}\|^2 & \le \EB \tU_{t+1}
			\le \beta_{0, t} \EB \tU_0 + \frac{16 \varrho_1 \Gamma_{11}}{\mu_G} \alpha_t +  c_{y,1} \beta_t + 
			c_{y,2}
			\frac{\beta_t^{2+2\Hsmooth}}{\alpha_t^2},
	\end{align}
	where $c_{y,1}$ and $c_{y,2}$ are defined in \eqref{eq:constanty}.
	Recall that 
	\[
	\EB \|\xhat_{t+1}\|^2 \le \left(1- \frac{\mu_F \alpha_t }{2}\right) \|\xhat_{t}\|^2  + \frac{c_{x,2}^2}{2\mu_{F}} \frac{\beta_{t}^2}{\alpha_{t}} \|\yhat_{t}\|^2 + 2\Gamma_{11} \alpha_{t}^2 + c_{x,3} \beta_{t}^2 + c_{x,4} \frac{\beta_t^{2+2\Hsmooth}}{\alpha_t}.
	\]
	Iterating this inequality yields that
	\begin{align}
		\label{eq:x-iterate}
		\EB\|\xhat_{t+1}\|^2 
		&\le\alpha_{0, t}	\EB\|\xhat_{0}\|^2  +
		\frac{c_{x,2}^2\kappa}{2\mu_{F}}\sum_{\tau=0}^{t}  \alpha_{\tau+1, t}\beta_{\tau} \EB\|\yhat_{\tau}\|^2 + \frac{8\Gamma_{11}}{\mu_F} \alpha_t 
		+ \frac{4c_{x,3}}{\mu_F}  \frac{\beta_t^2}{\alpha_{t}} + \frac{3 c_{x,4}}{\mu_F} \frac{\beta_t^{2 + 2\Hsmooth}}{\alpha_t^2} \notag\\
		&\overset{\eqref{eq:y-iterate}}{\le} 
		\frac{c_{x,2}^2\kappa}{2\mu_{F}}\sum_{\tau=0}^{t}  \alpha_{\tau+1, t}\beta_{\tau} \left(
		\beta_{0, {\tau-1}} \EB \tU_0
		+ \frac{16 \varrho_1 \Gamma_{11}}{\mu_G} \alpha_{\tau-1} + c_{y,1} \beta_{\tau-1} + c_{y,2} \frac{\beta_{\tau-1}^{2+2\Hsmooth}}{\alpha_{\tau-1}^2}
		\right) \notag \\
		& \quad \  +  \alpha_{0, t}	\EB\|\xhat_{0}\|^2+ \frac{8\Gamma_{11}}{\mu_F} \alpha_t +
		\frac{4c_{x,3}}{\mu_F}  \frac{\beta_t^2}{\alpha_{t}} + \frac{3 c_{x,4}}{\mu_F} \frac{\beta_t^{2+2\Hsmooth}}{\alpha_t^2}  \notag\\
		&\le \frac{7 L_H \LGy}{\LGx}  \beta_{0, t} \EB \tU_0 + \alpha_{0, t}\EB\|\xhat_{0}\|^2  + \frac{8\Gamma_{11}}{\mu_F} \alpha_t + c_{x,5} \beta_t + c_{x,6} \frac{\beta_t^2}{\alpha_t} + c_{x,7} \frac{\beta_t^{2+2\Hsmooth}}{\alpha_t^2},
	\end{align}
	where the first and last inequality use Lemma~\ref{lem:step-size-ineq}, $\frac{\beta_\tau}{\alpha_\tau} \le \kappa$, $\kappa \le \frac{\mu_F \mu_G}{28 L_H \LGx \LGy}$
	and the constants are defined in \eqref{eq:constantx2}.
	
	Note that 
	\begin{equation}
		\label{eq:U0}
		\EB \tU_0 = \frac{8\LGx^2\beta_0}{\mu_F\mu_G\alpha_0} \EB \|\xhat_{0}\|^2 + \EB\|\yhat_{0}\|^2
		\le \frac{8\LGx^2\kappa}{\mu_F\mu_G} \EB \|\xhat_{0}\|^2 + \EB\|\yhat_{0}\|^2 = \frac{ 2 \LGx \EB \|\xhat_{0}\|^2 }{7 L_H \LGy}+ \EB\|\yhat_{0}\|^2.
	\end{equation}
	Using~\eqref{eq:U0}, we then simplify~\eqref{eq:x-iterate} into
	\[
	\EB\|\xhat_{t+1}\|^2  \le \left( \alpha_{0,t} + 2\beta_{0,t}\right)\EB\|\xhat_{0}\|^2  + \frac{7 L_H \LGy}{\LGx} \beta_{0, t}\EB\|\yhat_{0}\|^2 + \frac{8\Gamma_{11}}{\mu_F} \alpha_t 
	+ c_{x,5} \beta_t + c_{x,6} \frac{\beta_t^2}{\alpha_t} + c_{x,7} \frac{\beta_t^{2+2\Hsmooth}}{\alpha_t^2},
	\]
	which together with $\alpha_{0,t}  \le \beta_{0,t}$ completes the proof.
\end{proof}

\subsection{Proof of Lemma~\ref{lem:step-size-ineq}}
\label{proof:step-size-ineq}

\begin{proof}[Proof of Lemma~\ref{lem:step-size-ineq}]
	\begin{enumerate}[(i)]
		\item This mainly follows from (i) in Lemma~\ref{lem:help}.
		\item We first show that $\left( \frac{\beta_{t-1}}{\beta_t} \right)^2 \le 1 + \frac{\mu_F}{8\kappa}\beta_t$.
		This is easy to prove by using  
		\[
		\left( \frac{\beta_{t-1}}{\beta_t} \right)^2
		\le \left( 1 + \frac{\mu_G}{32} \beta_t \right)^2
		\le \left( 1 + \frac{\mu_F}{32\kappa} \beta_t \right)^2
		\le 1 + \frac{\mu_F}{8\kappa}\beta_t.
		\]
		where the second inequality uses  $\kappa \le \frac{\mu_F}{\mu_G}$ and the last inequality uses $\beta_t \le  \iota_2 <\frac{64\kappa}{\mu_F}$ and $\frac{4\kappa}{\mu_F} = \frac{\iota_2}{3}$.
		Then the first inequality follows from (iii) in Lemma~\ref{lem:help}. 
		For the second inequality, note that $\sum_{\tau=0}^{t}  \alpha_{\tau+1, t}\beta_{\tau} \beta_{\tau-1}
		\le \left(1+\frac{\mu_G\iota_2}{32}\right)\sum_{\tau=0}^{t}  \alpha_{\tau+1, t}\beta_{\tau}^2 \le \left(1+\frac{\mu_G\iota_2}{32}\right)\frac{4}{\mu_F} \frac{\beta_t^2}{\alpha_t} \le \frac{6}{\mu_F}\frac{\beta_t^2}{\alpha_t}$.
		\item The first inequality follows from (iii) in Lemma~\ref{lem:help}.
		The second inequality can be established in a similar way as the first one by exchanging the states of $\alpha_t$ and $\beta_t$ and applying the fact that $\left( 1+\frac{\mu_G\iota_2}{8} \right)\frac{4}{\mu_F} \le \frac{4}{\mu_F} + \frac{\iota_2}{2\kappa} \le \frac{10}{\mu_F}$.
		\item Note that $\frac{\beta_t}{\alpha_t} \le \frac{\mu_F}{\mu_G}$, $1-\frac{\mu_F\alpha_t}{2} \le 1 - \frac{\mu_G\beta_t }{2}\le \left(1 - \frac{\mu_G}{4}\beta_t \right)^2$ and $ {1-\frac{\mu_G}{4}\beta_{t}}  \ge {1-\frac{\mu_G}{2}\iota_2} \ge 1- \frac{1}{2} = \frac{1}{2}$.
		It follows that
		\begin{align*}
			\sum_{\tau=0}^{t}  \alpha_{\tau+1, t}\beta_{\tau}  \beta_{0, \tau-1}
			&\le \sum_{\tau=0}^{t}  \beta_{\tau+1, t}^2\beta_{\tau}  \beta_{0, \tau-1} \\
			&\le \beta_{0, t} \sum_{\tau=0}^t \beta_{\tau+1, t}  \frac{\beta_{\tau}}{1-\frac{\mu_G}{4}\beta_{\tau}}\\
			&\le 2\beta_{0, t} \sum_{\tau=0}^t \beta_{\tau+1, t} \beta_{\tau} \le  \frac{8\beta_{0, t}}{\mu_G}.
		\end{align*}
\item
Define $\kappa_\tau = \frac{\beta_\tau}{\alpha_\tau}$. Then we have $\frac{\kappa_{\tau-1}}{\kappa_\tau} \le \frac{\beta_{\tau-1}}{\beta_\tau}$.
For $x \in (0, 0.1)$, one can check $(1+x)^{2+2 \Hsmooth } \le 1 + 4.8x$.
Since $\mu_G \beta_\tau \le \mu_G \iota_2 \le 1$,
the growth condition implies
\begin{align}
     \left( \frac{\beta_{\tau-1}}{\beta_\tau} \right)^{2 \Hsmooth} \left( \frac{\kappa_{\tau-1}}{\kappa_\tau} \right)^2 \left(1 - \frac{\mu_G \beta_\tau}{4} \right)
    & \le 
    \left( \frac{\beta_{\tau-1} }{ \beta_\tau } \right)^{2+2 \Hsmooth}
    \left( 1 - \frac{\mu_G \beta_\tau}{4} \right) \notag\\
    &\le \left( 1 + \frac{\mu_G \beta_\tau }{32} \right)^{2+2 \Hsmooth} \left( 1 - \frac{\mu_G \beta_\tau}{4} \right) \notag \\
    & = \left( 1 + \frac{3 \mu_G \beta_\tau }{20} \right) \left( 1 - \frac{\mu_G \beta_\tau}{4} \right)
    \le 1 - \frac{ \mu_G \beta_\tau }{10}.
    \label{lem:step-size-ineq:Eq1}
\end{align}
Then we have
\begin{align}
    \sum_{j=0}^t {\beta}_{j+1, t} \frac{\beta_j^{3+2\Hsmooth}}{\alpha_j^2}
    & = \sum_{j=0}^t \beta_j^{1 + 2 \Hsmooth} \kappa_j^2 \prod_{\tau=j+1}^t \left( 1 - \frac{\mu_G \beta_\tau}{4} \right) \notag \\
    & = \beta_t^{2 \Hsmooth} \kappa_t^2 \sum_{j=0}^t \beta_j \left( \frac{\beta_j}{\beta_t} \right)^{2\Hsmooth} \left( \frac{\kappa_j}{\kappa_t} \right)^2 \prod_{\tau=j+1}^t \left( 1 - \frac{\mu_G \beta_\tau}{4} \right) \notag \\
    & = \beta_t^{2 \Hsmooth} \kappa_t^2 \sum_{j=0}^t \beta_j \prod_{\tau=j+1}^t \left( \frac{\beta_{\tau-1} }{\beta_\tau} \right)^{2 \Hsmooth} \left( \frac{\kappa_{\tau-1}}{\kappa_\tau} \right)^2 \left( 1 - \frac{\mu_G \beta_\tau}{4} \right) \notag \\
    & \overset{\eqref{lem:step-size-ineq:Eq1}}{\le} \beta_t^{2\Hsmooth} \kappa_t^2 \sum_{j=0}^t \beta_j \prod_{\tau=j+1}^t \left( 1 - \frac{\mu_G \beta_\tau}{10} \right) \notag \\
    & = \frac{10 \beta_t^{2\Hsmooth} \kappa_t^2}{\mu_G} \left[ 1 - \prod_{\tau=0}^t \left( 1 - \frac{\mu_G \beta_\tau}{10} \right) \right]
    \le  \frac{10 \beta_t^{2 + 2 \Hsmooth}}{\mu_G \alpha_t^2}. \notag
\end{align}
\item 
Define $\kappa_\tau = \frac{\beta_\tau}{\alpha_\tau}$. Then we have $\frac{\kappa_{\tau-1}}{\kappa_\tau} \le \frac{\beta_{\tau-1}}{\beta_\tau}$.
For $x \in (0, 0.1)$, one can check $(1+x)^{2+2\Hsmooth} \le 1 + 4.8x$.
Since $\mu_G \beta_\tau \le \mu_G \iota_2 \le 1$ and $\frac{\beta_\tau}{\alpha_\tau} \le \kappa \le \frac{\mu_F}{\mu_G}$,
the growth condition implies
\begin{align}
     \left( \frac{\beta_{\tau-1}}{\beta_\tau} \right)^{2 \Hsmooth} \left( \frac{\kappa_{\tau-1}}{\kappa_\tau} \right)^2 \left(1 - \frac{\mu_F \alpha_\tau}{2} \right)
    & \le 
    \left( \frac{\beta_{\tau-1} }{ \beta_\tau } \right)^{2+2\Hsmooth}
    \left( 1 - \frac{\mu_F \alpha_\tau}{2} \right) \notag\\
    &\le \left( 1 + \frac{\mu_G \beta_\tau }{32} \right)^{2+2\Hsmooth } \left( 1 - \frac{\mu_F \alpha_\tau}{2} \right) \notag \\
    & = \left( 1 + \frac{\mu_G \beta_\tau }{6} \right) \left( 1 - \frac{\mu_F \alpha_\tau}{2} \right) \notag\\
    &\le \left( 1 + \frac{\mu_F \alpha_\tau }{6} \right) \left( 1 - \frac{\mu_F \alpha_\tau}{2} \right) \notag \\
    & \le 1 - \frac{ \mu_F \alpha_\tau }{3}.
    \label{lem:step-size-ineq:Eq2}
\end{align}
Then we have
\begin{align}
    \sum_{j=0}^t {\alpha}_{j+1, t} \frac{\beta_j^{2+2\Hsmooth}}{\alpha_j}
    & = \sum_{j=0}^t  \alpha_j \beta_j^{2\Hsmooth} \kappa_j^2 \prod_{\tau=j+1}^t \left( 1 - \frac{\mu_F \alpha_\tau}{2} \right) \notag \\
    & = \beta_t^{2 \Hsmooth} \kappa_t^2 \sum_{j=0}^t \alpha_j \left( \frac{\beta_j}{\beta_t} \right)^{2\Hsmooth} \left( \frac{\kappa_j}{\kappa_t} \right)^2 \prod_{\tau=j+1}^t \left( 1 - \frac{\mu_F \alpha_\tau}{2} \right) \notag \\
    & = \beta_t^{2 \Hsmooth} \kappa_t^2 \sum_{j=0}^t \alpha_j \prod_{\tau=j+1}^t \left( \frac{\beta_{\tau-1} }{\beta_\tau} \right)^{2 \Hsmooth } \left( \frac{\kappa_{\tau-1}}{\kappa_\tau} \right)^2 \left( 1 - \frac{\mu_F \alpha_\tau}{2} \right) \notag \\
    & \overset{\eqref{lem:step-size-ineq:Eq2}}{\le} \beta_t^{2\Hsmooth} \kappa_t^2 \sum_{j=0}^t \alpha_j \prod_{\tau=j+1}^t \left( 1 - \frac{\mu_F \alpha_\tau}{3} \right) \notag \\
    & = \frac{3 \beta_t^{2\Hsmooth} \kappa_t^2}{\mu_F} \left[ 1 - \prod_{\tau=0}^t \left( 1 - \frac{\mu_F \alpha_\tau}{3} \right) \right]
    \le  \frac{3 \beta_t^{2 + 2 \Hsmooth}}{\mu_F \alpha_t^2}. \notag
\end{align}
For the second inequality, we notice that 
\[
\frac{\beta_{\tau-1}^{1+2\Hsmooth} / \beta_\tau^{1+2\Hsmooth} }{\alpha_{\tau-1} / \alpha_\tau} \le \left( \frac{\beta_{\tau-1}}{\beta_\tau} \right)^{1+2\Hsmooth} \le \left(1 + \frac{\mu_G \beta_\tau}{32} \right)^3 \le \left( 1 + \frac{\mu_G \iota_2}{32} \right)^3 \le \frac{4}{3}.
\]
It follows that $\sum_{\tau=0}^t \alpha_{\tau+1, t} \beta_\tau \frac{\beta_{\tau-1}^{1 + 2\Hsmooth}}{\alpha_{\tau-1}} 
\le \frac{4}{3} \sum_{\tau=0}^t \alpha_{\tau+1, t} \frac{\beta_{\tau}^{2 + 2\Hsmooth}}{\alpha_{\tau}} 
\le \frac{4 }{\mu_F} \frac{\beta_t^{2 + 2\Hsmooth}}{\alpha_t^2}$.
\end{enumerate}
\end{proof}

\begin{lem}
	\label{lem:help}
	Let $\{ \alpha_t, \beta_t \}$ be nonincreasing positive numbers.
	\begin{enumerate}[(i)]
		\item If $\beta_{t} \le \frac{1}{a}$ and $\frac{\beta_{t-1}}{\beta_t} \le 1 + \frac{a}{2} \beta_t$ for all $t \ge 1$ and some $a >0$,
		$\sum_{j=0}^{t} \beta_j^2 \prod_{\tau=j+1}^{t} \left( 1- a \beta_{\tau} \right) \le \frac{2}{a}\beta_t$.
		\item If $\beta_t \le \kappa \alpha_t,  \alpha_{t} \le \frac{1}{a}$, and $\left( \frac{\beta_{t-1}}{\beta_t} \right)^2 \le 1 + \frac{a}{2\kappa}\beta_t$ for all $t \ge 1$ and some $a  >0$,
		$\sum_{j=0}^{t} \beta_j^2 \prod_{\tau=j+1}^{t} \left( 1- a \alpha_{\tau} \right) \le \frac{2}{a} \frac{\beta_{t}^2}{\alpha_{t}}$.
		\item If $\beta_{t} \le \frac{1}{a}$ and $\frac{\alpha_{t-1}}{\alpha_t} \le 1 + \frac{a}{2} \beta_t$ for all $t \ge 1$ and some $a >0$, 
		$\sum_{j=0}^{t} \beta_j\alpha_j \prod_{\tau=j+1}^{t} \left( 1- a \beta_{\tau} \right) \le \frac{2}{a } \alpha_t$.
	\end{enumerate}
\end{lem}

\begin{proof}[Proof of Lemma~\ref{lem:help}]
	See the proof of  \citet[Lemma~14]{kaledin2020finite}.
\end{proof}

\setcounter{section}{4}

\section{Omitted Proofs in Section~\ref{sec:decouple}
}
\label{proof:de}
In this section, we give the detailed proof of Theorem~\ref{thm:decouple-short}.
For generality, with $\Hsmooth$ define in Assumption~\ref{assump:smoothH}, we assume $\Hsmooth \in (0,1]$ instead of $\Hsmooth \in [0.5, 1]$ throughout this section.
We first present the formal version of Theorem~\ref{thm:decouple-short}.

\begin{thm}[Formal version of Theorem~\ref{thm:decouple-short}]
	\label{thm:decouple}
	Suppose that Assumptions~\ref{assump:smooth:FH} -- 
	\ref{assump:stepsize-new} hold. 
	In particular, we assume $\Hsmooth \in (0,1]$ instead of $\Hsmooth \in [0.5, 1]$ in Assumption~\ref{assump:smoothH}.
	Then we have for all $t \ge 0$,
	\begin{align}
		\EB \| \xhat_{t+1} \|^2 & \le \ccde_{x,0} \prod_{\tau=0}^t \left( 1 - \frac{\mu_G \beta_\tau}{4} \right) + \ccde_{x,1}\, \alpha_t + \ccde_{x,2}\frac{\beta_t^{2+2\Hsmooth}}{\alpha_t^2},
		\label{thm:decouple:x-formal} \\
		\| \EB \xhat_{t+1} \yhat_{t+1}^\top \| 
		& \le \ccde_{xy,0} \prod_{\tau=0}^t \left( 1 - \frac{\mu_G \beta_\tau}{4} \right)  
		+ \ccde_{xy,1} \,
		\beta_t
		+ \ccde_{xy,2} \frac{\beta_t^{1+2\Hsmooth} }{\alpha_t} \nonumber
		\\
		& \qquad 
		+ \left( \ccde_{xy,3} \, \alpha_t \beta_t + \ccde_{xy,4}    \frac{\beta_t^{5+4\Hsmooth} }{\alpha_t^5} \right) \left( \frac{\alpha_t}{\beta_t} \right)^\frac{2}{\Fsmooth}, \label{thm:decouple:xy-formal} \\
		\EB \| \yhat_{t+1} \|^2 & \le \ccde_{y,0} \prod_{\tau=0}^t \left( 1 - \frac{\mu_G \beta_\tau}{4} \right) 
		+ \ccde_{y,1}\, \beta_t + \ccde_{y,2} \frac{\beta_t^{1+2\Hsmooth} }{\alpha_t} \nonumber \\
		& \qquad 
		+ \left( \ccde_{y,3}\, \alpha_t \beta_t + \ccde_{y,4} \frac{ \beta_t^{5+4\Hsmooth} }{\alpha_t^5} \right) \left( \frac{\alpha_t}{\beta_t} \right)^{\frac{2}{\Fsmooth} } \notag \\
		& \qquad 
		+ \left( \ccde_{y,5}\, \alpha_t \beta_t + \ccde_{y,6} \frac{ \beta_t^{5+4\Hsmooth} }{\alpha_t^5} \right) \left( \frac{\alpha_t}{\beta_t} \right)^{\frac{1}{\Gsmooth} }, \label{thm:decouple:y-formal}
	\end{align}
	where $\{ \ccde_{x,i} \}_{i \in [2] \cup \{ 0 \} }$, $\{ \ccde_{xy,i} \}_{i \in [4] \cup \{ 0 \} }$ and $\{ \ccde_{y,i} \}_{i \in [6] \cup \{ 0 \} }$ are problem-dependent constants defined in 
	 \eqref{eq:constanx-ccde-012}, \eqref{eq:constantxy-ccde-01234} and \eqref{eq:constanty-ccde-0123456}.
	
	\end{thm}

\def\cprod{C_{prod}}

Under Assumption~\ref{assump:stepsize-new}~\ref{assump:stepsize-new:other}, we have
\begin{equation}\label{eq:cprod_def}
    \cprod := \sup_{t \ge 0} \frac{\prod_{\tau=0}^t \left( 1 - \frac{\mu_G \beta_\tau}{4} \right)}{\alpha_{t}^2} < \infty.
\end{equation}
With $\Hsmooth \ge 0.5$ and the constants bounds in Assumption~\ref{assump:stepsize-new}~\ref{assump:stepsize-new:constant}, the constants in \eqref{thm:decouple:x}, \eqref{thm:decouple:xy} and \eqref{thm:decouple:y} can be defined as
\begin{align}\label{eq:thm:decouple-constants}
	\begin{split}
		\CM_{x} & = \ccde_{x,0}\, \cprod\, \iota_1 + \ccde_{x,1} + \ccde_{x,2}\, \iota_2^{2\Hsmooth-1} \kappa^3,\\ 
		\CM_{xy,1} & = \ccde_{xy,1} + \ccde_{xy,2}\, \iota_2^{2\Hsmooth-1} \kappa, \quad
		\CM_{xy,2} = \ccde_{xy,0}\, \cprod\, \kappa^{\frac{2}{\Fsmooth}-1} + \ccde_{xy,3} + \ccde_{xy,4}\, \iota_2^{4\Hsmooth-2} \, \kappa^6,\\
		\CM_{y,1} & = \ccde_{y,1} + \ccde_{y,2}\, \iota_2^{2\Hsmooth-1} \kappa, \quad
		\CM_{y,2} = \ccde_{y,0}\, \cprod\, \kappa^{\frac{2}{\Fsmooth}-1} + \ccde_{y,3} + \ccde_{y,4}\, \iota_2^{4\Hsmooth-2} \, \kappa^6, \\
		\CM_{y,3} & =  \ccde_{y,5} + \ccde_{y,6}\, \iota_2^{4\Hsmooth-2} \, \kappa^6.
	\end{split}
\end{align}

Next, we give an example choice of the step sizes in Corollary~\ref{cor:decouple-rates}.
	\begin{equation}
	\label{eq:poly0-decou}
	\begin{split}
		& \alpha_t = \frac{128}{(\Fsmooth \wedge \Gsmooth) \mu_G \kappa(t+T_0)^{a}} ~\text{and}~\beta_t   = \frac{128}{(\Fsmooth \wedge \Gsmooth) \mu_G (t+T_0)^b}~ \\
		& \text{with}~
		a,b \in (0,1],~
		1 \le \frac{b}{a} \le 1 + \frac{\Fsmooth}{2} \wedge \Gsmooth,
		\TM_1 = \frac{128}{\mu_G (\kappa \iota_1 \wedge \iota_2 \wedge \frac{\rho}{\kappa})}
		~\text{and}~
		T_0 \ge
		\left\lceil \TM_1^{1/a}
		\right\rceil.
	\end{split}
\end{equation}

We emphasize that in the proof, our primary focus is on the order of the mean squared error, rather than optimizing the constants. The constants in \eqref{eq:thm:decouple-constants} serve only to provide an upper bound that ensures Theorem~\ref{thm:decouple} holds, and may be quite loose in certain cases, such as SGD with averaging in \eqref{eq:SGD-both}. The step size choice in \eqref{eq:poly0-decou} is thus presented as a feasible example for general cases. For specific examples, however, the optimal values of $\alpha_{0}$, $\beta_{0}$ and $T_0$, 
can vary significantly.
For example, in \eqref{eq:SGD-both},  $\beta_{0} $ and $T_0 $ can both be set as $1$.

Before present the detailed proof, we first give an operator decomposition lemma, which states that 
with Assumption~\ref{assump:near-linear}, we could decompose the nonlinear operators $F$ and $G$ as the linear parts plus higher-order error terms.
\begin{lem}[Operator decomposition]
	\label{lem:decom}
	Suppose that Assumption~\ref{assump:near-linear} holds.
	With $\xhat_t$ and $\yhat_t$ defined in \eqref{alg:xyhat}, we have the following results.
	\begin{enumerate}[(i)]
		\item $F(x_t, y_t) = B_1 \xhat_t + R_t^F$ with $\|R_t^F\| \le \SBF (\|\xhat_{t}\|^{1+\Fsmooth} + \|\yhat_t\|^{1+\Fsmooth}) $.
		It follows that $\xhat_t - \alpha_t  F(x_t, y_t) = (I-\alpha_t B_1) \xhat_t - \alpha_t R_t^F$.
		We further have $\|I-\alpha_t B_1\| \le  1- \mu_F \alpha_t$ if $0 \le \alpha_t \le \frac{\mu_F}{L_F^2}$.
		\item $G(H(y_t), y_t) =  B_3\yhat_{t} + R_t^{GH}$ with $\|R_t^{GH}\| \le \SBG \| \yhat_t \|^{1+\Gsmooth}$.  It follows that $\yhat_{t} - \beta_t G(H(y_t), y_t) = (I-\beta_t B_3)\yhat_{t} - \beta_t R_t^{GH}$.
		We further have $\|I-\beta_t B_3\| \le  1- \mu_G \beta_t$ if $0 \le \beta_t \le \frac{\mu_G}{\LGy^2}
		$.
		\item $G(x_t, y_t) = B_2\xhat_t + B_3\yhat_t + R_t^{G}$ with $\|R_t^{G}\| \le \SBG (\|\xhat_{t}\|^{1+\Gsmooth} + \|\yhat_t\|^{1+\Gsmooth})$. 
	\end{enumerate}
\end{lem}

\subsection{Proof of Lemma~\ref{lem:xhat-new}}
\label{proof:de:xhat-new}

\begin{proof}[Proof of Lemma~\ref{lem:xhat-new}]
We first present the specific forms of constants: 
\begin{equation}
\label{eq:constantx:new}
  \begin{aligned}
		\cde_{x,1} & =  4 L_{H}^2 \LGy^2,\  
        \cde_{x,2} = 2 d_x L_H \LGy, \ \cde_{x,3} =  2 L_H^2\Gamma_{22}, \  
        \cde_{x,4} = \frac{96 d_x S_H^2 \Gamma_{22}^{1+\Hsmooth}}{\mu_F}, \\
        \cde_{x,5} & = 4 d_x \Hholder (\LGx \wedge \LGy) \propto S_H, \ \cde_{x.6} = 4 d_x \SBF L_H (\LGx \wedge \LGy) \propto \SBF, \\ \cde_{x,7} & = \frac{d_x L_H \SBG^2}{\LGx} \propto \SBG^2, \  \cde_{x,8} = \frac{8 d_x S_H^2 (\LGx^{2+2\Hsmooth} \wedge \LGy^{2+2\Hsmooth}) }{L_H^{1+\Hsmooth} \LGx^{1+\Hsmooth}} \propto S_H^2.
    \end{aligned}
\end{equation}

	The decomposition in~\eqref{lem:xhat:Eq1} implies that
	\begin{align}
		\begin{split}
        \EB	\|\xhat_{t+1}\|^2
        &\overset{\eqref{lem:xhat:Eq1a} + \eqref{lem:xhat:Eq1b}}{\le}  \left(1-\frac{3\mu_{F}\alpha_{t}}{2} \right) \EB \|\xhat_{t}\|^2 +  c_{x,1}\beta_{t}^2 \EB \|\yhat_{t}\|^2     +  2\beta_{t}^2L_H^2\Gamma_{22} +  2\alpha_{t}^2\Gamma_{11}\\
				&\qquad + 2 \EB\underbrace{\langle\xhat_{t}-\alpha_t F(x_{t},y_{t}), H(y_{t}) - H(y_{t+1}) \rangle}_{\spadesuit},
        \end{split}
        \label{lem:xhat-new:Eq1}
	\end{align}
 where we have also used the fact that $4L_{H}^2 \LGx^2 \beta_{t}^2 \le  \frac{ \mu_F \alpha_{t}}{4}$.
The proof is almost identical to that of Lemma~\ref{lem:xhat} except that we take additional care on the cross term $\EB \spadesuit$.
		
By Lemma~\ref{lem:decom}, we have $\xhat_t - \alpha_t  F(x_t, y_t) = (1-\alpha_t B_1) \xhat_t - \alpha_t R_t^F$ with $\|R_t^F\| \le \SBF (\|\xhat_{t}\|^{1+\Fsmooth} + \|\yhat_t\|^{1+\Fsmooth})$ and $ y_t - y_{t+1} = \yhat_{t}-\yhat_{t+1} = \beta_t (B_2\xhat_{t} + B_3 \yhat_{t} + R_t^G + \psi_t)$ with $\|R_t^G\| \le \SBG (\|\xhat_{t}\|^{1+\Gsmooth} + \|\yhat_t\|^{1+\Gsmooth})$.
		By Assumption~\ref{assump:smoothH}, it follows that
		\begin{align*}
			H(y_{t+1}) - H(y_{t}) &= \nablaH(y_{t}) (y_{t+1}-y_{t}) + R_t^H 
			= \nablaH(y^{\star}) (y_{t+1}-y_{t}) + R_t^{\nablaH} + R_t^H
		\end{align*}
where $\|R_t^H\| \le S_H\|y_{t+1}-y_{t}\|^{1+\Hsmooth}$ and $ R_t^{\nablaH} =  (\nablaH(y_{t}) - \nablaH(y^{\star}))(y_{t+1}-y_{t})$.
Therefore, we have
\begin{align}
	\label{eq:x-corss}
	\EB \spadesuit &= \EB \tr(\spadesuit) = -\tr\left( \EB (H(y_{t+1}) - H(y_{t})) (\xhat_{t}-\alpha_t F(x_{t},y_{t}))^\top  \right)\notag \\
	&\le d_x  \left\| \EB (  \nablaH(y^{\star}) (y_{t+1}-y_{t}) + R_t^{\nablaH} + R_t^H ) ((I-\alpha_t B_1) \xhat_t - \alpha_t R_t^F)^\top \right\| \notag \\
	\begin{split}
			&\le d_x \cdot \bigg[
		\underbrace{\vphantom{\frac{a}{b}} \left\| \nablaH(y^{\star})  \EB (y_{t+1}-y_{t}) \xhat_t^\top (1-\alpha_t B_1)^{\top}  \right\|}_{\spadesuit_1} + 
		\underbrace{\vphantom{\frac{a}{b}} \EB \|R_t^H\|\|\xhat_{t}-\alpha_t F(x_{t},y_{t}) \|}_{\spadesuit_2} \\
		&\qquad + 
		\alpha_t  \underbrace{\vphantom{\frac{a}{b}} \left\| \nablaH(y^{\star})  \EB (y_{t+1}-y_{t})  (R_t^F)^{\top}  \right\|}_{\spadesuit_3} + \underbrace{\vphantom{\frac{a}{b}} \|\EB R_t^{\nablaH}   \left(
		\xhat_{t}-\alpha_t F(x_{t},y_{t}) \right)^\top \|}_{\spadesuit_4}
		\bigg].
	\end{split}
\end{align}
We then analyze the four terms $\{\spadesuit_i\}_{i \in [4]}$ on the right-hand side of the last inequality.
\begin{prop}
	\label{prop:simple}
	For any random variable $X \ge 0$ and real number $a > 0$, it follows that
	\[
	2\EB X^3 \le a\EB X^2 +  \frac{1}{a}\EB X^4.
	\]
\end{prop}
\begin{itemize}
	\item 
	For the term $\spadesuit_1$, using $\|I-\alpha_t B_1\| \le 1 - \mu_F \alpha_t \le 1$ and $\| \nablaH (y^\star) \| \le L_H$, we have
	\begin{align}
		\label{eq:x-cross-1}
		\spadesuit_1 
		& =	\beta_t	\left\| \EB  \nablaH(y^{\star})  (B_2\xhat_{t} + B_3 \yhat_{t} + R_t^G )  \xhat_t^\top (I-\alpha_t B_1)^{\top}  \right\| \notag \\
		& \le \beta_t L_H \left(
		\LGx \EB \|\xhat_{t}\|^2 + \LGy \|\EB \yhat_{t} \xhat_{t}^\top\| +   \EB\|R_t^G\|\|\xhat_{t}\|
		\right) \notag \\
		& \overset{\eqref{eq:3to4}}{\le} \beta_t L_H \left(
		2\LGx \EB \|\xhat_{t}\|^2 + \LGy \|\EB \yhat_{t} \xhat_{t}^\top\| +   \frac{\SBG^2}{2 \LGx} \left( \EB \|\xhat_{t}\|^{2+2\Gsmooth} + \EB \|\yhat_{t}\|^{2+2\Gsmooth}  \right)
		\right).
	\end{align}
Here the last inequality uses the following result which could be obtained by Proposition~\ref{prop:simple},
\begin{align}
	\label{eq:3to4}
	\EB\|R_t^G\|\|\xhat_{t}\| 
	&\le \LGx \EB\|\xhat_{t}\|^2 + \frac{ \SBG^2 }{2 \LGx } \left( \EB \|\xhat_{t}\|^{2+2\Gsmooth} + \EB \|\yhat_{t}\|^{2+2\Gsmooth}  \right).
\end{align}
	\item For the term $\spadesuit_2$, it follows that
 	\begin{align}
				\label{eq:x-cross-2}
\spadesuit_2
&\le S_H \cdot \EB \|y_{t+1}-y_{t}\|^{1+\Hsmooth} \|\xhat_{t}-\alpha_t F(x_{t},y_{t}) \| \notag \\
&\overset{\eqref{lem:xhat:Eq1a}+\eqref{lem:xhat:Eq1b1}}{\le}  
2 S_H \beta_t^{1+\Hsmooth} \left( 2 \LGx^{1+\Hsmooth} \EB\|\xhat_{t}\|^{2+\Hsmooth} + 2 \LGy^{1+\Hsmooth} \EB\|\yhat_{t}\|^{1+\Hsmooth} \|\xhat_{t}\| + \Gamma_{22}^{\frac{1+\Hsmooth}{2}} \EB\|\xhat_{t}\| \right) \notag \\
&\le  2L_H^{1+\Hsmooth} \LGx^{1+\Hsmooth} \beta_t^{1+\Hsmooth} \EB \|\xhat_{t}\|^2 + \frac{4 S_H^2 \beta_t^{1+\Hsmooth} }{L_H^{1+\Hsmooth} \LGx^{1+\Hsmooth}} \left( \LGx^{2+2\Hsmooth} \EB \|\xhat_{t}\|^{2+2\Hsmooth} +  \LGy^{2+2\Hsmooth} \EB \|\yhat_{t}\|^{2+2\Hsmooth} \right) \notag \\
& \quad \ + \frac{\mu_F \alpha_t}{12 d_x} \EB \| \xhat_t \|^2 + \frac{48 d_x S_H^2 \Gamma_{22}^{1+\Hsmooth}}{\mu_F} \frac{\beta_t^{ 2+2\Hsmooth}}{ \alpha_t} \notag \\
&\le \frac{\mu_F\alpha_t}{6 d_x} \EB \|\xhat_{t}\|^2 + \frac{4 S_H^2 \beta_t^{1+\Hsmooth} }{L_H^{1+\Hsmooth} \LGx^{1+\Hsmooth} } \left( \LGx^{2+2\Hsmooth} \EB \|\xhat_{t}\|^{2+2\Hsmooth} +  \LGy^{2+2\Hsmooth} \EB \|\yhat_{t}\|^{2+2\Hsmooth} \right) \notag \\
& \quad \ + \frac{48 d_x S_H^2 \Gamma_{22}^{1+\Hsmooth}}{\mu_F} \frac{\beta_t^{ 2+2\Hsmooth}}{ \alpha_t} ,
	\end{align}
where the second inequality also use $(a+b)^{1+\delta} \le 2 (a^{1+\delta} + b^{1+\delta})$ for any $a,b \ge 0$ and $\delta \in [0,1]$, the third inequality uses Cauchy–Schwarz inequality,
and the last inequality uses $(L_{H} \LGx \beta_{t})^{1+\Hsmooth}
\le L_{H} \LGx \beta_{t}
\le  \frac{ \mu_F \alpha_{t}}{24d_x} < 1$.
	\item For the term $\spadesuit_3$, it follows that
	\begin{align}
		\label{eq:x-cross-3}
		\spadesuit_3 
		&\le \beta_t L_H \SBF \EB \|G(x_t, y_t)\|  \cdot \left( \|\xhat_{t}\|^{1+\Fsmooth} + \|\yhat_{t}\|^{1+\Fsmooth}   \right)\notag \\
		&\overset{\eqref{lem:xhat:Eq1b1}}{\le}\beta_t L_H \SBF \EB  \left( \LGx \|\xhat_{t}\| + \LGy \|\yhat_{t}\| \right) \cdot \left( \|\xhat_{t}\|^{1+\Fsmooth} + \|\yhat_{t}\|^{1+\Fsmooth}   \right)\notag \\
		&\le   2 \beta_t \SBF L_H (\LGx \wedge \LGy) \left(  \EB \|\xhat_{t}\|^{2+\Fsmooth} + \EB\|\yhat_{t}\|^{2+\Fsmooth} \right).
	\end{align}
	\item For the term $\spadesuit_4$, by the definition of $R_t^{\nabla H}$, it follows that
		\begin{align}
		\label{eq:x-cross-4}
		\spadesuit_4 
		&\overset{ \eqref{assump:H:holder} + \eqref{lem:xhat:Eq1a} }{\le} \beta_t \Hholder \EB \|\yhat_{t}\|^{\Hsmooth} \|G(x_t, y_t)\| \|\xhat_{t}\|  \notag \\
		&\overset{\eqref{lem:xhat:Eq1b1}}{\le}  \beta_t \Hholder \left( \LGx \EB\|\xhat_{t}\|^2\|\yhat_{t}\|^{\Hsmooth} + \LGy \EB\|\yhat_{t}\|^{1+\Hsmooth} \|\xhat_{t}\| \right)  \notag \\
		&\le 2 \beta_t \Hholder (\LGx \wedge \LGy) \left(  \EB \|\xhat_{t}\|^{2+\Hsmooth} + \EB\|\yhat_{t}\|^{2+\Hsmooth} \right),
	\end{align}
    where the last step uses 
Young's inequality.
\end{itemize}

Plugging~\eqref{eq:x-cross-1},~\eqref{eq:x-cross-2},~\eqref{eq:x-cross-3}, and~\eqref{eq:x-cross-4} into~\eqref{eq:x-corss}, we have that
\begin{align}
	\EB \spadesuit 
    & \overset{ \eqref{eq:constantx:new} + \eqref{eq:delta_x} }{\le} \frac{ \mu_F \alpha_t}{4} \EB \| \xhat_t \|^2
    + \frac{ \cde_{x,2} \beta_t }{2}
    \| \EB \xhat_t \yhat_t^\top \|   
    + \frac{\cde_{x,4}}{2} \frac{\beta_t^{2+2\Hsmooth} }{\alpha_t} + \frac{\Delta_{x,t}}{2}, \label{lem:xhat-new:Eq2}
\end{align}
where the inequality also uses $L_H \LGx \beta_t \le \frac{\mu_F \alpha_t}{ 24 d_x }$ and 
we use the constants $\cde_{x,5}$ to $\cde_{x,8}$ defined in \eqref{eq:constantx:new} to hide the problem-dependent coefficients
to yield the expression of $\Delta_{x,t}$ in \eqref{eq:delta_x}.
Moreover, by Proposition~\ref{prop:ensure-linearity}, we have $\cde_{x,5} \propto S_H$.
Then plugging \eqref{lem:xhat-new:Eq2} into \eqref{lem:xhat-new:Eq1} gives
\begin{align*}
    \EB	\|\xhat_{t+1}\|^2 &\overset{\eqref{lem:xhat-new:Eq1}}{\le}  \left(1-\frac{3\mu_{F}\alpha_{t}}{2} \right) \EB \|\xhat_{t}\|^2 +  \cde_{x,1} \beta_{t}^2 \EB \|\yhat_{t}\|^2     +  2\beta_{t}^2L_H^2\Gamma_{22} +  2\alpha_{t}^2\Gamma_{11} + 2 \EB{\spadesuit} \\
    & \overset{ \eqref{eq:constantx:new} + \eqref{lem:xhat-new:Eq2} }{\le} \left( 1 - \mu_F \alpha_t  \right) \EB \| \xhat_t \|^2 + \cde_{x,1} \beta_t^2 \EB \| \yhat_t \|^2 + \cde_{x,2} \beta_t \| \EB \xhat_t \yhat_t^\top \| + 2 \Gamma_{11} \alpha_t^2 \\
    & \qquad + \cde_{x,3} \beta_t^2 + \cde_{x,4} \frac{\beta_t^{2+2\Hsmooth}}{\alpha_t} + \Delta_{x,t}, 
\end{align*}
which is the desired result.
\end{proof}

\subsection{Proof of Lemma~\ref{lem:yhat-new}}
\label{proof:de:yhat-new}

\begin{proof}[Proof of Lemma~\ref{lem:yhat-new}]
The decomposition in~\eqref{lem:yhat:Eq1} implies that
\begin{align}
		\EB \|\yhat_{t+1}\|^2 
  \label{lem:yhat-new:Eq1}
  \begin{split}
  & \overset{ \eqref{lem:yhat:Eq1a} + \eqref{lem:yhat:Eq1b} }{ \le } (1 - \mu_G \beta_t) \EB \| \yhat_{t} \|^2 + \LGx^2 \beta_t^2 \EB \| \xhat_{t} \|^2 + \beta_t^2 \Gamma_{22}  \\
  & \quad \ + 2 \beta_t \EB \langle \yhat_{t} - \beta_{t} G(H(y_{t}),y_{t}), G(H(y_{t}),y_{t}) - G(x_{t},y_{t})   \rangle
\end{split}
\end{align}
The proof is almost identical to that of Lemma~\ref{lem:yhat} except that we take additional care on the last term.

By Lemma~\ref{lem:decom}, we have 
\begin{align*}
\yhat_{t} - \beta_t G(H(y_t), y_t) &= (I-\beta_t B_3)\yhat_{t} - \beta_t R_t^{GH}\\
G(H(y_t), y_t) &= B_3 \yhat_t + R_t^{GH}\\
G(x_t, y_t) &= B_2\xhat_t + B_3\yhat_t + R_t^{G}
\end{align*}

with $\|R_t^{GH}\| \le \SBG \|\yhat_t\|^{1+\Gsmooth}$ and $\|R_t^{G}\| \le \SBG (\|\xhat_{t}\|^{1+\Gsmooth} + \|\yhat_t\|^{1+\Gsmooth})$. 
Therefore, we have
\begin{align}
    & \quad \ \EB \langle \yhat_{t} - \beta_{t} G(H(y_{t}),y_{t}), G(H(y_{t}),y_{t}) - G(x_{t},y_{t})   \rangle \notag \\
    & \le d_y \left\| \EB ( (I-\beta_t B_3)\yhat_{t} - \beta_t R_t^{GH} )
    ( - B_2 \xhat_{t} + R_t^{GH} - R_t^G )^\top
    \right\| \notag \\
    & \le d_y \Big[ (1 - \mu_G \beta_t ) \|B_2 \| \| \EB \yhat_{t} \xhat_{t}^\top \|
    + (1 - \mu_G \beta_t ) \EB \| \yhat_t \|
    (\| R_t^{GH} \| + \| R_t^G \| ) \notag
    \\
    & \quad \ + \beta_t \| B_2 \| \EB \|R_t^{GH} \| \| \xhat_t \|
    + \beta_t \EB \| R_t^{GH} \| 
    (\| R_t^{GH} \| + \| R_t^G \| ) 
    \Big] \notag \\
    \label{lem:yhat-new:Eq2}
    \begin{split}
    & \le d_y \LGx \| \EB \xhat_t \yhat_t^\top \| + \frac{\mu_G}{6} \EB \| \yhat_t \|^2 + \frac{\LGx^2}{2} \beta_t \EB \| \xhat_t \|^2  \\ 
    & \quad \ + \SBG^2 \left( \frac{15 d_y^2}{2 \mu_G} + \frac{d_y^2}{2} \beta_t + 4 d_y \beta_t \right) \left( \EB \| \xhat_t \|^{2+2\Gsmooth} + \EB \| \yhat_t \|^{2+2\Gsmooth} \right),
    \end{split}
\end{align}
where the second inequality uses $\| I - \beta_t B_3 \| \le 1 - \mu_G \beta_t$
and the last inequality uses Cauchy–Schwarz inequality.
Plugging \eqref{lem:yhat-new:Eq2} into \eqref{lem:yhat-new:Eq1} yields
\begin{align*}
    \EB \|\yhat_{t+1}\|^2 
    & \overset{ \eqref{lem:yhat-new:Eq1} }{ \le } (1 - \mu_G \beta_t) \EB \| \yhat_{t} \|^2 + \LGx^2 \beta_t^2 \EB \| \xhat_{t} \|^2 + \beta_t^2 \Gamma_{22}  \\
    & \quad \ + 2 \beta_t \EB \langle \yhat_{t} - \beta_{t} G(H(y_{t}),y_{t}), G(H(y_{t}),y_{t}) - G(x_{t},y_{t})   \rangle \\
    & \overset{ \eqref{lem:yhat-new:Eq2} }{\le}
    \left( 1 - \frac{2 \mu_G \beta_t}{3} \right) \EB \| \yhat_t \|^2 + 2 \LGx^2 \beta_t^2 \EB \| \xhat_t \|^2 + 2 d_y \LGx \beta_t \| \EB \xhat_t \yhat_t^\top \| + \Gamma_{22} \beta_t^2 + \Delta_{y,t},
\end{align*}
where we use $\Delta_{y,t}$ defined in \eqref{eq:delta_y} to collect the remaining terms.
\end{proof}

\subsection{Proof of Lemma~\ref{lem:xyhat}}
\label{proof:de:xyhat}

\begin{proof}[Proof of Lemma~\ref{lem:xyhat}]
We first present the specific forms of the constants:
\begin{equation}
\label{eq:constant-xy}
 \begin{aligned}
 \cde_{xy,1} & = L_H \LGx + \Gamma_{22}^\frac{1+\Hsmooth}{2} S_H + 6 \iota_2 L_H \LGy^2, \
	\cde_{xy,2} = (2+L_H) \Sigma_{22} + 4 \iota_2^{\Hsmooth} \Gamma_{22}^\frac{2+\Hsmooth}{2}, \
 \cde_{xy,3} = S_H \Gamma_{22}^\frac{1+\Hsmooth}{2}, \\
     \cde_{xy,4}
    & = 2 \Hholder (\LGx \wedge \LGy) + 8 \iota_2^\Hsmooth S_H ( \LGx^{1+\Hsmooth} \wedge \LGy^{1+\Hsmooth}) +  16 \iota_2^{1+\Hsmooth} S_H (\LGx^{2+\Hsmooth} + \LGy^{2+\Hsmooth}) \propto S_H.
 \end{aligned}
\end{equation}

	Using the update rules~\eqref{eq:xhat_update} and~\eqref{eq:yhat_update}, we have
\begin{align*}
	\EB&\left[\xhat_{t+1}\yhat_{t+1}^\top \right] 
	=\EB \underbrace{\vphantom{\frac{a}{b}}  \left(   \xhat_{t} -  \alpha_{t}F(x_{t},y_{t})  \right)\left( \yhat_{t}  - \beta_{t}G(H(y_{t}),y_{t}) \right)^\top  }_{\blacklozenge_1} + \alpha_{t}\beta_{t} \EB \left[ \xi_t  \psi_t^\top\right]\\
		& + \EB\underbrace{ \vphantom{\frac{a}{b}}
		 \beta_{t} \left(   \xhat_{t} -  \alpha_{t}F(x_{t},y_{t})  \right)\left(  G(H(y_{t}),y_{t})-G(x_{t},y_{t}) \right)^\top
	}_{\blacklozenge_2} 
+\EB  \underbrace{ \vphantom{\frac{a}{b}} \left(H(y_{t}) - H(y_{t+1})\right)^\top \yhat_{t+1}}_{\blacklozenge_3}.
\end{align*}
It then follows that
\begin{equation}
	\label{eq:xy-eq0}
	\|\EB\xhat_{t+1}\yhat_{t+1}^\top\| \le \|\EB\blacklozenge_1\| + \|\EB\blacklozenge_2\| + \|\EB\blacklozenge_3\| + \alpha_t \beta_t \Sigma_{12} .
\end{equation}
We then analyze $\{\|\EB\blacklozenge_i\|\}_{i \in [3]}$ in the following respectively.

To analyze $\blacklozenge_1$, we make use of Lemma~\ref{lem:decom} and obtain that
\begin{align*}
	\blacklozenge_1 
	&=(1-\alpha_t B_1) \xhat_t\yhat_{t}^\top (1-\beta_t B_3)^\top - \beta_t(1-\alpha_t B_1) \xhat_t \left(R_t^{GH}\right)^\top \\
	&\qquad -\alpha_tR_t^F \yhat_{t}^\top (1-\beta_t B_3)^\top +  \alpha_t \beta_t R_t^F \left( R_t^{GH}\right)^\top.
\end{align*}
Taking expectation and then the spectrum norm on both sides, we have
\begin{align}
\|\EB \blacklozenge_1 \| 
&\le (1-\mu_F\alpha_t)(1-\mu_G\beta_t)\|\EB \xhat_t\yhat_{t}^\top \| 
+  \beta_t (1-\mu_F\alpha_t) \SBG \EB\|\xhat_{t}\| 
\|\yhat_t\|^{1+\Gsmooth} \notag\\
&\qquad + \alpha_t (1-\mu_G\beta_t) \SBF \EB \|\yhat_{t}\| \left(\|\xhat_t\|^{1+\Fsmooth} + \|\yhat_t\|^{1+\Fsmooth} \right) \notag \\
& \qquad +  \alpha_t\beta_t \SBF \SBG \EB\left( \|\xhat_t\|^{1+\Fsmooth} + \|\yhat_t\|^{1+\Fsmooth} \right) \| \yhat_t \|^{1+\Gsmooth} \notag\\
&\le (1-\mu_F\alpha_t)(1-\mu_G\beta_t)\|\EB \xhat_t\yhat_{t}^\top \|   + \Delta_{xy,t}^{(1)},	\label{eq:xy-eq1}
\end{align}
Here we use $\Delta_{xy,t}^{(1)}$ to denote the higher-order residual collecting all the remaining terms for simplicity.
More specifically, it follows that
\begin{align*}
	\Delta_{xy,t}^{(1)} &= 
	 \beta_t (1-\mu_F\alpha_t) \SBG \EB\|\xhat_{t}\| 
  \|\yhat_t\|^{1+\Gsmooth} 
     +  \alpha_t (1-\mu_G\beta_t) \SBF \EB \|\yhat_{t}\| \left(\|\xhat_t\|^{1+\Fsmooth} + \|\yhat_t\|^{1+\Fsmooth} \right) \\
	 &\qquad + \alpha_t\beta_t \SBF \SBG \EB \| \yhat_t \|^{1+\Gsmooth} \left( \|\xhat_t\|^{1+\Fsmooth} + \|\yhat_t\|^{1+\Fsmooth} \right).
\end{align*}
Now we can apply Young's inequality to decouple the cross terms and obtain
\begin{align}\label{eq:delta_xy_1}
    \Delta_{xy,t}^{(1)}
    & \le \beta_t \SBG ( \EB \| \xhat_t \|^{2+\Gsmooth} + \EB \| \yhat_t \|^{2+\Gsmooth} )
    + 2 \alpha_t \SBF (\EB \| \xhat_t \|^{2+\Fsmooth} + \EB \| \yhat_t \|^{2+\Fsmooth}) \notag \\
    & \qquad
    + 2 \alpha_t \beta_t \SBF \SBG  (\EB \| \xhat_t \|^{2 + \Fsmooth + \Gsmooth} + \EB \| \yhat_t \|^{2 + \Fsmooth + \Gsmooth}).
\end{align}

To analyze $\blacklozenge_2$, we have
\begin{align}
		\label{eq:xy-eq2}
\|\blacklozenge_2\| 
&\overset{\eqref{lem:xhat:Eq1a}}{\le}  \beta_t \sqrt{1{-}\mu_F \alpha_{t}}  \|\xhat_{t}\| \cdot \|G(H(y_{t}),y_{t})-G(x_{t},y_{t})\|  
\overset{\eqref{assump:G:smooth}}{\le} \beta_t\sqrt{1{-}\mu_F \alpha_{t}} \cdot \LGx \|\xhat_{t}\|^2.
\end{align}
To analyze $\blacklozenge_3$, we will use the near linearity in Assumption~\ref{assump:near-linear} again.
 By Assumption~\ref{assump:smoothH}, we have that
 \begin{align*}
 		H(y_{t+1}) - H(y_t) &= \nabla H(y_t) (y_{t+1}-y_{t}) + R_t^H 
 		=\nabla H(y^{\star}) (y_{t+1}-y_{t}) + R_t^{\nabla H}   + R_t^H,
 \end{align*}
where $R_t^H$ and $R_t^{\nabla H}$ are defined by
\begin{gather*}
	R_t^H := H(y_{t+1}) - H(y_t) -  \nabla H(y_t) (y_{t+1}-y_{t}), \\
	R_t^{\nabla H} := \left(\nabla H(y_t) -\nabla H(y^{\star}) \right)(y_{t+1}-y_{t}).
\end{gather*}
with $R_t^H$ satisfies  $\|R_t^H\| \le S_H \|y_{t+1}-y_{t}\|^{1+\Hsmooth}$.
Note that $y_{t+1} -y_{t}  = \beta_t (G(x_t, y_t) + \psi_t)$.
Then
\begin{align*}
\|\EB \blacklozenge_3 \|
&\le\beta_t\|\nabla H(y^{\star}) \EB G(x_t, y_t) \yhat_{t}^\top \|+ \beta_t^2 \| \nabla H(y^{\star}) \EB G(x_t, y_t)G(x_t, y_t)^\top\| \\
&\qquad + \beta_t^2\|  \nabla H(y^{\star})  \EB \psi_t \psi_t^\top \| +\|\EB R_t^{\nabla H} \yhat_{t+1}^\top \| + S_H \EB\left[ \|y_{t+1}-y_{t}\|^{1+\Hsmooth} \|\yhat_{t+1}\|\right] \\
&\le \beta_t L_H  
\| \EB \underbrace{ G(x_t, y_t) \yhat_{t}^\top}_{\Diamond_1} \|
+  \beta_t^2L_H ( \EB\underbrace{\| G(x_t, y_t)\|^2}_{\Diamond_2} +  
{\Sigma_{22}} ) \\
&\qquad+\|\EB \underbrace{ R_t^{\nabla H} \yhat_{t+1}^\top }_{\Diamond_3}\| + S_H \EB \underbrace{ \|y_{t+1}-y_{t}\|^{1+\Hsmooth} \|\yhat_{t+1}\|}_{\Diamond_4}.
\end{align*}
To proceed with the proof, we then analyze $\{ \Diamond_i \}_{i \in [4]}$ respectively in the following.

\begin{itemize}
	\item For the term $\Diamond_1$, we use Lemma~\ref{lem:decom} and obtain
	\begin{align*}
	\|\EB \Diamond_1\| 
 & \le \LGx \|\EB \xhat_t\yhat_{t}^\top\| +  \LGy \EB \|\yhat_{t}\|^2  +  \SBG (\EB \|\xhat_t\|^{1+\Gsmooth} \|\yhat_t\| 
	+ \EB \|\yhat_{t}\|^{2+\Gsmooth} ).
	\end{align*}
\item For the term $\Diamond_2$, the inequality~\eqref{lem:xhat:Eq1b1} implies that
\begin{align}
	\EB \Diamond_2 =
	\EB\|G(x_t, y_t)\|^2 \le 2 \LGx^2 \EB\|\xhat_t\|^2+ 2 \LGy^2 \EB\|\yhat_t\|^2.
\end{align}
\item For the term $\Diamond_3$, 
\eqref{assump:smooth:FH:ineqH} and \eqref{assump:H:holder} imply
$\| \nablaH (y_t) - \nablaH (y^\star) \| \le \min \{ \Hholder \| \yhat_t \|^{\Hsmooth}, 2L_H \}$,
it follows
\begin{align*}
\| \EB \Diamond_3 \| 
&\le \beta_t \| \EB \left(\nabla H(y_t) -\nabla H(y^{\star}) \right) G(x_t, y_t) \yhat_{t}^\top \| +  \beta_t^2
\| \EB \left(\nabla H(y_t) -\nabla H(y^{\star}) \right) \psi_t\psi_t^\top\|  \\
&\qquad + \beta_t^2
\| \EB \left(\nabla H(y_t) -\nabla H(y^{\star}) \right) G(x_t, y_t) G(x_t, y_t)^\top\| \\
&\le \beta_t \Hholder \EB \|G(x_t, y_t)\| \| \yhat_{t}\|^{1+\Hsmooth} + 2 L_H \beta_t^2 ( \EB\|G(x_t, y_t)\|^2 + 
{\Sigma_{22}} ) \\
&\overset{\eqref{lem:xhat:Eq1b1}}{\le} \beta_t \Hholder ( \LGx \EB \|\xhat_t\| \|\yhat_t\|^{1+\Hsmooth} + \LGy \EB\|\yhat_t\|^{2+\Hsmooth})  + 2 L_H \beta_t^2 \Sigma_{22} \\
&\qquad + 4 \beta_t^2L_H\left( \LGx^2 \EB\|\xhat_t\|^2 + \LGy^2 \EB\|\yhat_{t}\|^2    \right).
\end{align*}
\item For the term $\Diamond_4$, it follows
\begin{align*}
\EB \Diamond_4
&\le \beta_t^{1+\Hsmooth} \EB\left( \|G(x_t, y_t)\| + \|\psi_t\|\right)^{1+\Hsmooth} \|\yhat_{t}\|  + \beta_t^{2+\Hsmooth} \EB\left( \|G(x_t, y_t)\| + \|\psi_t\|\right)^{2+\Hsmooth} \\
&\le 2 \beta_t^{1+\Hsmooth} 
\EB \|G(x_t, y_t)\|^{1+\Hsmooth} \|\yhat_{t}\|  {+} 2\Gamma_{22}^\frac{1+\Hsmooth}{2} \beta_t^{1+\Hsmooth} \EB \|\yhat_{t}\|   {+} 4 \beta_t^{2+\Hsmooth} \EB\left( \|G(x_t, y_t)\|^{2+\Hsmooth} {+} \|\psi_t\|^{2+\Hsmooth} \right)\\
&\overset{\eqref{lem:xhat:Eq1b1}}{\le} 4\beta_t^{1+\Hsmooth} (\LGx^{1+\Hsmooth} \EB\|\xhat_{t}\|^{1+\Hsmooth} \|\yhat_{t}\| + \LGy^{1+\Hsmooth} \EB\|\yhat_{t}\|^{2+\Hsmooth} )+\Gamma_{22}^\frac{1+\Hsmooth}{2}  (\beta_t \EB \|\yhat_{t}\|^2+ \beta_t^{1+2\Hsmooth} )   \\
&\qquad + 4 \beta_t^{2+\Hsmooth} \left( 4 \LGx^{2+\Hsmooth}
\EB\|\xhat_{t}\|^{2+\Hsmooth} + 4 \LGy^{2+\Hsmooth} \EB\|\yhat_{t}\|^{2+\Hsmooth} + \Gamma_{22}^{\frac{2+\Hsmooth}{2}}
\right),
\end{align*}
where we have used 
$(a+b)^{\gamma} \le 2^{\gamma-1} (a^\gamma+b^{\gamma})$ for any non-negative $a, b$ and $\gamma \ge 1$.
\end{itemize}
Putting these pieces together and noting $\beta_t \le \iota_2$, we have that
\begin{align}
	\label{eq:xy-eq4}
		\|\EB \blacklozenge_3\| &\le \beta_t L_H \LGx \|\EB \xhat_{t}\yhat_{t}^\top\|   + 6 \beta_t^2 L_H \LGx^2 \EB\|\xhat_{t}\|^2 + \beta_t \cde_{xy,1} \EB\|\yhat_{t}\|^2 + \cde_{xy,2} \beta_t^2
        +\cde_{xy,3} \beta_t^{1+2\Hsmooth} + \Delta_{xy,t}^{(2)},
\end{align}
where $\cde_{xy,1}, \cde_{xy,2}, \cde_{xy,3}$ are constants defined in~\eqref{eq:constant-xy} and $\Delta_{xy,t}^{(2)}$ is a higher-order residual covering all the remaining terms in~\eqref{eq:xy-eq4}.
By Young's inequality, we can derive the following upper bound for $\Delta_{xy,t}^{(2)}$
\begin{align}
    \label{eq:delta_xy_2}
    \Delta_{xy,t}^{(2)}
    & \le 2 \beta_t L_H \SBG (\EB \| \xhat_t \|^{2+\Gsmooth} + \EB \| \yhat_t \|^{2+\Gsmooth})
    + 16 \beta_t^{2+\Hsmooth} S_H \LGx^{2+\Hsmooth} \EB \| \xhat_t \|^{2+\Hsmooth} \notag \\
    & \qquad + \beta_t (\Hholder  \LGx + 4 \beta_t^\Hsmooth S_H \LGx^{1+\Hsmooth} ) (\EB \| \xhat_t \|^{2+\Hsmooth} + \EB \| \yhat_t \|^{2+\Hsmooth} ) \notag \\
    & \qquad + \beta_t \left( \Hholder \LGy + 4 \beta_t^\Hsmooth S_H \LGy^{1+\Hsmooth} + 16 \beta_t^{1+\Hsmooth} S_H \LGy^{2+\Hsmooth} \right) \EB \| \yhat_t \|^{2+\Hsmooth} \notag \\
    & \le 2 \beta_t L_H \SBG (\EB \| \xhat_t \|^{2+\Gsmooth} + \EB \| \yhat_t \|^{2+\Gsmooth}) + \beta_t \cde_{xy,4}
    (\EB \| \xhat_t \|^{2+\Hsmooth} + \EB \| \yhat_t \|^{2+\Hsmooth} ),
\end{align}
where $\cde_{xy,4}$ is defined in \eqref{eq:constant-xy}.

Now, we are ready to establish this lemma.
Plugging~\eqref{eq:xy-eq1},~\eqref{eq:xy-eq2} and~\eqref{eq:xy-eq4} into~\eqref{eq:xy-eq0}, we have that
\begin{align*}
	\|\EB\xhat_{t+1}\yhat_{t+1}^\top\| 
	&\le \left(1-\frac{\mu_F\alpha_t}{2}\right)\|\EB \xhat_t\yhat_{t}^\top \|  + \beta_t \left( \LGx \EB \|\xhat_{t}\|^2 + \cde_{xy,1} \EB\|\yhat_{t}\|^2\right) + \Sigma_{12} \alpha_t \beta_t \\
    & \qquad + \cde_{xy,2} \beta_t^2 + \cde_{xy,3} \beta_t^{1+2\Hsmooth} + \Delta_{xy,t},
 \end{align*}
where the inequality uses the following facts
\begin{itemize}
	\item Since $\frac{\beta_t}{\alpha_t} \le \kappa \le \frac{\mu_F}{2L_H \LGx}$, $(1-\mu_F\alpha_t)(1-\mu_G\beta_t) + \beta_t L_H \LGx \le 1-\mu_F\alpha_t +  \beta_t L_H \LGx  \le 1 -\frac{\mu_F\alpha_t}{2}$. 
	\item Since $\frac{\beta_t}{\alpha_t} \le \kappa \le \frac{\mu_F}{12 L_H \LGx}$, $\sqrt{1-\mu_F \alpha_{t}} + 6 \beta_t L_H \LGx \le 1-\frac{\mu_F\alpha_t}{2} + 6 \beta_t L_H \LGx \le 1$.
	\item 
 Combining \eqref{eq:delta_xy_1} and \eqref{eq:delta_xy_2}, we have
 \begin{align*}
     \Delta_{xy,t}^{(1)} + \Delta_{xy,t}^{(2)}
     & \le 2 \alpha_t \SBF (\EB \| \xhat_t \|^{2+\Fsmooth} + \EB \| \yhat_t \|^{2+\Fsmooth}) + \beta_t \SBG (1+2L_H) ( \EB \| \xhat_t \|^{2+\Gsmooth} + \EB \| \yhat_t \|^{2+\Gsmooth} )  \notag \\
    & \quad
    {+} \beta_t \cde_{xy,4}
    (\EB \| \xhat_t \|^{2+\Hsmooth} {+} \EB \| \yhat_t \|^{2+\Hsmooth} ) {+} 2 \alpha_t \beta_t \SBF \SBG (\EB \| \xhat_t \|^{2 + \Fsmooth + \Gsmooth} {+} \EB \| \yhat_t \|^{2 + \Fsmooth + \Gsmooth}) \\
    & =: \Delta_{xy,t}.
 \end{align*}
\end{itemize}
We complete the proof.
\end{proof}

\subsection{Proof of Lemma~\ref{lem:xhat-quartic}}
\label{proof:de:xhat-quartic}

\begin{proof}[Proof of Lemma~\ref{lem:xhat-quartic}]
We first present the specific forms of the constants:
\begin{equation}
\label{eq:constantx-quartic}
\begin{aligned}
    \cde_{xx,1} = 24 L_H \LGy, \ 
    \cde_{xx,2} = 40 L_H^2 \LGy^2, \ 
    \cde_{xx,3} = 1280 L_H^4 \LGy^4,\ 
    \cde_{xx,4} 
    = \frac{64 S_H^2 \Gamma_{22}^{1+\Hsmooth}}{\mu_F}.
\end{aligned}
\end{equation}

From the proof of Lemma~\ref{lem:xhat}, we have
\begin{align*}
\begin{split}
    \|\xhat_{t+1}\|^2 & \overset{ \eqref{lem:xhat:Eq1} }{=} \underbrace{ \|\xhat_{t}  - \alpha_{t}F(x_{t},y_{t})\|^2 }_{\CM_1} + 
    \underbrace{ \left\|H(y_{t}) - H(y_{t+1}) - \alpha_{t}\xi_{t}\right\|^2 }_{\CM_2} \\
    &\quad \ + \underbrace{ 2 \langle\xhat_{t}-\alpha_t F(x_{t},y_{t}), H(y_{t}) - H(y_{t+1}) \rangle }_{\CM_3} + \underbrace{ 2\alpha_{t} \langle \xhat_{t} - \alpha_{t}F(x_{t},y_{t}), - \xi_{t}\rangle }_{\CM_4}. 
\end{split}
\end{align*}
Taking the square of both sides yields
\begin{align}
    \| \xhat_{t+1} \|^4
    & \le \CM_1^2 + 4 \CM_2^2 + 4 \CM_3^2 + 4 \CM_4^2 + 2 \CM_1 (\CM_2 + \CM_3 + \CM_4), \label{lem:xhat-quartic:Eq1}
\end{align}
where the last inequality is due to $(a + b + c)^2 \le 3 (a^2 + b^2 + c^2)$ for any $a, b \in \RB$.
We then analyze these terms respectively in the following.
\begin{itemize}
\item For $\CM_1^2$, we have
\begin{align}
    \CM_1^2
    & \overset{ \eqref{lem:xhat:Eq1a}}{\le} 
    \left( 1 - \frac{7 \mu_F \alpha_t}{2} + \frac{ 49 \mu_F^2 \alpha_t^2 }{16} \right) 
    \| \xhat_t \|^4.
    \label{lem:xhat-quartic:Eq1-11}
\end{align}

\item For $\CM_2^2$, taking the expectation w.r.t. $\FM_t$, we have
\begin{align}
    \EB \left[ \CM_2^2 \,|\, \FM_t \right]
    & \overset{ \eqref{assump:smooth:FH:ineqH} }{\le} 8 \left( L_H^4 \beta_t^4 \EB \left[ \| G(x_t, y_t) - \psi_t \|^4 \,|\, \FM_t \right] + \alpha_t^4 \Gamma_{11}^2 \right) \notag \\
    & \overset{ \eqref{lem:xhat-quartic:Eq-22b} }{\le} 8 \left( 5 L_H^4 \beta_t^4 \| G(x_t, y_t) \|^4 + 7 L_H^4 \beta_t^4 \Gamma_{22}^2 + \alpha_t^4 \Gamma_{11}^2 \right),
    \label{lem:xhat-quartic:Eq1-22a}
\end{align}
where the first inequality also uses $(a+b)^4 \le 8 (a^4 + b^4)$ for any $a,b \in \RB$ and the last inequality uses \eqref{lem:xhat-quartic:Eq-22b} below
\begin{align}
\label{lem:xhat-quartic:Eq-22b}
    \EB \left[ \| G(x_t, y_t) - \psi_t \|^4 \right]
    \le 5 \| G(x_t, y_t) \|^4 + 7 \Gamma_{22}^2.
\end{align}
To derive \eqref{lem:xhat-quartic:Eq-22b}, we first 
notice that
\begin{align*}
    \| G(x_t, y_t) - \psi_t \|^4
    & \le \| G(x_t, y_t) \|^4 + 6 \| G(x_t, y_t) \|^2 \| \psi_t \|^2 + 4 \| G(x_t, y_t) \| \| \psi_t \|^3 + \| \psi_t \|^4  \\
    & \quad \ - 4 \| G(x_t, y_t) \|^2 \inner{G(x_t, y_t)}{\psi_t} \\
    & \le 5 \| G(x_t, y_t) \|^4 + 7 \| \psi_t \|^4 - 4 \| G(x_t, y_t) \|^2 \inner{G(x_t, y_t)}{\psi_t},
\end{align*}
where the last inequality follows from Young's inequality.
Taking the conditional expectation gives \eqref{lem:xhat-quartic:Eq-22b}.

Then we plug \eqref{lem:xhat:Eq1b1} into \eqref{lem:xhat-quartic:Eq1-22a} and obtain
\begin{align}
    \EB \left[ \CM_2^2 \,|\, \FM_t \right]
    & \le 320 L_H^4 \LGx^4 \beta_t^4 \| \xhat_t \|^4 + 320 L_H^4 \LGy^4 \beta_t^4 \| \yhat_t \|^4 + 56 L_H^4 \beta_t^4 \Gamma_{22}^2 + 8 \alpha_t^4 \Gamma_{11}^2,
    \label{lem:xhat-quartic:Eq1-22}
\end{align}
where the inequality also uses $(a+b)^4 \le 8 (a^4 + b^4)$ for any $a,b \in \RB$.

\item For $\CM_3^2$, taking the expectation w.r.t. $\FM_t$, we have
\begin{align}
    \EB \left[ \CM_3^2 \,|\, \FM_t \right] 
    & \overset{ \eqref{assump:smooth:FH:ineqH} + \eqref{lem:xhat:Eq1a} }{\le} 4 L_H^2 \| \xhat_t \|^2 \EB \left[ \| y_t - y_{t+1} \|^2 \,|\, \FM_t \right] \notag \\
    & \overset{\eqref{lem:xhat:Eq1b1}}{\le} 4 L_H^2 \beta_t^2 \| \xhat_t \|^2 \left( 2 \LGx^2 \| \xhat_t \|^2 + 2 \LGy^2 \| \yhat_t \|^2 + \Gamma_{22} \right) \notag \\
    & = 8 L_H^2 \LGx^2 \beta_t^2 \| \xhat_t \|^4 + 8 L_H^2 \LGy^2  \beta_t^2 \| \xhat_t \|^2 \| \yhat_t \|^2 + 4 L_H^2 \Gamma_{22} \beta_t^2 \| \xhat_t \|^2.     \label{lem:xhat-quartic:Eq1-33}
\end{align}

\item For $\CM_4^2$, taking the expectation w.r.t. $\FM_t$, we have
\begin{align}
    \EB \left[ \CM_4^2 \,|\, \FM_t \right]
    \le 4 \alpha_t^2 \| \xhat_t - \alpha_t F(x_t, y_t) \|^2 \EB \left[ \| \xi_t \|^2 \,|\, \FM_t \right]
    \overset{\eqref{lem:xhat:Eq1a}}{\le} 4 \alpha_t^2 \Gamma_{11} \| \xhat_t \|^2.
    \label{lem:xhat-quartic:Eq1-44}
\end{align}

\item For $\CM_1 \CM_2$, taking the expectation w.r.t. $\FM_t$, we have
\begin{align}
    \EB \left[ \CM_1 \CM_2 \,|\, \FM_t \right]
    & = \CM_1 \EB \left[ \CM_2 \,|\, \FM_t \right] \notag \\
    & \overset{ \eqref{lem:xhat:Eq1a} + \eqref{lem:xhat:Eq1b} }{\le} \| \xhat_t \|^2 \left( 4L_{H}^2 \LGx^2 \beta_{t}^2\|\xhat_{t}\|^2 + 4 L_{H}^2 \LGy^2 \beta_{t}^2\|\yhat_{t}\|^2     +  2\beta_{t}^2L_H^2\Gamma_{22} +  2\alpha_{t}^2\Gamma_{11}\right) \notag \\
    & \le 4 L_H^2 \LGx^2 \beta_t^2 \| \xhat_t \|^4 + 4 L_H^2 \LGy^2 \beta_t^2 \| \xhat_t \|^2 \| \yhat_t \|^2 + 2 L_H^2 \Gamma_{22} \beta_t^2 \| \xhat_t \|^2 + 2 \Gamma_{11} \alpha_t^2 \| \xhat_t \|^2.
    \label{lem:xhat-quartic:Eq1-12}
\end{align}

\item For $\CM_1 \CM_3$, taking the expectation w.r.t. $\FM_t$, we have
\begin{align}
    \EB \left[ \CM_1 \CM_3 \,|\, \FM_t \right]
    & = \CM_1 \EB \left[ \CM_3 \,|\, \FM_t \right] \notag \\
    & \overset{ \eqref{lem:xhat:Eq1a} + \eqref{lem:xhat:Eq1c} }{\le} \| \xhat_t \|^2
    \left( { \mu_F \alpha_{t}}\|\xhat_{t}\|^2 +  14 L_H \LGy \beta_t  \|\xhat_{t}\| \|\yhat_{t}\| +  \frac{32 S_H^2 \Gamma_{22}^{1+\Hsmooth}}{\mu_F} \frac{\beta_t^{2+\Hsmooth} }{\alpha_t} \right) \notag \\
    & \le {\mu_F \alpha_t} \| \xhat_t \|^4 + 14 L_H \LGy \beta_t \| \xhat_t \|^3 \| \yhat_t \| + \frac{32 S_H^2 \Gamma_{22}^{1+\Hsmooth}}{\mu_F} \frac{\beta_t^{2+2\Hsmooth}}{\alpha_t} \| \xhat_t \|^2.
    \label{lem:xhat-quartic:Eq1-13}
\end{align}

\item For $\CM_1 \CM_4$, taking the expectation w.r.t. $\FM_t$, we have
\begin{align}
    \EB \left[ \CM_1 \CM_4 \,|\, \FM_t \right]
    & = \CM_1 \EB \left[ \CM_4 \,|\, \FM_t \right]  = 0.
    \label{lem:xhat-quartic:Eq1-14}
\end{align}
\end{itemize}
Moreover, since $\alpha_t \le \frac{1}{12 \mu_F}$ and $\frac{\beta_t^2}{\alpha_t} \le \frac{\mu_F}{200 L_H^2 \LGx^2}$, we have
\begin{align*}
     1 - \frac{3 \mu_F \alpha_t}{2} + \frac{49 \mu_F^2 \alpha_t^2}{16} + 40 L_H^2 \LGx^2 \beta_t^2 + 1280 L_H^4 \LGx^4 \beta_t^4 
     \le 1 - 
     {\mu_F \alpha_t}.
\end{align*}
Taking the expectation on both sides of \eqref{lem:xhat-quartic:Eq1} w.r.t. $\FM_t$ and
plugging \eqref{lem:xhat-quartic:Eq1-11}, \eqref{lem:xhat-quartic:Eq1-22}
to \eqref{lem:xhat-quartic:Eq1-14} into it,
together with the above inequality and the definition of $\{ c_{xx,i} \}_{i \in [4]}$ in \eqref{eq:constantx-quartic},
we obtain
\eqref{lem:xhat-quartic:Ineq}.
\end{proof}

\subsection{Proof of Lemma~\ref{lem:yhat-quartic}}
\label{proof:de:yhat-quartic}

\begin{proof}[Proof of Lemma~\ref{lem:yhat-quartic}]
From the proof of Lemma~\ref{lem:yhat}, we have
\begin{align*}
    \|\yhat_{t+1}\|^2 &
    \overset{ \eqref{lem:yhat:Eq1} }{=} \underbrace{ \|\yhat_{t} - \beta_{t}G(H(y_{t}),y_{t})\|^2 }_{\DM_1} + \underbrace{ \left\|\beta_{t} \left(G(H(y_{t}),y_{t}) - G(x_{t},y_{t})\right)  - \beta_{t}\psi_{t}\right\|^2 }_{\DM_2} \notag\\
    & \quad \ + \underbrace{ 2\beta_{t} \langle \yhat_{t} - \beta_{t} G(H(y_{t}),y_{t}), G(H(y_{t}),y_{t}) - G(x_{t},y_{t})   \rangle }_{\DM_3} 
    + \underbrace{ 2\beta_{t} \langle \yhat_{t} - \beta_{t}G(H(y_{t}),y_{t}),- \psi_{t}  \rangle }_{\DM_4}. 
\end{align*}
Taking the square of both sides yields
\begin{align}
    \| \yhat_{t+1} \|^4
    & \le \DM_1^2 + 4 \DM_2^2 + 4 \DM_3^2 + 4 \DM_4^2 + 2 \DM_1 (\DM_2 + \DM_3 + \DM_4), 
    \label{lem:yhat-quartic:Eq1}
\end{align}
where the last inequality is due to $(a + b + c)^2 \le 3 (a^2 + b^2 + c^2)$ for any $a, b \in \RB$.
We then analyze these terms respectively in the following.
\begin{itemize}
\item For $\DM_1^2$, we have
\begin{align}
\label{lem:yhat-quartic:Eq1-11}
    \DM_1^2 \overset{\eqref{lem:yhat:Eq1a}}{\le}
    (1 - 2 \mu_G \beta_t + \mu_G^2 \beta_t^2 ) \| \yhat_t \|^4
    \le \left( 1 - \frac{3 \mu_G \beta_t }{2} \right) \| \yhat_t \|^4,
\end{align}
where the last inequality is due to $\beta_t 
\le  \frac{1}{2 \mu_G}$.

\item For $\DM_2^2$, taking the expectation w.r.t. $\FM_t$, we have
\begin{align*}
    \EB \left[ \DM_2^2 \,|\, \FM_t \right]
    = \beta_t^4 \EB \left[ \| G( H(y_t), y_t ) - G(x_t, y_t) - \psi_t \|^4 \,|\, \FM_t \right].
\end{align*}
Similar to the derivation of \eqref{lem:xhat-quartic:Eq-22b}, we can obtain
\begin{align*}
    \EB \left[ \| G(H(y_t), y_t) - G(x_t, y_t) - \psi_t \|^4 \right]
    \le 5 \| G(x_t, y_t) - G(H(y_t), y_t) \|^4 + 7 \Gamma_{22}^2.
\end{align*}
It follows that
\begin{align}
\label{lem:yhat-quartic:Eq1-22}
    \EB \left[ \DM_2^2 \,|\, \FM_t \right]
    & \le 5 \beta_t^4 \| G(x_t, y_t) - G(H(y_t), y_t) \|^4 + 7 \beta_t^4 \Gamma_{22}^2 
    \overset{ \eqref{assump:G:smooth} }{\le} 5 \LGx^4 \beta_t^4  \| \xhat_t \|^4 + 7 \beta_t^4 \Gamma_{22}^2.
\end{align}

\item For $\DM_3^2$, we have
\begin{align}
\label{lem:yhat-quartic:Eq1-33}
    \DM_3^2
    & \le 4 \beta_t^2 \| \yhat_t - \beta_t G(H(y_t), y_t) \|^2 \| G(H(y_t), y_t) - G(x_t, y_t) \|^2 
    \overset{ \eqref{assump:G:smooth} + \eqref{lem:yhat:Eq1a} }{\le} 4 \LGx^2 \beta_t^2 \| \yhat_t \|^2 \| \xhat_t \|^2.
\end{align}

\item For $\DM_4^2$, taking the expectation w.r.t. $\FM_t$, we have
\begin{align}
\label{lem:yhat-quartic:Eq1-44}
    \EB \left[ \DM_4^2 \,|\, \FM_t \right]
    \le 4 \beta_t^2 \| \yhat_t - \beta_t G(H(y_t), y_t) \|^2 \EB \left[ \| \psi_t \|^2 \,|\, \FM_t \right]
    \overset{ \eqref{lem:yhat:Eq1a} }{\le} 4 \Gamma_{22} \beta_t^2 \| \yhat_t \|^2.
\end{align}

\item For $\DM_1 \DM_2$, taking the expectation w.r.t. $\FM_t$, we have
\begin{align}
\label{lem:yhat-quartic:Eq1-12}
    \EB \left[ \DM_1 \DM_2 \,|\, \FM_t \right]
    = \DM_1 \EB \left[ \DM_2 \,|\, \FM_t \right]
    \overset{ \eqref{lem:yhat:Eq1a} + \eqref{lem:yhat:Eq1b} }{\le} \LGx^2\beta_{t}^2\|\xhat_{t}\|^2 \| \yhat_t \|^2 + \Gamma_{22} \beta_{t}^2 \| \yhat_t \|^2. 
\end{align}

\item For $\DM_1 \DM_3$, taking the expectation w.r.t. $\FM_t$, we have
\begin{align}
\label{lem:yhat-quartic:Eq1-13}
    \EB \left[ \DM_1 \DM_3 \,|\, \FM_t \right]
    = \DM_1 \EB \left[ \DM_3 \,|\, \FM_t \right]
    \overset{ \eqref{lem:yhat:Eq1a} + \eqref{lem:yhat:Eq1c} }{\le} 
    2 \LGx \beta_t \|\xhat_{t}\| \|\yhat_{t}\|^3.
\end{align}

\item For $\DM_1 \DM_4$, taking the expectation w.r.t. $\FM_t$, we have
\begin{align}
\label{lem:yhat-quartic:Eq1-14}
    \EB \left[ \DM_1 \DM_4 \,|\, \FM_t \right]
    = \DM_1 \EB \left[ \DM_4 \,|\, \FM_t \right] = 0.
\end{align}
\end{itemize}
Taking the expectation on both sides of \eqref{lem:yhat-quartic:Eq1} w.r.t. $\FM_t$ and
plugging \eqref{lem:yhat-quartic:Eq1-11}
to \eqref{lem:yhat-quartic:Eq1-14} into it yields \eqref{lem:yhat-quartic:Ineq}.
\end{proof}

\subsection{Proof of Lemma~\ref{lem:x+y_quartic}}
\label{proof:de:x+y-quartic}

\begin{proof}[Proof of Lemma~\ref{lem:x+y_quartic}]
To characterize the $L_4$-convergence rate, we 
define the Lyapunov function 
$
    V_{t} = \varrho_3  \frac{\beta_t}{\alpha_t} \| \xhat_t \|^4 + \| \yhat_t \|^4
    \text{ with }
    \varrho_3 = \frac{54 \LGx^4}{\mu_F \mu_G^3}.
    $

\noindent
\textbf{Derive the one-step descent.}
We first employ Lemmas~\ref{lem:xhat-quartic} and \ref{lem:yhat-quartic} to derive the one-step descent of $V_t$.
Since $\frac{\beta_t}{\alpha_t} \le \kappa$, we have 
\begin{align}
\label{lem:x+y_quartic:Eq0}
\begin{split}
20 \Gamma_{11} \alpha_t^2 + 20 L_H^2 \Gamma_{22} \beta_t^2
& \le 
\cde_{xx,5} \alpha_t^2, \quad
32 \alpha_t^4 \Gamma_{11}^2 + 224 L_H^4 \beta_t^4 \Gamma_{22}^2
\le 
\cde_{xx,6} \alpha_t^4.
\end{split}
\end{align}
with
\begin{align*}
    \cde_{xx,5}
    = 20 \Gamma_{11} + 20 L_H \Gamma_{22} \kappa^2, \ 
    \cde_{xx,6}
    = 32 \Gamma_{11}^2 + 224 L_H^4 \Gamma_{22}^2 \kappa^4.
\end{align*}
As a result of $\frac{\beta_{t+1}}{\alpha_{t+1}} \le \frac{\beta_t}{\alpha_t}$,
we have
\begin{align}
    \begin{split}
    \EB \left[ V_{t+1} \,|\, \FM_t \right] & \overset{ \eqref{lem:xhat-quartic:Ineq} + \eqref{lem:yhat-quartic:Ineq} }{ \le} V_t 
    - \underbrace{ \varrho_3 \mu_F \beta_t \| \xhat_t \|^4 \vphantom{\frac{\beta_t^3}{\alpha_t}} }_{\EM_0}
    - \frac{3 \mu_G \beta_t}{2} \| \yhat_t \|^4
    + \underbrace{ 4 \LGx \beta_t \| \xhat_t \| \| \yhat_t \|^3 \vphantom{\frac{\beta_t^2}{\alpha_t}}}_{\EM_1}
    + \underbrace{ \varrho_3 \cde_{xx,1} \frac{\beta_t^2}{\alpha_t}  \| \xhat_t \|^3 \| \yhat_t \| }_{\EM_2}  \\
    & \qquad + \underbrace{ 18 \LGx^2 \beta_t^2 \| \xhat_t \|^2 \| \yhat_t \|^2 \vphantom{\frac{\beta_t^3}{\alpha_t}} }_{\EM_3}
    + \underbrace{ \varrho_3 \cde_{xx,2} \frac{\beta_t^3}{\alpha_t} \| \xhat_t \|^2 \| \yhat_t \|^2 }_{\EM_4}
    + \underbrace{ 20 \LGx^4 \beta_t^4 \| \xhat_t \|^4 \vphantom{\frac{\beta_t^3}{\alpha_t}} }_{\EM_5}
    + \underbrace{ \varrho_{3} \cde_{xx,3} \frac{\beta_t^5}{\alpha_t} \| \yhat_t \|^4 }_{\EM_6}  \\
    & \qquad + \varrho_3 \cde_{xx,4} \frac{\beta_t^{3+2\Hsmooth}}{\alpha_t^2} \| \xhat_t \|^2
    + \varrho_3 \cde_{xx,5} \alpha_t \beta_t 
    \| \xhat_t \|^2 + 18 \Gamma_{22} \beta_t^2 \| \yhat_t \|^2
    + \varrho_3 \cde_{xx,6} \alpha_t^3 \beta_t 
    + 28 \Gamma_{22}^2 \beta_t^4.
    \end{split}
    \label{lem:x+y_quartic:Eq1}
\end{align}
We then analyze $\EM_0$ to $\EM_6$ respectively in the following.
\begin{itemize}
\item For $\EM_0$, we have $\frac{\beta_t}{\alpha_t} 
\le \frac{\mu_F}{5 \mu_G}$.
It follows that
\begin{align}
\label{lem:x+y_quartic:Eq1-0}
    \EM_0 
    \ge \frac{19}{20} 
    \varrho_3 \mu_F \beta_t \| \xhat_t \|^4 + \frac{ 
    \varrho_3 \mu_G \beta_t^2}{4 \alpha_t} \| \xhat_t \|^4.
\end{align}

\item For $\EM_1$, by Young's inequality, we have
\begin{align}
\label{lem:x+y_quartic:Eq1-1}
    \EM_1 
    & \le \frac{\LGx^4 \beta_t  }{ \lambda_1^4 \mu_G^3 } \| \xhat_t \|^4
    + 3 \lambda_1^{4/3} \mu_G \beta_t \| \yhat_t \|^4
    \le \frac{ \varrho_3 \mu_F \beta_t }{2} \| \xhat_t \|^4 + \mu_G \beta_t \| \yhat_t \|^4,
\end{align}
where the last inequality is by setting $\lambda_1^{4/3} = 1/3$.

\item For $\EM_2$, by Young's inequality, we have
\begin{align*}
    \EM_2 
    & = \varrho_3 \left( \beta_t^{3/4} \lambda_2 \| \xhat_t \|^3 \right) \left( c_{xx,1} \lambda_2^{-1} \beta_t^{5/4} \alpha_t^{-1} \| \yhat_t \| \right) \notag 
    \le \frac{3 \lambda_2^{4/3} \varrho_3 \beta_t }{4} \| \xhat_t \|^4 + \frac{ \varrho_3 (\cde_{xx,1})^4 \beta_t^5 }{4 \lambda_2^4 \alpha_t^4} \| \yhat_t \|^4.
\end{align*}
Since $\frac{\beta_t}{\alpha_t} \le \frac{\mu_F \mu_G}{200 L_H \LGx \LGy  
}$, 
we have
$\frac{\beta_t^4}{\alpha_t^4}
= \frac{\mu_F^4 \mu_G^4}{ 200^4 L_H^4 \LGx^4 \LGy^4 }
\le \frac{ 54 \cdot 24^4 \mu_F^3 \mu_G }{200^4 \varrho_3 (\cde_{xx,1})^4 }$.
Then setting $\lambda_2^{4/3} = \mu_F / 4$ yields
\begin{align}
\label{lem:x+y_quartic:Eq1-2}
    \EM_2
    \le \frac{ 3 \varrho_3 \mu_F \beta_t }{16} \| \xhat_t \|^4 + \frac{4^2 \cdot 54 \cdot 24^4 \mu_G \beta_t }{200^4} \| \yhat_t \|^4
    \le \frac{ 3 \varrho_3 \mu_F \beta_t }{16} \| \xhat_t \|^4 + 0.18 \mu_G \beta_t
    \| \yhat_t \|^4.
\end{align}

\item For $\EM_3$, by Cauchy-Schwarz inequality, we have
$
    \EM_3 
    \le \frac{\varrho_3 \mu_F \beta_t}{8} \| \xhat_t \|^4 + \frac{8 \cdot 81 \LGx^4 \beta_t^3 }{\mu_F \varrho_3} \| \yhat_t \|^4. 
    $
Since $\beta_t \le 
\frac{1}{14 \mu_G}$, 
we have 
$\beta_t^2 \le \frac{1}{14^2 \mu_G^2} = \frac{\varrho_3 \mu_F \mu_G}{14^2 \cdot 54 \LGx^4 }$.
It follows that
\begin{align}
\label{lem:x+y_quartic:Eq1-3}
    \EM_3 
    \le \frac{\varrho_3 \mu_F \beta_t}{8} \| \xhat_t \|^4 + \frac{8 \cdot 81\mu_G \beta_t}{14^2 \cdot 54}  \| \yhat_t \|^4
    \le \frac{\varrho_3 \mu_F \beta_t}{8} \| \xhat_t \|^4 + 0.062 \mu_G \beta_t
    \| \yhat_t \|^4.
\end{align}

\item For $\EM_4$, by Cauchy-Schwarz inequality, we have
$
    \EM_4 
    \le \frac{ \varrho_3 \mu_F \beta_t }{8} \| \xhat_t \|^4
    + \frac{2 \varrho_3 (\cde_{xx,2})^2 \beta_t^5 }{\mu_F \alpha_t^2} \| \yhat_t \|^4.
    $
Since $\alpha_t \le \frac{1}{12 \mu_F}$ 
and $\frac{\beta_t}{\alpha_t} \le \frac{\mu_F \mu_G}{200 L_H \LGx \LGy}$, we have
$ \frac{\beta_t^4}{\alpha_t^2}
\le \frac{ \mu_F^2 \mu_G^4 }{ 12^2 \cdot 200^4 L_H^4 \LGx^4 \LGy^4 }
\le \frac{54 \cdot 40^2 \mu_F \mu_G }{ 12^2 \cdot 200^4 \varrho_3 (\cde_{xx,2})^2 } $.
It follows that
\begin{align}
\label{lem:x+y_quartic:Eq1-4}
    \EM_4 
    \le \frac{ \varrho_3 \mu_F \beta_t }{8} \| \xhat_t \|^4
    + \frac{108 \cdot 40^2 \mu_G \beta_t }{12^2 \cdot 200^4} \| \yhat_t \|^4
    \le \frac{ \varrho_3 \mu_F \beta_t }{8} \| \xhat_t \|^4
    + 0.001 \mu_G \beta_t
    \| \yhat_t \|^4.
\end{align}

\item For $\EM_5$, since $\beta_t \le 
\frac{1}{14 \mu_G }$, we have
$\beta_t^3 \le \frac{1}{14^3 \mu_G^3} \le \frac{ \varrho_3 \mu_F }{ 14^3 \cdot 54 \LGx^4 }$. It follows that
\begin{align}
\label{lem:x+y_quartic:Eq1-5}
    \EM_5 \le 0.001 { \varrho_3 \mu_F \beta_t} \| \xhat_t \|^4.
\end{align}

\item For $\EM_6$, since 
$\alpha_t \le \frac{1}{12 \mu_F}$
and $\frac{\beta_t}{\alpha_t} \le \frac{\mu_F \mu_G}{ 200 L_H \LGx \LGy }$, we have
$\frac{\beta_t^4}{\alpha_t}
\le \frac{ \mu_F \mu_G^4 }{ 12^3 \cdot 200^4 L_H^4 \LGx^4 \LGy^4}
\le \frac{54 \cdot 1280 \mu_G}{ 12^3 \cdot 200^4 \varrho_3 \cde_{xx,3} }
\le 0.001
\frac{ \mu_G}{ \varrho_3 \cde_{xx,3}} $.
It follows that
\begin{align}
\label{lem:x+y_quartic:Eq1-6}
    \EM_6 \le 0.001
    \mu_G \beta_t
    \| \yhat_t \|^4.
\end{align}

\end{itemize}

Plugging \eqref{lem:x+y_quartic:Eq1-0} to \eqref{lem:x+y_quartic:Eq1-6} into \eqref{lem:x+y_quartic:Eq1} 
and taking the expectation yields
\begin{align}
    \EB 
    V_{t+1} 
    & \le \left( 1 - \frac{\mu_G \beta_t}{4} \right) \EB V_{t} + \left( \varrho_3 \cde_{xx,4} \frac{\beta_t^{3+2\Hsmooth}}{\alpha_t^2} + \varrho_3 \cde_{xx,5} \alpha_t \beta_t + 18 \Gamma_{22} \beta_t^2 \right) \left( \EB \| \xhat_t \|^2 + \EB \| \yhat_t \|^2 \right) \notag \\
    & \quad \ + \varrho_3 \cde_{xx,6}  \alpha_t^3 \beta_t  + 28 \Gamma_{22}^2 \beta_t^4. \label{lem:x+y_quartic:Eq2-origin}
\end{align}
Since we assume $\prod_{\tau=0}^t\left(1-\frac{\mu_{G}\beta_{\tau}}{4} \right) = \OM(\alpha_t)$, then 
Theorem~\ref{thm:first} with the definitions of $c_{x,7}$ in \eqref{eq:constantx2} and $c_{y,2}$ in \eqref{eq:constanty} implies
\begin{equation}
\label{lem:x+y_quartic:Eq2.05}
\begin{aligned}
    \EB \| \xhat_t \|^2 + \EB \| \yhat_t \|^2 
    \le c_+  \alpha_t + c_- \frac{\beta_t^{2+2\Hsmooth}}{\alpha_t^2} \text{ with } c_- \propto S_H^2 \Gamma_{22}^{1+\Hsmooth}.
\end{aligned}
\end{equation}
Also note that $\frac{\beta_t}{\alpha_t} \le \kappa$.
Then we have
\begin{align}
    & \quad \ \left( \varrho_3 \cde_{xx,4} \frac{\beta_t^{3+2\Hsmooth}}{\alpha_t^2} + \varrho_3 \cde_{xx,5} \alpha_t \beta_t + 18 \Gamma_{22} \beta_t^2 \right) \left( \EB \| \xhat_t \|^2 + \EB \| \yhat_t \|^2 \right) \notag \\
    & \le \beta_t \left[ \varrho_3 \cde_{xx,4} \frac{\beta_t^{2+2\Hsmooth}}{\alpha_t^2} + (\varrho_3 \cde_{xx,5} + 18 \Gamma_{22} \kappa) \alpha_t \right]  \left( c_+ \alpha_t + c_- \frac{\beta_t^{2+2\Hsmooth}}{\alpha_t^2} \right). 
    \label{lem:x+y_quartic:Eq2.1}
\end{align}
Plugging \eqref{lem:x+y_quartic:Eq2.1} into \eqref{lem:x+y_quartic:Eq2-origin} and using $\alpha_t \le \iota_1$, $\beta_t \le \iota_2$, $\frac{\beta_t}{\alpha_t} \le \kappa$ and $(c_1 a + c_2 b ) (c_3 a + c_4 b) \le c_1 c_3 a^2 + c_2 c_4 b^2 + (c_2 c_3 + c_1 c_4) (a^2 + b^2)/2$ for $a, b, c_1, c_2, c_3, c_4 \ge 0$ , we obtain
\begin{align}
    \EB 
    V_{t+1} 
    & \le \left( 1 - \frac{\mu_G \beta_t}{4} \right) \EB V_{t} + 
    \cde_{yy,1}\, \alpha_t^2 \beta_t + \cde_{yy,2} \frac{\beta_t^{5+4\Hsmooth}}{\alpha_t^4},
    \label{lem:x+y_quartic:Eq2}
\end{align}
where the constants $\cde_{yy,1}$ and $\cde_{yy,2}$ are defined as
\begin{equation}
\label{eq:constanty-quartic-12}
\begin{aligned}
    \cde_{yy,1} & = (\varrho_3 \cde_{xx,5} + 18 \Gamma_{22} \kappa) c_+ + \frac{ \varrho_3 \cde_{xx,4} c_+  + (\varrho_3 \cde_{xx,5} + 18 \Gamma_{22} \kappa) c_- }{2} 
    + \varrho_3 \cde_{xx,6} \iota_1 + 28 \Gamma_{22}^2 \kappa^2 \iota_2, \\
    \cde_{yy,2} 
    & = \varrho_3 \cde_{xx,4} c_- + \frac{ \varrho_3 \cde_{xx,4} c_+  + (\varrho_3 \cde_{xx,5} + 18 \Gamma_{22} \kappa) c_- }{2} 
    =  S_H^2 \Gamma_{22}^{1+\Hsmooth} \cdot h_1(S_H, \Gamma_{22})
\end{aligned}
\end{equation}
for some function $h_1(\cdot, \cdot)$.
Here the last inequality follows from the definitions of $\cde_{xx,4}$ in \eqref{eq:constantx-quartic} and $c_-$ in \eqref{lem:x+y_quartic:Eq2.05}.
\vspace{0.2cm}

\noindent
\textbf{Establish the convergence rates.} With the one-step descent inequality \eqref{lem:x+y_quartic:Eq2}, we could establish the convergence rates of $\EB \| \xhat_t \|^4$ and $\EB \| \yhat_t \|^4$.
For simplicity, we introduce (as we did in the proof of Theorem~\ref{thm:first})
$
{\alpha}_{j, t} = \prod_{\tau=j}^t\left(1-\frac{\mu_{F}\alpha_{\tau}}{2} \right)~\text{and}~
{\beta}_{j, t} = \prod_{\tau=j}^t\left(1-\frac{\mu_{G}\beta_{\tau}}{4} \right).
$
Iterating \eqref{lem:x+y_quartic:Eq2} and applying Lemma~\ref{lem:step-size-ineq2-1}~(i) yields
\begin{align*}
    \EB V_{t+1}
    & \le {\beta}_{0,t}\, \EB V_0 + \frac{8 \cde_{yy,1}}{\mu_G} \alpha_t^2 + \frac{10 \cde_{yy,2}}{\mu_G} \frac{\beta_t^{4+4\Hsmooth}}{\alpha_t^4}.
\end{align*}
The proof of Lemma~\ref{lem:step-size-ineq2-1}
is similar to that of Lemma~\ref{lem:step-size-ineq} and is omitted.

\begin{lem}[Step sizes inequalities]
\label{lem:step-size-ineq2-1}
Let ${\alpha}_{j, t} = \prod_{\tau=j}^t\left(1-\frac{\mu_{F}\alpha_{\tau}}{2} \right)$ and
${\beta}_{j, t} = \prod_{\tau=j}^t\left(1-\frac{\mu_{G}\beta_{\tau}}{4} \right)$.
Under Assumption~\ref{assump:stepsize-new},
it holds that
\begin{enumerate}[(i)]
    \item $\sum_{j=0}^t {\beta}_{j+1, t}\, \alpha_j^2 \beta_j \le \frac{8 \alpha_t^2}{\mu_G}$ and $\sum_{j=0}^t \beta_{j+1, t} \frac{\beta_j^{5+4\Hsmooth} }{\alpha_j^4} \le \frac{10 \beta_t^{4+4\Hsmooth}}{\mu_G \alpha_t^4}$.
    
    \item $\sum_{j=0}^t {\alpha}_{j+1, t}\, \alpha_j^3 \le \frac{4 \alpha_t^2}{\mu_F}$,
    $ \sum_{j=0}^t {\alpha}_{j+1, t}\, \alpha_j \beta_{0, j-1} \le \frac{8 \beta_{0,t}}{\mu_F}$ and
    $\sum_{j=0}^t \alpha_{j+1,t} \frac{\beta_j^{4+4\Hsmooth}}{\alpha_j^3} \le \frac{ 3\beta_t^{4+4\Hsmooth} }{\mu_F \alpha_t^4}$
\end{enumerate}

\end{lem}
Then we obtain the convergence rate of $\EB \| \yhat_t \|^4$ as shown in the following
\begin{align}
\label{lem:x+y_quartic:Eq3}
    \EB \| \yhat_t \|^4
    \le \EB V_{t}
    \le \beta_{0,t-1}\, \EB V_0 + \frac{8 \cde_{yy,1} }{\mu_G} \alpha_{t-1}^2 + \frac{10 \cde_{yy,2}}{\mu_G} \frac{\beta_t^{4+4\Hsmooth}}{\alpha_t^4}.
\end{align}
Since $\kappa \le \frac{\mu_F \mu_G}{200 L_H \LGx \LGy }$,
we have
\begin{align*}
    \EB V_0
    = \varrho_3 \frac{\beta_0}{\alpha_0} \EB \| \xhat_0 \|^4 + \EB \| \yhat_0 \|^4
    \le \frac{ 54 \LGx^4 \kappa }{\mu_F \mu_G^3} \EB \| \xhat_0 \|^4 + \EB \| \yhat_0 \|^4
    \le \cde_{yy,2} 
    \EB \| \xhat_0 \|^4 + \EB \| \yhat_0 \|^4,
\end{align*}
where $\cde_{yy,2} = \frac{\LGx^3}{3 \mu_G^2 L_H \LGy} $.
As a result,
\begin{align*}
    \EB \| \yhat_t \|^4
    \le \beta_{0, t-1} \left( \cde_{yy,2} \EB \| \xhat_0 \|^4 + \EB \| \yhat_0 \|^4 \right) + \frac{8 \cde_{yy,1}}{\mu_G} \alpha_{t-1}^2 + \frac{10 \cde_{yy,2}}{\mu_G} \frac{\beta_t^{4+4\Hsmooth}}{\alpha_t^4}.
\end{align*}
Plugging \eqref{lem:x+y_quartic:Eq1-2},
\eqref{lem:x+y_quartic:Eq1-4}, 
\eqref{lem:x+y_quartic:Eq1-6} and \eqref{lem:x+y_quartic:Eq0} into \eqref{lem:xhat-quartic:Ineq}
and taking the expectation, we obtain
\begin{align}
    \EB \| \xhat_{t+1} \|^4
    & \le \left(1 - \frac{\mu_F \alpha_t}{2} \right) \EB \| \xhat_t \|^4 
    + \frac{\mu_G \alpha_t}{5 \varrho_3} \EB \| \yhat_t \|^4 
    + \left( \cde_{xx,4} \frac{\beta_t^{2+2\Hsmooth}}{\alpha_t} +  \cde_{xx,5} \alpha_t^2 \right) \EB \| \xhat_t \|^2  + \cde_{xx,6} \alpha_t^4  \notag \\
    & \overset{ \eqref{lem:x+y_quartic:Eq2.05} + \eqref{lem:x+y_quartic:Eq3} }{\le} \left(1 - \frac{\mu_F \alpha_t}{2} \right) \EB \| \xhat_t \|^4
    + \frac{\mu_G \EB V_0}{ 5 \varrho_3 } \beta_{0,t-1} \alpha_t + \cde_{xx,7}\, \alpha_t^3 + \cde_{xx,8} \frac{\beta_t^{4 + 4\Hsmooth}}{\alpha_t^3},
    \label{lem:x+y_quartic:Eq4}
\end{align}
where the last inequality also uses $\alpha_t \le \iota_1$, $\beta_t \le \iota_2$, 
$\frac{\alpha_{t-1}}{\alpha_t} \le 1 + \frac{\mu_F \alpha_t}{8} \le 1 + \frac{\mu_F \iota_1}{8}$,
$\frac{\beta_{t-1}}{\beta_t} \le 1 + \frac{\mu_G \beta_t}{64} \le 1 + \frac{\mu_G \iota_2}{64}$,
$ \frac{ \beta_{t-1}^{a} \alpha_{t-1}^{-b} }{ 
\beta_t^{a} \alpha_t^{-b} } \le \left( \frac{\beta_{t-1}}{\beta_t} \right)^a $ for any $a,b >0$,
and the constants are defined as
\begin{equation}\label{eq:constantx-quartic-78}
\begin{aligned}
    \cde_{xx,7}
    & = \frac{8 \cde_{yy,1}}{5  \varrho_3 } \left( 1 + \frac{\mu_F \iota_1}{8} \right)^2 + \left( \cde_{xx,5} c_+ + \frac{ \cde_{xx,5} c_- + \cde_{xx,4} c_+ }{2} \right)  
    \left[ 1 + \frac{\mu_F \iota_1}{8} + \left( 1 + \frac{\mu_G \iota_2}{64} \right)^4  \right]
    + \cde_{xx,6}  \iota_1, \\
    \cde_{xx,8} 
    & = \frac{10 \cde_{yy,2}}{\mu_G} \left( 1 + \frac{\mu_G \iota_2}{64} \right)^8 + \left( \cde_{xx,4} c_- + \frac{ \cde_{xx,5} c_- + \cde_{xx,4} c_+ }{2} \right) 
    \left[ 1 + \frac{\mu_F \iota_1}{8} + \left( 1 + \frac{\mu_G \iota_2}{64} \right)^4  \right] \\
    & \lesssim S_H^2 \Gamma_{22}^{1+\Hsmooth} \cdot h_2(S_H, \Gamma_{22})
\end{aligned}
\end{equation}
for some function $h_2(\cdot,\cdot)$.
Here the last inequality follows from the definitions of $\cde_{xx,4}$ in \eqref{eq:constantx-quartic}, $c_-$ in \eqref{lem:x+y_quartic:Eq2.05} and $\cde_{yy,2}$ in \eqref{eq:constanty-quartic-12}.
Iterating \eqref{lem:x+y_quartic:Eq4} yields
\begin{align*}
    \EB \| \xhat_{t+1} \|^4
    & \le \alpha_{0, t} \EB \| \xhat_0 \|^4 + \frac{8 \mu_G}{5 \varrho_3 \mu_F} \beta_{0,t} \EB V_0 + \frac{4 \cde_{xx,7} }{\mu_F} \alpha_t^2 + \frac{3 \cde_{xx,8}}{\mu_F} \frac{\beta_t^{4+4\Hsmooth}}{\alpha_t^4},
\end{align*}
where the inequality also applies Lemma~\ref{lem:step-size-ineq2-1}~(ii).
Since $\frac{\mu_G \beta_j}{\mu_F \alpha_j} \le \frac{\mu_G}{\mu_F} \kappa \le \frac{1}{2}$ for any $j$, it holds that $\alpha_{0,t}  \le \beta_{0,t}$.
Recall the definition of $V_0$ and $\varrho_3 = \frac{54 \LGx^4}{\mu_F \mu_G^3}$. Then we can obtain 
\begin{align*}
    \EB \| \xhat_{t+1} \|^4
    & \le 2 \beta_{0,t} \EB \| \xhat_0 \|^4 + \frac{\mu_G^4}{27 \LGx^4} \beta_{0,t} \EB \| \yhat_0 \|^4 + \frac{4 \cde_{xx,7}}{\mu_F} \alpha_t^2 + \frac{3 \cde_{xx,8}}{\mu_F} \frac{\beta_t^{4+4\Hsmooth}}{\alpha_t^4}.
\end{align*}
This completes the proof.
\end{proof}

\subsection{Proof of 
Theorem~\ref{thm:decouple}}
\label{proof:de:decouple}

\begin{proof}[Proof of 
Theorem~\ref{thm:decouple}
]
Since the conditions of Theorem~\ref{thm:decouple} are stronger than those of Theorem~\ref{thm:first}.
The first inequality follows from Theorem~\ref{thm:first} with 
\begin{equation}\label{eq:constanx-ccde-012}
\begin{aligned}
    \ccde_{x,0} = 3 \EB\|\xhat_{0}\|^2  +  \frac{7 L_H \LGy \EB\|\yhat_{0}\|^2 }{\LGx}, \ 
    \ccde_{x,1} = \frac{8 \Gamma_{11}}{\mu_F} + c_{x,5} \kappa + c_{x,6} \kappa^2, \ 
    \ccde_{x,2} = c_{x,7} \propto S_H^2 \Gamma_{22}^{1+\Hsmooth},
\end{aligned}
\end{equation}
where the constants $\{ c_{x,i} \}_{i\in [7] \setminus [4] }$ are defined in \eqref{eq:constantx2}. 

The proof of the other two inequalities is divided into three parts.
We first use Lemma~\ref{lem:x+y_quartic} to analyze the high-order terms defined in Lemmas~\ref{lem:xhat-new} to \ref{lem:xyhat},
then derive the convergence rate of $\| \EB \xhat_t \yhat_t^\top \| 
$ and finally use this to establish the rate of $ \EB \| \yhat_t \|^2 
$.
\vspace{0.2cm}

\noindent
\textbf{Analyze the high-order terms.}
Since we assume $\prod_{\tau=0}^t\left(1-\frac{\mu_{G}\beta_{\tau}}{4} \right) = \OM(\alpha_t^2)$, then
Lemma~\ref{lem:x+y_quartic} with the definitions of $\ccde_{xx,8}$ in \eqref{eq:constantx-quartic-78} and $\ccde_{yy,2}$ in \eqref{eq:constanty-quartic-12} implies
\begin{align}
\label{lem:decouple:Eq0}
    \EB \| \xhat_t \|^4 + \EB \| \yhat_t \|^4
    \le c_{+,2}\, \alpha_t^2 + c_{-,2} \frac{\beta_t^{4+4\Hsmooth}}{\alpha_t^4} \text{ with } c_{-,2} = S_H^2 \Gamma_{22}^{1+\Hsmooth} \cdot h_3(S_H, \Gamma_{22})
\end{align}
for some function $h_3(\cdot, \cdot)$.
Note that the terms in $\Delta_{x,t}, \Delta_{y,t}$ and $\Delta_{xy,t}$ all have their exponents lying between $2$ and $4$. Then we could 
apply Jensen's inequality to control them by the fourth-order terms or
apply Young's inequality to bound them by second-order and fourth-order terms.

We first focus on $\Delta_{x,t}$, which is defined in \eqref{eq:delta_x}.
By Jensen's inequality and Assumption~\ref{assump:stepsize-new}, for $z = \| \xhat_t \|$ or $\| \yhat_t \|$ and $\gamma \in [0,2]$, we have
\begin{align}
\label{lem:decouple:Eq0.1}
    \EB z^{2+\gamma} \le (\EB z^4)^\frac{2+\gamma}{4}
    \le c_{+,2}^\frac{2+\gamma}{4} \alpha_t^{1+\frac{\gamma}{2}} 
    + c_{-,2}^\frac{2+\gamma}{4} \left( \frac{\beta_t^{2+2\Hsmooth}}{\alpha_t^2} \right)^{1+\frac{\gamma}{2}}
    \le c_{+,2}^\frac{2+\gamma}{4} \iota_1^\frac{\gamma}{2} \alpha_t 
    + c_{-,2}^\frac{2+\gamma}{4} \iota_2^{\Hsmooth\gamma} \kappa^\gamma \frac{\beta_t^{2+2\Hsmooth}}{\alpha_t^2}.
\end{align}
Substituting \eqref{lem:decouple:Eq0.1} into \eqref{eq:delta_x}, we obtain
\begin{align}
\label{eq:delta_x_upper}
    \Delta_{x,t}
    \le \cde_{x,9} \alpha_t \beta_t + \cde_{x,10} \frac{\beta_t^{3+2\Hsmooth}}{\alpha_t^2},
\end{align}
where the constants are defined as
\begin{equation}
\label{eq:constantx-de-910}
\begin{aligned}
    \cde_{x,9}
    & = 2 \cde_{x,5} c_{+,2}^\frac{2+\Hsmooth}{4} \iota_1^\frac{\Hsmooth}{2}
    + 2 \cde_{x,6} c_{+,2}^\frac{2+\Fsmooth}{4} \iota_1^{ 1 + \frac{\Fsmooth}{2} }
    + 2 \cde_{x,7} c_{+,2}^\frac{1+\Gsmooth}{2} \iota_1^\Gsmooth
    + 2 \cde_{x,8} c_{+,2}^\frac{1+\Hsmooth}{2} \iota_1^\Hsmooth \iota_2^\Hsmooth, \\
    \cde_{x,10}
    & = 2 \cde_{x,5} c_{-,2}^\frac{2+\Hsmooth}{4} \iota_2^{\Hsmoothsq} \kappa^\Hsmooth  
    + 2 \cde_{x,6} c_{-,2}^\frac{2+\Fsmooth}{4} \iota_1 \iota_2^{\Hsmooth \Fsmooth} \kappa^\Fsmooth  
    + 2 \cde_{x,7} c_{-,2}^\frac{1+\Gsmooth}{2} \iota_2^{2 \Hsmooth \Gsmooth} \kappa^{2\Gsmooth} 
    + 2 \cde_{x,8} c_{-,2}^\frac{1+\Hsmooth}{2} \iota_2^{ 2\Hsmoothsq + \Hsmooth} \kappa^{2\Hsmooth} \\
    & = S_H \Gamma_{22}^{\frac{1+\Hsmooth}{2}} \cdot h_4(S_H, \Gamma_{22})
\end{aligned}
\end{equation}
for some function $h_4(\cdot, \cdot)$.
Here the last inequality follows from the definitions of $\cde_{x,5}$ to $\cde_{x,8}$ in \eqref{eq:constantx:new} and $c_{-,2}$ in \eqref{lem:decouple:Eq0}.

\vspace{0.2cm}
For $\Delta_{y,t}$ defined in \eqref{eq:delta_y}, we have 
\begin{align}
\label{eq:delta_y_rep}
    \Delta_{y,t}
    \le \cde_{y,1} \beta_t ( \EB \| \xhat_t \|^{2+2\Gsmooth} + \EB \| \yhat_t \|^{2+2\Gsmooth} )
    \text{ with } 
    \cde_{y,1} =  \SBG^2 \left( \frac{15 d_y^2 }{\mu_G} + d_y^2 \iota_2 + 8 d_y \iota_2 \right).
\end{align}
By Young's inequality with $p = \frac{1}{1-\Gsmooth}$ and $q = \frac{1}{\Gsmooth}$ , for $z = \| \xhat_t \|$ or $\| \yhat_t \|$, we have
\begin{align}
    \cde_{y,1} \beta_t z^{2+2\Gsmooth} 
    & \le \frac{4 L_H \LGx \beta_t^2}{\mu_F \alpha_t} z^2 + \cde_{y,2} \frac{\beta_t^2}{\alpha_t} \left( \frac{\alpha_t}{\beta_t} \right)^\frac{1}{\Gsmooth} z^4 \text{ with } \cde_{y,2} \propto \SBG^\frac{2}{\Gsmooth},
    \label{eq:constanty-de-2}
\end{align}
where we use $\cde_{y,2} 
$ to hide the problem-dependent coefficients.
Plugging this inequality and \eqref{lem:decouple:Eq0} into \eqref{eq:delta_y_rep} and noting that $\frac{\beta_t}{\alpha_t} \le \frac{\mu_F \mu_G}{24 L_H \LGx}$, we obtain
\begin{align}
\label{eq:delta_y_upper}
    \Delta_{y,t} 
    & \le \frac{\mu_G \beta_t}{6} \EB \| \yhat_t \|^2 + \frac{4 L_H \LGx \beta_t^2}{\mu_F \alpha_t} \EB \| \xhat_t \|^2 
    + \cde_{y,2} \left( c_{+,2} \alpha_t \beta_t^2 + c_{-,2} \frac{ \beta_t^{6 + 4\Hsmooth} }{\alpha_t^5} \right) \left( \frac{\alpha_t}{\beta_t} \right)^\frac{1}{\Gsmooth}.
\end{align}

Finally, we concentrate on $\Delta_{xy,t}$, which is defined in \eqref{eq:delta_xy}.
We first tackle the first term.
By Young's inequality with $p = \frac{2}{2 - \Fsmooth}$ and $q = \frac{2}{\Fsmooth}$, for $z$ = $\| \xhat_t \|$ or $\| \yhat_t \|$, we have
\begin{align}
    2 \alpha_t \SBF z^{2+\Fsmooth}
    & \le \frac{\LGx \beta_t}{4} z^2 + \cde_{xy,5}\, \beta_t \left( \frac{\alpha_t}{\beta_t} \right)^\frac{2}{\Fsmooth} z^4 \text{ and } \cde_{xy,5} \propto \SBF^\frac{2}{\Fsmooth},
    \label{eq:constantxy-de-5}
\end{align}
where we use $\cde_{xy,5} 
$ 
to hide the problem-dependent coefficients.
Other terms can be controlled by \eqref{lem:decouple:Eq0.1}.
Recall that $\alpha_t \le \iota_1$ and $\frac{\beta_t}{\alpha_t} \le \kappa$.
Plugging the above inequality and \eqref{lem:decouple:Eq0.1} into \eqref{eq:delta_xy} yields
\begin{align}
\label{eq:delta_xy_upper}
    \Delta_{xy,t}
    & \le \LGx \beta_t (\EB \| \xhat_t \|^2 {+} \EB \| \yhat_t \|^2) 
    {+} \cde_{xy,6}\, \alpha_t \beta_t {+} \cde_{xy,7} \frac{\beta_t^{3+2\Hsmooth}}{\alpha_t^2}
    {+} \left( \cde_{xy,8} \alpha_t^2 \beta_t {+} \cde_{xy,9} \frac{ \beta_t^{5+4\Hsmooth} }{\alpha_t^4} \right)\! \left( \frac{\alpha_t}{\beta_t} \right)^\frac{2}{\Fsmooth},
\end{align}
where the constants are defined as
\begin{equation}
\label{eq:constantxy-de-6789}
\begin{aligned}
    \cde_{xy,6} & = 2 \SBG (1+2L_H) c_{+,2}^\frac{2+\Gsmooth}{4} \iota_1^\frac{\Gsmooth}{2} 
    + 2 \cde_{xy,4} c_{+,2}^\frac{2+\Hsmooth}{4} \iota_1^\frac{\Hsmooth}{2} 
    + 4 \SBF \SBG\, c_{+,2}^\frac{2+\Fsmooth + \Gsmooth}{4} \iota_1^{1 + \frac{\Fsmooth + \Gsmooth}{2}} ,\\
    \cde_{xy,7} & = 2 \SBG (1{+}2L_H) c_{-,2}^\frac{2+\Gsmooth}{4} \iota_2^{\Hsmooth \Gsmooth} \kappa^\Gsmooth   
    {+} 2 \cde_{xy,4} c_{-,2}^\frac{2+\Hsmooth}{4} \iota_2^{\Hsmoothsq} \kappa^\Hsmooth   
    {+} 4 \SBF \SBG \,c_{-,2}^\frac{2+\Fsmooth + \Gsmooth}{4} \iota_1 \iota_2^{\Hsmooth(\Fsmooth + \Gsmooth)} \kappa^{\Fsmooth + \Gsmooth} \\
    & 
    = ( \SBG + S_H + \SBF \SBG )  S_H \Gamma_{22}^{\frac{1+\Hsmooth}{2}} \cdot h_5(S_H, \Gamma_{22}) ,\\
    \cde_{xy,8} & = \cde_{xy,5} c_{+,2} \propto \SBF^\frac{2}{\Fsmooth} , \quad \cde_{xy,9} = \cde_{xy,5} c_{-,2} 
    = \SBF^\frac{2}{\Fsmooth}  S_H \Gamma_{22}^{\frac{1+\Hsmooth}{2}} \cdot h_6(S_H, \Gamma_{22})
\end{aligned}
\end{equation}
for some functions $h_5(\cdot, \cdot)$ and $h_6(\cdot, \cdot)$.
Here the inequalities follows from the definitions of $\cde_{xy,4}$ in \eqref{eq:constantx:new}, $\cde_{xy,5}$ in \eqref{eq:constantxy-de-5} and $c_{-,2}$ in \eqref{lem:decouple:Eq0}. 

\vspace{0.2cm}

\noindent
\textbf{Derive the upper bound of $\| \EB \xhat_t \yhat_t^\top \|$.}
With the upper bounds of the higher-order terms, we could derive the upper bound of $\| \EB \xhat_t \yhat_t^\top \|$.
Substituting \eqref{eq:delta_x_upper}, \eqref{eq:delta_y_upper} and \eqref{eq:delta_xy_upper}  into
\eqref{lem:xhat-new:Ineq}, 
\eqref{lem:yhat-new:Ineq} and
\eqref{lem:xyhat:Ineq} respectively yields
\begin{align}
    \EB \|\xhat_{t+1}\|^2 
    & \le \left(1 - \mu_F \alpha_t 
    \right) \EB \|\xhat_{t}\|^2 
    + \cde_{x,1} \beta_t^2 \EB \|\yhat_{t}\|^2 +  \cde_{x,2} \beta_t \| \EB \xhat_t \yhat_t^\top \| + 2\Gamma_{11} \alpha_{t}^2 \notag  \\
    & \quad \ + \cde_{x,3}\, \beta_{t}^2 +  \cde_{x,4} \frac{\beta_t^{2+2\Hsmooth}}{\alpha_t} + \cde_{x,9}\, \alpha_t \beta_t + \cde_{x,10} \frac{\beta_t^{3+2\Hsmooth}}{\alpha_t^2}, 
    \label{lem:decouple:Eq-x} \\
    \EB \|\yhat_{t+1}\|^2 
    & \le \left(1- \frac{ \mu_G \beta_t}{2} \right) \EB \|\yhat_{t}\|^2 
    + \left( \frac{4 L_H \LGx \beta_t^2}{\mu_F \alpha_t} + 2 \LGx^2 \beta_t^2 \right) \EB \|\xhat_{t}\|^2 
    + 2 d_y \LGx \beta_t \| \EB \xhat_t \yhat_t^\top \|
    \notag \\
    & \quad \  + \Gamma_{22} \beta_{t}^2
    + \cde_{y,2} \left( c_{+,2}\, \alpha_t \beta_t^2 + c_{-,2} \frac{ \beta_t^{6 + 4\Hsmooth} }{\alpha_t^5} \right) \left( \frac{\alpha_t}{\beta_t} \right)^\frac{1}{\Gsmooth},
    \label{lem:decouple:Eq-y}\\
    \| \EB \xhat_{t+1} \yhat_{t+1} ^\top\| 
    & \le \left(1-\frac{\mu_F\alpha_t}{2}\right) \|\EB \xhat_t\yhat_{t}^\top \|  + 2 \LGx \beta_t \EB \|\xhat_{t}\|^2 + (\cde_{xy,1} + \LGx) \beta_t \EB\|\yhat_{t}\|^2 + \Sigma_{12} \alpha_t \beta_t \notag \\
    & 
    + \cde_{xy,2} \beta_t^2 + \cde_{xy,3} \beta_t^{1+ 2\Hsmooth} 
    + \cde_{xy,6} \,\alpha_t \beta_t + \cde_{xy,7} \frac{\beta_t^{3+2\Hsmooth}}{\alpha_t^2} 
    +  \left( \cde_{xy,8}\, \alpha_t^2 \beta_t + \cde_{xy,9}\frac{\beta_t^{5+4\Hsmooth} }{\alpha_t^4} \right) \left( \frac{\alpha_t}{\beta_t} \right)^\frac{2}{\Fsmooth}. 
    \label{lem:decouple:Eq-xy}
\end{align}
Define the Lyapunov function 
$
    W_t = \varrho_4 \frac{\beta_t}{\alpha_t} \EB \| \xhat_t \|^2 + \| \EB \xhat_t \yhat_t^\top \| 
    \text{ with }
    \varrho_4 = \frac{6 \LGx}{\mu_F}.
$
Since 
$\frac{\beta_t}{\alpha_t} \le \kappa
    \le \frac{\mu_F \mu_G}{24 d_x L_H \LGx \LGy}
    \wedge \frac{\mu_F}{5 \mu_G}$
and $c_{x,4} = 2 d_x L_H \LGy$,
one can check
\begin{align}
    \label{lem:decouple:Eq-tmp}
    \varrho_4 c_{x,4} \frac{\beta_t^2}{\alpha_t}
    \le \frac{72 d_x L_H \LGx \LGy }{\mu_F^2}  \cdot \kappa^2 \cdot \frac{ \mu_F \alpha_t}{6}
    \le \frac{\mu_F \alpha_t}{6} \text{ and }
    2 \LGx = \frac{\mu_F \varrho_4}{3}.
\end{align}
Combining \eqref{lem:decouple:Eq-x},
    \eqref{lem:decouple:Eq-xy}, \eqref{lem:decouple:Eq-tmp}
and $\frac{\beta_{t+1}}{\alpha_{t+1}} \le \frac{\beta_t}{\alpha_t}$,
we can obtain
\begin{align}
    W_{t+1}
    & \le
    \left( 1 - \frac{\mu_F \alpha_t}{3} \right) W_t
    + \cde_{xy,10} \beta_t \EB \| \yhat_t \|^2
    + \cde_{xy,11} \alpha_t \beta_t
    + \cde_{xy,12} \beta_t^{1+2\Hsmooth} \notag \\
    & \qquad 
    +  \left( \cde_{xy,8} \alpha_t^2 \beta_t + \cde_{xy,9} \frac{ \beta_t^{5+4\Hsmooth} }{\alpha_t^4} \right) \left( \frac{\alpha_t}{\beta_t} \right)^\frac{2}{\Fsmooth},
    \label{lem:decouple:Eq1}
\end{align}
where we also use Assumption~\ref{assump:stepsize-new} and $\varrho_4 = \frac{6 \LGx}{\mu_F}$,
and the constants are defined as
\begin{equation}
\label{eq:constantxy-de-101112}
\begin{aligned}
    \cde_{xy,10}
    & = \cde_{xy,1} + \LGx + \varrho_4 \cde_{x,1} \rho , \\
    \cde_{xy,11} 
    & = 2 \varrho_4 \Gamma_{11} + \Sigma_{12} + \varrho_4 \cde_{x,3} \kappa^2 + \varrho_4 \cde_{x,9} \kappa 
    + \cde_{xy,2} \kappa + \cde_{xy,6}, \\
    \cde_{xy,12}
    & = \cde_{xy,3} + \varrho_4 \cde_{x,4} \kappa^2 + \varrho_4 \cde_{x,10} \kappa^3 + \cde_{xy,7} \kappa^2 
    = S_H \Gamma_{22}^\frac{1+\Hsmooth}{2} \cdot h_7(S_H, \Gamma_{22})
\end{aligned}
\end{equation}
for some function $h_7(\cdot, \cdot)$, and the last step follows from the definitions of $\cde_{xy,3}$ in \eqref{eq:constant-xy}, $\cde_{x,4}$ in \eqref{eq:constantx:new}, $\cde_{x,10}$ in \eqref{eq:constantx-de-910} and $\cde_{xy,7}$ in \eqref{eq:constantxy-de-6789}. 

For simplicity, we let
$
\tilde{\alpha}_{j, t} = \prod_{\tau=j}^t\left(1-\frac{\mu_{F}\alpha_{\tau}}{3} \right), \ 
{\beta}_{j, t} = \prod_{\tau=j}^t\left(1-\frac{\mu_{G}\beta_{\tau}}{4} \right)
\text{ and }
\tilde{\beta}_{j, t} = \prod_{\tau=j}^t\left(1-\frac{ \mu_{G}\beta_{\tau}}{2} \right).
$
Note that under our assumptions, \eqref{lem:x+y_quartic:Eq2.05} holds, i.e., 
\begin{align}
\label{lem:decouple:Eq2}
\EB \| \xhat_t \|^2 + \EB \| \yhat_t \|^2  \le c_+  \alpha_t + c_- \frac{\beta_t^{2+2\Hsmooth}}{\alpha_t^2} 
\le c_+ \alpha_t + c_- \kappa^2 \beta_t^{2\Hsmooth}
\text{ with } c_- \propto S_H^2 \Gamma_{22}^{1+\Hsmooth}.
\end{align}
Plugging \eqref{lem:decouple:Eq2} into \eqref{lem:decouple:Eq1} and iterating, we obtain
\begin{align}
    W_{t+1}
    & \le 
    \ccde_{xy,0}\,{\beta}_{0,t} +
    \ccde_{xy,1} \,
    \beta_t
    + \ccde_{xy,2} \frac{\beta_t^{1+2\Hsmooth} }{\alpha_t}
    + \left( \ccde_{xy,3} \, \alpha_t \beta_t + \ccde_{xy,4}    \frac{\beta_t^{5+4\Hsmooth} }{\alpha_t^5} \right) \left( \frac{\alpha_t}{\beta_t} \right)^\frac{2}{\Fsmooth}, \notag
\end{align}
where the inequality applies (i) and (ii) in Lemma~\ref{lem:step-size-ineq2-2}, whose proof is deferred to Appendix~\ref{proof:de:step-size-ineq2}, and also uses $\tilde{\alpha}_{0, t} \le {\beta}_{0,t}$, and the constants are defined as 
\begin{equation}
\label{eq:constantxy-ccde-01234}
\begin{aligned}
    \ccde_{xy,0} & = \varrho_4 \kappa \EB \| \xhat_0 \|^2 +  \| \EB \xhat_0 \yhat_0^\top \|, \ 
    \ccde_{xy,1} =  \frac{6(\cde_{xy,10} c_+ + \cde_{xy,11}) }{\mu_F}, \\
    \ccde_{xy, 2} & = \frac{ 4(\cde_{xy,10} c_- \kappa^2 + \cde_{xy,12} ) }
    {\mu_F} = S_H \Gamma_{22}^\frac{1+\Hsmooth}{2} \cdot h_{xy,2}(S_H, \Gamma_{22}), \\ 
    \ccde_{xy,3} & = \frac{ 12\cde_{xy,8} }{\mu_F} \propto \SBF^\frac{2}{\Fsmooth}, \ 
     \ccde_{xy,4}  = \frac{9 \cde_{xy,9} }{\mu_F} 
     = \SBF^\frac{2}{\Fsmooth} S_H \Gamma_{22}^{\frac{1+\Hsmooth}{2}} \cdot h_{xy,4}(S_H, \Gamma_{22})
\end{aligned}
\end{equation}
for some functions $\{ h_{xy,i}(\cdot, \cdot) \}_{i=2,4}$.
Here we have used the definitions of $c_-$ in \eqref{lem:x+y_quartic:Eq2.05}, 
$\cde_{xy,8}$, $\cde_{xy,9}$ in \eqref{eq:constantxy-de-6789} and $\cde_{xy,12}$ in \eqref{eq:constantxy-de-101112}.

\setcounter{lem}{3}
\begin{lem}[Step sizes inequalities]
\label{lem:step-size-ineq2-2}
Let 
$\tilde{\alpha}_{j, t} = \prod_{\tau=j}^t\left(1-\frac{\mu_{F}\alpha_{\tau}}{3} \right)$,  
$\tilde{\beta}_{j,t} = \prod_{\tau=j}^t \left( 1 - \frac{ \mu_G \beta_\tau}{2} \right)$ and ${\beta}_{j,t} = \prod_{\tau=j}^t \left( 1 - \frac{ \mu_G \beta_\tau}{4} \right)$.
Under Assumption~\ref{assump:stepsize-new},
it holds that
\begin{enumerate}[(i)]
    \item 
    $\sum_{j=0}^t \tilde{\alpha}_{j+1, t}\, \alpha_j \beta_j \le \frac{6 \beta_t}{\mu_F} $
    and $\sum_{j=0}^t \tilde{\alpha}_{j+1, t}\, \beta_j^{1+2\Hsmooth} \le \frac{4}{\mu_F} \frac{\beta_t^{1+2\Hsmooth} }{\alpha_t}$.
    
    \item 
    $\sum_{j=0}^t \tilde{\alpha}_{j+1, t}\, \alpha_{j}^2 \beta_j \left( \frac{\alpha_j}{\beta_j} \right)^\frac{2}{\Fsmooth}
    \le \frac{ 12 \alpha_t \beta_t }{\mu_F } \left( \frac{\alpha_t}{\beta_t} \right)^\frac{2}{\Fsmooth}$ and
    $\sum_{j=0}^t \tilde{\alpha}_{j+1, t}\, \frac{\beta_j^{5+4\Hsmooth}}{\alpha_j^4} \left( \frac{\alpha_j}{\beta_j} \right)^\frac{2}{\Fsmooth}
    \le \frac{ 9 \beta_t^{5+4\Hsmooth} } {\mu_F \alpha_t^5 } \left( \frac{\alpha_t}{\beta_t} \right)^\frac{2}{\Fsmooth}$.
    
    \item $\sum_{j=0}^t \tilde{\beta}_{j+1, t} \beta_j^2 \le \frac{4 \beta_t }{\mu_G} $, $\sum_{j=0}^t \tilde{\beta}_{j+1, t} \beta_j \beta_{0, j-1} \le \frac{8 \beta_{0,t}}{\mu_G}$ and 
    $\sum_{j=0}^t \tilde{\beta}_{j+1,t} \frac{\beta_t^{2+2\Hsmooth} }{\alpha_t} = \frac{ 3 \beta_t^{1+2\Hsmooth} }{ \mu_G \alpha_t} $.
    
    \item  $\sum_{j=0}^t \tilde{\beta}_{j+1, t}\, \alpha_{j} \beta_j^2 \left( \frac{\alpha_j}{\beta_j} \right)^\frac{2}{\Fsmooth}
    \le \frac{ 6 \alpha_t \beta_t }{\mu_G } \left( \frac{\alpha_t}{\beta_t} \right)^\frac{2}{\Fsmooth}$ and
    $\sum_{j=0}^t \tilde{\beta}_{j+1, t}\, \frac{\beta_j^{6+4\Hsmooth}}{\alpha_j^5} \left( \frac{\alpha_j}{\beta_j} \right)^\frac{2}{\Fsmooth}
    \le \frac{ 6 \beta_t^{5+4\Hsmooth} } {\mu_G \alpha_t^5 } \left( \frac{\alpha_t}{\beta_t} \right)^\frac{2}{\Fsmooth}$.

    \item  $\sum_{j=0}^t \tilde{\beta}_{j+1, t}\, \alpha_{j} \beta_j^2 \left( \frac{\alpha_j}{\beta_j} \right)^\frac{1}{\Gsmooth}
    \le \frac{ 6 \alpha_t \beta_t }{\mu_G } \left( \frac{\alpha_t}{\beta_t} \right)^\frac{1}{\Gsmooth}$ and
    $\sum_{j=0}^t \tilde{\beta}_{j+1, t}\, \frac{\beta_j^{6+4\Hsmooth}}{\alpha_j^5} \left( \frac{\alpha_j}{\beta_j} \right)^\frac{2}{\Fsmooth}
    \le \frac{ 6 \beta_t^{5+4\Hsmooth} } {\mu_G \alpha_t^5 } \left( \frac{\alpha_t}{\beta_t} \right)^\frac{1}{\Gsmooth}$.
\end{enumerate}

\end{lem}

As a result, $ \| \EB \xhat_{t+1} \yhat_{t+1}^\top \| \le  W_{t+1}$ implies
\begin{align}
    \| \EB \xhat_{t+1} \yhat_{t+1}^\top \| 
    & \le \ccde_{xy,0}\, {\beta}_{0,t} + 
    \ccde_{xy,1} \,
    \beta_t
    + \ccde_{xy,2} \frac{\beta_t^{1+2\Hsmooth} }{\alpha_t}
    + \left( \ccde_{xy,3} \, \alpha_t \beta_t + \ccde_{xy,4}    \frac{\beta_t^{5+4\Hsmooth} }{\alpha_t^5} \right) \left( \frac{\alpha_t}{\beta_t} \right)^\frac{2}{\Fsmooth}. 
    \label{lem:decouple:Eq3}
\end{align}
Since $\kappa \le \frac{\mu_F \mu_G}{200 L_H \LGx \LGy}$, we have $\ccde_{xy,0} \le \frac{ \mu_G }{33 L_H \LGy} \EB \| \xhat_0 \|^2 + \| \EB \xhat_0 \yhat_0^\top \|$.
Thus we obtain the second inequality.
\vspace{0.2cm}

\noindent
\textbf{Derive the upper bound of $\EB \| \yhat_t \|^2$.}
Now we are prepared to establish the convergence rate of $\EB \| \yhat_t \|^2$.
From \eqref{lem:decouple:Eq3} 
we can obtain
\begin{align}
    \| \EB \xhat_{t} \yhat_{t}^\top \| 
    & \le 
    \ccde_{xy,0}\,     {\beta}_{0,t-1} {+}
    \ccde_{xy,1} \, \zeta
    \beta_t
    {+} \ccde_{xy,2}\, \zeta^3 \frac{\beta_t^{1+2\Hsmooth} }{\alpha_t}
    {+} \left( \ccde_{xy,3}\, \zeta^{\frac{2}{\Fsmooth} + 2} \alpha_t \beta_t {+} \ccde_{xy,4}\, \zeta^{\frac{2}{\Fsmooth} + 9} \frac{\beta_t^{5+4\Hsmooth} }{\alpha_t^5} \right)\! \left( \frac{\alpha_t}{\beta_t} \right)^\frac{2}{\Fsmooth}, \label{lem:decouple:Eq3.1}
\end{align}
where $\zeta = 1 + \frac{\mu_F \iota_1}{16} + \frac{\mu_G \iota_2}{64} $.
Plugging \eqref{lem:decouple:Eq3.1} and
\eqref{lem:decouple:Eq2} into \eqref{lem:decouple:Eq-y} and applying the upper bounds in Assumption~\ref{assump:stepsize-new},
with $\cde_{y,3} := \frac{4 L_H \LGx}{\mu_F} + 2 \LGx^2 \iota_1$,
we have
\begin{align}
    & \EB \| \yhat_{t+1} \|^2
    \le \left( 1 - \frac{ \mu_G \beta_t}{2} \right) \EB \| \yhat_t \|^2 + 2 d_y \LGx \beta_t {\beta}_{0,t-1}
    + \left( \cde_{y,3} c_+ + 2 d_y \LGx \ccde_{xy,1} \zeta  + \Gamma_{22} \right) \beta_t^2 \notag \\
    & 
    + \left( \cde_{y,3} c_- \kappa^2 +  2 d_y \LGx \ccde_{xy,2} \zeta^3   \right) \frac{ \beta_t^{2+2\Hsmooth} }{\alpha_t}  
    + 2 d_y \LGx  \left( \ccde_{xy,3} \zeta^{\frac{2}{\Fsmooth} + 2} \alpha_t \beta_t^2 + \ccde_{xy,4} \zeta^{\frac{2}{\Fsmooth} + 9 }  \frac{ \beta_t^{6+4\Hsmooth} }{\alpha_t^5} \right) \left( \frac{\alpha_t}{\beta_t} \right)^\frac{2}{\Fsmooth} \notag \\
    & 
    +  \cde_{y,2} 
    \left( c_{+,2} \alpha_t \beta_t^2 + c_{-,2} \frac{ \beta_t^{6+4\Hsmooth} }{\alpha_t^5} \right) \left( \frac{\alpha_t}{\beta_t} \right)^\frac{1}{\Gsmooth}. 
    \label{lem:decouple:Eq5}
\end{align}
Iterating \eqref{lem:decouple:Eq5} and applying (iii), (iv) and (v) in Lemma~\ref{lem:step-size-ineq2-2} yields
\begin{align}
    \EB \| \yhat_{t+1} \|^2
    & \le 
    \ccde_{y,0}\,{\beta}_{0,t} + 
    \ccde_{y,1}\, \beta_t + \ccde_{y,2} \frac{\beta_t^{1+2\Hsmooth} }{\alpha_t}
    + \left( \ccde_{y,3}\, \alpha_t \beta_t + \ccde_{y,4} \frac{ \beta_t^{5+4\Hsmooth} }{\alpha_t^5} \right) \left( \frac{\alpha_t}{\beta_t} \right)^{\frac{2}{\Fsmooth} } \notag \\
    & \quad \ 
    + \left( \ccde_{y,5}\, \alpha_t \beta_t + \ccde_{y,6} \frac{ \beta_t^{5+4\Hsmooth} }{\alpha_t^5} \right) \left( \frac{\alpha_t}{\beta_t} \right)^{\frac{1}{\Gsmooth} },
\end{align}
where 
the constants are defined as
\begin{equation}
\label{eq:constanty-ccde-0123456}
\begin{aligned}
    \ccde_{y,0} & = \EB \| \yhat_0 \|^2 + \frac{16 d_y \LGx  \ccde_{xy,0} }{\mu_G}, \ 
    \ccde_{y,1} 
    =  \frac{4 \left( \cde_{y,3} c_+ + 2 d_y \LGx \ccde_{xy,1} \zeta + \Gamma_{22} \right)}{\mu_G}, \\
    \ccde_{y,2}
    & = \frac{3 \left( \cde_{y,3} c_- \kappa^2 + 2 d_y \LGx \ccde_{xy,2} \zeta^3 \right)}{\mu_G} = S_H \Gamma_{22}^\frac{1+\Hsmooth}{2} \cdot h_{y,2}(S_H, \Gamma_{22}), \\
    \ccde_{y,3}
    & = \frac{ 6 (2 d_y \LGx \ccde_{xy,3} \zeta^{ \frac{2}{\Fsmooth} + 2 } )  }{\mu_G} \propto \SBF^\frac{2}{\Fsmooth}, \\ 
    \ccde_{y,4}
    & = \frac{ 6 (2 d_y \LGx \ccde_{xy,4} \zeta^{ \frac{2}{\Fsmooth} + 9 } )  }{\mu_G} 
    = \SBF^\frac{2}{\Fsmooth} S_H \Gamma_{22}^\frac{1+\Hsmooth}{2} \cdot h_{y,4}(S_H, \Gamma_{22}), \\
    \ccde_{y,5}
    & = \frac{ 6 \cde_{y,2} c_{+,2} }{\mu_G} \propto \SBG^\frac{2}{\Gsmooth}, \
    \ccde_{y,6} = \frac{6 \cde_{y,2} c_{-,2} }{\mu_G} 
    = \SBG^\frac{2}{\Gsmooth} S_H \Gamma_{22}^\frac{1+\Hsmooth}{2} \cdot h_{y,6}(S_H, \Gamma_{22})
\end{aligned}
\end{equation}
for some functions $\{h_{y,i} (\cdot, \cdot)\}_{i=2,4,6}$.
Here we have used the definitions of $c_-$ in \eqref{lem:x+y_quartic:Eq2.05},
$\ccde_{xy,2}$ to $\ccde_{xy,4}$ in \eqref{eq:constantxy-ccde-01234},
$\cde_{y,2}$ in \eqref{eq:constanty-de-2} and $c_{-,2}$ in \eqref{lem:decouple:Eq0}.
\end{proof}

\subsection{Proof of Corollary~\ref{cor:decouple-rates}}
\label{proof:de:rates}

\begin{proof}[Proof of Corollay~\ref{cor:decouple-rates}]
Define $\beta_{0,T} = \prod_{t=0}^T \left( 1 - \frac{\mu_G \beta_t}{4} \right)$. 
Clearly, $\frac{\beta_t}{\alpha_t}$ is non-increasing and upper bounded by $\kappa$.
By setting $\alpha_t, \beta_t, T_0$ as in~\eqref{eq:poly0-decou}, 
we have $\alpha_t \le \iota_1 \wedge \frac{\iota_2}{\kappa} \wedge \frac{\rho}{\kappa}^2$ for all $t\ge0$, so 
the constant bounds in Assumption~\ref{assump:stepsize-new} hold.
Since $(1+x)^{\gamma} \le 1 + \gamma x$ for $x \ge 0$ and $\gamma \in (0, 1]$, then for any $t \ge 1$, we have
\begin{align*}
    \frac{1}{\alpha_t} \left(\frac{\alpha_{t-1}}{\alpha_t} -1 \right) 
    & \le \frac{1}{\alpha_t}  \cdot \frac{a}{t-1+T_0} 
    \le \frac{\Fsmooth \mu_G \kappa a}{128} \frac{(t+T_0)^a}{t-1+T_0}
    \le \frac{\Fsmooth \mu_G \kappa}{64} \frac{1}{(t+T_0)^{1-a}} \le \frac{\Fsmooth \mu_F}{16}.
\end{align*}
Similarly, we can obtain $ \frac{1}{\beta_t} \left(\frac{\alpha_{t-1}}{\alpha_t} -1 \right) \le \frac{ (\Fsmooth \wedge \Gsmooth) \mu_G}{16}$ and $ \frac{1}{\beta_t} \left( \frac{\beta_{t-1}}{\beta_t} - 1 \right) \le \frac{\mu_G}{64}$.
Thus the growth condition in Assumption~\ref{assump:stepsize-new} holds.
For $\beta_{0,T}$, if $b = 1$, we have
$
    \beta_{0,T} 
    \le \prod_{t=0}^T\left(1-\frac{\mu_{G}\beta_{t}}{64} \right)
    \le \prod_{t=0}^T\left(1 - \frac{2}{t+T_0}\right) 
    \le \frac{ T_0^2 }{(T+T_0) (T+T_0-1)} 
    = \OM(\alpha_T^2).
    $
If $b < 1$, we have
$
    \beta_{0,T}
    \le \exp \left( - 32 \sum_{t=0}^T (t+ T_0)^{-b} \right) 
    \le \exp \left( - \frac{ (T+T_0+1)^{1-b} - T_0^{1-b} }{ (1-b) / 32 } \right) 
    = o(\alpha_T^2).
$
Then the conditions in Theorem~\ref{thm:decouple} hold.
Note that $\Hsmooth \ge 0.5$ implies $\frac{\beta_t^{1+2\Hsmooth}}{\alpha_t} = \OM (\beta_t)$ and $\frac{\beta_t^{5+4\Hsmooth}}{\alpha_t^5} = \OM (\alpha_t \beta_t)$;
$\frac{b}{a} \le 1 + \frac{\Fsmooth}{2} \wedge \Gsmooth \le 2 $ implies $ \left( \frac{\alpha_t}{\beta_t} \right)^{ \frac{2}{\Fsmooth} \vee \frac{1}{\Gsmooth} } = \OM (\frac{1}{\alpha_t}) $ and $\beta_{0,T} = \OM(\alpha_T^2) = \OM(\beta_T)$.
It follows that $\| \EB \xhat_T \yhat_T^\top \| = \OM (\beta_T)$ and $\EB \| \yhat_T \|^2 = \OM (\beta_T)$
The upper bound of $\EB \| \xhat_T \|^2$ is from Theorem~\ref{thm:decouple-short}.
\end{proof}

\subsection{Proof of Lemmas~\ref{lem:step-size-ineq2-1} and \ref{lem:step-size-ineq2-2}}
\label{proof:de:step-size-ineq2}

\begin{proof}[Proof of Lemma~\ref{lem:step-size-ineq2-1}]
The proof is similar to that of Lemma~14 in~\citep{kaledin2020finite}.
\begin{itemize}
\item 
For the first claim in (i), the growth condition implies
\begin{align}
    \left( \frac{\alpha_{\tau-1} }{ \alpha_\tau } \right)^2
    \left( 1 - \frac{\mu_G \beta_\tau}{4} \right)
    & \le \left( 1 + \frac{\mu_G \beta_\tau }{16} \right)^2 \left( 1 - \frac{\mu_G \beta_\tau}{4} \right) \notag \\
    & = \left( 1 + \frac{\mu_G \beta_\tau }{8} + \frac{ \mu_G^2 \beta_\tau^2 }{16^2} \right) \left( 1 - \frac{\mu_G \beta_\tau}{4} \right)
    \le 1 - \frac{ \mu_G \beta_\tau }{8}.
    \label{lem:step-size-ineq2:Eq1}
\end{align}
Then we have
\begin{align}
    \sum_{j=0}^t {\beta}_{j+1, t}\, \alpha_j^2 \beta_j
    & = \sum_{j=0}^t \alpha_j^2 \beta_j \prod_{\tau=j+1}^t \left( 1 - \frac{\mu_G \beta_\tau}{4} \right) \notag \\
    & = \alpha_t^2 \sum_{j=0}^t \beta_j \left( \frac{\alpha_j}{\alpha_t} \right)^2 \prod_{\tau=j+1}^t \left( 1 - \frac{\mu_G \beta_\tau}{4} \right) \notag \\
    & = \alpha_t^2 \sum_{j=0}^t \beta_j \prod_{\tau=j+1}^t \left( \frac{\alpha_{\tau-1} }{\alpha_\tau} \right)^2\left( 1 - \frac{\mu_G \beta_\tau}{4} \right) \notag \\
    & \overset{\eqref{lem:step-size-ineq2:Eq1}}{\le} \alpha_t^2 \sum_{j=0}^t \beta_j \prod_{\tau=j+1}^t \left( 1 - \frac{\mu_G \beta_\tau}{8} \right) \notag \\
    & = \frac{8 \alpha_t^2}{\mu_G} \left[ 1 - \prod_{\tau=0}^t \left( 1 - \frac{\mu_G \beta_\tau}{8} \right) \right]
    \le  \frac{8 \alpha_t^2}{\mu_G}. \notag
\end{align}

\item For the second claim in (i), 
we define $\kappa_\tau = \frac{\beta_\tau}{\alpha_\tau}$. Then we have $\frac{\kappa_{\tau-1}}{\kappa_\tau} \le \frac{\beta_{\tau-1}}{\beta_\tau}$.
For $x \in (0, 1/64)$, one can check $(1+x)^{4+4 \Hsmooth } \le 1 + 9x$.
Since $\mu_G \beta_\tau \le \mu_G \iota_2 \le 1$,
the growth condition implies
\begin{align}
     \left( \frac{\beta_{\tau-1}}{\beta_\tau} \right)^{4 \Hsmooth} \left( \frac{\kappa_{\tau-1}}{\kappa_\tau} \right)^4 \left(1 - \frac{\mu_G \beta_\tau}{4} \right)
    & \le 
    \left( \frac{\beta_{\tau-1} }{ \beta_\tau } \right)^{4+4 \Hsmooth}
    \left( 1 - \frac{\mu_G \beta_\tau}{4} \right) \notag\\
    &\le \left( 1 + \frac{\mu_G \beta_\tau }{64} \right)^{4+4 \Hsmooth} \left( 1 - \frac{\mu_G \beta_\tau}{4} \right) \notag \\
    & = \left( 1 + \frac{9 \mu_G \beta_\tau }{64} \right) \left( 1 - \frac{\mu_G \beta_\tau}{4} \right)
    \le 1 - \frac{ \mu_G \beta_\tau }{10}.
    \label{lem:step-size-ineq2:Eq3}
\end{align}
Then we have
\begin{align}
    \sum_{j=0}^t {\beta}_{j+1, t} \frac{\beta_j^{5+4\Hsmooth}}{\alpha_j^4}
    & = \sum_{j=0}^t  \beta_j^{1 + 4 \Hsmooth} \kappa_j^4 \prod_{\tau=j+1}^t \left( 1 - \frac{\mu_G \beta_\tau}{4} \right) \notag \\
    & = \beta_t^{4 \Hsmooth} \kappa_t^4 \sum_{j=0}^t \beta_j \left( \frac{\beta_j}{\beta_t} \right)^{4\Hsmooth}  \left( \frac{\kappa_j}{\kappa_t} \right)^4 \prod_{\tau=j+1}^t \left( 1 - \frac{\mu_G \beta_\tau}{4} \right) \notag \\
    & = \beta_t^{4 \Hsmooth} \kappa_t^4 \sum_{j=0}^t \beta_j \prod_{\tau=j+1}^t \left( \frac{\beta_{\tau-1} }{\beta_\tau} \right)^{4 \Hsmooth} \left( \frac{\kappa_{\tau-1}}{\kappa_\tau} \right)^4 \left( 1 - \frac{\mu_G \beta_\tau}{4} \right) \notag \\
    & \overset{\eqref{lem:step-size-ineq2:Eq3}}{\le} \beta_t^{4\Hsmooth} \kappa_t^4 \sum_{j=0}^t \beta_j \prod_{\tau=j+1}^t \left( 1 - \frac{\mu_G \beta_\tau}{10} \right) \notag \\
    & = \frac{10 \beta_t^{4\Hsmooth} \kappa_t^4}{\mu_G} \left[ 1 - \prod_{\tau=0}^t \left( 1 - \frac{\mu_G \beta_\tau}{10} \right) \right]
    \le  \frac{10 \beta_t^{4 + 4 \Hsmooth}}{\mu_G \alpha_t^4}. \notag
\end{align}

\item
For the first claim in (ii), the growth condition implies
\begin{align}
    \left( \frac{\alpha_{\tau-1} }{ \alpha_\tau } \right)^2
    \left( 1 - \frac{\mu_F \alpha_\tau}{2} \right)
    & \le \left( 1 + \frac{\mu_F \alpha_\tau }{16} \right)^2 \left( 1 - \frac{\mu_F \alpha_\tau}{2} \right) \notag \\
    & = \left( 1 + \frac{\mu_F \alpha_\tau }{8} + \frac{ \mu_F^2 \alpha_\tau^2 }{16^2} \right) \left( 1 - \frac{\mu_F \alpha_\tau}{2} \right)
    \le 1 - \frac{ \mu_F \alpha_\tau }{4}.
    \label{lem:step-size-ineq2:Eq2}
\end{align}
Then we have
\begin{align}
    \sum_{j=0}^t \alpha_{j+1, t}\, \alpha_j^3 
    & = \sum_{j=0}^t \alpha_j^3 \prod_{\tau=j+1}^t \left( 1 - \frac{\mu_F \alpha_\tau}{2} \right) \notag \\
    & = \alpha_t^2 \sum_{j=0}^t \alpha_j \left( \frac{\alpha_j}{\alpha_t} \right)^2 \prod_{\tau=j+1}^t \left( 1 - \frac{\mu_F \alpha_\tau}{2} \right) \notag \\
    & = \alpha_t^2 \sum_{j=0}^t \alpha_j \prod_{\tau=j+1}^t \left( \frac{\alpha_{\tau-1} }{\alpha_\tau} \right)^2\left( 1 - \frac{\mu_F \alpha_\tau}{2} \right) \notag \\
    & \overset{\eqref{lem:step-size-ineq2:Eq2}}{\le} \alpha_t^2 \sum_{j=0}^t \alpha_j \prod_{\tau=j+1}^t \left( 1 - \frac{\mu_F \alpha_\tau}{4} \right) \notag \\
    & = \frac{4 \alpha_t^2}{\mu_F} \left[ 1 - \prod_{\tau=0}^t \left( 1 - \frac{\mu_F \alpha_\tau}{4} \right) \right]
    \le  \frac{4 \alpha_t^2}{\mu_F}. \notag
\end{align}

\item 
For the second claim in (ii), we have
\begin{align*}
    \sum_{j=0}^t {\alpha}_{j+1, t}\, \alpha_j \beta_{0, j-1}
    & = \beta_{0,t} \sum_{j=0}^t \alpha_{j+1, t} \, \alpha_j \frac{1}{ (1 - \mu_G \beta_j / 4) \beta_{j+1, t}} \notag \\
    & = \beta_{0,t} \sum_{j=0}^t \frac{\alpha_j}{1 - \mu_G \beta_j / 4} 
    \prod_{\tau=j+1}^t \frac{1 - \mu_F \alpha_\tau / 2}{1 - \mu_G \beta_\tau / 4}.
\end{align*}
One can check that for $x \in (0, 1/2)$, $\frac{1}{1-x} \le 1 + 2x$.
Since $\mu_G \beta_\tau \le \mu_G \iota_2 \le \frac{1}{2}$, we have $\frac{1}{1 - \mu_G \beta_\tau / 4} \le 1 + \frac{\mu_G \beta_\tau}{2} 
$.
It follows that $\frac{1}{1 - \mu_G \beta_\tau / 4} \le 2$ and
\begin{align*}
    \frac{1 - \mu_F \alpha_\tau / 2}{1 - \mu_G \beta_\tau / 4}
    \le \left(1 - \frac{\mu_F \alpha_\tau}{2} \right) \left(1 + \frac{\mu_G \beta_\tau}{2} \right)
    \le 1 - \frac{\mu_F \alpha_\tau}{2} + \frac{\mu_G \beta_\tau}{2}
    \le 1 - \frac{\mu_F \alpha_\tau}{4}, 
\end{align*}
where the last inequality is due to $\frac{\mu_G \beta_\tau}{\mu_F \alpha_\tau} \le \frac{\mu_G}{\mu_F} \kappa \le \frac{1}{2}$.
Then we have
\begin{align*}
    \sum_{j=0}^t {\alpha}_{j+1, t}\, \alpha_j \beta_{0, j-1}
    & = \beta_{0,t} \sum_{j=0}^t \frac{\alpha_j}{1 - \mu_G \beta_j / 4} 
    \prod_{\tau=j+1}^t \frac{1 - \mu_F \alpha_\tau / 2}{1 - \mu_G \beta_\tau / 4} \\
    & \le 2 \beta_{0,t} \sum_{j=0}^t \alpha_j \prod_{\tau=j+1}^t \left( 1 - \frac{\mu_F \alpha_\tau}{4} \right) \\
    & = \frac{8 \beta_{0,t}}{\mu_F} \left[ 1 - \prod_{\tau=0}^t \left( 1 - \frac{\mu_F \alpha_\tau}{4} \right) \right] \le \frac{8 \beta_{0,t} }{\mu_F}.
\end{align*}

\item For the last claim in (ii), we
define $\kappa_\tau = \frac{\beta_\tau}{\alpha_\tau}$. Then we have $\frac{\kappa_{\tau-1}}{\kappa_\tau} \le \frac{\beta_{\tau-1}}{\beta_\tau}$.
For $x \in (0, 1/64)$, one can check $(1+x)^{4+4 \Hsmooth } \le 1 + 9x$.
Since $\mu_G \beta_\tau \le \mu_G \iota_2 \le 1$ and $\frac{\beta_\tau}{\alpha_\tau} \le \frac{\mu_F}{\mu_G}$,
the growth condition implies
\begin{align}
     \left( \frac{\beta_{\tau-1}}{\beta_\tau} \right)^{4 \Hsmooth} \left( \frac{\kappa_{\tau-1}}{\kappa_\tau} \right)^4 \left(1 - \frac{\mu_F \alpha_\tau}{2} \right)
    & \le 
    \left( \frac{\beta_{\tau-1} }{ \beta_\tau } \right)^{4+4 \Hsmooth}
    \left( 1 - \frac{\mu_F \alpha_\tau}{2} \right) \notag\\
    &\le \left( 1 + \frac{\mu_G \beta_\tau }{64} \right)^{4+4 \Hsmooth} \left( 1 - \frac{\mu_F \alpha_\tau}{2} \right) \notag \\
    & = \left( 1 + \frac{9 \mu_G \beta_\tau }{64} \right) \left( 1 - \frac{\mu_F \alpha_\tau}{2} \right) \notag \\
    & \le \left( 1 + \frac{9 \mu_F \alpha_\tau }{64} \right) \left( 1 - \frac{\mu_F \alpha_\tau}{2} \right)
    \le 1 - \frac{ \mu_F \alpha_\tau }{3}.
    \label{lem:step-size-ineq2:Eq3.5}
\end{align}
Then we have
\begin{align}
    \sum_{j=0}^t {\alpha}_{j+1, t} \frac{\beta_j^{4+4\Hsmooth}}{\alpha_j^3}
    & = \sum_{j=0}^t  \alpha_j \beta_j^{4 \Hsmooth} \kappa_j^4 \prod_{\tau=j+1}^t \left( 1 - \frac{\mu_F \alpha_\tau}{2} \right) \notag \\
    & = \beta_t^{4 \Hsmooth} \kappa_t^4 \sum_{j=0}^t \alpha_j \left( \frac{\beta_j}{\beta_t} \right)^{4\Hsmooth}  \left( \frac{\kappa_j}{\kappa_t} \right)^4 \prod_{\tau=j+1}^t \left( 1 - \frac{\mu_F \alpha_\tau}{2} \right) \notag \\
    & = \beta_t^{4 \Hsmooth} \kappa_t^4 \sum_{j=0}^t \alpha_j \prod_{\tau=j+1}^t \left( \frac{\beta_{\tau-1} }{\beta_\tau} \right)^{4 \Hsmooth} \left( \frac{\kappa_{\tau-1}}{\kappa_\tau} \right)^4 \left( 1 - \frac{\mu_F \alpha_\tau}{2} \right) \notag \\
    & \overset{\eqref{lem:step-size-ineq2:Eq3.5}}{\le} \beta_t^{4\Hsmooth} \kappa_t^4 \sum_{j=0}^t \alpha_j \prod_{\tau=j+1}^t \left( 1 - \frac{\mu_F \alpha_\tau}{3} \right) \notag \\
    & = \frac{3 \beta_t^{4\Hsmooth} \kappa_t^4}{\mu_F} \left[ 1 - \prod_{\tau=0}^t \left( 1 - \frac{\mu_F \alpha_\tau}{3} \right) \right]
    \le  \frac{3 \beta_t^{4 + 4 \Hsmooth}}{\mu_F \alpha_t^4}. \notag
\end{align}

\end{itemize}
    
\end{proof}

\begin{proof}[Proof of Lemma~\ref{lem:step-size-ineq2-2}]
The proof is similar to that of Lemma~14 in~\citep{kaledin2020finite}.

\begin{enumerate}[(i)]
\item 

\begin{itemize}
\item For the first claim in (i), since $\frac{\beta_t}{\alpha_t} \le \kappa \le \frac{\mu_F}{\mu_G} $, we have $\frac{\beta_{t-1}}{\beta_t} \le 1 + \frac{\mu_G \beta_t}{32} \le 1 + \frac{\mu_F \alpha_t}{32}$. Then by exchanging the states of $\alpha_t$ and $\beta_t$ in (iii) of Lemma~\ref{lem:help},
we have $\sum_{j=0}^t \tilde{\alpha}_{j+1, t}\, \alpha_j \beta_j  \le \frac{4 \beta_t}{\mu_F}$. 


\item For the last claim in (i),
we
define $\kappa_\tau = \frac{\beta_\tau}{\alpha_\tau}$. Then we have $\frac{\kappa_{\tau-1}}{\kappa_\tau} \le \frac{\beta_{\tau-1}}{\beta_\tau}$.
For $x \in (0, 0.1)$, one can check $(1+x)^{1+2 \Hsmooth } \le 1 + 4x$.
Since $\mu_G \beta_\tau \le \mu_G \iota_2 \le 1$ and $\frac{\beta_\tau}{\alpha_\tau} \le \frac{\mu_F}{\mu_G}$,
the growth condition implies
\begin{align}
     \left( \frac{\beta_{\tau-1}}{\beta_\tau} \right)^{2 \Hsmooth}  \frac{\kappa_{\tau-1}}{\kappa_\tau} \left(1 - \frac{\mu_F \alpha_\tau}{3} \right)
    & \le 
    \left( \frac{\beta_{\tau-1} }{ \beta_\tau } \right)^{1+2 \Hsmooth}
    \left( 1 - \frac{\mu_F \alpha_\tau}{3} \right) \notag\\
    &\le \left( 1 + \frac{\mu_G \beta_\tau }{64} \right)^{1+2 \Hsmooth} \left( 1 - \frac{\mu_F \alpha_\tau}{3} \right) \notag \\
    & = \left( 1 + \frac{ \mu_G \beta_\tau }{16} \right) \left( 1 - \frac{\mu_F \alpha_\tau}{3} \right) \notag \\
    & \le \left( 1 + \frac{ \mu_F \alpha_\tau }{16} \right) \left( 1 - \frac{\mu_F \alpha_\tau}{3} \right)
    \le 1 - \frac{ \mu_F \alpha_\tau }{4}.
    \label{lem:step-size-ineq2:Eq3.6}
\end{align}
Then we have
\begin{align}
    \sum_{j=0}^t 
    \tilde{\alpha}_{j+1, t} {\beta_j^{1+2\Hsmooth}}
    & = \sum_{j=0}^t \alpha_j \beta_j^{2 \Hsmooth} \kappa_j \prod_{\tau=j+1}^t \left( 1 - \frac{\mu_F \alpha_\tau}{3} \right) \notag \\
    & = \beta_t^{2 \Hsmooth} \kappa_t \sum_{j=0}^t \alpha_j \left( \frac{\beta_j}{\beta_t} \right)^{2\Hsmooth} \frac{\kappa_j}{\kappa_t} \prod_{\tau=j+1}^t \left( 1 - \frac{\mu_F \alpha_\tau}{3} \right) \notag \\
    & = \beta_t^{2 \Hsmooth} \kappa_t \sum_{j=0}^t \alpha_j \prod_{\tau=j+1}^t \left( \frac{\beta_{\tau-1} }{\beta_\tau} \right)^{2 \Hsmooth}  \frac{\kappa_{\tau-1}}{\kappa_\tau}  \left( 1 - \frac{\mu_F \alpha_\tau}{3} \right) \notag \\
    & \overset{\eqref{lem:step-size-ineq2:Eq3.6}}{\le} \beta_t^{2\Hsmooth} \kappa_t \sum_{j=0}^t \alpha_j \prod_{\tau=j+1}^t \left( 1 - \frac{\mu_F \alpha_\tau}{4} \right) \notag \\
    & = \frac{4 \beta_t^{2\Hsmooth} \kappa_t}{\mu_F} \left[ 1 - \prod_{\tau=0}^t \left( 1 - \frac{\mu_F \alpha_\tau}{4} \right) \right]
    \le  \frac{4 \beta_t^{1 + 2 \Hsmooth}}{\mu_F \alpha_t}. \notag
\end{align}
\end{itemize}

\item 
\begin{itemize}
\item For the first claim in (ii),
define $\zeta_\tau = \frac{\alpha_\tau}{\beta_\tau}$. Then we have $\frac{\zeta_{\tau-1}}{\zeta_\tau} \le \frac{\alpha_{\tau-1}}{\alpha_\tau}$.
For $x \in (0. 0.2)$, one can check $\exp(x) \le 1 + 1.2x$.
Since $\Fsmooth \le 1$, $\frac{\beta_\tau}{\alpha_\tau} \le \frac{\mu_F}{5 \mu_G}$ and $\mu_F \alpha_\tau \le \mu_F \iota_1 \le 1$,
the growth condition implies
\begin{align}
     \frac{\alpha_{\tau-1}}{\alpha_\tau} \frac{ \beta_{\tau-1} }{ \beta_\tau  }  \left( \frac{\zeta_{\tau-1}}{\zeta_\tau} \right)^\frac{2}{\Fsmooth} \left(1 - \frac{\mu_F \alpha_\tau}{3} \right)
    & \le \frac{\beta_{\tau-1}}{\beta_\tau} 
    \left( \frac{\alpha_{\tau-1} }{ \alpha_\tau } \right)^\frac{3}{\Fsmooth}
    \left( 1 - \frac{\mu_F \alpha_\tau}{3} \right) \notag \\
    & \le \left( 1+\frac{\mu_G \beta_\tau}{64} \right) \left( 1 + \frac{ \Fsmooth\mu_F \alpha_\tau }{16} \right)^\frac{3}{\Fsmooth} \left( 1 - \frac{\mu_F \alpha_\tau}{3} \right) \notag \\
    & \le \left( 1 + \frac{\mu_F \alpha_\tau}{80} \right) \exp \left( \frac{3 \mu_F \alpha_\tau}{16} \right) \left( 1 - \frac{\mu_F \alpha_\tau}{3} \right) \notag \\
    & \le \left( 1 + \frac{\mu_F \alpha_\tau}{80} \right) \left( 1 + \frac{3.6 \mu_F \alpha_\tau}{16}  \right) \left( 1 - \frac{\mu_F \alpha_\tau}{3} \right) \notag \\
    & = \left( 1 + \frac{\mu_F \alpha_\tau }{4} \right) \left( 1 - \frac{\mu_F \alpha_\tau}{3} \right)
    \le 1 - \frac{ \mu_F \alpha_\tau }{12}.
    \label{lem:step-size-ineq2:Eq4}
\end{align}
Then we have
\begin{align}
    \sum_{j=0}^t 
    \tilde{\alpha}_{j+1, t} \alpha_j^2 \beta_j \left( \frac{\alpha_j}{\beta_j} \right)^\frac{2}{\Fsmooth}
    & = \sum_{j=0}^t \alpha_j^{2} \beta_j \zeta_j^\frac{2}{\Fsmooth} \prod_{\tau=j+1}^t \left( 1 - \frac{\mu_F \alpha_\tau}{3} \right) \notag \\
    & = \alpha_t \beta_t \zeta_t^\frac{2}{\Fsmooth} \sum_{j=0}^t \alpha_j \frac{\alpha_j}{\alpha_t} \frac{\beta_j}{\beta_t} \left(\frac{\zeta_j}{\zeta_t} \right)^\frac{2}{\Fsmooth} \prod_{\tau=j+1}^t \left( 1 - \frac{\mu_F \alpha_\tau}{3} \right) \notag \\
    & = \alpha_t \beta_t \zeta_t^\frac{2}{\Fsmooth} \sum_{j=0}^t \alpha_j \prod_{\tau=j+1}^t \frac{\alpha_{\tau-1} }{\alpha_\tau} \frac{\beta_{\tau-1}}{\beta_\tau} \left( \frac{\zeta_{\tau-1}}{\zeta_\tau} \right)^\frac{2}{\Fsmooth} \left( 1 - \frac{\mu_F \alpha_\tau}{3} \right) \notag \\
    & \overset{\eqref{lem:step-size-ineq2:Eq4}}{\le} \alpha_t \beta_t \zeta_t^\frac{2}{\Fsmooth} \sum_{j=0}^t \alpha_j \prod_{\tau=j+1}^t \left( 1 - \frac{\mu_F \alpha_\tau}{12} \right) \notag \\
    & = \frac{12 \alpha_t \beta_t \zeta_t^\frac{2}{\Fsmooth} }{\mu_F} \left[ 1 - \prod_{\tau=0}^t \left( 1 - \frac{\mu_F \alpha_\tau}{12} \right) \right]
    \le  \frac{12 \alpha_t \beta_t}{\mu_F} \left( \frac{\alpha_t}{\beta_t} \right)^\frac{2}{\Fsmooth}. \notag
\end{align}

\item For the last claim in (ii), define 
$\kappa_\tau = \frac{\beta_\tau}{\alpha_\tau}$ and
$\zeta_\tau = \frac{\alpha_\tau}{\beta_\tau}$. Then we have $\frac{\kappa_{\tau-1} }{\kappa_\tau} \le \frac{ \beta_{\tau-1} }{\beta_\tau} $ and $\frac{\zeta_{\tau-1}}{\zeta_\tau} \le \frac{\alpha_{\tau-1}}{\alpha_\tau}$.
For $x \in (0. 0.2)$, one can check $\exp(x) \le 1+1.2x$ and for $x \in (0, 0.1)$, $(1+x)^{5+4\Hsmooth} \le 1 + 10x$.
Since $\frac{\beta_\tau}{\alpha_\tau} \le \frac{\mu_F}{5 \mu_G} $, $\mu_F \alpha_\tau \le \mu_F \iota_1 \le 1$ and $\mu_G \beta_\tau \le \mu_G \iota_2 \le 1$,
the growth condition implies
\begin{align}
     & \quad \ \left( \frac{ \beta_{\tau-1} }{ \beta_\tau  } \right)^{4 \Hsmooth}   \left( \frac{\kappa_{\tau-1}}{\kappa_\tau} \right)^5 \left( \frac{\zeta_{\tau-1}}{\zeta_\tau} \right)^\frac{2}{\Fsmooth} \left(1 - \frac{\mu_F \alpha_\tau}{3} \right) \notag \\
    & \le \left( \frac{\beta_{\tau-1} }{ \beta_\tau } \right)^{5+4\Hsmooth}
    \left( \frac{\alpha_{\tau-1} }{ \alpha_\tau } \right)^\frac{2}{\Fsmooth}
    \left( 1 - \frac{\mu_F \alpha_\tau}{3} \right) \notag \\
    & \le \left( 1 + \frac{\mu_G \beta_\tau }{64} \right)^{5+4\Hsmooth} \left( 1 + \frac{\Fsmooth \mu_F \alpha_\tau}{16} \right)^\frac{2}{\Fsmooth} \left( 1 - \frac{\mu_F \alpha_\tau}{3} \right) \notag \\
    & \le \left( 1 + \frac{5 \mu_G \beta_\tau }{32} \right) \exp\left( \frac{\mu_F \alpha_\tau}{8} \right) \left( 1 - \frac{\mu_F \alpha_\tau}{3} \right)  \notag \\
    & \le \left( 1 +\frac{\mu_F \alpha_\tau}{32} \right) \left( 1 + \frac{1.2 \mu_F \alpha_\tau}{8} \right) \left( 1 - \frac{\mu_F \alpha_\tau}{3} \right) \notag \\
    & = \left( 1 + \frac{2 \mu_F \alpha_\tau }{9} \right) \left( 1 - \frac{\mu_F \alpha_\tau}{3} \right)
    \le 1 - \frac{ \mu_F \alpha_\tau }{9}.
    \label{lem:step-size-ineq2:Eq4.1}
\end{align}
Then we have
\begin{align}
    & \quad \ \sum_{j=0}^t 
    \tilde{\alpha}_{j+1, t} \frac{\beta_j^{5+4\Hsmooth} }{ \alpha_j^4 } \left( \frac{\alpha_j}{\beta_j} \right)^\frac{2}{\Fsmooth} \notag \\
    & = \sum_{j=0}^t \alpha_j \beta_j^{4\Hsmooth} \kappa_j^5 \zeta_j^\frac{2}{\Fsmooth} \prod_{\tau=j+1}^t \left( 1 - \frac{\mu_F \alpha_\tau}{3} \right) \notag \\
    & = \beta_t^{4\Hsmooth} \kappa_t^5 \zeta_t^\frac{2}{\Fsmooth} \sum_{j=0}^t \alpha_j  \left( \frac{\beta_j}{\beta_t} \right)^{4\Hsmooth} \left( \frac{\kappa_j}{\kappa_t} \right)^5 \left(\frac{\zeta_j}{\zeta_t} \right)^\frac{2}{\Fsmooth} \prod_{\tau=j+1}^t \left( 1 - \frac{\mu_F \alpha_\tau}{3} \right) \notag \\
    & = \beta_t^{4\Hsmooth} \kappa_t^5 \zeta_t^\frac{2}{\Fsmooth} \sum_{j=0}^t \alpha_j \prod_{\tau=j+1}^t \left( \frac{\beta_{\tau-1}}{\beta_\tau} \right)^{4\Hsmooth}
    \left( \frac{\kappa_{\tau-1}}{\kappa_\tau} \right)^5 \left( \frac{\zeta_{\tau-1}}{\zeta_\tau} \right)^\frac{2}{\Fsmooth} \left( 1 - \frac{\mu_F \alpha_\tau}{3} \right) \notag \\
    & \overset{\eqref{lem:step-size-ineq2:Eq4.1}}{\le} \beta_t^{4\Hsmooth} \kappa_t^5 \zeta_t^\frac{2}{\Fsmooth} \sum_{j=0}^t \alpha_j \prod_{\tau=j+1}^t \left( 1 - \frac{\mu_F \alpha_\tau}{9} \right) \notag \\
    & = \frac{9 \beta_t^{4\Hsmooth} \kappa_t^5 \zeta_t^\frac{2}{\Fsmooth} } {\mu_F} \left[ 1 - \prod_{\tau=0}^t \left( 1 - \frac{\mu_F \alpha_\tau}{9} \right) \right]
    \le  \frac{9 \beta_t^{5+4\Hsmooth} }{\mu_F \alpha_t^5} \left( \frac{\alpha_t}{\beta_t} \right)^\frac{2}{\Fsmooth}. \notag
\end{align}
\end{itemize}

\item
\begin{itemize}
\item
The first claim in (iii) follows from (i) in Lemma~\ref{lem:help}.

\item The proof is the same as that of (iv) in Lemma~\ref{lem:step-size-ineq}.

\item For the last claim in (i),
we
define $\kappa_\tau = \frac{\beta_\tau}{\alpha_\tau}$. Then we have $\frac{\kappa_{\tau-1}}{\kappa_\tau} \le \frac{\beta_{\tau-1}}{\beta_\tau}$.
For $x \in (0, 0.1)$, one can check $(1+x)^{1+2 \Hsmooth } \le 1 + 4x$.
Since $\mu_G \beta_\tau \le \mu_G \iota_2 \le 1$ and $\frac{\beta_\tau}{\alpha_\tau} \le \frac{\mu_F}{\mu_G}$,
the growth condition implies
\begin{align}
     \left( \frac{\beta_{\tau-1}}{\beta_\tau} \right)^{2 \Hsmooth}  \frac{\kappa_{\tau-1}}{\kappa_\tau} \left(1 - \frac{\mu_G \beta_\tau}{2} \right)
    & \le 
    \left( \frac{\beta_{\tau-1} }{ \beta_\tau } \right)^{1+2 \Hsmooth}
    \left( 1 - \frac{\mu_G \beta_\tau}{2} \right) \notag\\
    &\le \left( 1 + \frac{\mu_G \beta_\tau }{64} \right)^{1+2 \Hsmooth} \left( 1 - \frac{\mu_G \beta_\tau}{2} \right) \notag \\
    & = \left( 1 + \frac{ \mu_G \beta_\tau }{16} \right) \left( 1 - \frac{\mu_G \beta_\tau}{2} \right) 
    \le 1 - \frac{ \mu_G \beta_\tau }{3}.
    \label{lem:step-size-ineq2:Eq5}
\end{align}
Then we have
\begin{align}
    \sum_{j=0}^t \tilde{\beta}_{j+1, t} \frac{\beta_j^{2+2\Hsmooth}}
    {\alpha_j}
    & = \sum_{j=0}^t \beta_j^{1+2 \Hsmooth} \kappa_j \prod_{\tau=j+1}^t \left( 1 - \frac{\mu_G \beta_\tau}{2} \right) \notag \\
    & = \beta_t^{2 \Hsmooth} \kappa_t \sum_{j=0}^t \beta_j \left( \frac{\beta_j}{\beta_t} \right)^{2\Hsmooth} \frac{\kappa_j}{\kappa_t} \prod_{\tau=j+1}^t \left( 1 - \frac{\mu_G \beta_\tau}{2} \right) \notag \\
    & = \beta_t^{2 \Hsmooth} \kappa_t \sum_{j=0}^t \beta_j \prod_{\tau=j+1}^t \left( \frac{\beta_{\tau-1} }{\beta_\tau} \right)^{2 \Hsmooth}  \frac{\kappa_{\tau-1}}{\kappa_\tau}  \left( 1 - \frac{\mu_G \beta_\tau}{2} \right) \notag \\
    & \overset{\eqref{lem:step-size-ineq2:Eq5}}{\le} \beta_t^{2\Hsmooth} \kappa_t \sum_{j=0}^t \beta_j \prod_{\tau=j+1}^t \left( 1 - \frac{\mu_G \beta_\tau}{3} \right) \notag \\
    & = \frac{3 \beta_t^{2\Hsmooth} \kappa_t}{\mu_G} \left[ 1 - \prod_{\tau=0}^t \left( 1 - \frac{\mu_G \beta_\tau}{3} \right) \right]
    \le  \frac{3 \beta_t^{1 + 2 \Hsmooth}}{\mu_G \alpha_t}. \notag
\end{align}
\end{itemize}

\item 
\begin{itemize}
\item For the first claim in (ii),
define $\zeta_\tau = \frac{\alpha_\tau}{\beta_\tau}$. Then we have $\frac{\zeta_{\tau-1}}{\zeta_\tau} \le \frac{\alpha_{\tau-1}}{\alpha_\tau}$.
For $x \in (0. 0.2)$, one can check $\exp(x) \le 1 + 1.2x$.
Since $\Fsmooth \le 1$ 
and $\mu_G \beta_\tau \le \mu_G \iota_2 \le 1$,
the growth condition implies
\begin{align}
     \frac{\alpha_{\tau-1}}{\alpha_\tau} \frac{ \beta_{\tau-1} }{ \beta_\tau  }  \left( \frac{\zeta_{\tau-1}}{\zeta_\tau} \right)^\frac{2}{\Fsmooth} \left(1 - \frac{\mu_G \beta_\tau}{2} \right)
    & \le \frac{\beta_{\tau-1}}{\beta_\tau} 
    \left( \frac{\alpha_{\tau-1} }{ \alpha_\tau } \right)^\frac{3}{\Fsmooth}
    \left( 1 - \frac{\mu_G \beta_\tau}{2} \right) \notag \\
    & \le \left( 1+\frac{\mu_G \beta_\tau}{64} \right) \left( 1 + \frac{ \Fsmooth \mu_G \beta_\tau }{16} \right)^\frac{3}{\Fsmooth} \left( 1 - \frac{\mu_G \beta_\tau}{2} \right) \notag \\
    & \le \left( 1 + \frac{\mu_G \beta_\tau}{64} \right) \exp \left( \frac{3 \mu_G \beta_\tau}{16} \right) \left( 1 - \frac{\mu_G \beta_\tau}{2} \right) \notag \\
    & \le \left( 1 + \frac{\mu_G \beta_\tau}{64} \right) \left( 1 + \frac{3.6 \mu_G \beta_\tau}{16}  \right) \left( 1 - \frac{\mu_G \beta_\tau}{2} \right) \notag \\
    & = \left( 1 + \frac{\mu_G \beta_\tau }{3} \right) \left( 1 - \frac{\mu_G \beta_\tau}{2} \right)
    \le 1 - \frac{ \mu_G \beta_\tau }{6}.
    \label{lem:step-size-ineq2:Eq6}
\end{align}
Then we have
\begin{align}
    \sum_{j=0}^t \tilde{\beta}_{j+1, t} \alpha_j \beta_j^2 \left( \frac{\alpha_j}{\beta_j} \right)^\frac{2}{\Fsmooth}
    & = \sum_{j=0}^t \alpha_j \beta_j^{2} \zeta_j^\frac{2}{\Fsmooth} \prod_{\tau=j+1}^t \left( 1 - \frac{\mu_G \beta_\tau}{2} \right) \notag \\
    & = \alpha_t \beta_t \zeta_t^\frac{2}{\Fsmooth} \sum_{j=0}^t \beta_j \frac{\alpha_j}{\alpha_t} \frac{\beta_j}{\beta_t} \left(\frac{\zeta_j}{\zeta_t} \right)^\frac{2}{\Fsmooth} \prod_{\tau=j+1}^t \left( 1 - \frac{\mu_G \beta_\tau}{2} \right) \notag \\
    & = \alpha_t \beta_t \zeta_t^\frac{2}{\Fsmooth} \sum_{j=0}^t \beta_j \prod_{\tau=j+1}^t \frac{\alpha_{\tau-1} }{\alpha_\tau} \frac{\beta_{\tau-1}}{\beta_\tau} \left( \frac{\zeta_{\tau-1}}{\zeta_\tau} \right)^\frac{2}{\Fsmooth} \left( 1 - \frac{\mu_G \beta_\tau}{2} \right) \notag \\
    & \overset{\eqref{lem:step-size-ineq2:Eq6}}{\le} \alpha_t \beta_t \zeta_t^\frac{2}{\Fsmooth} \sum_{j=0}^t \beta_j \prod_{\tau=j+1}^t \left( 1 - \frac{\mu_G \beta_\tau}{6} \right) \notag \\
    & = \frac{6 \alpha_t \beta_t \zeta_t^\frac{2}{\Fsmooth} }{\mu_G} \left[ 1 - \prod_{\tau=0}^t \left( 1 - \frac{\mu_G \beta_\tau}{6} \right) \right]
    \le  \frac{6 \alpha_t \beta_t}{\mu_G} \left( \frac{\alpha_t}{\beta_t} \right)^\frac{2}{\Fsmooth}. \notag
\end{align}

\item For the last claim in (iv), define 
$\kappa_\tau = \frac{\beta_\tau}{\alpha_\tau}$ and
$\zeta_\tau = \frac{\alpha_\tau}{\beta_\tau}$. Then we have $\frac{\kappa_{\tau-1} }{\kappa_\tau} \le \frac{ \beta_{\tau-1} }{\beta_\tau} $ and $\frac{\zeta_{\tau-1}}{\zeta_\tau} \le \frac{\alpha_{\tau-1}}{\alpha_\tau}$.
For $x \in (0. 0.2)$, one can check $\exp(x) \le 1+1.2x$ and for $x \in (0, 0.1)$, $(1+x)^{5+4\Hsmooth} \le 1 + 10x$.
Since 
$\mu_G \beta_\tau \le \mu_G \iota_2 \le 1$,
the growth condition implies
\begin{align}
     & \quad \ \left( \frac{ \beta_{\tau-1} }{ \beta_\tau  } \right)^{4 \Hsmooth}   \left( \frac{\kappa_{\tau-1}}{\kappa_\tau} \right)^5 \left( \frac{\zeta_{\tau-1}}{\zeta_\tau} \right)^\frac{2}{\Fsmooth} \left(1 - \frac{\mu_G \beta_\tau}{2} \right) \notag \\
    & \le \left( \frac{\beta_{\tau-1} }{ \beta_\tau } \right)^{5+4\Hsmooth}
    \left( \frac{\alpha_{\tau-1} }{ \alpha_\tau } \right)^\frac{2}{\Fsmooth}
    \left( 1 - \frac{\mu_G \beta_\tau}{2} \right) \notag \\
    & \le \left( 1 + \frac{\mu_G \beta_\tau }{64} \right)^{5+4\Hsmooth} \left( 1 + \frac{\Fsmooth \mu_G \beta_\tau}{16} \right)^\frac{2}{\Fsmooth} \left( 1 - \frac{\mu_G \beta_\tau}{2} \right) \notag \\
    & \le \left( 1 + \frac{5 \mu_G \beta_\tau }{32} \right) \exp\left( \frac{\mu_G \beta_\tau}{8} \right) \left( 1 - \frac{\mu_G \beta_\tau}{2} \right)  \notag \\
    & \le \left( 1 +\frac{5 \mu_G \beta_\tau}{32} \right) \left( 1 + \frac{1.2 \mu_G \beta_\tau}{8} \right) \left( 1 - \frac{\mu_G \beta_\tau}{2} \right) \notag \\
    & = \left( 1 + \frac{ \mu_G \beta_\tau }{3} \right) \left( 1 - \frac{\mu_G \beta_\tau}{2} \right)
    \le 1 - \frac{ \mu_G \beta_\tau }{6}.
    \label{lem:step-size-ineq2:Eq7}
\end{align}
Then we have
\begin{align}
    & \quad \ \sum_{j=0}^t \tilde{\beta}_{j+1, t} \frac{\beta_j^{6+4\Hsmooth} }{ \alpha_j^5 } \left( \frac{\alpha_j}{\beta_j} \right)^\frac{2}{\Fsmooth} \notag \\
    & = \sum_{j=0}^t \beta_j^{1+4\Hsmooth} \kappa_j^5 \zeta_j^\frac{2}{\Fsmooth} \prod_{\tau=j+1}^t \left( 1 - \frac{\mu_G \beta_\tau}{2} \right) \notag \\
    & = \beta_t^{4\Hsmooth} \kappa_t^5 \zeta_t^\frac{2}{\Fsmooth} \sum_{j=0}^t \beta_j  \left( \frac{\beta_j}{\beta_t} \right)^{4\Hsmooth} \left( \frac{\kappa_j}{\kappa_t} \right)^5 \left(\frac{\zeta_j}{\zeta_t} \right)^\frac{2}{\Fsmooth} \prod_{\tau=j+1}^t \left( 1 - \frac{\mu_G \beta_\tau}{2} \right) \notag \\
    & = \beta_t^{4\Hsmooth} \kappa_t^5 \zeta_t^\frac{2}{\Fsmooth} \sum_{j=0}^t \beta_j \prod_{\tau=j+1}^t \left( \frac{\beta_{\tau-1}}{\beta_\tau} \right)^{4\Hsmooth}
    \left( \frac{\kappa_{\tau-1}}{\kappa_\tau} \right)^5 \left( \frac{\zeta_{\tau-1}}{\zeta_\tau} \right)^\frac{2}{\Fsmooth} \left( 1 - \frac{\mu_G \beta_\tau}{2} \right) \notag \\
    & \overset{\eqref{lem:step-size-ineq2:Eq7}}{\le} \beta_t^{4\Hsmooth} \kappa_t^5 \zeta_t^\frac{2}{\Fsmooth} \sum_{j=0}^t \beta_j \prod_{\tau=j+1}^t \left( 1 - \frac{\mu_G \beta_\tau}{6} \right) \notag \\
    & = \frac{6 \beta_t^{4\Hsmooth} \kappa_t^5 \zeta_t^\frac{2}{\Fsmooth} } {\mu_G} \left[ 1 - \prod_{\tau=0}^t \left( 1 - \frac{\mu_G \beta_\tau}{6} \right) \right]
    \le  \frac{6 \beta_t^{5+4\Hsmooth} }{\mu_G \alpha_t^5} \left( \frac{\alpha_t}{\beta_t} \right)^\frac{2}{\Fsmooth}. \notag
\end{align}
\end{itemize}

\item The proof of (v) is nearly the same as that of (iv), except that we replace $\Fsmooth$ by $2\Gsmooth$.

\end{enumerate}

\end{proof}

\subsection{Analysis of Constants in Leading Terms}
\label{proof:de:const}

In this subsection, we provide the details for Remark~\ref{rema:const}.

\subsubsection{Derivation for (\ref{eq:const-leading-sum})}
we first analyze the constants $\CM_{x}$, $\CM_{xy,1}$ and $\CM_{y,1}$ appearing in Theorem~\ref{thm:decouple-short}, with their detailed expression in 
\eqref{eq:thm:decouple-constants}.
As mentioned in Remark~\ref{rema:const}, we focus on the diminishing step sizes $\alpha_t =\Theta(t^{-a})$ and $\beta_t = \Theta(t^{-b})$ with $0<a<b \le 1$ and $\frac{b}{a} < 1 + \frac{\Fsmooth}{2} \wedge \Fsmooth$ to capture the most essential dependence on the parameters.
Moreover,
as analyzed in Remark~\ref{rema:step-size}, we can focus on $t \ge t_0$ with a prescribed $t_0$, then the constants $\iota_1, \iota_2, \kappa, \rho$ in Assumption~\ref{assump:stepsize-new} are of the order $o(1)$ as $t_0 \to \infty$.
Then, in the expression of the $\CM_{x}$, $\CM_{xy,1}$ and $\CM_{y,1}$, all terms involving these constants can be viewed as higher-order infinitesimal (as $t_0 \to \infty$), as shown below
\begin{align*}
		\CM_{x} = \ccde_{x,1} + o(1),\quad 	\CM_{xy,1} = \ccde_{xy,1} + o(1),\quad
		\CM_{y,1} = \ccde_{y,1} + o(1).
\end{align*}
With $\ccde_{x,1} $ defined in \eqref{eq:constanx-ccde-012}, we have
\begin{equation}
\label{eq:constantx-leading}
    \CM_{x} = \frac{8 \Gamma_{11}}{\mu_F}  + o(1).
\end{equation}

Next, we analyze $\CM_{xy,1} = \ccde_{xy,1} + o(1)$.
Although $\ccde_{xy,1} = \frac{6}{\mu_F}(\cde_{xy,10} c_+ + \cde_{xy,11})  $ as defined in \eqref{eq:constantxy-ccde-01234},
we can improve the constant $c_+$ to $o(1)$ by applying the results in Theorem~\ref{thm:decouple-short}.
Note that with $1 < \frac{b}{a} < 1 + \frac{\Fsmooth}{2} \wedge \Fsmooth$, we have $\EB \| \yhat_{t} \|^2 = o(\alpha_t)$ from the proof of Corollary~\ref{cor:decouple-rates}.
If we plug this upper bound into \eqref{lem:decouple:Eq1} instead of plugging \eqref{lem:decouple:Eq2} into \eqref{lem:decouple:Eq1} as we did in the proof of Theorem~\ref{thm:decouple}, we can obtain
$$\ccde_{xy,1} = \frac{6}{\mu_F}\left(\cde_{xy,10} \cdot o(1) + \cde_{xy,11} \right)  = \frac{6\cde_{xy,11} }{\mu_F} + o(1). $$
Then with $\cde_{xy,11}$ defined in \eqref{eq:constantxy-de-101112} and $\cde_{xy,6}$ defined in \eqref{eq:constantxy-de-6789},
we have
\begin{equation}
\label{eq:constantxy-leading}
    \CM_{xy,1} 
    = \ccde_{xy,1} + o(1) = \frac{6\cde_{xy,11} }{\mu_F} + o(1)
    = \frac{6}{\mu_F} \left( \frac{12 \LGx \Gamma_{11} }{\mu_F} + \Sigma_{12} \right) + o(1).
\end{equation}

Finally, we analyze $\CM_{y,1} = \ccde_{y,1} + o(1)$.
Although 
$\ccde_{y,1} 
=  \frac{4}{\mu_G} \left( \cde_{y,3} c_+ + 2 d_y \LGx \ccde_{xy,1} \zeta + \Gamma_{22} \right)$
as defined in \eqref{eq:constanty-ccde-0123456},
we can also improve the constant $c_+$.
To that end, the expression of $\CM_{x}$ in \eqref{eq:constantx-leading} implies  $\EB \| \xhat_{t} \|^2 = \frac{8 \Gamma_{11}}{\mu_F} \alpha_{t-1} + o(\alpha_{t-1}) = \frac{8 \Gamma_{11}}{\mu_F} \alpha_{t} + o(\alpha_{t})$, where we have used $1\le \frac{\alpha_{t-1}}{\alpha_t} \le 1 + \frac{\mu_F}{16} \alpha_t = 1+o(1)$.
If we plug \eqref{lem:decouple:Eq3.1} and
this upper bound into \eqref{lem:decouple:Eq-y} when we deriving \eqref{lem:decouple:Eq5} in the proof of Theorem~\ref{thm:decouple},
we can obtain
\begin{align*}
    \ccde_{y,1} 
=  \frac{4}{\mu_G} \left[ \cde_{y,3} \Big( \frac{8 \Gamma_{11} }{\mu_F} + o(1) \Big) + 2 d_y \LGx \ccde_{xy,1} \zeta + \Gamma_{22} \right]
\end{align*}
Moreover, with $\zeta = 1 + \frac{\mu_F \iota_1}{16} + \frac{\mu_G \iota_2}{64} = 1 + o(1) $ defined below \eqref{lem:decouple:Eq3.1},  $\cde_{y,3} = \frac{4 L_H \LGx}{\mu_F} + 2 \LGx^2 \iota_1 = \frac{4 L_H \LGx}{\mu_F} + o(1) $ defined above \eqref{lem:decouple:Eq5}, and the expression of $\ccde_{xy,1}$ in \eqref{eq:constantxy-leading},
we have
\begin{align*}
    \CM_{y,1} = \ccde_{y,1} + o(1)
    = \frac{4}{\mu_G} \left[ \frac{32 L_H \LGx \Gamma_{11} }{\mu_F^2}  + \frac{12 d_y \LGx}{\mu_F} \left( \frac{12 \LGx \Gamma_{11} }{\mu_F} + \Sigma_{12} \right) + \Gamma_{22} \right] + o(1).
\end{align*}

\subsubsection{Derivation for (\ref{eq:tr-cov-x-y}) and (\ref{eq:norm-cov-xy})}

We first give the explicit expressions for $\Sigma_x$, $\Sigma_y$, and $\Sigma_{x,y}$.
\citet{mokkadem2006convergence}
assume that $\alpha_n = \Theta(n^{-a})$, $\beta_n = \Theta(n^{-b})$ with $1/2 < a < b \le 1$ and
\begin{equation*}
    \EB \left[ \begin{pmatrix}
        \xi_{t} \xi_{t}^\top & \xi_{t} \psi_{t}^\top \\
        \psi_{t} \xi_{t}^\top & \psi_{t} \psi_{t}^\top
    \end{pmatrix}  \Bigg|\, \FM_{t} \right]
    \overset{a.s.}{\to} 
    \begin{pmatrix}
        \Sigma_{\xi} & \Sigma_{\xi, \psi} \\
        \Sigma_{\xi, \psi}^\top & \Sigma_{\psi}
    \end{pmatrix}.
\end{equation*}
For brevity, we omit other regular conditions.
Then the asymptotic covariance matrices $\Sigma_x$ and $\Sigma_y$ in \eqref{eq:nonlinear-clt} have the following expressions
\begin{align}
\Sigma_x &= \int_0^\infty \exp(-B_1 s) \Sigma_{\xi} \exp(- B_1^\top s)  ds \label{eq:cov-x}
\\
\Sigma_y &= \int_0^\infty \exp \left( - \Big(B_3 - \frac{\invdiffslow I}{2} \Big) s \right) \widetilde{\Sigma}_{\psi} \exp \left( - \Big(B_3^\top - \frac{\invdiffslow I}{2} \Big) s \right) ds.\label{eq:cov-y}
\end{align} 
Here $\invdiffslow = \lim_{n \to \infty} \beta_{n+1}^{-1} - \beta_{n}^{-1}$ and $\invdiffslow > 0$ only when $b=1$.
$ \widetilde{\Sigma}_{\psi}$ is the asymptotic covariance of the modified noise  $ \breve{\psi}_t = \psi_t - B_2 B_1^{-1} \xi_t$ mentioned in Remark~\ref{rema:const} and has the following expression
\begin{equation}\label{eq:cov-psi-tilde}
        \widetilde{\Sigma}_{\psi} := \Sigma_{\psi} - B_2 B_1^{-1} \Sigma_{\xi,\psi} - \Sigma_{\xi, \psi}^\top B_1^{-\top} B_2^\top + B_2 B_1^{-1} \Sigma_{\xi} B_1^{-\top} B_2^\top.
    \end{equation}
Since $x^\star = H(y^\star)$, $\beta_t = o(\alpha_t)$, and $H$ is $L_H$-Lipschitz continuous, we have $\|H(y_t) - H(y^\star) \| \le L_H \| y_t - y^\star \| = o_p(\alpha_t^{1/2})$ and consequently
$\alpha_t^{-1/2} \xhat_t = \alpha_t^{-1/2} (x_t - x^\star) - \alpha_t^{-1/2} (H(y_t) - H(y^\star)) = \alpha_t^{-1/2} (x_t - x^\star) = o_p(1)$.
Then $\alpha_t^{-1/2} \xhat_t $ has the same asymptotic distribution as $\alpha_t^{-1/2}(x_t - x^\star)$ in \eqref{eq:nonlinear-clt}.
If we further assume that $\{ \alpha_t^{-1} \|\xhat_t \|^2 \}_{t=1}^\infty$ and $\{ \beta_t^{-1} \| \yhat_t \|^2\}_{t=1}^\infty$ are asymptotically uniformly integrable, then we can obtain $\lim_{t\to\infty} \alpha_t^{-1} \EB \| x_t \|^2 = \tr(\Sigma_x)$ and $\lim_{t\to\infty} \beta_t^{-1} \EB \| y_t \|^2 = \tr(\Sigma_y)$~\citep[Theorem~2.20]{van2000asymptotic}.
For $\Sigma_{x,y} = \lim_{t\to\infty} \beta_t^{-1} \EB [x_t y_t^\top]$, 
\citet[Theorem~2.6]{konda2004convergence} show that for the linear case, $\Sigma_{x,y}$ satisfies the following equation 
\begin{equation}\label{eq:cov-x-y}
    B_1 \Sigma_{x,y} + \Sigma_x B_2^\top = \Sigma_{\xi,\psi}
\end{equation}
Under the aforementioned uniform integrability condition, we have
$\lim_{t\to\infty} \beta_t^{-1} \|\EB x_t y_t^\top\| = \| \Sigma_{x,y} \|$.
Next, we derive the upper bounds for $\tr(\Sigma_x)$, $\tr(\Sigma_y)$ and $\|\Sigma_{x,y} \|$.

\textbf{The upper bound for $\tr(\Sigma_x)$}.
For $\Sigma_x$ define in \eqref{eq:cov-x}, we first derive an exponential upper bound for $\|e^{-B_1s} \|$ and $\|e^{-B_1^{\top}s} \| $.
By Proposition~\ref{prop:ensure-linearity}, we have $\frac{B_1 + B_1^{\top}}{2} \succeq \mu_F I$.
For any \(v\in\mathbb{R}^d\), define \(w(s) := e^{-B_1^{\top}s} v\). Then we have
\begin{align*}
\frac{d}{ds}\|w(s)\|^2
= \frac{d}{ds}\, v^{\top} e^{-B_1 s} e^{-B_1^{\top} s} v 
= - v^{\top} e^{-B_1 s}(B_1 + B_1^{\top}) e^{-B_1^{\top} s} v 
\le -2\mu_F \|w(s)\|^2.
\end{align*}
By Gr\"onwall's inequality, we obtain $\|w(s)\|^2 \le \|v\|^2 e^{-2\mu_Fs}$.
Since $v$ is arbitrary, this implies $\|e^{-B_1^{\top}s} \| \le e^{-\mu_F s}$, $\forall s\ge 0$.
Similarly, we have $\|e^{-B_1 s}\| \le e^{-\mu_F s}$, $\forall s\ge 0$.
Next, we bound the trace of the integrand. Using the fact that
\(\tr(AB) \le \|A\|\,\tr(B)\) when \(B\succeq 0\), we obtain
\begin{align*}
\tr\!\big(e^{-B_1 s}\Sigma_{\xi}e^{-B_1^{\top} s}\big)
&= \tr\!\big( e^{-B_1^{\top} s} e^{-B_1 s}\Sigma_{\xi}\big)
\le \big\|e^{-B_1^{\top} s} e^{-B_1 s}\big\| \,\tr(\Sigma_{\xi}) \\
&\le \|e^{-B_1^{\top} s}\|\,\|e^{-B_1 s}\|\,\tr(\Sigma_{\xi})
\le e^{-2\mu_F s}\,\tr(\Sigma_{\xi}).
\end{align*}
By Assumption~\ref{assump:noise-s}, we have $\tr(\Sigma_{\xi}) \le \Gamma_{11}$.
Then integrating over \(s\in[0,\infty)\) yields
\[
\tr(\Sigma_x)
= \int_{0}^{\infty} \tr\!\big(e^{-B_1 s}\Sigma_{\xi}e^{-B_1^{\top} s}\big)\,ds
\le \tr(\Sigma_{\xi})\int_{0}^{\infty} e^{-2\mu_F s}\,ds
= \frac{\tr(\Sigma_{\xi})}{2\mu_F}
\le \frac{\Gamma_{11}}{2\mu_F}.
\]

\textbf{The upper bound for $\tr(\Sigma_y)$}.
Recall that $\Sigma_y$ is defined in \eqref{eq:cov-y}.
Under Assumption~\ref{assump:stepsize-new}, we have $\invdiffslow < \mu_G/2$ and thus $\frac{B_3 + B_3^\top - \invdiffslow I}{2} \succeq \frac{\mu_G}{2} I$.
Similar to the former derivation, we can obtain $\tr(\Sigma_y)
\le \frac{\tr(\widetilde{\Sigma}_{\psi})}{\mu_G}$.
It remains to bound $\tr(\widetilde{\Sigma}_{\psi})$.
Recall that $\widetilde{\Sigma}_{\psi}$ is defined in \eqref{eq:cov-psi-tilde}. 
Letting $A:=B_2B_1^{-1}$, then we have
$\widetilde{\Sigma}_{\psi}
=\Sigma_{\psi} - A\Sigma_{\xi,\psi}-\Sigma_{\xi,\psi}^{\top}A^{\top}
+ A\Sigma_{\xi}A^{\top}$.
By Assumption~\ref{assump:noise-s}, we have $\tr(\Sigma_{\xi}) \le \Gamma_{11}$,
$\tr(\Sigma_{\psi}) \le \Gamma_{22}$ and $\|\Sigma_{\xi,\psi}\| \le \Sigma_{12}$.
Then we have
\begin{align*}
\tr(\widetilde{\Sigma}_{\psi})
&= \tr(\Sigma_{\psi})
- \tr(A\Sigma_{\xi,\psi})
- \tr(\Sigma_{\xi,\psi}^{\top}A^{\top})
+ \tr(A\Sigma_{\xi}A^{\top}) \\
&= \tr(\Sigma_{\psi})
-2\,\tr(A\Sigma_{\xi,\psi})
+ \tr(A\Sigma_{\xi}A^{\top}) 
\le \Gamma_{22}
+2\big|\tr(A\Sigma_{\xi,\psi})\big|
+\tr(A\Sigma_{\xi}A^{\top}).
\end{align*}
For the last term, since $\Sigma_{\xi}\succeq 0$, we have
\[
\tr(A\Sigma_{\xi}A^{\top})
=\tr(A^{\top}A\,\Sigma_{\xi})
\le \|A^{\top}A\|\,\tr(\Sigma_{\xi})
=\|A\|^{2}\,\tr(\Sigma_{\xi})
\le \|A\|^{2}\Gamma_{11}.
\]
For the second term, using $|\tr(X)|\le d_y \|X\|$ for $X\in\mathbb{R}^{d_y\times d_y}$, we obtain
\[
\big|\tr(A\Sigma_{\xi,\psi})\big|
\le d_y\,\|A\Sigma_{\xi,\psi}\|
\le d_y\,\|A\|\,\|\Sigma_{\xi,\psi}\|
\le d_y\,\|A\|\,\Sigma_{12}.
\]
Combining the above yields
\begin{equation*}
\tr(\widetilde{\Sigma}_{\psi})
\le \Gamma_{22}
+2d_y\,\|A\|\,\Sigma_{12}
+\|A\|^{2}\Gamma_{11}.
\end{equation*}
Finally, we bound $\|A\| = \|B_2 B_1^{-1} \|$.
By Proposition~\ref{prop:ensure-linearity}, we have $\|B_2\| \le \LGx$, 
$\frac{B_1+B_1^{\top}}{2}\succeq \mu_F I$, and
$\frac{B_3+B_3^{\top}}{2}\succeq \mu_G I$.
For any $x\neq 0$,
$\mu_F\|x\|^2 \le x^{\top}\frac{B_1+B_1^{\top}}{2}x
= x^{\top}B_1 x
\le \|B_1 x\|\,\|x\|$.
It follows that $\|B_1^{-1} B_1 x \| = \| x\|\le \frac{1}{\mu_F} \| B_1 x \|$.
Because $B_1$ is invertible and $x$ is arbitrary, we have $\|B_1^{-1}\| \le \frac{1}{\mu_F}$.
Therefore,
$\|A\|=\|B_2B_1^{-1}\| \le \|B_2\|\,\|B_1^{-1}\|
\le \frac{\LGx}{\mu_F}$.
Combining the above analysis yields
\[
\tr(\Sigma_y)
\le
\frac{1}{\mu_G}
\left[
\Gamma_{22}
+2d_y\,\frac{\LGx}{\mu_F}\,\Sigma_{12}
+\Big(\frac{\LGx}{\mu_F}\Big)^{2}\Gamma_{11}
\right].
\]

\textbf{The upper bound for $\|\Sigma_{x,y}\|$}.
Because $\Sigma_{x,y}$ satisfies \eqref{eq:cov-x-y},
we have $\Sigma_{x,y} = B_1^{-1} (\Sigma_{\xi,\psi} - \Sigma_x B_2^\top) $.
We have established $\|B_1^{-1}\| \le \frac{1}{\mu_F}$ and $\tr(\Sigma_x) \le \frac{\Gamma_{11}}{2\mu_F}$.
By Assumption~\ref{assump:noise-s} and Proposition~\ref{prop:ensure-linearity},
$\| \Sigma_{\xi,\psi} \| \le \Sigma_{12}$ and $\|B_2 \| \le \LGx$.
With $\|\Sigma_x \| \le \tr(\Sigma_x)$, we have
\begin{align*}
    \|\Sigma_{x,y} \|
    \le \frac{1}{\mu_F} \left( \| \Sigma_{\xi,\psi} \| + \|\Sigma_{x}\| \|B_2^\top \| \right)
    \le \frac{1}{\mu_F} \left( \Sigma_{12} + \frac{\LGx \Gamma_{11}}{2 \mu_F} \right).
\end{align*}

\section{Proof for the Lower Bound}
\label{append:proof:lower}

In this section, we present the proof of Proposition~\ref{prop:lower}.
This proof also relies on the convergence rates for the MSE and fourth-order moments without local linearity in Theorem~\ref{thm:first} and Lemma~\ref{lem:x+y_quartic}.
Under the conditions in Proposition~\ref{prop:lower}, one can check that the conditions of Theorem~\ref{thm:first} and Lemma~\ref{lem:x+y_quartic}, especially Assumptions~\ref{assump:stepsize-new-weak} and \ref{assump:stepsize}, are satisfied.
Moreover, Proposition~\ref{prop:lower} together with Theorem~\ref{thm:first} implies both $\EB |\xhat_t|^2$ and $\EB |\yhat_t|^2$ are of the order $\Theta(\alpha_t)$.

\begin{proof}[Proof of Proposition~\ref{prop:lower}]
Under the conditions of Proposition~\ref{prop:lower},
the update rule of two-time-scale SA becomes
\begin{align*}
		x_{t+1} &= x_{t} - \alpha_{t}\left(x_{t}-y_{t} + \xi_{t}\right),  \\   
		y_{t+1} &= y_{t} - \beta_{t}\left(y_t - |x_t - y_t|\sign(y_t)\right). 
\end{align*}
Note that for this example, $x^\star = y^\star = 0 \in \RB$, $\xhat_{t} = x_t - y_t$ and $\yhat_t = y_t$.
Correspondingly, the update for the errors term becomes
\begin{align}
    \xhat_{t+1} & = (1-\alpha_t) \xhat_t - \alpha_t \xi_t + y_t - y_{t+1} \nonumber \\
    & = (1-\alpha_t) \xhat_t - \alpha_t \xi_t + \beta_t \yhat_t - \beta_t |\xhat_t| \sign(\yhat_t), \label{eq:counter-x} \\
    \yhat_{t+1} & = (1-\beta_t) \yhat_t + \beta_t |\xhat_t| \sign(\yhat_t). \label{eq:counter-y}
\end{align}
Because this example satisfies Assumptions~\ref{assump:smooth:FH}  -- \ref{assump:smoothH} and \ref{assump:noise-s} with $L_H=L_F=\LGx = \LGy = \mu_F = \mu_G=1$, $S_H = 0$ and $\Hsmooth = 1$,
then by Theorem~\ref{thm:first} and Lemma~\ref{lem:x+y_quartic}, we have $\EB|\xhat_t|^2 + \EB|\yhat_t|^2 = \OM(\alpha_t)$ and $\EB |\xhat_t |^4 = \OM(\alpha_t^2)$.
By Assumption~\ref{assump:stepsize-new-weak}, we have $\alpha_t \le 1/12$ and $\beta_t \le 1/14$.
We also have $\beta_t / \alpha_t \le 1/ 200$.

The remaining proof proceeds in three steps. First, we show that $\EB | \xhat_t |^2 = \Omega(\alpha_t)$. Second, we establish that $\EB |\xhat_t| = \Omega(\sqrt{\alpha_t})$. Finally, we prove that $\EB |\yhat_t| = \Omega(\sqrt{\beta_t})$, and consequently $\EB | \yhat_t|^2 \ge (\EB |\yhat_t|)^2 = \Omega(\beta_t)$.
\vspace{0.2cm}

\noindent
\textbf{Step~1: Prove $\EB | \xhat_t |^2 = \Omega(\alpha_t)$.}
Squaring both sides of \eqref{eq:counter-x} and taking the expectation yields
\begin{align*}
    \EB | \xhat_{t+1} |^2
    & = (1-\alpha_t)^2 \EB |\xhat_t|^2 + \alpha_t^2 \Sigma_t + \beta_t^2 \EB |\yhat_t|^2 + \beta_t^2 \EB |\xhat_t|^2 + 2 \beta_t (1-\alpha_t) \EB \xhat_t \yhat_t \\
    & \quad \ -2 \beta_t (1-\alpha_t)\EB \xhat_t |\xhat_t| \sign(\yhat_t) - 2 \beta_t^2 \EB |\xhat_t \yhat_t| \\
    & \ge (1-2\alpha_t-2\beta_t) \EB |\xhat_t|^2 + \alpha_t^2 \Sigma_1 - 2\beta_t(1+\beta_t) \EB |\xhat_t \yhat_t|.
\end{align*}
By the AM-GM inequality, we have
    $2\beta_t(1+\beta_t) \EB |\xhat_t \yhat_t|
    \le \alpha_t \EB|\xhat_t|^2 + \frac{\beta_t^2(1+\beta_t)^2}{\alpha_t} \EB |\yhat_t|^2$.
Recall that $\beta_t \le 1/14$ and $\beta_t / \alpha_t \le 1/200$.
It follows that
\begin{align*}
    \EB | \xhat_{t+1} |^2
    \ge (1-4 \alpha_t) \EB |\xhat_t|^2 + \alpha_t^2 \Sigma_1 - \frac{2 \beta_t^2}{\alpha_t} \EB |\yhat_t|^2.
\end{align*}
Recall that $\EB |\yhat_t|^2 = \OM(\alpha_t)$ and $\beta_t / \alpha_t \to 0$.
Then there exists $t_0$ such that $\forall t \ge t_0$, we have
\begin{align*}
    \EB | \xhat_{t+1} |^2
    & \ge (1-4 \alpha_t) \EB |\xhat_t|^2 + \frac{\alpha_t^2 \Sigma_1}{2}  \\
    & \ge \EB |\xhat_{t_0}|^2 \prod_{i=t_0}^t (1-4\alpha_i) + \frac{\Sigma_1}{2} \sum_{i=t_0}^t \alpha_i^2 \prod_{j=i+1}^t (1-4\alpha_j). 
\end{align*}
Assumption~\ref{assump:stepsize-new-weak}
implies $\alpha_i$ is non-increasing and
$\alpha_t^{-1} \le \alpha_{t-1}^{-1} + \mu_F / 16$. Then we can obtain $\alpha_t^{-1} = \OM(t)$ and consequently $\alpha_t = \Omega(t^{-1})$.
By telescoping, we have
\begin{align*}
    \sum_{i=t_0}^t \alpha_i^2 \prod_{j=i+1}^t (1-4\alpha_j) 
    & \ge \alpha_t \sum_{i=t_0}^t \alpha_i \prod_{j=i+1}^t (1-4\alpha_j) 
    = \frac{\alpha_t}{4} \left[ 1 - \prod_{j=t_0}^t (1-4\alpha_j) \right] \\
    & \ge \frac{\alpha_t}{4} \left[1 - \exp\Big(-4\sum_{j=t_0}^t \alpha_j\Big)\right].
\end{align*}
Since $\alpha_t = \Omega(t^{-1})$, then there exists $t_1$ such that $\forall t \ge t_1$, $\sum_{i=t_0}^t \alpha_i^2 \prod_{j=i+1}^t (1-4\alpha_j) \ge \alpha_t / 8$.
Thus, $ \EB |\xhat_{t}|^2 = \Omega(\alpha_t)$.
\vspace{0.2cm}

\noindent
\textbf{Step~2: Prove $\EB |\xhat_t| = \Omega(\sqrt{\alpha_t})$.}
Combining the results of Theorem~\ref{thm:first} and Step~1 yields $\EB |\xhat_t|^2 = \Theta(\alpha_t)$.
Moreover, combining the result of Lemma~\ref{lem:x+y_quartic} and $\EB |\xhat_t|^4 \ge (\EB |\xhat_t|^2)^2$ also yields $\EB |\xhat_t|^4 = \Theta(\alpha_t^2)$.
Thus, there exists $\gamma_1> 0$ such that $\forall t, (\EB | \xhat_t |^2)^2 / \EB |\xhat_t|^4 \ge \gamma_1$.
To apply this result to give a lower bound for $\EB|\xhat_t|$, we need the following inequality.
\begin{prop}[Paley–Zygmund inequality]
 If $Z\ge 0$ is a random variable with finite variance and $\theta \in [0,1]$, then $\PB(Z > \theta\, \EB Z) \ge(1-\theta)^2 (\EB Z)^2 / \EB Z^2$.
\end{prop}
Let $Z = |\xhat_t|^2$ and $\theta \in (0,1)$.
Then $$\PB\left( |\xhat_t| \ge \sqrt{\theta\, \EB |\xhat_t|^2} \right) = \PB\left( |\xhat_t|^2 \ge \theta\, \EB |\xhat_t|^2 \right)
\ge (1-\theta)^2 \gamma_1.$$
Thus, $$\EB |\xhat_t| \ge \sqrt{\theta\, \EB |\xhat_t|^2}\,  \PB\left( |\xhat_t| \ge \sqrt{\theta\, \EB |\xhat_t|^2} \right)
\ge (1-\theta)^2 \sqrt{\theta}\, \gamma_1 \sqrt{\EB |\xhat_t|^2} = \Omega(\sqrt{\alpha_t}).$$
\vspace{0.2cm}

\noindent
\textbf{Step~3: Prove $\EB |\yhat_t| = \Omega(\sqrt{\beta_t})$.}
Since $\EB |\xhat_t| = \Omega(\sqrt{\alpha_t})$, there exists $\gamma_2 > 0$ such that $\forall t,\ \EB |\xhat_t| \ge \gamma_2 \sqrt{\alpha_t}$.
Now we apply this result to derive the lower bound for $\EB |\yhat_t|$.
First, we note that the two terms on the right-hand side of \eqref{eq:counter-y} have the same sign, leading to $\sign(\yhat_{t+1}) = \sign(\yhat_t)$.
Because $y_0 \neq y^\star$, we have $\sign(\yhat_0) \neq 0$.
Using $\beta_t \le 1/14$, which follows from Assumption~\ref{assump:stepsize-new-weak}, we obtain that $\sign(\yhat_t) \neq 0$ for all $t \ge 0$.
Thus, we have
\begin{align*}
    \EB|\yhat_{t+1}|
    & =\EB \yhat_{t+1} \sign(\yhat_t)
    = (1-  \beta_t) \EB |\yhat_t| + \beta_t \EB|\xhat_t|
    \ge (1-  \beta_t) \EB |\yhat_t| + \gamma_2 \beta_t \sqrt{\alpha_t} \\
    & \ge \EB |\yhat_0| \prod_{i=0}^t(1-\beta_i) +\gamma_2 \sum_{i=0}^t \beta_i \sqrt{\alpha_i} \prod_{j=i+1}^t(1-\beta_j).
\end{align*}
Similar to the analysis in Step~1, for the second term, we have
\begin{align*}
   \sum_{i=0}^t \beta_i \sqrt{\alpha_i} \prod_{j=i+1}^t(1-\beta_i)
   \ge \sqrt{\alpha_t} \sum_{i=0}^t \beta_i \prod_{j=i+1}^t(1-\beta_i)
   \ge \sqrt{\alpha_t}  \left[1 - \exp\Big(-\sum_{j=0}^t \beta_j\Big)\right]
\end{align*}
Since $\beta_t = \Omega(t^{-1})$, then there exists $t_2$ such that $\forall t \ge t_2$, $\sum_{i=0}^t \beta_i \sqrt{\alpha_i} \prod_{j=i+1}^t(1-\beta_i) \ge \sqrt{\alpha_t} / 2$.
Thus, $ \EB |\yhat_{t}| = \Omega(\sqrt{\alpha_t})$ and 
$\EB | \yhat_{t} |^2 \ge \EB (|\yhat_{t}|)^2 = \Omega(\alpha_t)$.

\end{proof}

\end{document}